\documentclass{article}
\title{Finite length for unramified $\GL_2$: beyond multiplicity one, non-semisimple case}
\author{Lucrezia Bertoletti}
\date{}

\usepackage[a4paper]{geometry}

\usepackage{amsmath, bbm, amssymb} 
\usepackage{mathrsfs} 
\usepackage{comment} 
\usepackage{tikz-cd} 
\usepackage{mathtools}

\usepackage{amsthm, thmtools}
\usepackage{hyperref, cleveref}
\usepackage[shortlabels]{enumitem}
\usepackage[
backend=biber,
style=numeric,
sorting=nyt,
minnames=5,
maxnames=5
]{biblatex}
\usepackage[all]{xy}

\theoremstyle{plain}
\newtheorem{theorem}{Theorem}[subsection]
\newtheorem{lemma}[theorem]{Lemma}
\newtheorem{proposition}[theorem]{Proposition}
\newtheorem{corollary}[theorem]{Corollary}

\newtheorem{deflem}[theorem]{Definition-lemma}

\newtheorem{newthm}{Theorem}[section]

\newtheorem{example}[theorem]{Example}
\newtheorem{remark}[theorem]{Remark}

\theoremstyle{definition}
\newtheorem{definition}[theorem]{Definition}
\newtheorem{newdefinition}[newthm]{Definition}

\newcommand{\defeq}{:=} 
\newcommand{\xtwoheadrightarrow}[2][]{
\xrightarrow[#1]{#2}\mathrel{\mkern-14mu}\rightarrow
}

\DeclareMathOperator{\id}{id} 
\DeclareMathOperator{\im}{im} 
\DeclareMathOperator{\coker}{coker}

\DeclareMathOperator{\Hom}{Hom} 
\DeclareMathOperator{\Ext}{Ext} 
\DeclareMathOperator{\End}{End}

\DeclareMathOperator{\Gal}{Gal} 
\DeclareMathOperator{\Tr}{Tr} 
\DeclareMathOperator{\Fr}{Fr} 
\DeclareMathOperator{\Frob}{Frob} 
\DeclareMathOperator{\tr}{tr} 
\DeclareMathOperator{\Norm}{Norm} 

\DeclareMathOperator{\rad}{rad} 
\DeclareMathOperator{\soc}{soc} 
\DeclareMathOperator{\cosoc}{cosoc} 
\DeclareMathOperator{\JH}{JH} 
\DeclareMathOperator{\Ind}{Ind} 
\DeclareMathOperator{\cInd}{c-Ind} 
\DeclareMathOperator{\Ord}{Ord} 
\DeclareMathOperator{\Sym}{Sym}

\DeclareMathOperator{\GL}{GL} 
\DeclareMathOperator{\gr}{gr} 

\DeclareMathOperator{\lcm}{lcm}

\newcommand{\rep}{\overline{\rho}} 
\newcommand{\rss}{\overline{\rho}^{\mathrm{ss}}} 
\newcommand{\Wr}{W(\rep)} 
\newcommand{\Wss}{W(\rss)} 
 
\newcommand{\Dss}{\mathscr{D}^{\mathrm{ss}}} 

\newcommand{\DZero}{D_0(\rep)}
\newcommand{\DZeroEll}[1][\ell]{D_0(\rep)_{#1}}
\newcommand{\DZeroFil}[1][\ell]{D_0(\rep)_{\le #1}} 
\newcommand{\DiagrFil}[1][\ell]{D(\rep)_{\le #1}} 
\newcommand{\DZerogr}[1][\ell]{D_0(\rep)_{#1}} 
\newcommand{\DOne}{D_1(\rep)}
\newcommand{\DOneEll}[1][\ell]{D_1(\rep)_{#1}}
\newcommand{\DOneFil}[1][\ell]{D_1(\rep)_{\le #1}} 

\newcommand{\Diagr}{D(\rep)}

\newcommand{\Dx}[1][\sigma]{D_{0,{#1}} (\rep)}
\newcommand{\DOnex}[1][\sigma]{D_{1,{#1}} (\rep)}
\newcommand{\DxFil}[2][\sigma]{D_{0,{#1}} (\rep)_{\le #2}}
\newcommand{\Dxgr}[2][\sigma]{D_{0,{#1}} (\rep)_{#2}}

\newcommand{\DssZero}{D_0(\rss)}
\newcommand{\DssZeroEll}[1][\ell]{D_0(\rss)_{#1}}
\newcommand{\DssZeroFil}[1][\ell]{D_0(\rss)_{\le #1}} 
\newcommand{\Dssx}[1][\sigma]{D_{0,{#1}}(\rss)}
\newcommand{\DssOnex}[1][\sigma]{D_{1,{#1}} (\rss)}
\newcommand{\DssOne}[1][\sigma]{D_{1}(\rss)}

\newcommand{\aGal}[1][F]{\Gal(\overline{#1}/#1)} 
\newcommand{\Fres}{\mathbb{F}_q} 
\newcommand{\GR}[1][K]{\mathrm{GL}_2(\mathcal{O}_{#1})} 
\newcommand{\GF}[1][K]{\GL_2(#1)} 
\newcommand{\Gres}{\GL_2(\Fres)} 
\newcommand{\GRZ}[1][K]{\GR[#1] #1 ^\times} 
\newcommand{\MR}[1][K]{\mathrm{M}_2(\mathcal{O}_{#1})} 
\DeclareMathOperator{\Gt}{\widetilde{\Gamma}} 

\newcommand{\inertia}{I(\overline{K}/K)} 

\DeclareMathOperator{\Mat}{Mat}
\DeclareMathOperator{\rk}{rk}
\newcommand{\lgR}{\lg_{\GR}} 
\newcommand{\lgF}{\lg_{\GF}} 
\newcommand{\socR}{\soc_{\GR}} 
\newcommand{\socF}{\soc_{\GF}} 
\newcommand{\socG}{\soc_{\Gamma}} 
\newcommand{\socI}{\soc_{I/Z_1}} 
\newcommand{\cosocF}{\cosoc_{\GF}} 
\newcommand{\cosocG}{\cosoc_{\Gamma}} 
\newcommand{\radG}{\rad_{\Gamma}} 

\newcommand{\IndI}{\Ind_{I}^{\GR}}
\newcommand{\IndB}{\mathrm{Ind}^{\Gamma}_{B(\Fres)}} 
\newcommand{\IndBK}{\mathrm{Ind}^{\GF}_{B(K)}}

\DeclareMathOperator{\ProjG}{Proj_{\Gamma}}

\DeclareMathOperator{\F}{\mathbb{F}} 
 
\newcommand{\mytwist}{\det(\rep)\omega^{-1}}

\newcommand{\ip}{\iota_\pi} 
\newcommand{\alphaI}{\alpha} 
\newcommand{\bm}[1][\mu]{\beta_{#1}} 
\newcommand{\btm}[1][\mu]{\widetilde{\beta}_{#1}} 

\newcommand{\V}{V} 
\newcommand{\VFil}[1][\pi',\ell]{V(#1)} 
\newcommand{\Vgr}[1][\pi',\ell]{\frac{V(#1)}{V(#1+1)}} 
\newcommand{\Vell}[1][\pi',\ell]{V(#1)}

\newcommand{\rFil}[1][\pi', \ell]{r(#1)}

\newcommand{\Wchi}[2][J_1,J_2]{W(\chi_{#2},\chi_{#2}^{#1})}
\newcommand{\tell}[1][\ell]{\tau^{(f+1)}_{#1}}

\newcommand{\WJ}[1][J, \underline{i}]{W_{#1}'}
\newcommand{\PWJ}[1][J, \underline{i}]{W_{#1}}
\newcommand{\VJ}[1][J, \underline{i}]{V_{#1}}
\newcommand{\QJ}[1][J, \underline{i}]{Q_{#1}}
\newcommand{\KJ}[1][J, \underline{i}]{K_{#1}}
\newcommand{\pJo}[1][J, \underline{i}]{\psi^0_{#1}}
\newcommand{\pJ}[1][J, \underline{i}]{\psi_{#1}}
\newcommand{\pbJ}[1][J, \underline{i}]{\overline{\psi}_{#1}}
\newcommand{\Isigma}[1][\sigma_{\underline{m}}]{I(\sigma_{(J-1)^{\mathrm{ss}}}, #1)}

\newcommand{\mo}{{\left(\begin{smallmatrix}0 & 1 \\ 1 & 0\end{smallmatrix}\right)}}
\newcommand{\mpi}{{\left(\begin{smallmatrix}p & 0 \\ 0 & 1\end{smallmatrix}\right)}}

\newcommand{\mI}{\mathfrak m}
\newcommand{\mK}{\mathfrak m_{K_1}}

\newcommand{\xJ}[1][J, \underline{i}]{x_{#1}}

\newcommand{\E}[1][2f]{\mathrm{E}^{#1}_{\Lambda} }
\newcommand{\Eg}[1][2f]{\mathrm{E}^{#1}_{\grL} }

\newcommand{\Rbar}{\overline{R}}
\newcommand{\Rfil}[1][\ell]{\Rbar^{#1}}

\newcommand{\PhG}{\Phi\Gamma^{\text{ét} }_{\mathbb{F} } }
\newcommand{\PhGhat}{\widehat{\Phi\Gamma}^{\text{ét} }_{\mathbb{F} }} 

\newcommand{\DA}{D_A}
\newcommand{\Dvee}{D^\vee_\xi}

\newcommand{\Nnull}{N(\rep)}
\newcommand{\Nfil}[1][\ell]{N^{#1}(\rep)}

\newcommand{\Jrep}{J_{\rep}}

\newcommand{\Yn}[1][\lambda]{Y(#1)}

\newcommand{\Zn}[1][\lambda]{Z(#1)}

\newcommand{\atwo}[1][\ell]{\mathfrak{a}^{#1}}

\newcommand{\grm}{\gr_{\mI}}
\newcommand{\grL}{\gr(\Lambda)}

\newcommand{\mpp}[1][p_0]{m_{\mathfrak #1}}
\newcommand{\FX}{\F(\!(X)\!)}

\newcommand{\dimFX}{\dim_{\FX}}

\newcommand{\repgl}{\overline{r}} 
\newcommand{\mgl}{\mathfrak{m}_{\overline{r}}}
\newcommand{\Moo}{M_\infty}
\newcommand{\moo}{\mathfrak m_\infty}
\newcommand{\Roo}{R_\infty}

\DeclareMathOperator{\Art}{Art}

\begin{document}
\maketitle

\begin{abstract}
Let $p$ be a prime number and 
$K$ a finite unramified extension of 
$\mathbb{Q}_p$. 
We study the smooth mod $p$ representations 
of $\mathrm{GL}_2(K)$ appearing 
in a tower of mod $p$ Hecke eigenspaces 
of the cohomology of Shimura curves,
under mild genericity assumptions but notably 
no multiplicity one assumption at tame level,
and prove that they are of finite length, 
thereby extending the results of \cite{BHHMS4}
to higher multiplicity.
In the companion article \cite{Ber}
we investigated the case
where the local Galois representation 
attached to the Hecke eigensystem 
is semisimple; 
this article treats the non-semisimple case.
\end{abstract}

\tableofcontents
\section{Introduction}
	\label{sec:introduction}
\subsection{Main results}
	\label{sec:main-results} 
Fix a prime number $p$,
a totally real number field $F$ unramified
at places above $p$, and a quaternion algebra
$D$ with centre $F$ which is split at all places
above $p$ and at exactly one infinite place.
We denote by $\mathbb{A}_F^{ \infty }$
the ring of finite adèles of $F$.
Fix one place $v\mid p$, and a compact open
subgroup $V^v$ of 
$(D \otimes_{F}\mathbb{A}_F^{\infty,v})^{\times}$. 
If $V_v$ is any compact open subgroup
of $(D \otimes_{F}F_v)^{\times} \cong \GL_2(F_v)$,
then $V^vV_v$ is a compact open subgroup of
$(D \otimes_{F}\mathbb{A}^\infty_F)^{\times}$,
and we can consider the associated smooth
projective Shimura curve $X_{V^vV_v}$, 
which is defined over $F$, with the conventions of
\cite[§3.1]{BD14}.

If $\F$ is a finite extension of $\F_p$
(which we always tacitly assume 
to be sufficiently large),
and if
$\repgl \colon \Gal( \overline{F}/F)
\to \GL_2(\F)$ is a continuous absolutely
irreducible Galois representation,
then we study
the admissible smooth representation
\begin{equation}
	\label{eq:piShi-def}
\pi(V^{v}) \defeq \varinjlim_{V_v} \Hom_{\Gal(\overline{F}/F)}(\repgl, H^1_{ \text{ét}} (X_{V^vV_v} \times_{F}\overline{F}, \F))
\end{equation}
of $\GL_2(F_v)$ over $\F$,
where we let $V_v$ vary in the set of
compact open subgroups of  
$(D \otimes_{F}F_v)^{\times} \cong \GL_2(F_v)$.
The reader should note that,
despite our notation, $\pi(V^{v})$
depends on all global choices, and not only on
the prime-to-$v$ level.
We restrict ourselves to those
$\repgl$ that give a nonzero $\pi$.
These representations have seen some
attention in recent years
(e.g.\ in \cite{HW22}, \cite{BHHMS1}, \cite{BHHMS2}, \cite{Wang23}), 
though for the most part they
remain mysterious when $F_v \neq \mathbb{Q}_p$.
An important contribution to
 our understanding comes from 
the recent paper \cite{BHHMS4},
establishing that \eqref{eq:piShi-def} has
finite length, under
some genericity assumption on $\repgl$,
and a multiplicity one assumption on 
$\pi$ (which is called the 
\emph{minimal case}).

In \cite{Ber} we extend most of the results
in \cite{BHHMS4}
when the multiplicity one assumption is dropped,
but assuming that the restriction
$\repgl|_{\aGal[F_v]} $
of $\repgl$ to the decomposition subgroup at $v$ 
is semisimple.
In this article,
which is a companion to
\cite{Ber},
we investigate the remaining (and harder) case,
the \emph{non-semisimple} case.
From now we assume,
unless otherwise indicated,
that $\repgl|_{\aGal [F_v]} $
is non-semisimple.
In the rest of the introduction we 
present our results.

We set $K \defeq F_v$, $f \defeq [K:\mathbb{Q}_p]$,
and $q \defeq p^{f},$
and we denote by $\omega$ the 
mod $p$ cyclotomic character of
$\Gal(\overline{K}/K)$,
which we regard as a character of $K^{\times}$
via local class field theory,
 with the convention that uniformisers 
are sent to geometric
Frobenius elements.
We denote by $\omega_f,$ $\omega_{2f} $
Serre's fundamental characters of the 
inertia subgroup $\inertia$ of 
$ \Gal(\overline{K}/K)$ of level
$f$, $2f$ respectively.
Following \cite{BHHMS4}, we say that $\repgl$
is \emph{generic} if the following conditions
are satisfied
\begin{enumerate}[(i)]
\item 
\label{point:generic(i)} 
the restriction 
$\repgl|_{\Gal(\overline{F}/F(\mu_p))} $ is absolutely irreducible;
\item 
\label{point:generic(ii)} 
	if $w\nmid p$ is a place where either
$D$ or $\repgl$ is ramified, then
the framed deformation ring of 
$\repgl|_{\Gal(\overline{F}_w/F_w)} $
over the Witt vectors $W(\F)$ is formally smooth;
\item 
\label{point:generic(iii)} 
	the restriction 
$\repgl|_{\inertia} $ of $\repgl$
to the inertia subgroup at $v$
is, up to twist, of the form
\[
\begin{pmatrix}
\omega_{f}^{\sum_{j=0} ^{f-1}(r_j+1)p^j }	
 & \ast \\
0 & 1
\end{pmatrix}, 
\]
with $\max \{12,2f+1\} <r_j<p- \max \{15,2f+4\} $,
and $\ast \neq 0$.
\end{enumerate}
The bounds on the $r_i$ of 
\ref{point:generic(iii)} all come
from \cite{BHHMS4}: we don't need any stronger
bound to make our arguments work.

The following is the main result of this article.
\begin{theorem}[\Cref{thm:global-part-fl}]
	\label{thm:pre-main-0}
Assume $\repgl$ generic.
Then, the representation $\pi(V^{v})$
has finite length as a $\GF$-representation.
\end{theorem}
Observe that we don't need any hypothesis
on the prime-to-$v$ level $V^{v}$.

The proof of \Cref{thm:pre-main-0}
reduces ``by dévissage''
to the study of some subquotients of \eqref{eq:piShi-def},
which we now describe.
For every $w \neq v$ lying over $p$,
choose an (absolutely) irreducible
$\GR[w]$-representation
$\sigma_w$,
and set $\sigma_p^{v} \defeq 
\bigotimes_{ \begin{substack}
w\mid p, \\ w\neq v
\end{substack} } \sigma_w $.
Then, our strategy 
(which mirrors the proof of 
\cite[Corollary~8.4.6]{BHHMS1})
is to show that $\pi(V^{v})$ admits 
a \emph{finite} filtration whose subquotients 
are submodules
of $\GF$-representations
of the form
\begin{equation}
	\label{eq:piShi-dev} 
\pi(V^{v},\sigma_p^{v})\defeq \varinjlim_{V_v} 
\Hom_{\prod_{
w\mid p, w\neq v
} \GL_2(\mathcal{O}_{F_w} )}
\left(
\sigma_p^{v},
\Hom_{\Gal(\overline{F}/F)}
\left(\repgl, H^{1}_{\text{ét}}(X_{V^vV_v} \times_{F}\overline{F}, \F)\right)\right).
\end{equation}
Then, we show that every such
$\pi(V^{v},\sigma_p^{v})$
is of finite length, concluding the proof.

The two theorems below describe in more precise terms
the properties of $\pi(V^{v},\sigma_p^{v})$.
Remember that
by \cite[Theorem~6.3(ii)]{Wang23}
there is a unique integer
$r \ge 1$, called the \emph{multiplicity},
such that $\dim_{\F} \Hom_{\GR}(\sigma, \pi) \in \{0,r\}$
for every (absolutely) irreducible representation
$\sigma$ of $\GR$ over $\F$.

From now on we let 
$\rep \defeq \repgl^\vee |_{\Gal(\overline{K}/K)},$
where $\rep^\vee $ is the dual of $\rep$.
Write $\rep \cong 
\left( \begin{smallmatrix} 
\chi_1& * \\
0 & \chi_2 \\
\end{smallmatrix} \right)$,
for two 
$\Gal(\overline{K}/K)$-characters 
$\chi_1 ,\chi_2 $ over $\F$.
Let $B \subseteq \GL_{2, \mathcal{O}_K}$
denote the algebraic subgroup of
upper-triangular matrices,
and let $\pi_0 , \pi_f$
be the principal series
$\pi_0 \defeq 
\IndBK (\chi_2 \otimes \chi_1 \omega^{-1})$
and
$\pi_f \defeq 
\IndBK (\chi_1 \otimes \chi_2 \omega^{-1})$.
Set $\V \defeq \F^r$.
While \cite[Theorem~1.1.1(iii)]{BHHMS4}
assumes that $r=1$,
the following theorem
is a generalisation for arbitrary
$r \ge 1$.
\begin{theorem}[{\Cref{thm:PS}}]
	\label{thm:pre-main-1} 
Assume $\repgl$ generic,
and let $\pi$ be a $\GF$-representation
of the form \eqref{eq:piShi-dev}.
Then, $\pi$ is of the form
\[
\pi \cong 
\left( 
\begin{xy} (0,0)*+{
\V \otimes_{\F} \pi_0 
}="a";
(15,0)*+{\pi'}="b";
(30,0)*+{
\V \otimes_{\F} \pi_f
}="c"; 
{\ar@{-}"a";"b"};
{\ar@{-}"b";"c"}\end{xy}
\right),
\]
where the notation means that
$\socF \pi \cong  \V \otimes_{\F} \pi_0$,
$\cosocF \pi \cong  \V \otimes_{\F} \pi_f$,
and that $\pi' \cong 
\ker(\pi/\V \otimes_{\F} \pi_0 \twoheadrightarrow 
\V \otimes_{\F} \pi_f)$.
Moreover, the length of $\pi'$ 
is bounded above by
$r \cdot (f -1 )$ and
every Jordan-H\"older constituent
of $\pi'$ is supersingular.
\end{theorem}
Ideally, we should expect that
for every irreducible subquotient $\overline{\pi}$
of $\pi'$ we have $[\pi : \overline{\pi}]=r$.
However, such a result seems currently out of reach.

Let $K_1\unlhd \GR $ 
be the set of matrices with trivial reduction mod $p $, 
$I\le \GR $ 
the set of matrices which are upper triangular mod $p $,
and $I_1\unlhd I $ 
the set of matrices which are upper unipotent mod $p $.
Define $Z_1 := (1+ p \mathcal{O} _{K} )\cdot \id $
to be the centre of $I_1 $,
and
$\Lambda \defeq \F [\![I_1/Z_1]\!]  $
to be the completed group algebra of $I_1/Z_1 $,
which is a noetherian noncommutative local ring of Krull
dimension $3f$,
whose maximal ideal we denote by $\mI$.
By \cite[Theorem~1.9]{BHHMS1} 
and by \cite[Theorem~6.3(ii)]{Wang23},
$\pi$ has a central character.
In particular, for all subquotients
$\pi'$ of $\pi$, 
the linear dual $ \Hom_{\F}(\pi', \F)$
is a $\Lambda$-module,
which is moreover finitely generated
since $\pi$ is admissible.
Recall that a nonzero finitely generated 
$\Lambda$-module $M$ is Cohen-Macaulay of
grade $c \ge 0$ if
$ \Ext^{i}_{\Lambda}(M, \Lambda)\neq 0$ if
and only if $i = c$.
The following theorem
is a generalisation 
of \cite[Theorem~1.1.3]{BHHMS4}
for arbitrary $r \ge 1$.
\begin{theorem}
  \label{thm:pre-main-2}
Assume $\repgl$ generic,
let $\pi$ be a $\GF$-representation
of the form \eqref{eq:piShi-dev},
let $\pi_1 \subsetneq \pi_0 $
be two nonzero subrepresentations of $\pi$,
and let $\pi_2 \defeq \pi_0 /\pi_1 $.
\begin{enumerate}[(i)]
\item 
The $\F$-linear dual
$\pi_2 ^\vee$ of $\pi_2$
is a Cohen-Macaulay $\Lambda$-module 
of grade $2f$.
\item 
The $\GF$-representation
$\pi_2 $ is generated 
by its $K_1 $-invariants $\pi_2 ^{K_1}$.
\item 
If $\rss$ denotes the semisimplification
of $\rep$,
and if $\Wss$
denotes the set of
irreducible $\GR$-representations
defined in \cite[§§3.1-2]{BDJ10},
then we have
\begin{equation}
  \label{eq:pre-Dvee-rk}
\dimFX \Dvee (\pi_0 )
= \sum_{\tau \in \Wss} 
[\pi_0 ^{K_1} : \tau], 
\end{equation}
where $\Dvee (\pi_0 )$
is the cyclotomic $(\varphi, \Gamma)$-module
associated to $\pi_0 $
in \cite[\S 2.1.1]{BHHMS2}.
\end{enumerate}
\end{theorem}
\Cref{thm:pre-main-2}(i)
is \Cref{prop:CM-via-N}(ii),
\Cref{thm:pre-main-2}(ii)
is \Cref{prop:CM-via-N}(i),
and \Cref{thm:pre-main-2}(iii)
follows from \Cref{cor:Wang}.
We can even compute
$\dimFX \Dvee (\pi_2 )$
using
\Cref{thm:pre-main-2}(iii)
and the exactness of $\Dvee (-)$.
However, 
the natural injection
$\pi_0^{K_1} / \pi_1^{K_1} \hookrightarrow 
\pi_2 ^{K_1}$
is not surjective in general
(this is unique to the non-semisimple case),
so it is not a priori clear that
\eqref{eq:pre-Dvee-rk}
holds for $\pi_2 $.
When $r = 1$
this is \cite[Corollary~1.1.3]{BHHMS5},
for general $r \ge 1$
we expect \eqref{eq:pre-Dvee-rk} to still hold,
though at time of writing we lack a proof.

\subsection{Outline of the article}
We begin with the study of \eqref{eq:piShi-dev},
which we denote simply by $\pi$.

In \Cref{sec:K1-inv}
we study the subrepresentations of $\pi$,
mainly in terms of their diagram,
in the sense of \cite{Pas04}.
Remember that to a Galois representation $\rep$
as in \Cref{sec:main-results}
one can associate a diagram 
$\Diagr=(\DOne \hookrightarrow \DZero)$,
which we take to be the diagram 
$\mathcal{D} (\pi_{\text{glob}} (\rep)) $
of \cite[Theorem~1.3]{DL21}
(or equivalently the one of
\cite[Theorem~1.3.2]{BHHMS2}).
We refer to \Cref{sec:diagrams}
for the definition of diagram,
but in short $\Diagr$
should be regarded as the datum of
a $\GRZ$-representation $\DZero$
together with an action of 
$\Pi \defeq 
\left( 
\begin{smallmatrix} 0 & 1 \\ p & 0 \end{smallmatrix}
\right) $
on $\DZero^{I_1}$.
Crucially, if $\rep$ is generic then
there is a unique positive integer $r \ge 1$
such that the following four hypotheses hold:
\begin{enumerate}[(i)]
	\item \label{hypothesis:pre-i}
(\cite[Theorem~1.3]{DL21} when $r = 1$,
\Cref{thm:diag} when $r > 1$)
There is an isomorphism of diagrams
\begin{equation}
  \label{eq:ip-ante-litteram}
  \ip \colon 
  (\pi^{I_1} \hookrightarrow  \pi^{K_1})
  \cong \Diagr^{\oplus r},
\end{equation}
where the diagram structure on
$\V \otimes_{\F} \Diagr$ 
is induced from the one on $\Diagr$.
	\item \label{hypothesis:pre-ii}
(\cite[Proposition~6.4.6]{BHHMS1} and \cite[Theorem~6.3(ii)]{Wang23})
If 
$\chi:I\to \F^\times  $ is a smooth character
appearing in 
$\pi^{I_1}=\pi[\mI]$,
then we have an equality of multiplicities
		\[
			[\pi[\mI^3]\colon \chi] = [\pi[\mI]\colon \chi]=r.
		\]

	\item \label{hypothesis:pre-iii}
	(\cite[Theorem~8.2]{HW22} 
and \cite[Theorem~6.3(i)]{Wang23})
		The linear dual $\pi^\vee
\defeq \Hom_{\F}(\pi, \F)$ is \textit{essentially self-dual} of grade $2f$,
	i.e.\ there exists a $\GL_2(K)$-equivariant isomorphism of $\Lambda$-modules
\begin{equation}
  \label{eq:pre-kappa-pi}
\Ext^{2f}_{\Lambda}(\pi^\vee,\Lambda)\cong \pi^\vee \otimes \chi_\pi.
\end{equation}	

	\item \label{hypothesis:pre-iv}
(\cite[Proposition~2.6.3]{BHHMS4})
If $\chi \colon I \to \F ^\times $ is a 
smooth character and $i \ge 0$, 
then $\Ext^i_{I/Z_1} (\chi,\pi)\neq 0$
if and only if $[\pi[\mI]:\chi]\neq 0$,
in which case
\[
	\dim_{\F} \Ext^i_{I/Z_1} (\chi,\pi)=
	\binom{2f}{i} r.
\]
\end{enumerate}

Although the representation \eqref{eq:piShi-dev}
is a global object, the arguments are local:
we take any 
admissible smooth representation of $\GF$
with a central character
satisfying \ref{hypothesis:pre-i} to \ref{hypothesis:pre-iv},
and deduce some of its properties, 
culminating in \Cref{thm:pre-main-1} 
and \Cref{thm:pre-main-2}.
Note that \ref{hypothesis:pre-i} 
is stronger than condition (i) from
\cite[\S 3.3.2]{BHHMS2};
we impose this stronger condition 
in order to make weight cycling arguments possible.

Fix $\pi_1 $ a subrepresentation of $\pi$,
then $\iota_\pi(\pi_1 ^{I_1} 
\hookrightarrow \pi_1 ^{K_1})$
is a subdiagram of 
$\Diagr^{\oplus r}\cong \F^{r} \otimes_{\F} \Diagr$,
and in \Cref{sec:K1-inv}
we determine what subdiagrams can occur.
More precisely, 
recall that $\Diagr$ 
admits a increasing filtration
$(\DiagrFil[\ell])_{\ell = 0}^{f}$
of subdiagrams,
defined in the paragraph after
\cite[Remarque~5.10]{Hu16}
(and denoted there by $D_\ell$).
Then, we show in \Cref{thm:srep-structure}
that the subdiagram
$\iota_\pi(\pi_1 ^{ I_1} \hookrightarrow 
\pi_1 ^{ K_1}) \subseteq \F^{r} \otimes_{\F} \Diagr$ 
is always of the form
\begin{equation}
  \label{eq:pre-srep-K1}
\iota_\pi(\pi_1^{ I_1} \hookrightarrow 
\pi_1 ^{ K_1}) =
\sum_{ \ell=0 }^{ f} \VFil[\pi_1 ,\ell] \otimes_{\F} 
\DiagrFil[\ell] \subseteq \F^{r} \otimes_{\F}\Diagr, 
\end{equation}
for some uniquely determined
$\F$-vector subspaces 
$\VFil[\pi_1 ,\ell] \subseteq \F^r$,
for  $0 \le \ell \le f$.
The technical difficulty that emerges when $r>1$ 
is then to coordinate all our constructions
to ensure their compatibility
with these $\F$-vector subspaces.

In \Cref{sec:Wang-with-mult}
we set out to prove
\Cref{thm:pre-main-2}(iii),
and in doing so we show how to 
extend the majority of
\cite{Wang24} to the case $r \ge 1$.
Accordingly, the family 
$\{\xJ[J, \underline{i}]\}$
of elements of $\pi$ 
constructed in \cite[Theorem~6.3]{Wang24}
now depends on the choice of an element $w \in \F$,
and it becomes important to show 
that all the constructions at play
do not introduce any ``mixed terms''
involving multiple $w \in \F$,
which requires careful bookkeeping
of the various isomorphisms at play.

In \Cref{sec:CM-and-fl},
we prove some of the properties that
a nonzero subrepresentation 
$\pi_1$ of $\pi$ satisfies.
First, in \Cref{sec:CM},
we determine the structure of
$\grm(\pi_1 ^{\vee}) \defeq 
\bigoplus_{ n \ge 0} \mI^{n}\pi_1 ^{\vee}
/\mI^{n+1}\pi_1 ^{\vee}$
as a module over
$\grL \defeq
\bigoplus_{ n \ge 0} \mI^{n}/\mI^{n+1}$.
In particular, 
in \Cref{prop:CM-via-N}
we find that it is Cohen-Macaulay
of grade $2f$, and 
\Cref{thm:pre-main-2}(i)
follows from
\cite[Proposition~III.2.2.4]{LvO}.
Proving this statement required multiple adjustments:
constructing the diagram
\eqref{diag:thetas} involves a study
of certain $H$-eigenspaces occurring inside
$\bigoplus_{ n = 0}^{f} \mI^{n}\pi_1^{ \vee}
/\mI^{n+1}\pi_1^{ \vee}$,
and how they relate
to the $\F$-linear subspaces 
$\VFil[\pi_1, \ell] \subseteq \F^r$
of \eqref{eq:pre-srep-K1}.
Moreover,
after defining $\widetilde{\pi}_1$
as the unique 
subrepresentation of $\pi$ satisfying
\[
  \widetilde{\pi}_1  = 
  \ker\left( \pi
  \xrightarrow[\sim]{\kappa_\pi}
  \Ext^{2f}_{\Lambda}\left( \pi^\vee \otimes 
  \chi_\pi, \Lambda\right)^\vee
  \twoheadrightarrow 
  \Ext^{2f}_{\Lambda}\left(
  \pi_1 ^{\vee} \otimes \chi_\pi,
  \Lambda\right)^\vee
  \right),
\]
where $\kappa_\pi$ is the $\F$-linear dual
of \eqref{eq:pre-kappa-pi} 
(twisted by $\chi_\pi$),
we compute the $\F$-dimension
of $\VFil[\widetilde{\pi}_1, \ell]$
by adapting
the argument of \cite[Proposition~3.2.2]{BHHMS4}
in the semisimple case.
Note that the argument of 
Step 4 of the proof of
\cite[Proposition~4.4.3]{BHHMS4},
which computes
$\dim_{\F} \VFil[\widetilde{\pi}_1, \ell]$
(which are completely determined
by the integer $i( \widetilde{\pi}_2 )$
of \emph{loc.\ cit.})
purely using numerical invariants,
does not extend to the case $r >1$,
mainly for the reason that for $r > 1$ 
and for a fixed integer $m \ge 0$
there are in general multiple integer solutions
$(r(0), \dots, r(f))$ such that $0 \le r(\ell) \le r$
and $r (f) \le r (f -1) \le \dots \le r(1) \le r(0)$
to the equation
$\sum_{\ell = 0}^{f} 
r(\ell) \cdot \binom{f}{\ell} = m$.

In \Cref{sec:fl}
we show that $\pi_1$ is generated by $\pi_1^{K_1}$,
from which \Cref{thm:pre-main-2}(ii)
easily follows,
and establish the upper bound 
$\lgF (\pi_1 ) \le \sum_{\ell = 0}^{f} 
\dim_{\F} \VFil[\pi_1, \ell]$
on the length of $\pi_1$
(and in particular $\lgF (\pi) \le r \cdot (f+1)$).
Finally, in \Cref{thm:PS}
we conclude the proof of \Cref{thm:pre-main-1}.
A difference with the semisimple case is that
we need a preliminary study 
(done in \Cref{cor:tilde-involution})
of the contruction $\pi_1 \mapsto \widetilde{\pi}_1 $
introduced in the previous parargaph
to compute $\cosocF \pi$;
as an immediate consequence 
of \emph{loc.\ cit.}\ we can show that
the subquotient $\pi'$ 
of \Cref{thm:pre-main-1}
satisfies the analogue of
hypothesis \ref{hypothesis:pre-iii}.

In \Cref{sec:global-diagram} we show
that hypothesis \ref{hypothesis:pre-i} holds
for arbitrary $r \ge 1$.
The main obstacle rests in showing that
the action of
$\Pi$ on $\ip \colon \pi^{I_1 } \cong 
\F^r \otimes_{\F} \DOne$
``respects lines'',
i.e. that it restricts to $L \otimes_{\F} \DOne$
for all lines $L \subseteq \F^r$.
This is already known in the semisimple case,
cf.\ \cite[Theorem~1.8]{BHHMS2};
the argument in the non-semisimple case 
is largely unchanged,
but we should remark that it rests upon 
certain freeness results for patched modules
due to Yitong Wang.
Finally, in \Cref{sec:global-fl} we show
the finite length of \eqref{eq:piShi-def}
by ``dévissage'',
reducing ourselves to the finite length of
\eqref{eq:piShi-dev},
thus concluding the proof of
\Cref{thm:pre-main-0}.

\paragraph{Acknowledgements:}
The author wishes to give her most sincere thanks 
to Christophe Breuil and Stefano Morra
for their priceless encouragements and guidance,
and for reviewing several versions of this paper.
The author extends her thanks to 
Yitong Wang for his precious remarks
and to Florian Herzig
for his helpful suggestions and comments.
The author is supported by an FMJH grant.

\section{Preliminaries}
  \label{sec:preliminaries}
In this section we introduce 
the notation and the material
that will be used throughout the text.
We fix $K/\mathbb{Q}_p $ an unramified extension of (inertial) degree $f$, and let $q \defeq  p^f $.
We let
 $I \le \GR$ be the subgroup of matrices
 that are upper triangular modulo $p$,
$I_1\unlhd I$ be its maximal pro-$p$ subgroup, consisting of those matrices that are unipotent modulo $p$,
and we let $K_1\unlhd I_1$ be the subgroup of matrices that are sent to the identity modulo $p$,
which is also the pro-$p$ radical of $\GR $.
In particular, $\GR/K_1 \cong \Gres$,
which we also denote by $\Gamma$.
We let $N$ be the normaliser of $I$ in $\GL_2(K)$,
which is generated by $I$,  $K^\times $, 
and the matrix 
$\Pi \defeq 
\left( 
\begin{smallmatrix} 0 & 1 \\ p & 0 \end{smallmatrix} \right) $.

Fix a finite extension $\F$ of $\mathbb{F}_p$,
which we assume to be sufficiently large.
The following completed group algebra
will play an important role.
\begin{newdefinition}
If $Z_1 := (1+ p \mathcal{O} _{K} )\cdot \id $ denotes the centre of $I_1 $,
then we let
$\Lambda$
be the completed group algebra 
$\Lambda \defeq \mathbb{F} [\![I_1/Z_1]\!]$,
whose unique maximal ideal we call $\mI$,
and which we endow with the $\mI$-adic topology.
\end{newdefinition}
We know from
\cite{Clozel17}
and from
\cite[Theorem 5.3.4]{BHHMS1}, 
while keeping the notation of 
\cite[paragraph~before~Remark~3.17]{BHHMS2},
 that the graded ring $\grL \defeq \grm(\Lambda )$
associated to the $\mI$-adic filtration on 
$\Lambda $ is an Auslander-regular
noncommutative algebra over $\mathbb{F} $
generated by the variables $\{y_i, z_i, h_i\mid 0\le i \le f-1\}$,
which commute for different indices,
and are subject to the commutator relations 
\begin{align}
[y_i,z_i]=h_i, \quad
  \label{eq:grL-commutator-relations}
[h_i, y_i]= [h_i, z_i]=0. 
\end{align}
For the definition of Auslander-regularity,
see \cite[Definition~III.2.1.3]{LvO},
\cite[Definition~III.2.1.7]{LvO}.
In particular, $\Lambda$ is Auslander-regular
by \cite[Theorem~III.2.2.5]{LvO}.
Then, we define
\begin{align}
  \label{eq:R-definition}
  R  \defeq& \grL  /(h_0,\dots, h_{f-1} ) \cong \F [y_i,z_i\mid 0\le i \le  f-1], \\
  \label{eq:R-bar-definition}
 \Rbar  \defeq& \grL/J = R/(y_i z_i \mid 0\le i \le  f-1)\\ 
\nonumber \cong&
\F [y_i,z_i\mid 0\le i < f]/(y_iz_i\mid 0\le i \le  f-1),
\end{align}
where $J:=(h_i,y_iz_i\mid 0\le i \le  f-1)$.

We are interested in the following invariant, which is defined in the paragraph before 
\cite[Lemma~3.32]{BHHMS2}.
\begin{newdefinition}
  \label{def:multiplicity}
Let $M$ be a finitely generated module over $\grL$, 
annihilated by some power $J^n$ of $J$.
We define the \emph{multiplicity} of $M$ at a minimal prime $\mathfrak{q} $
of $\Rbar$ by
\[
  m_{\mathfrak{q} }(M) \defeq  \sum_{i=0} ^{n-1} \lg_{\rep_{\mathfrak{q} } } (J^iM/J^{i+1} M)_{\mathfrak{q} } . 
\]
Here, $\lg_{\Rbar_{\mathfrak{q}}}$ denotes the length of an $\Rbar_{\mathfrak{q}}$-module.  
\end{newdefinition}
Of the $2^f$ minimal prime ideals of $\Rbar $,
we fix in particular $\mathfrak{p} _0 \defeq  (z_i\mid 0\le i \le  f -1)$,
and we remind the reader that $m_{\mathfrak{q} }$
is additive,
cf.\ \cite[Lemma~3.32]{BHHMS2}.

\subsection{Diagrams}
  \label{sec:diagrams}
In this section, we review the theory of diagrams
and in particular the construction of the diagram
$\Diagr$ of
\cite[Theorem~13.8]{BP12}.

If $\chi:I \to \mathbb{F}^\times$ is a smooth character, we define
$\chi^s(\cdot) \defeq \chi(\Pi \cdot \Pi^{-1} )$.
Remember that
$\chi$ factors through a character of 
the finite torus $H \defeq I/I_1$.

\begin{definition}[{\cite[Definition~9.7]{BP12}}]
  \label{eq:diag-def}
  A \emph{diagram} is a triple $(D_0, D_1, r)$,
where $D_0$ is a smooth representation of $\GR K^ \times $ over $\mathbb{F}$, $D_1$ is a smooth representation of $N$ over $\mathbb{F}$,
    and $r:D_1\to D_0$ is an $IK^ \times $-equivariant map.
\end{definition}

\begin{example}
 If $\pi $ is a smooth representation of $\GR $, then it induces the diagram
    \[
      (\pi^{I_1}
\xhookrightarrow{r} \pi^{K_1}),
    \]
  where $r $ is the canonical inclusion.  
\end{example}

We are also interested in another diagram:
the diagram 
$\Diagr $ of 
 \cite[Proposition~13.1]{BP12}. We now recall its construction.
Fix once and for all an embedding
$\mathbb{F} _{q^2} \hookrightarrow \mathbb{F}  $.
If $\inertia $ denotes the inertia subgroup of $\Gal(\overline{K}/K) $,
we denote by 
\begin{align*}
  \omega_f \colon  &\inertia \twoheadrightarrow \mathbb{F}_q^\times  \hookrightarrow \mathbb{F}^\times , \\
  \omega_{2f} \colon  &\inertia \twoheadrightarrow \mathbb{F}_{q^2}^\times  \hookrightarrow \mathbb{F}^\times 
\end{align*}
Serre's fundamental characters of level
$f$, $2f$ respectively.
We also denote by
$\omega \colon \Gal (\overline{K}/K)
\twoheadrightarrow \mathbb{F}_p^\times 
\hookrightarrow \mathbb{F}^\times$
Serre's fundamental character of level
$1$, i.e.\ the cyclotomic character modulo $p$.

Following the second paragraph of
\cite[§1.3]{BHHMS4},
we say that 
 a continuous Galois representation 
 $ \rep:\Gal(\overline{K}/K)\to \GL_2(\mathbb{F} )$ 
is $n$-\emph{generic},
where $n \ge 0$ is a nonnegative integer,
if, up to twist, 
$\rep|_{\inertia}$
is not isomorphic to $\omega \oplus 1$,
and if 
one of the following holds:
\begin{enumerate}
  \item \emph{(the reducible case)}
the restriction $\rep|_{\inertia}  $ is isomorphic, up to twist, to
\[
  \begin{pmatrix}
  \omega_{f}^{\sum_{j=0} ^{f-1}(r_j+1)p^j }  
  & \ast \\
  0 & 1
  \end{pmatrix} ,
\]
for some integers $r_j $ with
$ n\le r_j \le p-3-n$.
\item \emph{(the irreducible case)}
the restriction $\rep|_{\inertia}  $ is isomorphic, up to twist, to
  \[
    \begin{pmatrix}
      \omega_{2f}^{\sum_{j=0} ^{f-1}(r_j+1)p^j }  
    & 0 \\
    0 & 
  \omega_{2f}^{\sum_{j=0} ^{f-1}(r_j+1)p^{j+f} }  
    \end{pmatrix} ,
  \]
for some integers $r_j $ with
$n+1\le r_0 \le p-2-n$,
and with
$n\le r_j \le p-3-n$
for $j>0 $.
\end{enumerate}

To a $0$-generic $\rep$ 
we can associate a set of Serre weights $\Wr $,
following \cite[§§3.1-2]{BDJ10}.
Remember that a Serre weight
$\sigma $ is an isomorphism class of smooth irreducible representations of 
$\GR $ over $\mathbb{F} $,
or equivalently of $\Gamma$.
Recall that every constituent
of $\DZero$ shares the same central character
(cf.\ for example
\cite[Lemma~4.1(i)]{Wang24},
after noticing that $\sigma_{\underline{b}}$
in \emph{loc.\ cit.}\ is defined in the paragraph
after \cite[Eq.~(3)]{Wang24}
in terms of the \emph{extension graph} $\Lambda_W^\mu$
of \emph{loc.\ cit.}),
which can be computed
by choosing the Serre weight
parametrised by 
$(x_0 , \dots, x_{f-1}) \in \mathscr{D}$
in \eqref{eq:D-def-Jrep} below.
We obtain the $I$-character
$\left( \begin{smallmatrix}
a & b \\
pc & d
\end{smallmatrix}
 \right)\mapsto \overline{a}^{\sum_{j = 0}^{f-1} r_j p^j}$
 (where $\overline{a}$ is the reduction
 of $a$ mod $p$),
in particular the centre $\mathcal{O}_K \cdot \id$
of $\GR$ acts on $\DZero$
by the character
$a \cdot \id
\mapsto \overline{a}^{\sum_{j = 0}^{f-1} r_j p^j}$,
which corresponds to
the $\inertia$-character
$\mytwist = \omega_f^{\sum_{j = 0}^{f-1} r_j p^j}$
under local class field theory,
normalised so that uniformisers correspond to geometric Frobeniuses.

A classification of Serre weights
can be found in \cite[Proposition~2.17]{NewYork}:
it states that
Serre weights can be uniquely written as 
\begin{equation}
\Sym^{s_0} \F^2 \otimes_{\F} 
(\Sym^{s_1} \F^2)^{\Fr} \otimes_{\F} 
\cdots \otimes_{\F} 
(\Sym^{s_{f-1}} \F^2)^{\Fr^{f-1}} \otimes_{\F} 
{\det}^m,
\label{eq:weight-characterisation}
\end{equation}
for $0\le s_i \le p-1 $ and for $0\le m < p^f-1 $.
Here, the notation 
$(\Sym^{s_{i}} \F^2)^{\Fr^{i}}$
denotes the base change
$(\Sym^{s_{i}} \mathbb{F}_q^2) \otimes_{\mathbb{F}_q, \Fr^i} \F$
via the field homomorphism 
$\Fr^i \colon \mathbb{F}_q \hookrightarrow \F$
sending $x \mapsto x^{p^i}$.
We denote the Serre weight
\eqref{eq:weight-characterisation}
by $(s_0,\dots, s_{f-1} )\otimes \det^m
$.
From this characterisation one can also see
that 
$\sigma^{I_1} $ is always 1-dimensional,
and that if the character $\chi:I\to \F^\times  $ describes the action of $I $ on $\sigma^{I_1}$, 
then we can recover $\sigma $ from $\chi $
as long as $\chi\neq \chi^s $.
Note that the genericity assumption on 
$\rep $ makes sure that for all $\sigma\in \Wr $
we have $\chi^s \neq \chi$.

Following \cite[§2]{HW22}, we say that
a Serre weight is $n$-\emph{generic},
where $n \ge 0$ is a nonnegative integer,
if the integers $s_i$ of \eqref{eq:weight-characterisation},
for $i \in \{0, \dots, f-1\}$,
satisfy
$n \le s_i \le p-2-n$.
We say that a smooth character 
$ \chi \colon I \to \F^{\times}$
is $n$-generic if $\chi = \sigma^{I_1}$
for an $n$-generic Serre weight $\sigma$.
If $\rep$ is $n$-generic, then 
we remark that any $\sigma \in \Wr$
is $n$-generic, and any $\lambda$ belonging to the set 
$\mathscr{P}$ 
defined in \eqref{eq:P-def-Jrep} below
is $(n-1)$-generic if $n \ge 1$.

Suppose that $\rep$ is $0$-generic.
In order to  construct the diagram $\Diagr $, 
we begin with the following definition, coming from \cite[p.~65, last paragraph]{BP12}:
\begin{deflem}[{\cite[Corollary~3.12]{BP12}}]
\label{deflem:I}
If $\sigma,\tau$ 
are two Serre weights,
then there exists at most one (up to isomorphism) indecomposable $\Gamma $-representation 
with socle $\sigma $ and cosocle $\tau $,
and such that $\sigma $ appears with multiplicity one.
We let $I(\sigma, \tau) $ be this representation, if it exists, and we let $I(\sigma, \tau)=0  $ otherwise.
\end{deflem}
Recall that the socle filtration 
$(\soc_i(M))_{i \ge 0} $ of a $\Gamma$-module
$M$ is defined inductively as follows:
$\soc_{0}(M)=0$, and for $i >  0$,
$\soc_i(M)$ is the preimage
of $\socG (M/\soc_{i-1} (M))$ in $M$.
If there exists an $\ell \ge 0$ such that
$\soc_\ell(M)=M$, then we call the
smallest such $\ell$ the 
\emph{Loewy length} of $M$.

Then, for any $I(\sigma, \tau)$,
we let $\ell(\sigma, \tau)\in 
\mathbb{Z} _{>0} \cup \infty $ 
 be $\infty $ if $I(\sigma, \tau)=0$,
and otherwise the Loewy length of $I(\sigma,\tau)$.
Finally, we set $\ell(\rep, \tau) \defeq  
\min \left\{ \ell(\sigma,\tau)
\mid \sigma\in \Wr \right\}$.
If $\ell(\rep,\tau)<\infty $, this minimum 
is realised by exactly one $\sigma \in \Wr $, 
cf.\ \cite[Lemma~12.8]{BP12}.
In this case,
following \cite[Paragraph before Lemma~13.3]{BP12}
we set $I(\rep,\tau) \defeq  I(\sigma,\tau) $.

Define $\Dx $ as the direct limit 
\begin{equation}
\label{eq:diagr-from-I}
\Dx \defeq \varinjlim_{\le} I(\rep, \tau) ,
\end{equation}
where the direct limit is taken over
the set of $I(\rep, \tau) $ with socle $\sigma$,
to which we give the partial order defined in 
\cite[Paragraph after Lemma~13.3]{BP12}.
It is clear by construction that 
$\socG \Dx[\sigma] = \sigma$.
Then, according to \cite[Proposition~13.4]{BP12}, 
we can define $\DZero $ as the direct sum 
\begin{equation} 
  \label{eq:diagr-def} 
\DZero  \defeq  \bigoplus_{\sigma \in \Wr} \Dx[\sigma],  
\end{equation}
which is multiplicity free
by \cite[Corollary~13.5]{BP12}.
We regard all these as representations of $\GR $, by inflation.

Consider the character
$\mytwist \colon \Gal(\overline{K}/K) \to \F^{\times}$,
which can be identify with a character 
$K^\times \cdot \id\to  \F^\times $
via local class field theory.
Then, we let $K^\times$ act on $\DZero$ by $\mytwist$.
Note that this extends the previous action 
of $\mathcal{O}_K ^{\times} \cdot \id$ on $\DZero$,
by the paragraph before
\eqref{eq:weight-characterisation}.
We define 
\begin{equation}
  \label{eq:diagr1-def} 
\DOnex[\sigma] \defeq \Dx[\rho]^{I_1},
\quad \DOne \defeq 
\DZero^{I_1}= \bigoplus_{\sigma \in \Wr} \DOnex[\sigma].
\end{equation}
This determines the diagram $\Diagr$
up to the choice of certain parameters,
namely the scalars $\nu_i$
defined in the second paragraph of
\cite[§6]{Bre11}.
We take these parameters to be the ones of
\cite[Theorem~1.3]{DL21}
(cf.\ the paragraph after \emph{loc.\ cit.}).
In any case, notice that $\Pi$ will send
an eigenvector $v\in \DOne$
 with $H$-eigencharacter $\chi$ 
to an eigenvector $\Pi v$
with $H$-eigencharacter $\chi^s$.

From now on, 
we will always assume that $\rep$ is reducible
and $0$-generic.
Following \cite[§4]{Bre14},
if $\lambda = (\lambda_0(x_0), \dots,\lambda_{f-1}(x_{f-1}))$
is an $f$-tuple, with $\lambda_{i} (x_i)
\in  \mathbb{Z} \pm x_i$
for $i \in \{0, \dots, f-1\}$,
we define
\begin{equation}
  \label{eq:e-lambda-def}
\begin{aligned}
e(\lambda) &\defeq \frac{1}{2} 
\left( \sum _{i=0}^{f-1} p^{i}(x_i-\lambda_i(x_i))\right)
\text{ if }\lambda_{f-1}(x_{f-1})
\in \mathbb{Z}+x_{f-1} , \\
e(\lambda) &\defeq \frac{1}{2} 
\left(p^{f}-1+ \sum _{i=0}^{f-1} p^{i}(x_i-\lambda_i(x_i))\right)
\text{ if }\lambda_{f-1}(x_{f-1})
\in \mathbb{Z}-x_{f-1}.
\end{aligned}
\end{equation}

Recall the set $\mathscr{P}$ parametrising $\DOne$
defined in \cite[§4]{Bre14} (and denoted there by $\mathscr{PD}$),
and the subset $\mathscr{D} \subseteq \mathscr{P}$
parametrising $\Wr$.
To a given $\lambda \in \mathscr{P}$
(which is of the form $ \lambda_{i} (x_i)
\in  \mathbb{Z} \pm x_i$
for $i \in \{0, \dots, f-1\}$)
we associate the Serre weight
$\sigma_{\lambda} \defeq(\lambda(r_0), \dots,\lambda(r_{f-1}))
\otimes \det^{e(\lambda)(r_0,\dots,r_{f-1})}$,
and let 
$\chi_{\lambda} \colon I \to \F^{\times}$
be the character acting on 
$\sigma_\lambda^{I_1}$.
It follows from \cite[Proposition~4.2]{Bre14} 
that $\lambda \mapsto \chi_\lambda$
is a bijection between $\mathscr{P}$
and the set of eigencharacters
of $\DOne$,
and it follows from \cite[Proposition~A.3]{Bre14} 
that $\lambda \mapsto \sigma_\lambda$
is a bijection between $\mathscr{D}$ and $\Wr$.

We let $\mathscr{D}^{\mathrm{ss}}
\subseteq \mathscr{P}^{\mathrm{ss}}$
be the corresponding sets for the semisimplification 
$\rep^{\mathrm{ss}}$ of $\rep$,
so $\mathscr{P} \subseteq \mathscr{P}^{\mathrm{ss}}$
and $\mathscr{D} \subseteq \mathscr{D}^{\mathrm{ss}}$.

More precisely, recall that
$\mathscr{P} ^{\mathrm{ss}}$ is the set of $f$-tuples
$ \left( \lambda_0 (x_0 ), \dots, \lambda_{f-1}(x_{f-1}) \right)$ such that,
for all $i \in \{0, \dots,f-1\}$:
\begin{enumerate}[(i)]
\item 
  \label{condition:Pss-i}
$\lambda_i(x_i) \in \{x_i,x_i+1, x_i+2 , p-3-x_i, p-2-x_i, p-1 - x_i\}$;
\item 
  \label{condition:Pss-ii}
$\lambda_i(x_i) \in \{x_i, x_i+1, x_i+2\} \implies
\lambda_{i+1}(x_{i+1}) \in \{x_{i+1} , x_{i+1}+2,  p-2-x_{i+1}\}$;
\item 
  \label{condition:Pss-iii}
$\lambda_i(x_i) \in \{p-3-x_i, p-2-x_i, p-1- x_i\} \implies
\lambda_{i+1}(x_{i+1})\in \{x_{i+1}+1, p-3-x_{i+1}, p-1 - x_{i+1} \}$,
\end{enumerate}
with the conventions 
$x_f = x_0$, $\lambda_f(x_f)=\lambda_0(x_0)$.
Moreover, $\mathscr{D} ^{\mathrm{ss}} \subseteq \mathscr{P} ^{\mathrm{ss}}$ is the subset
such that $\lambda_i (x_i) \in \{x_i,x_i+1,p-3-x_i,p-2-x_i\}$
for all $i \in \{0, \dots,f-1\}$.

Then, 
it follows from \cite[Proposition~4.2]{Bre14} that
there exists a unique subset $\Jrep \subseteq \{0, \dots,f-1\}$
such that 
\begin{align}
  \label{eq:D-def-Jrep}
\mathscr{D} &= \left\{ \lambda \in \mathscr{D} ^{\mathrm{ss}} \mid \lambda_i (x_i) \in \{x_i+1, p-3- x_i\} \implies i \in \Jrep \right\}, \\
  \label{eq:P-def-Jrep}
\mathscr{P} &= \left\{ \lambda \in \mathscr{P} ^{\mathrm{ss}} \mid \lambda_i (x_i) \in \{x_i+2, p-3- x_i\} \implies i \in \Jrep \right\}.
\end{align}
In particular, $ |\Wr| = 2^{|\Jrep|}$,
and $\rep$ is semisimple if and only if 
$\Jrep = \{0, \dots,f -1 \}$.

To a Serre weight $\sigma\in \Wr$ 
we can associate a positive integer $\ell(\sigma)$
as follows. 
\begin{definition}
  \label{def:J-and-ell}
If $\lambda \in \mathscr{P} ^{\mathrm{ss}}$,
following \cite[Eq.~(10)]{BHHMS4} we define 
\begin{equation}
  \label{eq:J-lambda-def}
J_{\lambda} \defeq \{j \in \{0, \dots, f-1\}
\mid
\lambda_j(x_j) \in \{x_j + 1, x_j + 2,
p-3-x_j\}\},
\end{equation}
and we set the length $\ell(\lambda) $ 
of $\lambda$ as 
$\ell(\lambda) \defeq |J_{\lambda}|$.
By \Cite[§11]{BP12},
the assignment $\lambda \mapsto J_{\lambda}$
gives a bijection between $\mathscr{D}^{\mathrm{ss}}$
and the set of subsets of $\{0, \dots, f-1\}$.
Since $\mathscr{D} ^{\mathrm{ss}}$ also parametrises 
$\Wss$, we sometimes write $J_{\tau} \defeq J_{\lambda}$
and $\ell(\tau) \defeq \ell(\lambda)$
for the unique $\lambda \in \mathscr{D} ^{\mathrm{ss}}$ parametrising $\tau \in \Wss$.

For $0 \le \ell \le f$, we set 
$\mathscr{P} ^{\mathrm{ss}}_{\ell} \defeq 
\{\lambda \in \mathscr{P} ^{\mathrm{ss}} \mid \ell(\lambda)=\ell\} $,
$\mathscr{P}_{\ell} \defeq 
\mathscr{P} \cap \mathscr{P} ^{\mathrm{ss}}_{\ell}$,
and we set 
$\Wss_{\ell} \defeq \{\sigma\in \Wss \mid \ell(\sigma)=\ell\}$,
$\Wr_{\ell} \defeq \Wr \cap \Wss_{\ell}$.

By \emph{loc.\ cit.}\ the assignment $\lambda \mapsto J_\lambda$
restricts to a bijection between $\mathscr{D}$
and the set of subsets of $\Jrep$.
\end{definition}

Suppose now and for the rest of the section
that $\rep$ is reducible nonsplit.
For $\ell \in \{-1, \dots, f\}$,
let $\DZeroFil[\ell]$ be the largest
$\Gamma$-subrepresentation of $\DZero$
such that no $\tau \in \Wss$
with $\ell(\tau) > i$ occurs as a subquotient of
$\DZeroFil[\ell]$.
By \cite[Proposition~5.2]{Hu16},
the filtration
\begin{equation}
  \label{eq:DZero-filtration-definition}
0 =  
\DZeroFil[-1] \subsetneq
\DZeroFil[0] \subsetneq \dots 
\subsetneq
\DZeroFil[i] \subsetneq \dots \subsetneq
\DZeroFil[f] = \DZero
\end{equation}
is the unique filtration of the 
$\Gamma$-representation $\DZero$ such that,
for all $\ell \in \{0, \dots, f\}$,
the $\Gamma$-representation
$\DZerogr[\ell] \defeq \DZeroFil[\ell] / \DZeroFil[\ell-1]$
embeds inside 
$\DssZeroEll[\ell] \defeq 
\bigoplus_{\tau \in \Wss_\ell} \Dssx[\tau]$,
and 
\begin{equation}
    \label{eq:socle-DZerogr}
\socR \DZerogr[\ell] = 
\bigoplus_{\tau \in \Wss_\ell} \tau.
\end{equation}
For $\ell \in \{0, \dots, f\}$,
set $\DssZeroFil[\ell] \defeq 
\bigoplus_{\ell' \le \ell} \DssZeroEll[\ell']$.
We recall
\cite[Eq.~(67), (68)]{BHHMS4}
(for $i = \ell$) below:
\begin{align}
  \label{eq:D-Dss-gr-comparison}
\JH(\DZerogr[\ell]) &= \JH(\DZero) \cap 
\JH(\DssZeroEll[\ell]),\\
  \label{eq:D-Dss-Fil-comparison}
\JH(\DZeroFil[\ell]) &= \JH(\DZero) \cap 
\JH(\DssZeroFil[\ell]).
\end{align}
Since $\DZero$ is multiplicity free,
the decomposition \eqref{eq:diagr-def}
implies that $\DZeroFil[\ell]$ decomposes as
\begin{equation}
  \label{eq:DFil-decomposition}
\DZeroFil[\ell] = \bigoplus_{\sigma \in \Wr} 
\DxFil[\sigma]{\ell},
\end{equation}
with $\DxFil[\sigma]{\ell} \defeq 
\Dx[\sigma] \cap \DZeroFil[\ell]$.
Similarly, $\DZerogr[\ell] $ decomposes as
\begin{equation}
\label{eq:Dgr-decomposition}
\DZerogr[\ell] = \bigoplus_{\tau \in \Wss_\ell} 
\Dxgr[\tau]{\ell},
\end{equation}
with $\Dxgr[\tau]{\ell} \defeq \DZerogr \cap \Dssx[\tau]$.

We also recall 
\cite[Eq.~(70)]{BHHMS4}
for $i = \ell$:
\[
\frac{\DxFil[\sigma]{\ell}}{\DxFil[\sigma]{\ell -1}}
= \bigoplus_{\tau \in \Wss_{\ell},\ J_\sigma = \Jrep \cap J_\tau} 
\Dxgr[\tau]{\ell}.
\]

We conclude this section by introducing
weight cycling on $\Wr$.
The following notation comes from \cite[p.~13]{BP12}.
\begin{definition}
\label{def:s-at-D0-level}
Given a Serre weight $\sigma $,
if the character $\chi:I\to \F ^\times$
describes the action of $I$ on $\sigma^{I_1} $,
then there exists a unique Serre weight, which we denote by $\sigma^{[s]} $,
such that $\sigma^{[s]} $ is distinct from $\sigma $
and such that $\chi^s$ describes the action of $I $ on $(\sigma^{[s]})^{I_1} $.
\end{definition}
Like before, if $\sigma\in \Wss $,
and if $\chi \in \mathscr{D} ^{\mathrm{ss}}$
is the character given by the action of $I $ on $\sigma^{I_1} $,
then $\chi^s \neq \chi$ 
follows from the genericity assumption on $\rep $.
In particular, the condition 
$\sigma^{[s]}\neq \sigma$
is automatic.

For a subset $J \subseteq \{0, \dots, f-1\}$,
set $\delta (J) \defeq  \{j \mid j+1 \in J\}$
(where the indices are taken modulo $f$).
If $\lambda \in \mathscr{D} ^{\mathrm{ss}}$
has corresponding subset $J_\lambda \subseteq \{0, \dots,f-1\}$,
then we let $\delta(\lambda) \in \mathscr{D} ^{\mathrm{ss}}$
be defined by $\delta(\lambda) _j \defeq \lambda_{j+1}$
for all $j \in \{0, \dots, f-1\}$
(where the indices are taken modulo $f$),
and we notice that $\delta(J_\lambda)$
is the subset corresponding to $\delta(\lambda)$.
Notice that this notation is compatible with
the $\delta_r(-)$ 
defined in the second paragraph of \cite[§15]{BP12}.
If moreover $\sigma = \sigma_\lambda \in \Wss$
corresponds to $\lambda$, 
then we write $\delta(\sigma) \in \Wss$
for the Serre weight corresponding to $\delta(\lambda)$.
By \cite[Lemma~15.2]{BP12}
applied to the case $\tau=\sigma$
(which in \emph{loc.\ cit.}\ are defined
in the paragraph before \cite[Lemma~15.1]{BP12};
note that in this case
we have $\mathcal{S} ^+=\mathcal{S} ^-=\emptyset $,
with the notation of the paragraph before
\cite[Lemma~15.1]{BP12}),
we see that $\delta(\sigma)$ is the unique 
Serre weight in $\Wss$ 
such that there exists a $\GR$-equivariant inclusion
$I(\delta(\sigma), \sigma^{[s ] })
\hookrightarrow \DssZero$.

\subsection{\texorpdfstring{$(\varphi, \Gamma)$}{(phi, Gamma)}-modules}
In this section we review the functor $\Dvee$
of \cite{Bre15}
taking values in (pro-)étale 
$(\varphi,\Gamma)$-modules.
Let 
\begin{equation}
  \label{eq:N0-def}
N_0 \defeq   \begin{pmatrix} 1 & \mathcal{O} _K \\ 
0 & 1\end{pmatrix}  \cong \mathcal{O} _K,
\end{equation}
which comes with the trace map 
$N_0 \cong \mathcal{O} _K 
\xtwoheadrightarrow{\Tr}
\mathbb{Z} _p,$
and let $N_1$ be its kernel.
\begin{definition}
  \label{def:Colmez_generalised}  
If $T_{2} \subseteq \GL_{2,K}  $
denotes the maximal torus of diagonal matrices, then we define the morphisms of algebraic groups
\begin{align*}
  \xi: \mathbb{G} _m 
  & \to T_{2}   
  &
  \theta: T_{2} 
  & \to \mathbb{G} _m \\
  x & \mapsto \begin{pmatrix} x & \\ & 1 \end{pmatrix},
  &
  \begin{pmatrix} x_1 & \\ & x_2 \end{pmatrix} & \mapsto x_1 ,
\end{align*}
which coincide with 
\cite[example~2.3]{BHHMS2} for $n=2$.
Observe that, for $x\in \mathbb{Z} _p\setminus \{0\} $,
we have
$\xi(x)N_1\xi(x^{-1} )\subseteq N_1 $.
\end{definition}

If $\pi$ is a smooth representation  of $\GF$
over $\F$,
we can regard $\pi^{N_1}$
as a smooth representation of 
$N_0/N_1 \cong \mathbb{Z} _p$.
With this in mind, from now on we fix the identification
\[
  \F[\![ N_0/N_1]\!] \xrightarrow{\sim} \F [\![ \mathbb{Z} _p ]\!] \cong \F [\![ X ]\!],
\]
where the last isomorphism sends $X$ to $[1] - 1$.

Define $\F [\![X]\!] [F]$ to be the noncommutative polynomial ring over $\F [\![X]\!] $ with the relation $FS(X)= S(X^p)F$ for $S(X)\in \mathbb{F} [\![X]\!] $.
Then, 
we enrich the structure of  $\F [\![X]\!] $-module on
$\pi^{N_1} $
to an $\F [\![X]\!] [F]$-module structure
by making $F$ act as 
$F(v) = \sum_{n_1\in N_1/{\xi(p)} N_1 {\xi(p)}^{-1}} 
n_1\xi(p)v \in \pi^{N_1 } $ on $v\in \pi^{N_1 }  $.

We can also endow $\pi^{N_1} $ with an action of $\mathbb{Z} _p ^ \times  $, 
by making $x \in \mathbb{Z} _p ^ \times  $
act by $\xi(x) $.
The fact that, for 
$x\in \mathbb{Z} _p ^ \times   $,
$\xi(x) $ normalises $N_1$ implies that
the action of $\mathbb{Z} _p ^ \times   $ commutes with $F$, while
the identity
\[
  \begin{pmatrix} x & 0 \\0 & 1 \end{pmatrix}   
  \begin{pmatrix} 1 & 1 \\0 & 1 \end{pmatrix} 
  \begin{pmatrix} x ^{-1} &0 \\0 & 1 \end{pmatrix} 
  =
  \begin{pmatrix} 1 & x \\0 & 1  \end{pmatrix},
\]
implies the commutation relation
\[
  \xi(x) \circ (1+X) = (1+X)^x \circ \xi(x)
\]
of endomorphisms of $\pi^{N_1} $.

We adopt the convention of
\cite{Bre15}
of denoting by $\PhG$ the category of 
finite-dimensional étale $(\varphi, \Gamma) $-modules
over $\F (\!(X)\!)  $,
and of denoting by $\PhGhat $ its pro-completion. 
Again following \emph{loc.\!\ cit.},
we now define a contravariant functor $\Dvee (-)$.
\begin{definition}
  \label{def:Dvee}
Given $\pi$ a smooth representation of $\GF $
over $\F $,
we can consider the set of all finitely generated 
$\F [\![X]\!] [F] $-submodules
$M \subseteq \pi^{N_1} $
which are stable by the action of  
$\mathbb{Z} _p^\times  $ 
and which are moreover \emph{admissible},
i.e.\ such that $M[X]= \{m\in M\mid Xm=0\} $
is finite dimensional over $\F  $.
If $M^\vee \defeq  \Hom_{\F } (M,\F ) $ is the 
algebraic $\F$-linear dual of $M $,
we define
\[
  D^\vee_\xi(\pi) \defeq  \varprojlim_{M} M^\vee[X ^{-1}]
\]
to be the inverse limit over this set.
It naturally lies in $\PhGhat$ by
\cite[Proposition~2.7(i)]{Bre15}.
\end{definition}

\subsection{The main four hypotheses}
  \label{sec:hypotheses}
Throughout this section, we assume that 
$\rep$ is reducible 
and $0$-generic.
We will use the following
set of hypotheses on
an admissible smooth representation 
$\pi$ of $\GF$ over $\mathbb{F}$
with a central character:
\begin{enumerate}[(i)]
  \item \label{hypothesis:i}
    there is an integer $r\ge 1$ such that
there is an isomorphism of diagrams
\[
(\pi^{I_1} \hookrightarrow  \pi^{K_1})
\cong \Diagr^{\oplus r};
\]
  \item \label{hypothesis:ii} for any $\lambda\in \mathscr{P} $ , we have an equality of multiplicities
\[
  [\pi[\mI^{3}]\colon \chi_{\lambda}] = [\pi[\mI]\colon \chi_\lambda];
\]

  \item \label{hypothesis:iii}
if we let 
$\E[j](-) \defeq  \Ext^j_{\Lambda} (-, \Lambda)
= \Ext^j_{I/Z_1} (-, \Lambda)$,
then the $\Lambda$-module
$\pi^\vee$ is \textit{essentially self-dual}
of grade $2f$,
i.e.\ there exists a $\GL_2(K)$-equivariant isomorphism of $\Lambda$-modules
\[
\E(\pi^\vee)\cong \pi^\vee \otimes (\mytwist)
\]  
where the action of $\GF$ on the
left-hand side is defined by Kohlhaase in 
\cite[Proposition~3.2]{Koh17};

  \item \label{hypothesis:iv}
for any smooth character $\chi \colon I \to \F ^\times $
and for $i \ge 0$
we have $\Ext^i_{I/Z_1} (\chi,\pi)\neq 0$
if and only if $[\pi[\mI]:\chi]\neq 0$,
in which case
\[
  \dim_{\F} \Ext^i_{I/Z_1} (\chi,\pi)=
  \binom{2f}{i} r,
\]
where $r \ge 1$ is the integer of \ref{hypothesis:i},
and where $\Ext^{i}_{I/Z_1}(\chi,-) $ 
is computed in the category of smooth $I/Z_1$-representations.
\end{enumerate}

Notice that, by the proof of \cite[Corollary~5.3.5]{BHHMS1},
\ref{hypothesis:ii} implies that the 
$\grL$-module $\grm(\pi^\vee)
\defeq \bigoplus _{i \ge 0}
\mI^{i}\pi^\vee /\mI^{i+1}\pi^\vee$
is annihilated by $J$, and so its multiplicity 
$\mpp[q](\grm(\pi^\vee))$ 
(cf.\ \Cref{def:multiplicity})
is well-defined for each minimal prime
$\mathfrak{q}\subseteq \Rbar$.

\section{The rank of \texorpdfstring{$(\varphi, \Gamma)$}{(phi, Gamma)}-modules
associated to subrepresentations of \texorpdfstring{$\pi$}{pi}}
The main result of this section is 
\Cref{cor:Wang},
which computes the $\FX$-rank of the cyclotomic 
$(\varphi, \Gamma)$-module $\Dvee (\pi_1 )$
associated to a
subrepresentation $\pi_1 $ of $\pi$
in terms of certain numerical invariants 
defined in \eqref{eq:VFil-rFil-def} below.
Throughout this section we assume that 
$\rep$ is reducible and $0$-generic,
and we assume that $\pi$ satisfies
hypotheses
\ref{hypothesis:i} to \ref{hypothesis:iv}
of \Cref{sec:hypotheses}.

\subsection{\texorpdfstring{$K_1$}{K1}- and \texorpdfstring{$I_1$}{I1}-invariants of subrepresentations of \texorpdfstring{$\pi$}{pi}}
  \label{sec:K1-inv}
Throughout this section we assume that 
$\rep$ is reducible nonsplit and $0$-generic,
and we assume that $\pi$ satisfies
hypotheses
\ref{hypothesis:i} to \ref{hypothesis:iv}
of \Cref{sec:hypotheses}.
Remember that,
in the paragraph after \eqref{eq:diagr1-def},
we chose the structure of
diagram on $\Diagr$ to be the same
as \cite[Theorem~1.3]{DL21},
which by the last paragraph of
\cite[§3.4.1]{BHHMS2}
coincides with the one of
\cite[Theorem~3.93]{BHHMS2}.
Set $\V\defeq\F^{r}$.
By hypothesis \ref{hypothesis:i}
of \Cref{sec:hypotheses}
we fix once and for all
an isomorphism of diagrams
\begin{equation}
  \label{eq:ip-def}
\ip :
( \pi^{I_1}\hookrightarrow\pi^{K_1})
\cong
\V \otimes_{\F} 
\Diagr,
\end{equation}
where the diagram structure on
$\V \otimes_{\F} \Diagr$ 
is induced from the one on $\Diagr$.
If $\pi_1 $ is a subrepresentation of $\pi$,
then $\iota_\pi(\pi_1 ^{I_1} \hookrightarrow \pi_1 ^{K_1})$
is a subdiagram of $\V \otimes_{\F}\Diagr$,
and the goal of this section is to
determine which subdiagrams can occur.
This is accomplished in \Cref{thm:srep-structure}.

For $j \in \{0, \dots, f-1\}$,
we denote by $\alpha_j$ the 
$I$-character 
$\alpha_j \left( 
\left( 
\begin{smallmatrix}
a & b \\
pc & d
\end{smallmatrix}\right)\right)\defeq 
\big( \overline{a}\overline{d}^{-1} \big)^{p^j}$.
Given a character $\chi \colon I \to \F ^\times $
and two subsets $J_1 , J_2 \subseteq \{0, \dots,f-1 \}$
such that $J_1 \cap J_2 = \emptyset $,
set
\[
  \chi^{J_1 ,J_2 } \defeq \chi 
\prod_{j \in J_1 }  \alpha_j^{-1}
\prod_{j \in J_2 }  \alpha_j.
\]
We let $\Wchi{}$ be the $I$-representation of
\cite[Lemma~4.3.1]{BHHMS4},
i.e.\ the unique $I$-representation with 
socle $\chi$ and
cosocle $\chi^{J_1 ,J_2 } $ and such that 
the $d$-th socle layer is given by 
\[
\bigoplus_{
\substack{
J_1' \subseteq J_1, J_2' \subseteq J_2,\\
|J_1 '| + |J_2 '| = d} 
}
\chi
\prod_{j \in J_1' }  \alpha_j^{-1}
\prod_{j \in J_2' }  \alpha_j.
\]

Following \cite[§4.3.2]{BHHMS4},
for $\mu \in \mathscr{P}$ we define
\begin{align}
  \label{eq:Ymu-def}
\Yn[\mu] &\defeq \left\{ j \in \{ 0, \dots,f-1\} 
\mid \mu_j (x_j) \in \{x_j ,x_j +1 ,  p-2- x_j, p-3- x_j\}\right\} \cup \Jrep^c,
\\
  \label{eq:Zmu-def}
\Zn[\mu] &\defeq \left\{ j \in \{ 0, \dots,f-1\} 
\mid \mu_j (x_j) \in \{x_j +1 , x_j +2, p-1- x_j, p-2- x_j\}\right\} \cup \Jrep^c.
\end{align}

We want to construct an $I$-equivariant injection
 $\V \otimes_{\F} \Wchi[J_1,J_2]{\mu}
 \hookrightarrow \pi|_I$.
We begin with the following lemma.
\begin{lemma}
  \label{lemma:extend-to-Wchi}
Suppose that $\rep$ is $2$-generic.
Let $\mu \in \mathscr{P}$,
and let $J_1 \subseteq \Yn[\mu]$ and
$J_2 \subseteq \Zn[\mu]$ be subsets satisfying
$J_1 \cap J_2 = \emptyset $.
Then, we have
$\JH(W(\chi_\mu,\chi_\mu^{J_1,J_2})) \cap 
\JH(\pi^{I_1}) = \{\chi_\mu\}$,
and the restriction map
\begin{equation}
  \label{eq:Wchi-iso}
  \Hom_{I/Z_1}(\Wchi[J_1,J_2]{\mu}, \pi|_I)
\to \Hom_{I/Z_1}(\chi_\mu, \pi|_I)
=\Hom_{I/Z_1}(\chi_\mu, \pi^{I_1})
\end{equation}
is an isomorphism.
\end{lemma}

\begin{proof}
The first claim follows from 
\cite[Lemma~2.3.6(ii)]{BHHMS4} with $m=1$.
For the second one,
it is enough to show that the restriction map
$ \Hom_{I/Z_1}(\Wchi[J_1,J_2]{\mu}, \pi|_I)
\to \Hom_{I/Z_1}(\chi_\mu, \pi|_I)$ 
is an isomorphism, 
and we do this by showing that 
$\Ext^{\varepsilon}_{I/Z_1}
(\Wchi[J_1,J_2]{\mu}/\chi_\mu, \pi|_I)=0$
for $\varepsilon=0,1$.
The first statement implies that 
$\JH(\Wchi[J_1,J_2]{\mu}/\chi_\mu) \cap \JH(\pi^{I_1})
= \emptyset$,
we conclude by a dévissage on 
$\JH(\Wchi[J_1,J_2]{\mu}/\chi_\mu)$ together with
hypothesis \ref{hypothesis:iv} of
\Cref{sec:hypotheses}.
\end{proof}

By construction of $\mathscr{P}$
(cf.\ the paragraph after \eqref{eq:e-lambda-def}),
for all $\mu \in \mathscr{P}$
we can fix once and for all
a nonzero $H$-eigenvector $v_\mu \in \DOne$
of $H$-eigencharacter $\chi_\mu$,
and this choice in turn determines
an $H$-equivariant isomorphism
\begin{equation}
    \label{eq:v-iso-def}
\alphaI \colon  
\bigoplus_{ \mu \in \mathscr{P}}
\chi_{\mu}
 \xrightarrow{\sim} \DOne.
\end{equation}
Denote by 
$\alphaI_{\mu} \defeq \alphaI |_{\chi_\mu}
\colon \chi_\mu \hookrightarrow \DOne$
the restriction of $\alphaI$ to $\chi_\mu$.
As $\DOne$ is multiplicity free, we have
\begin{equation}
    \label{eq:explicit-isotypic-component}
  \Hom_{I/Z_1}(\chi_\mu, \DOne) = \F \alphaI_\mu.
\end{equation}
Suppose now that $\rep$ is $2$-generic,
and let $\mu \in \mathscr{P}$.
For all $J_1 \subseteq \Yn[\mu]$, 
$J_2 \subseteq \Zn[\mu]$ such that 
$J_1 \cap J_2 = \emptyset$,
a straightforward computation gives
\begin{align}
  \label{eq:iso-chi-isotypical}
\Hom_{I/Z_1}&(\V \otimes_{\F}\Wchi[J_1,J_2]{\mu}, \pi|_I)
 \cong \V^\vee \otimes_{\F} 
 \Hom_{I/Z_1}(\Wchi[J_1,J_2]{\mu}, \pi|_I)
\\
\nonumber
& \overset{\eqref{eq:Wchi-iso}}{\cong }
\V^\vee \otimes_{\F} 
\Hom_{I/Z_1}(\Wchi[J_1,J_2]{\mu}, \pi^{I_1 })
\xrightarrow[\sim]{\id_{\V^\vee } \otimes \ip}
\V^\vee \otimes_{\F} 
\Hom_{I/Z_1}(\chi_\mu, \V \otimes_{\F}\DOne) \\
\nonumber
&
\cong  \V^\vee \otimes_{\F} \V \otimes_{\F}
\Hom_{I/Z_1}(\chi_\mu, \DOne) 
\overset{\eqref{eq:explicit-isotypic-component}}{=} 
\V^\vee  \otimes_{\F} \V \otimes_{\F} \F \alphaI_\mu
\cong \Hom_{\F} \left( \V,\V \right) \otimes_{\F} \F \alphaI_\mu.
\end{align}
Denote by $\bm[\mu]$ the unique 
$I$-equivariant homomorphism 
$\bm[\mu] \colon \V \otimes_{\F} \Wchi[J_1,J_2]{\mu} 
\to \pi|_I$ 
that is sent to 
$\id_{\V} \otimes \alphaI_\mu \in 
\Hom_{\F} \left( \V,\V \right)
\otimes_{\F} \F \alphaI_\mu$
via the composition with \eqref{eq:iso-chi-isotypical}.
Notice that the composition 
$ \ip \circ (\bm[\mu]|_{\chi_\mu}) \colon
\V \otimes_{\F}\chi_\mu \to \V \otimes_{\F}\DOne$
is $\id_{\V} \otimes \alphaI_{\mu}$
by construction.
In particular, $\bm[\mu]$ is injective
because it is so on the $I$-socle
$\V \otimes_{\F}\socI \left(  
\Wchi[J_1,J_2]{\mu}\right) =
\V \otimes_{\F}\chi_\mu$.
For a nonzero $w \in \V$, 
we denote by $\bm[\mu,w]$ the restriction
of $\bm[\mu]$ to $\F w \otimes_{\F} 
\Wchi [J_1 , J_2 ]{\mu}$.

We also define $\btm[\mu]$ to be the $\GR$-equivariant homomorphism 
\begin{equation}
  \label{eq:beta-tilde-def}
\btm[\mu] \colon
\V \otimes_{\F}\IndI\Wchi[J_1,J_2]{\mu}\to\pi|_{\GR}
\end{equation}
corresponding to $\bm[\mu]$ 
under Frobenius reciprocity,
and we denote by $\btm[\mu, w]$
the restriction of $\btm[\mu]$ to 
$\F w \otimes_{\F} \IndI \Wchi[J_1,J_2]{\mu}$
(or equivalently the $\GR$-equivariant homomorphism
$\btm[\mu, w] \colon
\F w \otimes_{\F}\IndI\Wchi[J_1,J_2]{\mu}\to\pi|_{\GR}$
corresponding to $\bm[\mu, w]$ 
under Frobenius reciprocity).

Let $\mathcal{P} \defeq \mathcal{P} ( x_0 , \dots, x_{f -1 })$
be the set of $f$-tuples defined in 
\cite[\S 2]{BP12},
and let $\psi \colon I \to \F ^{\times}$
be a $1$-generic character 
(cf.\ the paragraph after \eqref{eq:weight-characterisation}
for the definition of $1$-genericity).
It follows from \eqref{eq:weight-characterisation}
and the genericity assumption 
that $\psi$ is uniquely
of the form
$\psi \colon \left( 
\begin{smallmatrix}
a & b \\
pc & d
\end{smallmatrix}\right)
\mapsto \overline{a}^{s} 
\eta (\overline{a}\overline{d})$,
for some character $\eta \colon \mathbb{F}_q ^{\times}
\to \F ^{\times}$
and some integer $s = \sum_{j = 0}^{f-1} p^j s_j$
with $1 \le s_j \le p -3 $.
We can explicitly compute $\psi^s$:
\begin{align}
    \label{eq:psi-s-computation}
\psi^s \left( 
\left( 
\begin{smallmatrix}
a & b \\
pc & d
\end{smallmatrix}\right)\right) 
&=  
\psi \left( 
\left( 
\begin{smallmatrix}
d & c \\
pb & a
\end{smallmatrix}\right)\right) =  
\overline{d}^{\sum_{j = 0}^{f-1} p^j s_j}
\eta (\overline{a}\overline{d})
= \\
\nonumber
& = 
\overline{a}^{\sum_{j = 0}^{f-1} p^j(p -1 - s_j)}
\eta(\overline{a} \overline{d})\cdot
(\overline{a} \overline{d})
^{\sum_{j = 0}^{f-1} p^j s_j} 
.
\end{align}
Then, by \cite[Lemma~2.2]{BP12}
applied to $\chi = \psi^s$
(hence $r_j = p-1- s_j$,
$\eta = \eta \cdot(\ )^{\sum_{j = 0}^{f-1} p^j s_j}$
in \emph{loc.\ cit.}\ by \eqref{eq:psi-s-computation}),
there is a bijection 
\begin{align}
  \label{eq:P-bijection-1}
\mathcal{P}\xrightarrow{\sim}
&
\JH \left( \IndI \psi \right) \\
\nonumber
\xi  \mapsto 
&
\xi (p -1 - s_0, \dots, p -1 - s_{f-1}) 
\otimes 
\operatorname{det}^{ 
e(\xi)(p- 1 -s_0 , \dots, p -1-s_{f-1})+ 
\sum_{j = 0}^{f-1} p^js_j}\eta = \\
  \label{eq:xi-xi-c-equality}
& = 
\xi^c \otimes 
\operatorname{det}^{ 
e(\xi^c)(s_0 , \dots, s_{f-1})}
\eta,
\end{align}
where $\xi^c \defeq 
\xi (p -1 - x_0, \dots, p -1 - x_{f-1}) \in \mathcal{P}$,
and where in \eqref{eq:xi-xi-c-equality} 
we have used the identity
$e (\xi^c) - e (\xi) \equiv
\sum_{j = 0}^{f -1} p^j \xi_j (x_j)
\mod (p^f -1)$
(which follows directly from \eqref{eq:e-lambda-def}
after an elementary computation)
together with \eqref{eq:psi-s-computation}.
Moreover, the assignment
\begin{equation}
  \label{eq:P-bijection-2}
\xi \in \mathcal{P} \mapsto 
\mathcal{S} (\xi) \defeq
\{j \in \{0, \dots, f-1\}  \mid 
\xi_j (x_j) \in \{x_j -1 , p -1 -x_j\}\} 
\end{equation}
gives a bijection between $\mathcal{P}$
and the set of subsets of $\{ 0, \dots, f -1 \}$.
The notation $\mathcal{S} (\xi)$ follows 
\cite[\S 19]{BP12},
but notice that it is related to the subset
$J(\xi)$ of 
\cite[\S 2]{BP12} and
\cite[\S 3]{HW22}
by
$ J(\xi) = 
\delta (\mathcal{S} (\xi))$
(cf.\ the last paragraph of \Cref{sec:diagrams}
for the definition of $\delta (-)$).
Our conventions follow the ones of 
\cite[Remark~3.1.1]{BHHMS4},
in particular $\emptyset $
parametrises the socle of $\IndI \psi$, and
$ \{ 0, \dots, f -1 \} $
parametrises the cosocle of $\IndI \psi$. 

We recall here \cite[Lemma~3.1.3]{BHHMS4} and
\cite[Lemma~4.1.1]{BHHMS4}.
\begin{lemma}[{\cite[Lemma~3.1.3]{BHHMS4}}]
  \label{eq:3.1.3-recalled}
If $\lambda \in \mathscr{P} ^{\mathrm{ss}}$, 
then $\chi_ \lambda$ occurs in $\DssOnex[\tau]$,
where $\tau \in \Wss$ is determined by 
$J_{\tau} = J_{\lambda}$ via \eqref{eq:J-lambda-def}.
Moreover, as a Jordan-H\"older constituent of $\IndI \chi _\lambda$,
$\tau$ is parametrised via 
\eqref{eq:P-bijection-1} and
\eqref{eq:P-bijection-2}
by the following subset of $\{ 0, \dots,f-1\}$:
\begin{equation}
  \label{eq:Xmu-ss-def}
X ^{\mathrm{ss}}(\lambda) \defeq 
\left\{ j \mid \lambda_j (x_j) \in \{x_j, x_j +1, p-2- x_j, p-3- x_j\}\right\}.
\end{equation}
\end{lemma}

\begin{lemma}[{\cite[Lemma~4.1.1]{BHHMS4}}]
  \label{eq:4.1.1-recalled}
If $\mu \in \mathscr{P}$, then $\chi_ \mu$ occurs in $\DOnex[\sigma]$,
where $\sigma \in \Wr$ is determined by $J_{\sigma} = \Jrep \cap J_{\mu}$ via \eqref{eq:J-lambda-def}.
Moreover, as a Jordan-H\"older constituent of $\IndI \chi _\mu$,
$\sigma$ is parametrised via 
\eqref{eq:P-bijection-1} and
\eqref{eq:P-bijection-2}
by the following subset of $\{ 0, \dots,f-1\}$:
\begin{equation}
  \label{eq:Xmu-def}
X(\mu) \defeq 
\left\{ j \mid \mu_j (x_j) \in \{x_j, p-2- x_j, p-3- x_j\}\right\} 
\cup 
\left\{ j \in \Jrep
\mid \mu_j (x_j) = x_j +1\right\}.
\end{equation}
\end{lemma}

If $\lambda \in \mathscr{P}^{\mathrm{ss}}$ parametrises $\chi_{\lambda}$,
then 
\begin{equation}
  \label{eq:lambda-s-def}
\lambda^{[s]} \defeq \left( 
p-1- \lambda_0 (x_0 ), \dots, p-1- \lambda_{f-1}(x_{f-1}) \right)
\in \mathscr{P}^{\mathrm{ss}}
\end{equation}
parametrises $\chi_{\lambda}^s$, 
using \eqref{eq:psi-s-computation}
(applied to $\xi = \chi_\lambda$)
and the identity 
$e (\lambda^{[s]}) - e (\lambda) \equiv
\sum_{j = 0}^{f -1} p^j \lambda_j (x_j)
\mod (p^f -1)$,
formally equivalent to the one
that appears two lines below \eqref{eq:xi-xi-c-equality};
the only difference is that $\xi, \xi^c \in \mathcal{P}$
and $\lambda, \lambda^{[s]} \in \mathscr{P}^{\mathrm{ss}}$
belong to different sets
(this also explains why we have chosen the notation
$\lambda^{[s]}$ instead of $\lambda^c$).

For $\tau \in \Wss$, we see from the last paragraph of 
\Cref{sec:diagrams}
that $I(\delta (\tau), \tau^{[s]}) 
\hookrightarrow \Dssx[\delta (\tau)]$.
Moreover, if $\tau$
is parametrised by 
$\lambda \in \mathscr{D}^{\mathrm{ss}}$,
then by \Cref{def:s-at-D0-level} 
(applied to $\sigma = \tau$)
we know that $\chi_\lambda^s$ occurs in $\tau^{[s]}$,
so a fortiori it occurs in
$I(\delta (\tau), \tau^{[s]})$,
so we deduce from \Cref{eq:3.1.3-recalled}
(applied to $\lambda = \lambda^{[s]}$,
using \eqref{eq:lambda-s-def})
that 
\begin{equation}
  \label{eq:delta-s-interplay}
\delta (J_\lambda) = J_{\delta(\lambda)}
= J_{\delta (\tau)} = J_{\lambda^{[s]}},
\end{equation}
where the first two equalities
follow from the last paragraph of
\Cref{sec:diagrams}.

The following lemma,
which will be used in the proof of
\Cref{thm:srep-structure},
generalises 
\cite[Lemma~4.3.9]{BHHMS4} to the case 
$r \ge 1$.
\begin{lemma}
  \label{lemma:beta-vector}
Keep the assumptions of 
\Cref{lemma:extend-to-Wchi},
and fix $w \in V \setminus \{0\}$. 
Then,
\begin{enumerate}[(i)]
\item the image of 
$\btm[\mu, w] \colon
\F w \otimes_{\F} \IndI
\Wchi[J_1,J_2]{\mu}
\to
\pi|_{\GR}$
has socle isomorphic to $\sigma \in \Wr$
where $\sigma$ is the Serre weight determined
by $J_{\sigma} = J_{\rep} \cap J_{\mu}$
(cf.\ \eqref{eq:J-lambda-def}), and
\begin{equation}
  \label{eq:soc-im-betatld-containment}
\ip\left(
\socR \im(\btm[\mu,w])\right)
= \F w \otimes_{\F}\sigma \subseteq 
\V \otimes_{\F}\DZero;
\end{equation}

\item 
we have $\sigma \in \JH(\IndI \chi_\mu)$,
parametrised (as a subquotient of
$\IndI \chi_\mu$) by $X(\mu)$, where $X(\mu)$
is defined in \eqref{eq:Xmu-def}.
\end{enumerate}
\end{lemma}
\begin{proof}
We essentially follow the proof of \cite[Lemma~4.3.9]{BHHMS4},
but need to be more careful about the different copies of the Serre weights of $\pi^{K_1}$.

(i)
Let $\btm[\mu,w]'$ be the restriction
of $\btm[\mu,w]$ to 
$\F w \otimes_{\F}\IndI\chi_\mu$.
We first show 
\eqref{eq:soc-im-betatld-containment}
with $\btm[\mu,w]'$ replacing
$\btm[\mu,w]$ in the left-hand side.

Let $\sigma \in \Wr$ be the Serre weight 
determined by $J_\sigma = J_{\rep} \cap J_\mu$
(cf.\ \eqref{eq:J-lambda-def}),
set $L \defeq \F w \subseteq \V$,
and follow $\btm[\mu,w]'
\in \Hom_{\GR}(L \otimes_{\F}\IndI\chi_\mu, \pi)$
in the following commutative diagram
\begin{equation}
  \label{eq:diagr-Wchi}
\begin{tikzcd}
\Hom_{\GR}(L \otimes_{\F}\IndI\!\chi_\mu, \pi) \ar[r, " \sim"] 
\ar[d, equals]
& \Hom_{I/Z_1}\!
(L \otimes_{\F}
\! \chi_\mu, \pi) 
\ar[d, equals] \\
\Hom_{\GR}(L \otimes_{\F}\IndI\!\chi_\mu, \pi^{K_1}) 
\ar[r, " \sim"] \ar[d, " \wr", "\ip"']
& \Hom_{I/Z_1}\!
(L \otimes_{\F}
\! \chi_\mu, \pi^{I_1})
\ar[d, " \wr"',  "\ip"]\\
\Hom_{\GR}(L \otimes_{\F}\IndI\!\chi_\mu, 
\V \otimes_{\F}
\! \DZero) 
\ar[r, " \sim"]
& \Hom_{I/Z_1}\!
(L \otimes_{\F}
\! \chi_\mu, \V \otimes_{\F}
\! \DOne) \\
\Hom_{\GR}(L \otimes_{\F}\IndI\!\chi_\mu, 
L \otimes_{\F}\!\Dx[\sigma]) 
\ar[r, " \sim"]
\ar[u, hook]
& \Hom_{I/Z_1}\!
(L \otimes_{\F}
\! \chi_\mu, L \otimes_{\F}
\!
\DOnex[\sigma]) \ar[u, hook],
\end{tikzcd}
\end{equation}
where the horizontal isomorphisms come from
Frobenius reciprocity.
We see that 
$\ip \circ (\btm[\mu,w]'|_{\chi_\mu})
\in \Hom_{I/Z_1}(L \otimes_{\F}\chi_\mu, 
\V \otimes_{\F}\DOnex[\sigma])$
by \cite[Lemma~4.1.1]{BHHMS4},
and its image is contained inside 
$L \otimes_{\F}\DOnex[\sigma]$
by construction of \eqref{eq:beta-tilde-def}.
In particular, 
$\ip \circ \btm[\mu,w]'
\in 
\Hom_{\GR}(L \otimes_{\F}\IndI \chi_\mu, 
L \otimes_{\F}\Dx[\sigma])$
by the commutativity of \eqref{eq:diagr-Wchi}.
The claim then follows from 
$\socR(\Dx[\sigma]) = \sigma$.

We conclude by showing that
\[
\JH \left( \IndI \left( 
\Wchi[J_1,J_2]{\mu}/\chi_\mu \right) \right) \cap 
\Wr = \emptyset,
\]
which follows from the first part of
\Cref{lemma:extend-to-Wchi} and
the fact that $\JH \left( \IndI \chi' \right)
\cap \Wr = \emptyset $ for any 
$\chi' \notin \pi^{I_1}$
by \cite[Proposition~4.2]{Bre14}.
Hence, 
$ \ip\left(
\socR \im(\btm[\mu,w])\right)$
can be no larger than
$ \ip\left(
\socR \im(\btm[\mu,w]')\right)$.

(ii)
That $\sigma \in \JH(\IndI \chi_\mu)$
follows from the proof of (i),
while the second part of the statement
follows from \Cref{eq:4.1.1-recalled}.
\end{proof}

We prepare to prove a generalisation of \cite[Lemma~4.3.11]{BHHMS4}.
Suppose that $\rep$ is $3$-generic,
so that $\chi_{\mu}$ is $2$-generic for any $\mu \in \mathscr{P}$.
Let $\lambda \in \mathscr{P} ^{\mathrm{ss}}$, and set
\begin{align}
  \label{eq:J1-J2-temp}
J_1  & \defeq  \{j \in  J_{\rep}^c  \mid 
\lambda_j(x_j)= p-3-x_j\} ,&
J_2  & \defeq  \{j \in J_{\rep}^c  \mid 
\lambda_j(x_j)= x_j+2\}.
\end{align}
Define $\mu \in \mathscr{P}^{\mathrm{ss}}$ by 
\begin{equation}
  \label{eq:mu-ad-hoc}
\mu_j(x_j) \defeq 
\begin{cases}
p-1- x_j & \text{if }j \in J_1  \\
x_j & \text{if }j \in J_2  \\
\lambda_j(x_j) & \text{otherwise}.
\end{cases}
\end{equation}
By the line after \cite[Eq.~(61)]{BHHMS4},
one has $\mu \in \mathscr{P}$,
$\chi_{\lambda} = \chi_\mu \prod_{j \in J_1 } \alpha_j^{-1} \prod_{j \in J_2 } \alpha_j$,
$|J_\mu| = |J_\lambda| - |J_1 | - |J_2 |$,
and $J_1 \subseteq \Yn[\mu]$,
$J_2 \subseteq \Zn[\mu]$, $J_1 \cap J_2 = \emptyset $.

Consider the morphism $\btm[\mu]$ of \eqref{eq:beta-tilde-def}
for
$J_1, J_2$ as in \eqref{eq:J1-J2-temp}. 
Fix a nonzero $w \in \V$, set $L \defeq \F w$,
and define $Q_L\defeq \im(\btm[\mu,w]) \subseteq \pi$,
which only depends on the line $L$.
If $\sigma \in \Wr$ is the Serre weight 
determined by $J_\sigma = J_{\rep} \cap J_\mu$,
then we have $\socR Q_L = \ip^{-1} \left( \F w \otimes_{\F} \sigma \right)$
by \Cref{lemma:beta-vector}.
Let $\tau \in \Wss$ 
be the Serre weight determined by 
$J_\tau = J_\lambda$,
with occurs inside $\IndI \chi_\lambda$
by \Cref{eq:3.1.3-recalled}.
The following lemma adapts
\cite[Lemma~4.3.11]{BHHMS4},
and will be used in the proof of
\Cref{thm:srep-structure}.
\begin{lemma}
  \label{lemma:I-and-QL}
Keep the above assumptions and assume that
$\rep$ is $6$-generic.
We have $L \otimes_{\F}I(\sigma,\tau) 
\subseteq \ip(Q_L^{K_1})$ and 
\[
\JH\left(Q_L/ \ip^{-1}\left( 
L \otimes_{\F}I(\sigma,\tau) \right)\right)
\cap \Wss = \emptyset.
\]
\end{lemma}
\begin{proof}
We know that $\IndI W(\chi_\mu, \chi_\lambda)$
is multiplicity free by \cite[Lemma~4.3.3]{BHHMS4},
and from \Cref{lemma:beta-vector}
we know that $Q_L$ has irreducible socle,
isomorphic to $\sigma$.
It follows that $Q_L$ 
(as an abstract $\GR$-representation) 
is the unique quotient of $\IndI W(\chi_\mu, \chi_\lambda)$
with $\GR$-socle $\sigma$,
in particular the abstract $\GR$-representation
$Q_L$ 
only depends on $\rep$, not on $\pi$
or (a fortiori) on $r$.

Therefore, $Q_L$ is isomorphic
to the $\GR$-representation $V$ of
\cite[Lemma~4.3.11]{BHHMS4},
and by \emph{loc.\ cit.}\ there exists an injection
$j \colon I(\sigma, \tau) \hookrightarrow Q_L$,
unique up to scalars,
which moreover takes values in $Q_L^{K_1}$
because $I(\sigma, \tau) \subseteq \DZero$
is $K_1 $-invariant.
We know from \emph{loc.\ cit.}\ that
\[
\JH\left(Q_L/ j\left( 
I(\sigma,\tau) \right)\right)
\cap \Wss = \emptyset,
\]
so it only remains to show that the composition
$I(\sigma, \tau) 
\xhookrightarrow{j} Q_L^{K_1 }
\xhookrightarrow{\ip} \V \otimes_{\F} \DZero$
has image $ L \otimes_{\F} I(\sigma, \tau)$.
We know from \Cref{lemma:beta-vector}
that $(\ip \circ j) (\sigma) = L \otimes_{\F} \sigma$,
and by \cite[Lemma~3.1.5(ii)]{Ber}
applied to $S = \F[\Gamma],
D = \DZero, W = \V, M = \im \left( \ip \circ j \right)$
we see that $\im (\ip \circ j) \subseteq 
L \otimes_{\F} \DZero$.
Since $L \otimes_{\F} \DZero$ is multiplicity free,
it contains exactly one 
$\GR$-subrepresentation isomorphic to 
$I (\sigma, \tau)$, namely 
$L \otimes_{\F} I (\sigma, \tau)$,
so we deduce that  $\im (\ip \circ j)$ coincides with
$L \otimes_{\F} I (\sigma, \tau)$.
\end{proof}

The following theorem generalises 
\cite[Theorem~4.3.15]{BHHMS4}.
\begin{theorem}
  \label{thm:srep-structure}
Assume that $\rep$ is $6$-generic.
Let $\pi_1 $ be a subrepresentation of $\pi$.
Then, there exists a unique $(f+1)$-step
decreasing filtration
$V_1 ( 0) \supseteq V_1 ( 1)
\supseteq \dots \supseteq  V_1 ( f)$
of (possibly zero) $\F$-vector subspaces of $\V$
such that 
\begin{equation}
  \label{eq:srep-K1-decomposition}
\ip(\pi_1^{K_1}) = \sum_{\ell=0}^f 
V_1 (\ell) \otimes_{\F}\DZeroFil[\ell].
\end{equation}
\end{theorem}
\begin{remark}
  \label{rmk:srep-structure-r=1}
When $\pi_1 = \pi$, we have $V_1 (\ell) = \V$
for all $\ell \in \{0, \dots, f\}$
by \eqref{eq:ip-def}.

Also, when $V = \F$ is $1$-dimensional,
then for every $\ell \in \{0, \dots,f \}$
we either have $V_1 (\ell) = 0$ or $\F$.
In particular,
there exists a unique integer $i_0 \in \{-1, \dots,f\}$
such that $V_1 (\ell) = 0$ if and only if $i_0 < \ell$.
This recovers the integer $i_0$ 
of \cite[Theorem~4.3.15]{BHHMS4}.
\end{remark}

\begin{proof}
For $0 \le \ell \le f$, 
we define $V_1( \ell)$ as the largest
$\F$-vector subspace of $\V$ such that
$V_1( \ell) \otimes_{\F}\DZeroFil[\ell]
\subseteq \ip(\pi_1^{ K_1})$.
We declare $V_1(-1) \defeq  \V$ by convention.

We mostly follow the proof of \cite[Theorem~4.3.15]{BHHMS4},
but need to work ``linewise''.
Therefore, for a fixed line $L \subseteq \V$, we set
\begin{equation}
  \label{eq:DL-def}
D_L \defeq \ip(\pi_1^{ K_1}) \cap
( L \otimes_{\F} \DZero),\quad
\ell(L) \defeq 
\max \{\ell \in \{-1, \dots, f\} \mid L \subseteq 
V_1( \ell) \},
\end{equation}
where the intersection on the left is taken inside
$\V \otimes_{\F} \DZero$.
We observe that, by maximality of $V_1(\ell)$,
\begin{equation}
  \label{eq:DL-inclusion-not-inclusion}
L \otimes_{\F} \DZeroFil[\ell(L)] \subseteq D_L,
\quad
L \otimes_{\F} \DZeroFil[\ell(L)+1] \not\subseteq D_L.
\end{equation}
By \cite[Lemma~3.1.5(i)]{Ber}
applied to $S= \F[\Gamma]$, $D= \DZero$,
$W=\V$, and $M_1, M_2$ the two sides of \eqref{eq:srep-K1-decomposition},
it is enough to prove that
the inclusion 
$L \otimes_{\F} \DZeroFil[\ell(L)] \subseteq D_L$
is an equality for all lines $L \subseteq \V$.

Suppose that there exists a line $L$
such that this is not the case.
Then, we may find a Serre weight $\tau$
which embeds into $D_L/ \left( L \otimes_{\F} 
\DZeroFil[\ell(L)] \right)$,
and in particular which embeds into 
$ \left( L \otimes_{\F}
\DZero \right)/ \left( L \otimes_{\F} 
\DZeroFil[\ell(L)] \right)$.
This implies that $\tau \in W(\rss)$
and moreover that $\ell(\tau) > \ell(L)$,
by construction of $\DZeroFil[\ell(L)]$
(cf.\ the paragraph of \eqref{eq:socle-DZerogr}).
Thus, there exists a Serre weight $\tau$
satisfying 
\begin{equation}
  \label{eq:tau-chosen}
\tau \in W(\rss) \cap \JH(D_L),
\quad \ell(\tau) > \ell(L),
\end{equation}
and we choose it so that $\ell(\tau)$ is minimal.

\paragraph{Step 1.}
We prove that $\ell(\tau) = \ell(L)+1$.

First, assume that $\tau \in \Wss \setminus \Wr$,
and let $\sigma ( \neq \tau)$ be the 
unique Serre weight in $\Wr$ such that 
$J_\sigma = J_{\rep} \cap J_\tau$
(via \eqref{eq:J-lambda-def}).
By \cite[Lemma~4.1.3]{BHHMS4} we have
$\tau \in \JH(\Dx[\sigma])$,
hence by \eqref{eq:diagr-from-I}
(cf.\ also the paragraph before \eqref{eq:diagr-from-I})
we have
$L \otimes_{\F} I(\sigma,\tau)
\hookrightarrow L \otimes_{\F}\DZero$,
and so $L \otimes_{\F} I(\sigma,\tau) \hookrightarrow D_L$.
By \cite[Lemma~4.1.3]{BHHMS4} we also have
\[
\JH(\radG(L \otimes_{\F}I(\sigma,\tau))) 
\subseteq \Wss \cap \JH(D_L),
\]
and $\JH(\radG(L \otimes_{\F} I(\sigma,\tau))) \neq 0$
since $\sigma \neq \tau$.
The minimality of $\tau$ implies that
$\ell(\tau') \le \ell(L)$
for any $\tau' \in \JH(\radG(I(\sigma,\tau)))$.
By the first statement of
\cite[Lemma~4.1.3]{BHHMS4}, and using
$|J_{\tau'}|=\ell(\tau')$ for
$\tau' \in \Wss$, this forces 
$\ell(\tau) \le \ell(L) +1$, hence 
$\ell(\tau)= \ell(L)+1$.

Next, assume that 
$\tau \in \Wr$, i.e.\ $\tau$ occurs in the
$\GR$-socle of $D_L$, or equivalently
$J_\tau \subseteq J_{\rep}$.
Let $\tau_1 \in \Wss$ be the Serre weight
uniquely determined by 
$J_{\tau_1} = \delta(J_\tau)$
(for the definition of $\delta$,
cf.\ the third paragraph of \Cref{def:s-at-D0-level}),
and $\sigma_1 \in \Wr$ be the Serre weight
uniquely determined by 
$J_{\sigma_1} = J_{\rep} \cap \delta(J_\tau)$.

Consider the Serre weight $\tau^{[s]}$ 
(for the definition of $\tau^{[s]}$,
cf.\ \Cref{def:s-at-D0-level}).
We have $\chi_{\tau^{[s]}} = \chi_{\tau}^s$
by \Cref{def:s-at-D0-level},
then using \eqref{eq:lambda-s-def}
and \eqref{eq:delta-s-interplay}
we compute 
$J_{\chi_\tau^s} = \delta(J_{\tau})$.
By the first statement in \Cref{eq:3.1.3-recalled} and
\Cref{eq:4.1.1-recalled}, 
we deduce that $\tau^{[s]}$ occurs in both 
$\Dssx[\tau_1]$ and $\Dx[\sigma_1]$.
Moreover, the image of the $\GR$-equivariant map 
$\IndI \chi_\tau^s \to D_L$, obtained by Frobenius reciprocity
from the inclusion $L \otimes_{\F} \chi_\tau^s \subseteq D_L$
(which in turn comes from the inclusion 
$L \otimes_{\F} \chi_\tau \subseteq D_L$
after applying $\Pi$ to both sides),
is multiplicity free with $\GR$-cosocle equal to
$\tau^{[s]}$.
There exists a unique such subrepresentation of
$L \otimes_{\F} \DZero$,
namely $L \otimes_{\F} I(\sigma_1, \tau^{[s]})$
(cf.\ \Cref{deflem:I}).
So, $L \otimes_{\F} I(\sigma_1, \tau^{[s]})
\subseteq D_L$,
and $\tau_1 \in \JH(I(\sigma_1, \tau^{[s]}))$
by \cite[Lemma~4.3.10]{BHHMS4}
applied to $\rss$ and $\chi_\mu = \chi_{\tau^{[s]}}$.
Note that $\ell(\tau_1 )= \ell(\tau) > \ell(L)$,
so $\tau_1 $ satisfies \eqref{eq:tau-chosen},
with $\ell(\tau_1 )$ minimal.
If again $\tau_1 \in \Wr$, i.e.\ $J_{\tau_1 } \subseteq \Jrep$,
we may continue this procedure to obtain $\tau_2$ and $\sigma_2 $.
Since $J_\tau \neq \emptyset $, $\Jrep \neq \{0, \dots, f-1\}$
(using that $\rep$ is reducible nonsplit)
we may continue this procedure until we arrive at some
$\tau_n$ with 
$J_{\tau_n} = \delta^n(J_\tau) \not \subseteq \Jrep$,
or equivalently $\tau_n \in \Wss \setminus \Wr$,
so we conclude $\ell(\tau) = \ell(\tau_n) = \ell(L) +1$ 
by the case of the first paragraph of Step 1.

\paragraph{Step 2.}
Let $\lambda \in \Dss$ be the element
corresponding to $\tau$, and define the 
$f$-tuple $\mu = (\mu_j(x_j))$
by $\mu_j(x_j) = p-1-x_j$ if $j \in J_1$
and $\mu_j(x_j) = \lambda_j(x_j) $
otherwise, where 
\[
J_1 \defeq \{j \in J_{\rep} ^c
\mid \lambda_j(x_j)=p-3-x_j\}.
\]

One checks as in
Step 2 of the proof of
\cite[Theorem~4.3.15]{BHHMS4}
that $\mu \in \mathscr{P}$,
that $\chi_\lambda =
\chi_\mu \prod_{j \in J_1} \alpha_j^{-1}$,
and moreover that  
$J_{\rep} \cap J_\lambda = J_\mu$.
Let 
\[
\widetilde{J}_1 \defeq 
\{0 \le j \le f \mid \lambda_j(x_j)
\in \!\{x_j+1,p-2-x_j\}\} = 
\{0 \le j \le f \mid \mu_j(x_j) \in \!
\{x_j+1,p-2-x_j\}\}.
\]
Then $J_1 \cap \widetilde{J}_1 = \emptyset$,
and if we set $J \defeq J_1 \sqcup \widetilde{J}_1$
then one checks that $J \subseteq \Yn[\mu]$, 
where $\Yn[\mu]$ is as in \eqref{eq:Ymu-def}.

Let $\chi'' \defeq \chi_\mu^{J, \emptyset}
= \chi_\mu \prod_{j \in J} \alpha_j^{-1}.$
Fix a nonzero vector $w \in L \setminus \{0\}$,
and consider the $\GR$-equivariant homomorphism
$\btm[\mu,w] \colon
L \otimes_{\F}\IndI (W(\chi_\mu,\chi''))
\to \pi|_{\GR}$
of \Cref{lemma:beta-vector}(i) 
(applied to $J_1 = J$, $J_2 = \emptyset $). 

\paragraph{Step 3.}
Let $Q_L\defeq \im(\btm[\mu,w])$,
which does not depend on 
$w \in L \setminus \{0\}$.
Notice that the $\GR$-representation
 $\IndI (W(\chi_\mu,\chi''))$
is $K_1$-invariant by \cite[Lemma~4.3.1]{BHHMS4},
so $Q_L\subseteq\pi^{K_1}$.
We claim that $Q_L \subseteq \pi_1 $,
and that $\ip(Q_L) \subseteq D_L$.

By \Cref{lemma:beta-vector}(i),
$\socR(\ip(Q_L)) = L \otimes_{\F}\sigma 
\subseteq \V \otimes_{\F}\DZero$,
where $\sigma \in \Wr$ is the unique Serre
weight such that 
$J_\sigma = J_{\rep} \cap J_\tau$.
By \cite[Lemma~3.1.5(ii)]{Ber}
applied to $S = \F[\Gamma],
D = \DZero, W = \V, M = \ip(Q_L)$
we have
$\ip(Q_L) \subseteq  
L \otimes_{\F}\Dx[\sigma]$.

It follows from
\cite[Proposition~4.3.6(i)]{BHHMS4}
that
\begin{equation}
  \label{eq:cosoc-QL}
\cosocG(Q_L) \cong  \bigoplus
_{J' \subseteq J, \tau^{J'} \in \JH(Q_L)} 
\tau^{J'},
\end{equation} 
where $\tau^{J'}$ denotes the $\GR$-cosocle of
$\IndI \left( \chi_\mu \prod_{j \in J'} \alpha_j^{-1} \right)$.
Since we are comparing $\ip(Q_L)$
and $D_L$ inside $L \otimes_{\F}\DZero$,
which is multiplicity free,
it suffices to show that $\tau^{J'}$
occurs in $D_L$ for each $J' \subseteq J$
satisfying $\tau^{J'} \in \JH(Q_L)$.

If $J' = J_1$ then $\chi'' = \chi_\lambda$
and $\bm[\mu,w] \colon L \otimes_{\F}W(\chi_\mu, \chi_\lambda) \to \pi^{K_1}|_I$
(cf.\ the paragraph above
\eqref{eq:beta-tilde-def})
takes values in $\pi_1$, 
using that $\ip (\im(\bm[\mu,w])) \subseteq L \otimes_{\F} \DZero$
is multiplicity free and its $I$-cosocle $\chi_\lambda$
occurs in $D_L|_I$.
It follows by Frobenius reciprocity that $\ip \circ \btm[\mu,w]$
takes values in $D_L$, 
so we may assume $J' \neq J_1$
in the following.

Let $Q_L' \defeq
\btm[\mu,w] (L \otimes_{\F} 
\IndI W(\chi_\mu, \chi_\lambda))
\subseteq Q_L$.
By \Cref{lemma:I-and-QL} we have
$L \otimes_{\F}I(\sigma,\tau) \subseteq \ip(Q'_L)$,
and 
\begin{equation}
  \label{eq:JH-QL-prime-mod-I}
\JH\left(\ip (Q_L') / 
\left( L \otimes_{\F}I(\sigma,\tau) \right)\right)
\cap \Wss = \emptyset.
\end{equation}
Moreover, we claim that
\begin{equation}
  \label{eq:JH-QL-mod-I}
\JH\left(\ip (Q_L) / 
\left( L \otimes_{\F}I(\sigma,\tau) \right)\right)
\cap \Wss = \emptyset.
\end{equation}
This follows by noting that if
$\chi' \in \JH(W(\chi_\mu, \chi'')) \setminus 
\JH(W(\chi_\mu, \chi_\lambda))$,
then $\chi' \notin \JH(\DssOne)$
by the explicit description of $\mathscr{P} ^{\mathrm{ss}}$,
so $\JH(\IndI \chi') \cap \Wss = \emptyset $ 
by \cite[Proposition~4.2]{Bre14}.

Now fix $J' \subseteq J$ such that 
$J' \neq J_1$ and $\tau^{J'} \in \JH(Q_L)$,
in particular $\tau^{J'} \neq \tau$.

Then, as $L \otimes_{\F} \Dx[\sigma]$ is
multiplicity free we must have 
$L \otimes_{\F} I(\sigma, \tau^{J'})
\subseteq Q_L$,
and we claim that
$L \otimes_{\F} I(\sigma, \tau^{J'})
\subseteq L \otimes_{\F} \DZeroFil[\ell(L)]
\subseteq D_L$.
Suppose that this is not the case. 
Then, the composition
$L \otimes_{\F} I(\sigma, \tau^{J'})
\to  \frac{L \otimes_{\F} \DZero}
{L \otimes_{\F} \DZeroFil[\ell(L)]}$
is nonzero, 
and by construction of $\DZeroFil[\ell(L)]$
(cf.\ the paragraph of \eqref{eq:socle-DZerogr})
there exists a Serre weight 
$\tau' \in \JH(I(\sigma,\tau^{J'}))$ 
lying inside $\Wss$
and such that $\ell(\tau') \ge \ell(L)+1$.
By \eqref{eq:JH-QL-mod-I} and
$L \otimes_{\F} I(\sigma, \tau^{J'})
\subseteq Q_L$
we deduce that $\tau'$ occurs inside
$L \otimes_{\F} I(\sigma,\tau)$, 
and the last statement of \cite[Lemma~4.1.3]{BHHMS4}
implies $\tau' = \tau$ 
(it cannot be the case that
$\ell(\tau') < \ell(\tau)=\ell(L)+1$).
In particular, $\tau \in 
\JH(I(\sigma,\tau^{J'}))$,
and \eqref{eq:cosoc-QL} implies that 
$\tau$ is in the cosocle of $\ip(Q_L)$,
which is multiplicity free,
hence it must occur in the cosocle of 
$I(\sigma,\tau^{J'})$ too.
But $\cosocG(I(\sigma,\tau^{J'}))=
\tau^{J'}$, so $\tau^{J'}= \tau$,
contradiction.
Hence, 
$L \otimes_{\F} I(\sigma, \tau^{J'})
\subseteq L \otimes_{\F} \DZeroFil[\ell(L)]
\subseteq D_L$
as claimed.

\paragraph{Step 4.}
We show that for $\tau$ as in \eqref{eq:tau-chosen}
we have an inclusion
\begin{equation}
  \label{eq:Dtau-DL-inclusion}
L \otimes_{\F} \Dxgr[\tau]{\ell(L)+1} \subseteq 
\frac{\ip(D_L)}{ 
L \otimes_{\F} \DZeroFil[\ell(L)]}.
\end{equation}
Notice that, by Step 1,
$\tau$ is a Serre weight in 
$\Wss_{\ell(L) +1} \cap \JH(D_L)$.
Since $L \otimes_{\F}\DZeroFil[\ell(L)+1]$
(and all its $\Gamma$-subquotients) is
multiplicity free, it is enough to show that
\begin{equation}
  \label{eq:DL-JH-inclusion}
\JH(\Dssx[\tau]) \cap \JH(\DZerogr[\ell(L) +1 ])
\overset{\eqref{eq:D-Dss-gr-comparison}}{= }
\JH(\Dxgr[\tau]{\ell(L) + 1})
\subseteq \JH(D_L).
\end{equation}

Define $\lambda$, $\mu$, $\chi''$
as in Step 2.
In Step 3 we have shown that the map $\btm[\mu,w]$
restricts to $\btm[\mu,w] \colon 
L \otimes_{\F} \IndI (W(\chi_\mu,\chi'')) \to \pi_1 $.
By Frobenius reciprocity, we have
$\bm[\mu,w] \colon L \otimes_{\F}W(\chi_\mu,\chi'') \to \pi_1 |_I$
(cf.\ the paragraph after \eqref{eq:iso-chi-isotypical}
for $\bm[\mu,w]$).
Notice that 
$W(\chi_\mu^{s}, \chi^{\prime \prime s})$,
as an $I$-representation, has the same underlying
$\F$-vector space as
$W(\chi_\mu, \chi'')$
but the two $I$-actions differ by conjugation by $\Pi$.
In particular, applying $\Pi$ to $\bm[\mu,w]$ we obtain
an $I$-equivariant map
$\gamma_w \colon L \otimes_{\F} W(\chi_\mu^{s},
\chi^{\prime \prime s}) \hookrightarrow 
\pi_1|_I$
(to be precise,
$\gamma_w(w \otimes x) \defeq  \Pi \bm[\mu,w](w \otimes x)$
for $x \in W(\chi_\mu, \chi'')
= W(\chi_\mu^{s}, \chi^{\prime \prime s})$).
By Frobenius reciprocity, we obtain a 
$\GR$-equivariant map
$\widetilde{\gamma}_w \colon L \otimes_{\F} \IndI
W(\chi_\mu^{s}, \chi^{\prime \prime s})
\to \pi_1$.

Let $\sigma_1 \in \Wr$ be the Serre weight
 such that $\chi_\mu^s$ occurs in 
$\DOnex[\sigma_1 ]$.
Then, $\sigma_1 $ occurs in 
$\IndI \chi_\mu^s$ and is parametrised by 
$X(\mu^{[s]})$
via \Cref{eq:4.1.1-recalled}
(cf.\ \eqref{eq:lambda-s-def}).
Similarly, let 
$\tau_1 = \delta(\tau) \in \Wss$
be such that $\chi_\lambda^s$
occurs in $\DssOnex[\tau_1 ]$,
then $\tau_1 $ occurs in $\IndI \chi_\lambda^s$
and is parametrised by $X^{\mathrm{ss}}(\lambda^{[s]})$
via \Cref{eq:3.1.3-recalled}.
By \Cref{lemma:beta-vector}
applied with $\mu = \mu^{[s]}, J_1 = \emptyset, J_2 = J$
we see that $\socR(\im(\widetilde{\gamma}_w))
= \sigma_1$.
By \Cref{lemma:I-and-QL} applied with
$\lambda = \lambda^{[s]} , \mu = \mu^{[s]}$,
and noting that $W(\chi_\mu^s, \chi_\lambda^s)
\hookrightarrow W(\chi_\mu^s, \chi^{\prime \prime s} )$,
we deduce that 
$L \otimes_{\F }I(\sigma_1 ,\tau_1) \subseteq 
\im(\widetilde{\gamma}_w)^{K_1}
\subseteq \ip(\pi_1^{ K_1})$.
In particular, $\tau_1 \in \JH(D_L)$.

Note that $J_{\sigma_1} =  J_{\rep} \cap J_{\tau_1}$,
using
$J_{\sigma_1 } = 
\Jrep \cap J_{\mu^{[s]}} = J_{\lambda^{[s]}}$
(by \Cref{eq:4.1.1-recalled}),
$J_{\lambda^{[s]}} 
 = J_{\tau_1 }$
(by \Cref{eq:3.1.3-recalled}),
and the identity 
$\Jrep \cap J_{\mu^{[s]}} = 
\Jrep \cap  J_{\lambda^{[s]}}$, 
which follows from Step 2
(namely, for all $j \in J_1 ^c \supseteq \Jrep$
we have $\lambda_j = \mu_j$,
hence $\lambda_j^{[s]} = \mu_j^{[s]}$).

Let $Q_{\tau_1} $ be the unique quotient
of $\IndI W(\chi_\mu^s, \chi^{\prime \prime s} )$
with socle $\delta(\tau)= \tau_1$.
Notice that the quotient 
$\IndI W(\chi_\mu^s, \chi^{\prime \prime s} )
\twoheadrightarrow Q_{\tau_1}$
factors through
\[
\IndI W(\chi_\mu^s, \chi^{\prime \prime s} )
\twoheadrightarrow
\IndI W(\chi_\lambda^s, \chi^{\prime \prime s} ).
\]
By \cite[Lemma~4.3.13]{BHHMS4} applied
to $J = J_1 $, $\chi = \chi_\lambda$,
and using that the $\GR$-representation 
$\widetilde{R}_{\widetilde{J}_1} (\chi_\lambda)$
of \emph{loc.\ cit.}\ is isomorphic to
$\IndI W(\chi_\lambda^s, \chi^{\prime \prime s} )$
(cf.\ the paragraph above \emph{loc.\ cit.}),
we deduce that $Q_{\tau_1} $ contains $\Dssx[\tau_1]$
(as an abstract $\GR$-representation).
If $Q_{\sigma_1}$ denotes
the unique quotient of
$\IndI W(\chi_\mu^s, \chi^{\prime \prime s} )$
with socle $\sigma_1$,
then $\tau_1$ occurs in 
$Q_{\sigma_1} \cong \im(\widetilde{\gamma}_w)$,
hence $Q_{\sigma_1} $ surjects onto 
$Q_{\tau_1} $, 
and hence contains 
$\Dssx[\tau_1]$
(as an abstract $\GR$-representation) as a subquotient.

The rest of the argument proceeds as in 
Step 4 of the proof of \cite[Theorem~4.3.15]{BHHMS4}.
By \cite[Lemma~4.3.14]{BHHMS4},
we have the chain of inclusions
\begin{align*}
\JH(\Dssx[\tau_1]) \cap \JH(\DZerogr[\ell (L) +1])
=  \JH(\Dssx[\tau_1]) \cap 
&\JH(\Dx[\sigma_1]) \subseteq  \\
\subseteq 
\JH(Q_{\sigma_1} ) \cap &\JH(\Dx[\sigma_1])
\subseteq \JH(Q_\sigma^{K_1} ),
\end{align*}
where the last inclusion follows from
\cite[Proposition~4.3.8]{BHHMS4},
applied to $\sigma = \sigma_1$ and 
varying $\tau$.

As $Q_{\tau_1}  \cong \im(\widetilde{\gamma}_w)
\subseteq D_L$,
then \eqref{eq:DL-JH-inclusion}
(hence \eqref{eq:Dtau-DL-inclusion})
holds if $\tau$ is replaced by $\tau_1$.
We can repeat the same argument with 
$\tau_1 = \tau \in \Wss \cap \JH(D_L)$,
which has still length $\ell(L)+1$,
hence we see that 
\eqref{eq:Dtau-DL-inclusion}
holds if $\tau$ is replaced by 
$\delta^n(\tau)$, for every integer $n > 0$.
Hence, it holds for $\tau$ itself, as 
$\delta( - )$ is periodic.

\paragraph{Step 5.}
We show that 
$L \otimes_{\F}\DZeroFil[\ell(L)+1] 
\subseteq \ip(D_L)$, 
which will contradict 
\eqref{eq:DL-inclusion-not-inclusion}.
As $D_L$ is multiplicity free,
and since $L \otimes_{\F}
\DZeroFil[\ell(L)]\subseteq \ip(D_L)$
by construction of $\ell(L)$,
it suffices to prove 
\begin{equation}
  \label{eq:JH-claimed-inclusion}
\bigsqcup_{\tau \in \Wss_{\ell(L)+1}} 
\JH(\Dxgr[\tau]{\ell(L) +1 })
\overset{\eqref{eq:Dgr-decomposition}}{=}
\JH(\DZerogr[\ell(L)+1])
\subseteq \JH(D_L).
\end{equation}
If moreover
$\tau \in \JH(D_L)$, this follows from
Step 4 (since $\tau$
satisfies \eqref{eq:tau-chosen}).
Therefore, we are reduced to showing that
for all $\tau \in \Wss_{\ell(L) +1 }$
we have $\tau \in \JH(D_L)$.

Fix $\tau$ as in \eqref{eq:tau-chosen},
write $J_{\tau} = \mathcal{S}_1 \sqcup \dots
\sqcup \mathcal{S}_t$, with 
$\mathcal{S}_i = \{a_i, a_i+1, \dots,
b_i = a_i + \ell_i -1\}$,
for
$0 \le a_1 < a_2< \dots< a_t < f$,
$\ell_i \defeq |\mathcal{S}_i|$,
and $b_i + 1 \notin J_{\tau}$
for each $1 \le i \le r$. 
Like before, the indices in $\{0, \dots, f-1\}$
are taken modulo  $f$.
In particular, 
$\ell(\tau) = \sum _{i = 1}^{t} \ell_i $.
Fix $1 \le i \le t$, and define an
$f$-tuple $\lambda$ by 
\begin{equation}
  \label{eq:lambda-def-again}
\lambda_j(x_j) = 
\begin{cases}
p-3-x_j & \text{if } j \in J_\tau \setminus \{b_i\}, \\
x_j + 1 & \text{if } j = b_i, \\
p-2-x_j & \text{if } j = b_i + 1, \\
p-1-x_j & \text{otherwise.}
\end{cases}
\end{equation}
Then, as in Step 5 of the proof of
\cite[Theorem~4.3.15]{BHHMS4}
one checks that
$\lambda \in \mathscr{P} ^{\mathrm{ss}}$,
that $|J_\lambda| = \ell(L) + 1$,
and that the unique Serre weight
$\tau_1  \in \Wss$ such that $\chi_\lambda^s$
contributes to $\Dssx[\tau_1 ]$
satisfies 
\[
J_{\tau_1}  = (J_\tau \setminus \{b_i\})
\sqcup \{b_i + 1\},
\]
so in particular $\ell(\tau_1) = \ell(\tau) = 
\ell(L) + 1.$

We want to show that $\tau_1 \in \JH(D_L)$, 
so that we can take
$\tau = \tau_1 $ in \eqref{eq:tau-chosen},
and so we can deduce by Step 4 that 
$\JH( \Dxgr[\tau_1]{\ell(L) +1 }) \subseteq \JH(D_L)$.
Define $\mu \in \mathscr{P}$ and $J_1 $ as in Step 2.
For a fixed $w \in L \setminus \{0\}$
consider the $\GR$-equivariant morphism
$\btm[\mu,w] \colon
L \otimes_{\F} \IndI (W(\chi_\mu, \chi_\lambda))
\hookrightarrow \pi$
of (the line below) \eqref{eq:beta-tilde-def},
with $J_2 = \emptyset$.
By \cite[Lemma~4.3.1(ii)]{BHHMS4}
$\IndI (W(\chi_\mu, \chi_\lambda))$
is $K_1 $-invariant, hence $\btm[\mu,w]$
takes values in
$\pi^{K_1}$,
and we let $Q_L \defeq \im \btm[\mu,w] \subseteq \pi^{K_1 }$
be its image.
We first claim that $\ip(Q_L) \subseteq D_L$.

As in Step 3,
one shows that $\ip(Q_L) \subseteq L \otimes_{\F} \Dx[\sigma]$,
where $\sigma \in \Wr$ is the unique Serre weight such that 
$J_\sigma = \Jrep \cap J_\tau$ via \eqref{eq:J-lambda-def}.
Keeping the notation of Step 3,
it follows from \cite[Proposition~4.3.6(i)]{BHHMS4}
(with $\chi = \chi_\mu$, $J_2  = \emptyset $)
that
\[
  \cosocG(Q_L) \cong \bigoplus
  _{J' \subseteq J_1 , \tau^{J'} \in \JH(Q_L)} 
  \tau^{J'},
\]
so it suffices to show that,
for all $J' \subseteq J_1 $,
$\tau^{J'} \in \JH(Q_L)$ implies
$\tau^{J'} \in \JH(D_L)$.

Let $\tau^{J'} \in \JH(Q_L)$ for some
$J' \subseteq J_1$, and define
an $f$-tuple $\mu'$ by
\[
\begin{cases}
\mu_j'(x_j) = \mu_j(x_j)-2 = \lambda_j(x_j),
 & \text{if } j \in J'\\
 \mu_j'(x_j) = \mu_j(x_j) & \text{otherwise },
\end{cases}
\]
so that $\chi_{\mu'} = \chi_\mu \prod _{j \in J'} \alpha_j^{-1} .$
Then, $\mu' \in \mathscr{P} ^{\mathrm{ss}}$
and $|J_{\mu'} | \le |J_\lambda| = \ell(L) +1$,
with equality holding if and only if $\mu' = \lambda$,
(or equivalently $J' = J_1$),
in which case $\tau^{J'} \in
\JH(\Dssx[\tau]) \cap \JH(\DZero)
= \JH(\Dxgr[\tau]{\ell(L)+1})$
by \Cref{eq:3.1.3-recalled}.
If $J' = J_1 $, then $\tau^{J'} \in \JH(D_L)$
by \eqref{eq:DL-JH-inclusion}
(choosing $\tau = \tau_1$ in \eqref{eq:tau-chosen}).
If instead $J' \subsetneq J_1$,
then $\tau^{J'}  \in \JH(\DssZeroFil[\ell(L)])
\cap \JH(\DZero) 
\overset{\eqref{eq:D-Dss-Fil-comparison}}{=} 
\JH(\DZeroFil[\ell(L)])
\subseteq \JH(D_L)$
by assumption.
This proves the claim.

We now show that $\tau_1  \in \JH(D_L)$.
One defines a $\GR$-equivariant morphism
$\widetilde{\gamma}_w \colon
L \otimes_{\F} \IndI W(\chi_\mu^s, \chi_\lambda^s)
\to \pi_1$
like in Step 4, namely we notice that 
$W(\chi_\mu^{s}, \chi_\lambda^{s})$
has the same underlying $\F$-vector space as
$W(\chi_\mu, \chi_\lambda)$
but the two $I$-actions differ by conjugation by $\Pi$.
Applying $\Pi$ to $\bm[\mu,w]$ we obtain
an $I$-equivariant map
$\gamma_w \colon L \otimes_{\F} W(\chi_\mu^{s},
\chi_\lambda^{s}) \hookrightarrow 
\pi_1|_I$
as in Step 4,
and by Frobenius reciprocity a 
$\GR$-equivariant map
$\widetilde{\gamma}_w \colon L \otimes_{\F}\IndI
W(\chi_\mu^{s}, \chi_\lambda^{s})
\to \pi_1$.
Define $\sigma_1 \in \Wr$,
as in the third paragraph of Step 4,
and notice that the definition of $\tau_1 \in \Wss$
(cf.\ the line after \eqref{eq:lambda-def-again})
is identical to the definition of $\tau_1 $
in the third paragraph of Step 4.

Finally, by \Cref{lemma:I-and-QL}
applied to $\lambda = \lambda^{[s]}$,
$\mu = \mu^{[s]}$
(and so $\sigma = \sigma_1 $, $\tau = \tau_1 $
in \emph{loc.\ cit.}),
we have
$L \otimes_{\F} I(\sigma_1 , \tau_1 ) \subseteq 
\ip \left( \im(\widetilde{\gamma}_w)^{K_1} \right)
\subseteq D_L$,
in particular $\tau_1$ occurs inside $D_L$.

Therefore, if we let
\[
\mathfrak{S} \defeq \{J_\tau \mid \tau
\in \Wss_{\ell(L) + 1} \!\cap \JH (D_L)\}
\subseteq  \mathcal{J}_{\ell(L) + 1} \defeq \{J \subseteq \{ 0, \dots,f -1 \} \mid |J| = \ell(L) + 1\},
\]
then we have proved that $\mathfrak{S}$ is a nonempty
set with the following property:
if $J \in \mathfrak{S}$, $j \in J$ and $j +1 \notin J$,
then 
$\tau_j J \defeq (J \setminus \{j\} ) \sqcup \{j+1\} \in 
\mathfrak{S}$.
We claim that this is enough to conclude that 
$\mathfrak{S} = \mathcal{J}_{\ell(L) + 1}$.
Indeed, suppose there exists 
$J_0 \in 
\mathcal{J}_{\ell(L) + 1} \setminus \mathfrak{S}$
(which implies $\ell (L) + 1 \le f$),
and let $j_0$ an element in $J_0$ such that 
$j_0 - 1 \notin J_0$
(such a $j_0$ exists, because 
$J_0 \subsetneq \{0, \dots, f-1\}$).
We give $\mathcal{P} (\{0 , \dots, f -1 \})$
the total order given by 
$J_1 \preceq J_2 $ if and only if 
either $|J_1 | < |J_2 |$,
or $|J_1 | = |J_2 |$ and 
$\sum_{j \in J_1 } 2^{j-j_0 } \le \sum_{j \in J_2} 2^{j- j_0}$.

Let $\mathfrak{S}_0 \defeq  \{J \in \mathfrak{S}\mid 
J_0  \setminus  J
\text{ is $\preceq$-minimal}\}$,
and let $J \defeq  \max_{\preceq} \mathfrak{S}_0$.
If we set $ \partial J \defeq J \setminus \delta (J)$,
$ \partial J_0  \defeq J_0  \setminus \delta (J_0 )$
(cf.\ the last paragraph of \Cref{sec:diagrams}
for the definition of $\delta (-)$),
then we distinguish two cases:
if $ \partial J \not\subseteq \partial J_0 $,
i.e.\ if there exists $j \in \partial J \setminus \partial J_0 $,
then we reach a contradiction upon observing that
$J' \defeq \tau_j J =
(J \setminus \{j\} ) \sqcup \{j+1\} \in \mathfrak{S}$
is such that either $J_0 \setminus J' \prec J_0  \setminus J$
(if $j+1 \in J_0 $), thus violating 
$J \in \mathfrak{S}_0 $;
or $J_0 \setminus J' =  J_0  \setminus J$ 
(if $j+1 \notin J_0 $, which together with
$j \notin \partial J_0$ implies $j \notin J_0$)
and $J' \in \mathfrak{S}_0$ is such that 
$J \prec J'$, 
thus violating the $\preceq$-maximality of $J$.
If instead
$ \partial J \subseteq \partial J_0 $,
then we can decompose
$J = \mathcal{S}_1 \sqcup \cdots \sqcup \mathcal{S}_t$
as in the paragraph of \eqref{eq:lambda-def-again}.
Since $|J| = |J_0 |$ 
and $J \neq J_0$, there exists $1 \le t' \le t$
such that $\mathcal{S}_{t'} \not \subseteq J_0$.
Write $\mathcal{S}_{t'} = \{a, a+1, \dots,
b = a + |\mathcal{S}_{t'}| -1\}$
for some $0 \le a \le f-1$,
and let $m \in \{0, \dots, |\mathcal{S}_{t'}|-1\}$ 
be the largest nonnegative integer
such that $a + m \notin J_0$.
If we set $J' \defeq  
\tau_{a+m} \tau_{a+m +1}\dots \tau_{b-1}\tau_b J$,
then we see that
either $J_0 \setminus  J' \prec J_0 \setminus J$
(precisely when $b + 1 \in J_0$),
thus violating $J \in \mathfrak{S}_0 $;
or $J' \in \mathfrak{S}_0$ and
$J \prec J'$, 
thus violating the $\preceq$-maximality of $J$.

This implies that
$\mathfrak{S} = \mathcal{J}_{\ell(L) + 1}$.
Since $\tau \mapsto J_\tau$
gives a bijection between $\Wss_{\ell(L) + 1}$
and $\mathcal{J}_{\ell(L) +1 }$,
this shows
the claim of Step 5,
cf.\ the line after 
\eqref{eq:JH-claimed-inclusion}.
\end{proof}

In the hypotheses of \Cref{thm:srep-structure},
and for $\ell \in \{-1, \dots, f+1\}$,
we set 
\begin{equation}
  \label{eq:VFil-rFil-def}
\VFil[\pi_1,\ell] \defeq 
\begin{cases}
V_1 ( \ell)& \text{if }\ell \in \{0, \dots, f\}, \\
\V & \text{if }\ell = -1, \\
0 & \text{if } \ell = f+1,
\end{cases}
\quad
\rFil[\pi_1,\ell] \defeq 
\dim_{\F} \VFil[\pi_1,\ell],
\end{equation}
where $V_1 (\ell)$ is the $\F$-vector subspace of 
\emph{loc.\ cit.}
We now give some immediate corollaries to 
\Cref{thm:srep-structure}.
\begin{corollary}
  \label{cor:K1-direct-sum}
Assume that $\rep$ is $6$-generic,
and let $\pi_1$ be a subrepresentation of $\pi$.
Then, there exists an isomorphism of 
$\GR$-representations
\[
\pi_1 ^{K_1} \cong  
\bigoplus_{\ell = 0}^  f 
\Vgr[\pi_1 ,\ell] \otimes_{\F}\DZeroFil[\ell].
\]
\end{corollary}
\begin{proof}
For all $\ell \in \{0, \dots, f\}$,
let $W_1 (\ell)$ be a complementary subspace of 
$\Vell[\pi_1 ,\ell+1]$ inside
$\Vell[\pi_1 ,\ell]$.
Then, 
we have an inclusion
\begin{equation}
    \label{eq:from-sum-to-direct-sum}
\bigoplus_{\ell = 0}^f 
W_1(\ell) \otimes_{\F}\DZeroFil[\ell] 
= \sum_{\ell = 0}^f 
W_1(\ell) \otimes_{\F}\DZeroFil[\ell] \subseteq 
\sum_{\ell = 0}^f 
\Vell[\pi_1 ,\ell] \otimes_{\F}\DZeroFil[\ell] 
= \ip (  \pi_1 ^{K_1}).
\end{equation}
Moreover, the composition
\begin{equation}
\bigoplus_{\ell = 0}^f 
W_1(\ell) \otimes_{\F}\DZeroFil[\ell] 
\overset{\eqref{eq:from-sum-to-direct-sum}}{ \hookrightarrow }
\sum_{\ell = 0}^f 
\Vell[\pi_1 ,\ell] \otimes_{\F}\DZeroFil[\ell] 
\twoheadrightarrow 
\bigoplus_{\ell = 0}^f 
\Vgr[\pi_1 ,\ell] \otimes_{\F}\DZeroFil[\ell]
\end{equation}
is an isomorphism
since $W_1(\ell)$ is isomorphic to $\Vgr[\pi_1 , \ell]$
as an $\F$-vector space.
Since all the $\GR$-representations involved
are finite-dimensional over $\F$,
we conclude for dimension reasons.
\end{proof}

If we set $\DOneFil[\ell] \defeq \DZeroFil[\ell]^{I_1}$
and $\DOneEll[\ell] \defeq \DZeroEll[\ell]^{I_1}$
for all $ \ell \in \{0, \dots, f+1\}$,
then the following corollary
describes the $I_1 $-invariants of
subrepresentations of $\pi$,
generalising \cite[Corollary~4.3.16]{BHHMS4} 
to the case $r \ge 1$.
\begin{corollary}
  \label{eq:I1-subrep}
Assume that $\rep$ is $6$-generic,
and let $\pi_1$ be a subrepresentation of $\pi$.
Then, the representation $\pi_1^{I_1}$
of $H = I/I_1 $ over $\F$ is isomorphic to
\begin{equation}
  \label{eq:srep-I1-decomposition}
\pi_1 ^{ I_1} 
\cong 
\bigoplus_{\ell =0} ^f 
\VFil[\pi_1 , \ell]\otimes_{\F}
\DOneEll[\ell].
\end{equation}
\end{corollary}
\begin{proof}
We have 
$\pi_1 ^{ I_1} \cong  
\bigoplus_{\ell = 0}^{f} 
\Vgr[\pi_1, \ell]\otimes_{\F} \DOneFil[\ell]$
by \Cref{cor:K1-direct-sum}.
If we show that
$\DOneFil[\ell] \cong \bigoplus_{0 \le \ell' \le \ell} \DOneEll[\ell']$ 
are isomorphic as $H$-representations over $\F$
for all $\ell \in \{0, \dots, f\}$,
then we can conclude with a simple computation:
\begin{align*}
\pi_1 ^{ I_1} \cong  
\bigoplus_{\ell = 0}^{f} 
\Vgr[\pi_1, \ell]\otimes_{\F} \DOneFil[\ell]
\cong 
\bigoplus_{\ell = 0}^{f} 
\bigoplus_{0 \le \ell' \le \ell} 
\left( \Vgr[\pi_1, \ell]\otimes_{\F} 
\DOneEll[\ell'] 
\right) \\
= 
\bigoplus_{\ell' = 0}^{f} 
\bigoplus_{\ell' \le \ell \le f} 
\left( \Vgr[\pi_1, \ell]\otimes_{\F} 
\DOneEll[\ell'] 
\right)
\cong  
\bigoplus_{\ell' = 0}^{f} 
\left( \VFil[\pi_1, \ell']\otimes_{\F} 
\DOneEll[\ell'] 
\right),
\end{align*}
where in the second line we have used 
the isomorphism
$\bigoplus_{\ell' \le \ell \le f}
\frac{\VFil[\pi_1,\ell]}{\VFil[\pi_1,\ell+1]}
\cong \VFil[\pi_1,\ell']$
of $\F$-vector spaces,
valid for all $\ell' \in \{0, \dots, f\}$.
To show that
$\DOneFil[\ell] \cong 
\bigoplus_{0 \le \ell' \le \ell} \DOneEll[\ell']$,
notice that both sides are semisimple
$H$-representations
(as every $H$-representation over $\F$ is)
and multiplicity free,
so it is enough to compare Jordan-H\"older constituents.
But it follows from \eqref{eq:D-Dss-gr-comparison}
that
$\JH(\DOneFil[\ell]) =
\left\{ \sigma \in \Wss \mid \ell (\sigma) \le \ell,
\ \chi_\sigma \in \mathscr{P} \right\}$,
and it follows from \eqref{eq:D-Dss-Fil-comparison}
that
$\JH(\DOneEll[\ell']) =
\left\{ \sigma \in \Wss \mid \ell (\sigma) = \ell',
\ \chi_\sigma \in \mathscr{P} \right\}$,
which immediately implies
$\JH(\DOneFil[\ell]) =
\bigsqcup_{0 \le \ell' \le \ell}\JH(\DOneEll[\ell'])$.
\end{proof}

The following corollary
describes the $\GR$-socle of
subrepresentations of $\pi$,
and generalises 
\cite[Corollary~4.3.17]{BHHMS4} 
to the case $r \ge 1$.
\begin{corollary}
Assume that $\rep$ is $6$-generic.
Let $\pi_1$ be a subrepresentation of $\pi$.
Then, 
\[
\socR \pi_1 
\cong  \bigoplus_{\ell =0} ^f 
\VFil[\pi_1, \ell]\otimes_{\F}
\big( \textstyle\bigoplus_{\sigma \in \Wr_\ell} 
\sigma \big).
\]
\end{corollary}
\begin{proof}
Follows directly from 
$\ip(\socR \pi_1) \overset{\eqref{eq:srep-K1-decomposition}}{=}
\sum_{\ell =0} ^f \VFil[\pi_1, \ell]\otimes_{\F}
\socR \left( \DZeroFil[\ell] \right)$
and from $\socR \left( \DZeroFil[\ell] \right)
= \bigoplus_{\ell' = 0} ^\ell 
\bigoplus_{\sigma \in \Wr_{\ell'}} \sigma$
(which in turn is a consequence of 
\eqref{eq:D-Dss-Fil-comparison}).
\end{proof}

We give a last corollary, that will be used 
in the proof of 
\Cref{thm:gen-by-inv}(ii).
\begin{corollary}
  \label{cor:VFil-monotone}
Assume that $\rep$ is $6$-generic,
and let $\pi_0, \pi_1 $ 
be two subrepresentations of $\pi$.
If $\pi_1^{K_1 } \subseteq \pi_0^{K_1 }$,
then we have
$\VFil[\pi_1 ,\ell] \subseteq \VFil[\pi_0 ,\ell]$
for all
$\ell \in \{0, \dots, f\}$.
\end{corollary}

\begin{proof}
Observe first that $\pi_1 ^{K_1 } \subseteq \pi_0 ^{K_1 }$
implies
\begin{align}
  \label{eq:inclusion-filtered-truncated}
\sum_{\ell' = 0}^{\ell} \VFil[\pi_1 ,\ell'] \otimes_{\F}
\DZeroFil[\ell']
&\overset{ \eqref{eq:srep-K1-decomposition}}{= } \ip(\pi_1 ^{K_1 }) \cap \DZeroFil[\ell] 
 \\
\nonumber
&\subseteq 
\ip(\pi_0 ^{K_1}) \cap \DZeroFil[\ell]
\overset{ \eqref{eq:srep-K1-decomposition}}{= }
\sum_{\ell' = 0}^{\ell} \VFil[\pi_0 ,\ell'] \otimes_{\F}
\DZeroFil[\ell']
\end{align}
for all $\ell \in \{0, \dots, f\}$.
Upon taking images through the 
$\GR$-equivariant surjection
$\V \otimes_{\F} \DZeroFil[\ell]
\twoheadrightarrow 
\frac{\V \otimes_{\F} \DZeroFil[\ell]}
{\V \otimes_{\F} \DZeroFil[\ell -1 ]}$,
the inclusion \eqref{eq:inclusion-filtered-truncated}
of $\GR$-submodules of $\V \otimes_{\F} \DZero$
induces an inclusion
\begin{align}
  \label{eq:inclusion-filtered-truncated-mod}
\VFil[\pi_1 ,\ell] \otimes_{\F}
\DZerogr[\ell] 
 &= 
 \frac{\V \otimes_{\F} \DZeroFil[\ell -1 ] + 
\VFil[\pi_1 ,\ell] \otimes_{\F} \DZeroFil[\ell]}
{\V \otimes_{\F} \DZeroFil[\ell -1 ] }
 \\
\nonumber
&=
\frac{\V \otimes_{\F} \DZeroFil[\ell -1 ] + 
\sum_{\ell' = 0}^{\ell} 
\VFil[\pi_1 ,\ell'] \otimes_{\F} \DZeroFil[\ell']}
{\V \otimes_{\F} \DZeroFil[\ell -1 ] }
\\
\nonumber
&\subseteq 
\frac{\V \otimes_{\F} \DZeroFil[\ell -1 ] + 
\sum_{\ell' = 0}^{\ell} 
\VFil[\pi_0 ,\ell'] \otimes_{\F} \DZeroFil[\ell']}
{\V \otimes_{\F} \DZeroFil[\ell -1 ] }
\\
\nonumber
&= 
\frac{\V \otimes_{\F} \DZeroFil[\ell -1 ] + 
\VFil[\pi_0 ,\ell] \otimes_{\F} \DZeroFil[\ell]}
{\V \otimes_{\F} \DZeroFil[\ell -1 ] }
= 
\VFil[\pi_0 ,\ell] \otimes_{\F}
\DZerogr[\ell],
\end{align}
which implies
$ \VFil[\pi_1 ,\ell] \subseteq 
\VFil[\pi_0 ,\ell]$.
\end{proof}

\begin{remark}
  \label{rmk:I1-does-not-reconstruct-pi}
As in \cite[Remark~4.3.18]{BHHMS4},
we remark that a subrepresentation $\pi_1 $ of $\pi$
is not in general determined by $\pi_1 ^{I_1 }$.
For example, when $\Jrep =  \emptyset$,
then it follows from the definitions that 
$|J_\lambda| \le f/2$ for all $\lambda \in \mathscr{P}$,
so \eqref{eq:srep-I1-decomposition} does not depend on
$\VFil[\pi_1 ,\ell]$ for $\ell >  f/2$,
hence $I_1$-invariants cannot distinguish 
between two subrepresentations $\pi_1 , \pi_1' $
such that 
$\VFil[\pi_1 ,\ell] = \VFil[\pi_1' ,\ell]$ 
for $\ell$ up to $f/2$, 
but such that 
$\VFil[\pi_1 ,\ell] \neq \VFil[\pi_1' ,\ell]$ 
for some $\ell >  f/2$.
This does occur in practice:
when $f = 1$, $r = 1$, if we take
$\pi_1 $ to be a principal series subrepresentation
of $\pi$,
and take $\pi_1 ' = \pi$,
then we have
$\VFil[\pi_1, 0] = \VFil[\pi_1 ' ,0] = \F$,
$0 = \VFil[\pi_1, 1] \subsetneq 
\VFil[\pi_1', 1] = \F$.

Similarly, $\pi_1$ is not in general determined
by $\socR \pi_1$.
\end{remark}

We prepare to state 
\Cref{prop:char-only-occurs} below.
Assume that $\rep$ is $ \max \{6,2f+1\}$-generic.
For $\lambda \in \mathscr{P}$, consider
the finite-dimensional representation
$\tau_{\lambda}^{(f+1)}$ of $I$ over $\F$
with socle $\chi_\lambda$
of \cite[Lemma~2.4.1]{BHHMS4} applied with $n = f+1$
(which is multiplicity free by
\cite[Corollary~2.4.3]{BHHMS4}),
and define $\tell[\ell] \defeq 
\bigoplus_{\lambda \in \mathscr{P}_{\ell}} 
\tau_\lambda^{(f+1)}$,
where $\mathscr{P}_\ell$ is defined 
in the paragraph after \eqref{eq:J-lambda-def}.
We know 
by \cite[Corollary~2.4.3(i)]{BHHMS4} (with $n = f+1$)
that $\bigoplus_{\ell = 0} ^f \tell$ is multiplicity free.

If $\pi_1$ is a subrepresentation of $\pi$,
for $d  \in \{ 0, \dots, f \}$ 
we set 
\begin{equation}
  \label{eq:tau1d-def}
\tau^{(f + 1)}(\pi_1) \defeq 
\bigoplus  _{\ell = 0}^f
\left( \VFil[\pi_1,\ell] \otimes_{\F}
\tell \right),
\quad
\tau_{\le d}^{(f + 1)}(\pi_1 ) \defeq 
\VFil[\pi_1, d] \otimes_{\F}
\left( \bigoplus _{\ell = 0}^d
\tell[\ell] \right)
\end{equation}
Notice that $\tau^{(f+1)}(\pi) =
\V \otimes_{\F} \bigoplus_{\lambda \in \mathscr{P}} \tau_\lambda^{(f+1)}$
is isomorphic to the $I$-representation
$\tau^{(f+1)}$ of \cite[Lemma~2.4.1]{BHHMS4}
(with $n = f+1$,
and noting that 
\emph{loc.\ cit.}\ does hold for $r \ge 1$).
Let
\begin{equation}
  \label{eq:i-tau-iso-def}
i \colon  \tau^{(f+1)}(\pi) = 
\bigoplus_{ \ell = 0} ^f 
\left( \V \otimes_{\F} \tell  \right)
 \hookrightarrow \pi|_I
\end{equation}
be the $I$-equivariant injection 
of \cite[Lemma~2.4.2]{BHHMS4} (with $n = f+1$).
By \emph{loc.\ cit.}\ we know that 
$i \colon  \tau^{(f+1)}(\pi)[\mI^{f+1}] 
\xrightarrow{\sim} \pi[\mI^{f+1}]$
is an isomorphism,
where $\mI \subseteq \Lambda$ 
is the unique the maximal ideal of $\Lambda$
(cf.\ \Cref{sec:diagrams}).

Moreover, the proof of \emph{loc.\ cit.}\ makes it clear
that \eqref{eq:i-tau-iso-def} can be chosen to be compatible with $\ip$,
in the sense that the composition
\begin{equation}
  \label{eq:i-ip-compatibility}
\V \otimes_{\F}
\left( \bigoplus_{ \lambda \in \mathscr{P}}
\chi_{\lambda} \right)
\cong 
\V \otimes_{\F}
\left( \bigoplus_{ \ell = 0} ^f 
\left(\tell\right)^{I_1 }  \right)
=  \tau^{(f+1)}(\pi)[\mI] 
 \xrightarrow[\sim]{i} \pi^{I_1} \xrightarrow[\sim]{\ip}
\V \otimes_{\F} \DOne
\end{equation}
is $\id_{\V} \otimes \alphaI$
(for the definition of the $H$-equivariant isomorphism 
$\alphaI$, cf.\ \eqref{eq:v-iso-def}).

We conclude the section with the following proposition,
which will be used in Step 1 of the proof of 
\Cref{prop:CM-via-N} below.
\begin{proposition}
  \label{prop:char-only-occurs}
Assume that $\rep$ is $ \max \{6,2f+1\}$-generic,
and keep the above notation.
Let $\pi_1 \subseteq \pi$ be a subrepresentation of $\pi$.
If $\lambda \in 
\mathscr{P} ^{\mathrm{ss}} \setminus 
\mathscr{P}$, and if we set
 $d \defeq \ell(\lambda) = |J_{\lambda}|$,
then the character $\chi_{\lambda}$
can only occur in the following 
$I$-subrepresentation of $\pi_1[\mI^{f+1}]$:
\[
i \left( 
\tau_{\le d}^{(f+1)}(\pi_1)[\mI^{f+1}] \right)
\cap \pi_1[\mI^{f+1}],
\]
where this intersection is taken inside $\pi[\mI^{f+1}]$.
\end{proposition}
\begin{remark}
  \label{rmk:char-only-occurs-r=1}
Notice that \Cref{prop:char-only-occurs}
generalises
\cite[Proposition~4.3.19]{BHHMS4}
to the case $r \ge 1$.
Indeed,
when $r = 1$, \eqref{eq:tau1d-def} becomes
\begin{equation}
  \label{eq:tau1d-minimal}
\tau^{(f + 1)}(\pi_1 ) =  
\bigoplus  _{\ell = 0}^{i_0 }
\tau_\ell^{(f+1)}, \quad
\tau_{\le d}^{(f + 1)}(\pi_1 ) =  
\begin{cases}
\bigoplus  _{\ell = 0}^{d}
\tau_\ell^{(f+1)},
 & \text{if }d \le i_0, \\
0 & \text{otherwise,}
\end{cases}
\end{equation}
where $i_0 $ is as in \cite[Theorem~4.3.15]{BHHMS4}
(cf.\ also \Cref{rmk:srep-structure-r=1}).
In particular, when $d = i_0 +1$ we have 
$\tau_{\le d}^{(f+1)}(\pi_1 ) = 0$, and we recover
\cite[Proposition~4.3.19]{BHHMS4}.
\end{remark}

\begin{proof}
By (the line below) \eqref{eq:i-tau-iso-def},
the statement can be reformulated as follows:
if $U$ is an $I$-subrepresentation of 
$\tau^{(f+1)}(\pi)[\mI^{f+1 } ]
= \V \otimes_{\F} \bigoplus_{\lambda \in \mathscr{P}}
\tau_\lambda^{(f+1)} [\mI^{f+1}]$ 
with $I$-cosocle $\chi_{\lambda}$ such that
$i(U) \subseteq \pi_1 [\mI^{f+1}]$,
then in fact $U \subseteq \tau_{\le d}^{(f+1 )}(\pi_1 ) $.

By \cite[Lemma~3.1.5(ii)]{Ber}
applied to $S= \Lambda$,
$D = 
\bigoplus_{\mu \in \mathscr{P}} 
\tau_{\mu}^{(f+1)} [\mI^{f+1}]$,
$W = \V$, $M = U$,
there is a unique $\mu' \in \mathscr{P}$
and a unique $\F$-vector subspace $W_{\mu'} \subseteq \V$
such that $\soc_{I} U = W_{\mu'} \otimes_{\F} \chi_{\mu'}$
and
$U \subseteq  W_{\mu'} \otimes_{\F} 
\tau_{\mu'}^{(f+1)} [\mI^{f +1 }]$.
We now claim that $W_{\mu'}$ is a line.

If $\psi$ is a Jordan-H\"older constituent of
$\tau_{\mu'}^{(f+1)} [\mI^{f+1}]$,
then we let 
$\tau_{\mu'}(\psi)$ be the unique $I$-subrepresentation of
$\tau_{\mu'}^{(f+1)} [\mI^{f+1}]$ with $I$-cosocle $\psi$,
and we let
$W_\psi \subseteq \V$ 
be the largest vector subspace of $W$
such that $W_\psi \otimes \tau_{\mu'} (\psi) 
\subseteq U$
(notice that $W_{\chi_{\mu'}} = W_{\mu'}$
by construction).
Then, we see that
\begin{equation}
  \label{eq:U-as-a-sum}
U = \sum_{\psi \in \JH(\tau_{\mu'}^{(f+1)} [\mI^{f+1}])}
W_\psi \otimes_{\F} \tau_{\mu'} (\psi)
\end{equation}
by \cite[Lemma~3.1.5(i)]{Ber}
applied to $S= \Lambda$,
$D = 
\bigoplus_{\mu' \in \mathscr{P}}
\tau_{\mu'}^{(f+1)} [\mI^{f+1}]$,
$W = \V$, and $M_1 , M_2 $ the two sides of
\eqref{eq:U-as-a-sum}.
Since $\cosoc_I U = \chi_\lambda$ by assumption,
we know that $W_{\chi_\lambda} =: L$ is a line.
Moreover, consider the projection
$U \twoheadrightarrow U/ (L \otimes_{\F} \tau_ {\mu'} (\chi_\lambda)) =: U'$.
We have that
$\cosoc_I U = \chi_\lambda$ surjects onto $\cosoc_I U'$,
however we also have 
$[U' : \chi_\lambda] = 
[U : \chi_\lambda] - [L \otimes_{\F} \tau_ {\mu'} (\chi_\lambda) : \chi_\lambda] = 1 - 1 = 0$,
so we must have $\cosoc_I U' = 0$, hence $U' = 0$,
hence $U \subseteq 
L \otimes_{\F} \tau_{\mu'} (\chi_\lambda) \subseteq 
L \otimes_{\F} \tau_{\mu'}^{(f+1)} [\mI^{f+1}]$.
This implies that
$W_{\psi} = L$
for all $\psi \in \JH(\tau_{\mu'}(\chi_\lambda))$,
and taking $\psi = \chi_{\mu'}$ 
we conclude that
$W_{\chi_{\mu'}} = W_{\mu'} = L$
is a line.

In particular, $U$ is multiplicity free
because $\tau_{\mu'}^{(f+1)}$ is
(by \cite[Corollary~2.4.3(i)]{BHHMS4} with $n = f+1$).
By definition of $\tau^{(f +1 )}_{\le d} (\pi_1 )$
(cf.\ \eqref{eq:tau1d-def}),
and using that $U \subseteq 
L \otimes_{\F} \tau_{\mu'}^{(f+1)} [\mI^{f+1}]$,
the condition
$U \subseteq \tau_{\le d}^{(f+1)}(\pi_1 )$
(which is equivalent to the statement,
by the first paragraph of this proof)
holds if and only if we have
\begin{equation}
  \label{eq:line-condition-restated}
\ell({\mu'}) \le d = \ell (\lambda), \quad
L \subseteq \VFil[\pi_1 , d].
\end{equation}
By \eqref{eq:i-ip-compatibility}
and by \Cref{thm:srep-structure},
we at least have $L \subseteq \VFil[\pi_1 ,\ell (\mu')]$.
Let  $J_1 $, $J_2 $ be as in \eqref{eq:J1-J2-temp},
and $\mu \in \mathscr{P} ^{\mathrm{ss}}$ 
as in \eqref{eq:mu-ad-hoc}.
As in the line after \eqref{eq:mu-ad-hoc},
we have $\mu \in \mathscr{P}$,
$\chi_{\lambda} = \chi_\mu \prod_{j \in J_1 } \alpha_j^{-1} \prod_{j \in J_2 } \alpha_j$,
$J_1 \subseteq \Yn[\mu]$,
$J_2 \subseteq \Zn[\mu]$, $J_1 \cap J_2 = \emptyset$,
and finally
$\ell (\mu) = |J_\mu| = |J_\lambda| - |J_1 | - |J_2 | = \ell (\lambda) - |J_1 | - |J_2 |$.
In particular, 
this proves the first half of 
\eqref{eq:line-condition-restated}.
We claim that $\mu = \mu'$.

Let $\sigma \in \Wr$ be the Serre weight
determined by $J_\sigma = J_{\rep}  \cap J_{\mu}$,
and $\tau \in \Wss$ be the Serre weight
determined by $J_\tau = J_{\lambda}$.
For a fixed vector $w \in L \setminus \{0\}$,
consider the $I$-equivariant map
$\bm[\mu,w] \colon L \otimes_{\F}W(\chi_\mu, \chi_{\lambda}) \to \pi|_I$  
defined in the paragraph below \eqref{eq:iso-chi-isotypical}
for 
$J_1 = \{j \in J_{\rep} ^c
\mid \lambda_j(x_j)=p-3-x_j\}$,
$J_2 = \emptyset $
(one checks directly that $\Jrep \cap J_\lambda = J_\mu$).

Note that by the proof of \cite[Lemma~4.3.1]{BHHMS4}
\begin{equation}
  \label{eq:W-inside-Inj}
W(\chi_\mu, \chi_\lambda) \subseteq 
(\operatorname{Inj}_{I/Z_1 }\chi_\mu)[\mI^{|J_1 | 
  + |J_2 | + 1}] \subseteq 
(\operatorname{Inj}_{I/Z_1 }\chi_\mu)[\mI^{f + 1}],
\end{equation}
and that it is the unique $I$-subrepresentation of
$(\operatorname{Inj}_{I/Z_1 }\chi_\mu)[\mI^{f + 1}]$
with $I$-cosocle $\chi_\lambda$.
In particular, $\bm[\mu,w]$ takes values in
$ \pi[\mI^{f+1}]$,
so we can consider $i^{-1}(\im (\bm[\mu,w]))$,
which by \eqref{eq:i-ip-compatibility}
and by construction of $\bm[\mu, w]$
(cf.\ the paragraph after 
\eqref{eq:iso-chi-isotypical})
is contained inside 
$L \otimes_{\F}\tau_{\mu}^{(f+1)} [\mI^{f+1}]$.
In conclusion, $\chi_\lambda$ appears as a 
Jordan-H\"older consituent both inside
$i^{-1}(\im(\bm[\mu,w])) \subseteq 
L \otimes_{\F}\tau_{\mu}^{(f+1)} [\mI^{f+1}]$,
and inside $U \subseteq L \otimes_{\F}\tau_{\mu'}^{(f+1)} [\mI^{f+1}]$.
By \cite[Corollary~2.4.3(i)]{BHHMS4} (with $n = f+1$),
we deduce that $\mu = \mu'$.

Notice that both $i^{-1}(\im (\bm[\mu,w]))$ and $U$
are $I$-subrepresentations of 
$L \otimes_{\F}\tau_{\mu}^{(f+1)} [\mI^{f+1}]$
with $I$-cosocle $\chi_\lambda$,
and so they coincide
by the line after \eqref{eq:W-inside-Inj}.
In particular, $\bm[\mu,w]$ takes values in $\pi_1|_I$,
since $i(U) \subseteq \pi_1 [\mI^{f +1 }]$
by hypothesis,
hence the $\GR$-equivariant morphism 
$\btm[\mu,w]$ 
of (the line below) \eqref{eq:beta-tilde-def}
also takes values in
$ \pi_1$. 
We let $Q_L = \im \btm[\mu,w]$ be the image of $\btm[\mu,w]$.

Note that $\chi_\lambda$ occurs inside
$\DssOnex[\tau]$ 
and that we have 
$\ell(\tau) = |J_{\tau} | = |J_{\lambda}| = d$
by \Cref{eq:3.1.3-recalled}.
By the first assertion of 
\Cref{lemma:I-and-QL},
we have $L \otimes_{\F}I(\sigma,\tau) \subseteq 
\ip(Q_L^{K_1 })
\subseteq \ip(\pi_1^{K_1})$.
If we define $\ell(L) \in \{-1, \dots, f\}$
as in \eqref{eq:DL-def},
then we deduce from \Cref{thm:srep-structure}
that $L \otimes_{\F}I(\sigma,\tau) \subseteq 
L \otimes_{\F} \DZeroFil[\ell(L)]$.
By \eqref{eq:D-Dss-Fil-comparison},
$\tau \in \Wss$, and the definition
of $\DssZeroFil[\ell(L)]$,
we see that $d = \ell(\tau) \le \ell(L)$,
or equivalently $L \subseteq \VFil[\pi_1 ,d]$
(cf.\ \eqref{eq:DL-def}),
thus proving the second half of 
\eqref{eq:line-condition-restated}.
\end{proof}

\subsection{Extending \texorpdfstring{\cite{Wang24}}{[Wang24]} to \texorpdfstring{$r \ge 1$}{r ≥ 1}}
  \label{sec:Wang-with-mult}
The main result of this section is 
\Cref{thm:rk-subrep} below,
which computes the $A$-rank of the multivariable
$(\varphi, \mathcal{O}_K ^\times )$-module 
$\DA (\pi_1 )$
associated to a 
subrepresentation $\pi_1 $ of $\pi$
in terms of the numerical invariants 
defined in \eqref{eq:VFil-rFil-def}.
Using 
$\dim_{\FX}(\Dvee (\pi_1 )) = \rk_A (\DA (\pi_1 ))$
we then deduce
\Cref{cor:Wang} below.

In order to prove \Cref{thm:rk-subrep},
in this section 
we extend most of the results from \cite{Wang24}
to the case $r \ge 1$.
In keeping with the assumptions of \cite{Wang24},
throughout this section we assume that 
$\rep$ is reducible and $(2f +1)$-generic.
We specifically allow the case 
$\rep$ reducible split.
Even though \cite{Wang24} only requires
that $\pi$ satisfy
\ref{hypothesis:i} and \ref{hypothesis:ii},
we will also assume hypotheses
\ref{hypothesis:iv}
of \Cref{sec:hypotheses} to hold.
The reason for this is that it is needed for
\Cref{thm:srep-structure}.

Recall the group $N_0 $ defined in
\eqref{eq:N0-def}.
There is an equality $\F [\![N_0 ]\!] 
= \F [\![Y_0 , \dots, Y_{f-1}]\!] $,
where $Y_j \defeq \sum_{a \in \mathbb{F}_q ^{\times}} 
a^{- p^j}
\left( 
\begin{smallmatrix}
1 & [a] \\
0 & 1
\end{smallmatrix}\right)
\in \F [\![N_0 ]\!]$
for $0 \le j \le f-1$.
For $\underline{i} = (i_0, \dots , i _{f-1})
 \in \mathbb{Z}^f$,
we set $\| \underline{i}\| \defeq
\sum_{j = 0}^{f-1} i_j$ 
and define $\underline{Y} ^{\underline{i}}
= \prod_{j = 0} ^{f-1} Y_j ^{i_j}
 \in \F [\![ N_0 ]\!] $.
We endow 
$\F [\![N_0 ]\!]$
with the $(Y_0 , \dots, Y_{f-1})$-adic topology, 
we endow
$\F [\![N_0 ]\!] [1/(Y_0 \cdots Y_{f-1})]$
with the unique topology that makes
$\F [\![N_0 ]\!] \subseteq 
\F [\![N_0 ]\!] [1/(Y_0 \cdots Y_{f-1})]$
an open neibourhood of $0$,
and denote by $A$ 
the topological ring given by the completion of
$\F [\![N_0 ]\!] [1/(Y_0 \cdots Y_{f-1})]$
with respect to this topology.
For an admissible smooth representation 
$\overline{\pi}$ of
$\GF$ over $\F$ with a central character,
we define $\DA (\overline{\pi})$ to be the completion
of $A \otimes_{\F [\![N_0 ]\!] } \overline{\pi}^\vee $
with respect to the tensor product topology,
where $\overline{\pi}^\vee $ is given the $\mI$-adic topology.
We endow $\overline{\pi}^\vee $ 
with the $\mathcal{O}_K ^{\times}$-action
given by $f \mapsto f \circ 
\left( 
\begin{smallmatrix}
a & 0 \\
0 & 1
\end{smallmatrix}\right)$
(for $a \in \mathcal{O}_K ^{\times}$),
which extends by continuity to $\DA (\overline{\pi})$,
and the $\F [\![N_0 ]\!] $-module map
$\overline{\pi}^\vee \to \F [\![N_0 ]\!] 
\otimes_{\varphi, \F [\![N_0 ]\!] } \overline{\pi}^\vee$
(where the $\F$-linear Frobenius 
$\varphi \colon \F [\![N_0 ]\!] \to \F [\![N_0 ]\!] $
comes from multiplication by $p$ on 
$N_0 = \mathcal{O}_K$)
given by $f \mapsto f \circ 
\left( 
\begin{smallmatrix}
p & 0 \\
0 & 1
\end{smallmatrix}\right)$
induces a continuous $A$-module map
\begin{equation}
  \label{eq:beta-DA-definition}
\beta \colon  \DA (\overline{\pi}) \to 
A \otimes_{\varphi, A} \DA (\overline{\pi}).
\end{equation}

Throughout the section we always follow the notation of 
\cite{Wang24}.
For the sake of clarity, 
we highlight some of the differences
in notation
in the following remark.
\begin{remark}
  \label{rmk:Wang-notation}
We begin by recalling some of the notation
from \cite[\S 2]{Wang24}.
For $P$ a statement,
we let $\delta _P = 1$ if $P$ is true,
and $0$ otherwise.
Set $\mathcal{J} \defeq \{0, \dots, f -1\}$,
and for $J \subseteq \mathcal{J}$
define $\underline{e}^J \in \mathbb{Z}^f$
as $e^J_j \defeq \delta_{j \in J}$,
and define 
$\sigma _J \defeq \sigma _{\underline{a}^J}$,
where $\underline{a}^J \in \mathbb{Z}^f$
is as in \cite[Eq.~(5)]{Wang24}
and where $\sigma _{\underline{b}}$
(for $\underline{b} \in \mathbb{Z}^f$)
is defined in the paragraph of 
\cite[Eq.~(4)]{Wang24}.

By \cite[Lemma~3.2(i)]{Wang24}
we have $\JH(\IndI \chi_\lambda^s) = 
\{F (\mathfrak{t}_\lambda (- \underline{b}) )
\mid \underline{0} \le \underline{b} \le\underline{1}\}$,
with $F (\mathfrak{t}_\lambda (- \underline{b}))$
corresponding to the subset $\{j \mid b_{j+1} = 1\}$
of the parametrisation of \cite[\S 2]{BP12}.
By the paragraph of 
\eqref{eq:P-bijection-2}
(and taking $\psi$ to be $\chi_\lambda^s$
in \eqref{eq:P-bijection-1}),
the subset
$\{j \mid b_{j} = 1\}$
corresponds to $\mathcal{S} (\xi)$
in the parametrisation of \eqref{eq:P-bijection-2}.
In particular, 
$ \underline{b} = \underline{e}^{ \emptyset } 
= \underline{0}$
parametrises the socle of $\IndI \chi_\lambda^s$,
i.e.\ the cosocle of $\IndI \chi_\lambda$,
and $\underline{b} = \underline{e}^{ \{0, \dots, f-1\}}
= \underline{1}$
parametrises the cosocle of $\IndI \chi_\lambda^s$,
i.e.\ the socle of $\IndI \chi_\lambda$.

It follows from 
\cite[Eq.~(4)]{Wang24}
that $J \mapsto \sigma_{\underline{e}^J}$
gives a bijection 
$\{J \subseteq \mathcal{J}\} \xrightarrow{\sim} \Wss$,
and restricts to a bijection 
$\{J \subseteq \Jrep\} \xrightarrow{\sim} \Wr$.
Composing with the bijection
$\mu \in \mathscr{D} ^{\mathrm{ss}}\mapsto J_\mu \subseteq \mathcal{J}$
of \eqref{eq:J-lambda-def}
gives $\sigma_\mu =  \sigma_{\underline{e}^{J_\mu}}$.
For $J \subseteq \mathcal{J}$,
there is also the Serre weight $\sigma_J$
defined in the line after
\cite[Eq.~(4)]{Wang24},
and $\chi_J$ the $I$-character acting on 
$\sigma_J^{I_1 }$.
A simple computation shows that the assignment 
$\mu \in \mathscr{D} ^{\mathrm{ss}}
\mapsto \lambda \in \mathscr{P}$
given by 
\begin{equation}
  \label{eq:lambda-bij-Jsh}
\lambda_j (x_j) \defeq 
\begin{cases}
p-1-x_j & \text{if }\mu_j (x_j) = p-3-x_j,
\ j \notin \Jrep \\
\mu_j (x_j) & \text{otherwise}
\end{cases}
\end{equation}
defines a bijection 
between $\Dss$ and 
$ \{\lambda \in \mathscr{P} \mid 
y_j \notin \mathfrak{a} (\lambda)
\ \forall j \in \mathcal{J}\}$
which is the inverse of the bijection 
in the proof of \cite[Lemma~3.59(ii)]{BHHMS2}
(note that these two bijections also hold for
$\rep$ reducible split).

For every subset $J \subseteq \mathcal{J}$,
recall the subsets
$J ^{\mathrm{ss}} \defeq  J \cap \Jrep$,
$J ^{\mathrm{nss}} \defeq  J \setminus  \Jrep 
= J \setminus J ^{\mathrm{ss}}$,
$J ^{\mathrm{sh}} \defeq  J \cap (J -1 )\cap \Jrep 
\subseteq J ^{\mathrm{ss}}$
defined in \cite[Eq.~(8)]{Wang24},
and the subset
$\partial J \defeq  J \setminus (J - 1)$
defined in \cite[Remark~5.11]{Wang24}.
Keeping the above notation, 
and using \cite[Eq.~(6),(7)]{Wang24},
it is a direct computation that 
\eqref{eq:lambda-bij-Jsh}
implies $\chi_{J_\mu} 
= \chi_\mu \alpha^{\underline{e}^{J ^{\mathrm{sh}}}}
= \chi_\lambda$,
hence
$\sigma_{J_\mu} = \sigma_\lambda$
because $\DOne$ is multiplicity free,
and by the $0$-genericity of $\rep$.
\end{remark}

Following the paragraph before
\cite[Proposition~4.2]{Wang24},
for each $J \subseteq \mathcal{J}$
we fix once and for all a nonzero $H$-eigenvector
$v_J \in \DOne$ with eigenvalue $\chi_J$.
By the second paragraph of
\Cref{rmk:Wang-notation}
we can pick $v_J$ to be equal to 
the $H$-eigenvector $v_\lambda \in \DOne$
of the paragraph of \eqref{eq:v-iso-def},
where $\lambda$ is defined as in 
\eqref{eq:lambda-bij-Jsh}
in terms of the unique $\mu \in \Dss$
such that $J = J_\mu$.

Recall that, by \cite[Proposition~4.2]{Wang24},
for all $J \subseteq \mathcal{J}$ and all
$\underline{i} \in \mathbb{Z}^f$ such that 
$\underline{0} \le \underline{i} \le \underline{f} - \underline{e}^{J ^{\mathrm{sh}}}$
there is a unique $H$-eigenvector $\underline{Y}^{- \underline{i}} v_J \in \DZero$
satisfying
\begin{align}
  \label{eq:YvJ-def-i}
Y_j^{i_j + 1} (\underline{Y}^{- \underline{i}} v_J) &= 0 \ \forall j \in \mathcal{J}; \\
\nonumber
  \label{eq:YvJ-def-ii}
\underline{Y}^{\underline{i}} (\underline{Y}^{- \underline{i}} v_J) &= v_J.
\end{align}
We let $\WJ$ be the $I$-subrepresentation of $\DZero$
generated by $\underline{Y}^{-\underline{i}} v_J$,
and we let $\PWJ$ be the $I$-representation
with underlying vector space $\WJ$
but action of $g \in I$ on $x \in \WJ$
given by $\Pi^{-1} g \Pi \cdot x$.

Let $J \subseteq \mathcal{J}$ and 
$\underline{i} \in \mathbb{Z}^{f}$ such that 
$\underline{0} \le \underline{i} \le \underline{f} - \underline{e}^{J^{\mathrm{sh}}}$.
Consider the $I$-equivariant map
$V \otimes_{\F} \WJ \subseteq V \otimes_{\F} \DZero  
\xrightarrow{\ip^{-1}} \pi$.
Applying $\Pi$, we obtain the following $I$-equivariant injection
\begin{equation}
  \label{eq:psi0-def}
 \pJo \colon  V \otimes_{\F} \PWJ \to \pi,
 \quad
\pJo(w \otimes x) \defeq  \Pi \ip^{-1}(w \otimes  x),
\end{equation}
for $w \in V,\ x \in \PWJ$.

Set $\VJ \defeq \IndI(\PWJ)$.
By Frobenius reciprocity applied to $\pJo$ we obtain a $\GR$-equivariant map
\begin{equation}
  \label{eq:psi-def}
\pJ \colon  V \otimes_{\F} \VJ \to \pi.
\end{equation}
By the proof of 
\cite[Lemma~3.48]{BHHMS2}
we have $\VJ$ multiplicity free, 
so we may define $\QJ$ as the unique quotient of $\VJ$
whose socle is the Jordan-H\"older constituent of $\IndI(\chi_J^s) \subseteq \VJ$
corresponding to $(J \bigtriangleup (J-1)^{\mathrm{ss}})-1$
via \eqref{eq:P-bijection-1} and
\eqref{eq:P-bijection-2},
and set $\KJ \defeq \ker(\VJ \twoheadrightarrow \QJ)$.
\begin{remark}
  \label{rmk:QJ-angle-bracket}
Keep the above notation.
It follows from the definitions that 
\begin{align}
\nonumber
\WJ & \cong  \left\langle I \cdot \underline{Y}^{-\underline{i}} v_J \right\rangle, \\ 
\nonumber
\pJo \left( \V \otimes_{\F} \PWJ \right)
&= \left\langle I \cdot \Pi\ip^{-1} \left( \V \otimes_{\F}
\underline{Y}^{-\underline{i}} v_J  \right)\right\rangle, \\
  \label{eq:pJ-image}
\pJ \left( V \otimes_{\F} \VJ \right) 
&= \left\langle \GR \cdot \Pi\ip^{-1} \left( \V \otimes_{\F}
\underline{Y}^{-\underline{i}} v_J  \right)\right\rangle.
\end{align}
The $I$-representation $\WJ$ (resp.\ $\PWJ$) is isomorphic to the 
$I$-representation $W'$ (resp.\ $W$) in the first paragraph of the proof of \cite[Lemma~5.1]{Wang24},
and the $\GR$-representation $\VJ$ is isomorphic to the 
$\GR$-representation $V$ of \emph{loc.\ cit.}
\Cref{lemma:Wang-first-lemma}(i)
below shows that 
the $\GR$-representation $\QJ$ is isomorphic to the 
$\GR$-representation 
$Q(\chi_J^s, \chi_J^s \alpha^{\underline{i}}, 
(J \bigtriangleup (J-1)^{\mathrm{ss}})-1)$
of \cite[Remark~5.2]{Wang24}.
\end{remark}

Fix $w \in \V$. 
For all 
$\underline{k} \in \mathbb{Z}_{\ge 0}^f$
and all
$\underline{i}' \in \mathbb{Z} _{\ge 0}^f$ such that 
$\underline{i}' \le \underline{i}$,
the right-hand side of 
\eqref{eq:pJ-image}
contains the element
$\underline{Y}^{\underline{k}} \mo
\Pi \ip^{-1}
\left( w \otimes \underline{Y}^{- \underline{i}'} v_J
\right)
= 
\underline{Y}^{\underline{k}} 
\mpi \ip^{-1}
\left( w \otimes \underline{Y}^{- \underline{i}'} v_J
\right)$,
where we have used that 
$\mpi = 
\mo
\left( 
\begin{smallmatrix}
0 & 1 \\
p & 0
\end{smallmatrix} \right) = 
\mo
\Pi$.
We now construct a preimage of this element 
inside $\V \otimes_{\F} \VJ$.
Following the notation of the paragraph before
\cite[Lemma~3.5]{NewYork},
if $\tau$ is a smooth $I$-representation,
for $g \in \GR$ and $x \in \tau$
we let $[g, x] \in \IndI \tau$
be the unique element with support on $I g^{-1}$
which sends $g^{-1}$ to $x$.
We have
$w \otimes \underline{Y}^{- \underline{i}} v_J
\in V \otimes_{\F} \PWJ$,
hence we can consider
$Y^{\underline{k}} \mo
\cdot
\big[\id, 
\big( w \otimes \underline{Y}^{- \underline{i}'} v_J
\big)\big]
\in V \otimes_{\F} \VJ$,
and compute
\begin{align}
  \label{eq:interpret-inside-VV}
\pJ 
&
\left(
Y^{\underline{k}} {\mo}
\cdot 
\big[\id, 
\big( w \otimes \underline{Y}^{- \underline{i}'} v_J
\big)\big]
\right)
= 
\underline{Y}^{\underline{k}} {\mo}
\pJo 
\left( w \otimes \underline{Y}^{- \underline{i}'} v_J
\right)
\\
\nonumber
&\overset{\eqref{eq:psi0-def}}{= }
\underline{Y}^{\underline{k}} {\mo}
\Pi \ip^{-1}
\left( w \otimes \underline{Y}^{- \underline{i}'} v_J
\right)
= \underline{Y}^{\underline{k}} 
{\mpi} \ip^{-1}
\left( w \otimes \underline{Y}^{- \underline{i}'} v_J
\right)
\in \pi,
\end{align}
where in the first identity we have used 
Frobenius reciprocity.

Analogously, we have
$w \otimes v_J
\in V \otimes_{\F} \PWJ$,
and we can compute that $\pJ$
sends
$\big[\id, 
\big( w \otimes v_J
\big)\big]
\in V \otimes_{\F} \VJ$
to $\ip^{-1}\Pi (w \otimes v_J)
= \ip^{-1}(w \otimes \Pi v_J) \in \pi$
(using that \eqref{eq:ip-def}
is an isomorphism of diagrams).
In other words,
if we let $\phi \defeq [\id, 1] \in \IndI (\chi_J^s)$
be the element 
defined in \cite[\S 2]{BP12},
then we have an $I$-equivariant injection
$\F w \otimes_{\F}\chi_J^s \hookrightarrow  
\F w \otimes_{\F} \PWJ$
(corresponding to the inclusion
$\F w \otimes_{\F} \chi_J \subseteq 
\F w \otimes_{\F} \WJ$
sending $w \otimes 1 $ to $w \otimes v_J$),
and taking inductions we obtain
a $\GR$-equivariant injection 
$\F w \otimes_{\F }\IndI (\chi_J^s)
\hookrightarrow  \F w \otimes_{\F} \VJ$
whose composition with $\pJ$
sends $w \otimes \phi$ to $\ip^{-1}(w \otimes \Pi v_J)$.

The following lemma generalises \cite[Lemma~5.1]{Wang24}.
\begin{lemma}
  \label{lemma:Wang-first-lemma}
Keep the above assumptions.
\begin{enumerate}[(i)]
\item 
The $\GR$-equivariant map 
$\pJ$ of \eqref{eq:psi-def} factors uniquely as
\[
\pJ \colon V \otimes_{\F} \VJ \twoheadrightarrow 
V \otimes_{\F} \QJ \xhookrightarrow{\pbJ}
\pi.
\]
\item
The $\GR$-representation $\QJ$ is multiplicity free,
with $\GR$-socle $\sigma_{(J-1) ^{\mathrm{ss}}}$
and $\GR$-cosocle $\sigma_{\underline{c}}$,
where
\[
c_j = (-1)^{1 - \delta_{j+1 \in J}} \left( 
2i_j +1 + \delta_{j \in (J-1) ^{\mathrm{ss}}} - \delta_{j \in J \triangle (J-1) ^{\mathrm{ss}}} \right),
\]
for all $j \in \mathcal{J}$,
is the same integer as in \cite[Eq.~(13)]{Wang24}

\item 
We have
\begin{equation}
  \label{eq:QJ-subquot}
 \frac{\QJ}
{\sum_{\underline{0}\le \underline{i}' < \underline{i}} \QJ[J, \underline{i}']}
\cong Q(\chi_J^s \alpha^{\underline{i}}, 
\{j \mid j+1 \in J \triangle (J -1 ) ^{\mathrm{ss}}, i_{j+1} = 0\} ),
\end{equation}
where the term on the right-hand side is defined in
\cite[Remark~5.2]{Wang24}.

\item 
For $\underline{m} \in \mathbb{Z}^{f}$ such that
all $m_j$ are between $\delta_{j \in (J -1 ) ^{\mathrm{ss}}}$
and $c_j$
there is a unique subrepresentation
$\Isigma$ of $\QJ$ with $\GR$-cosocle $\sigma_{\underline{m}}$.
Moreover, $\Isigma$ has constituents $\sigma_{\underline{b}}$
with each $b_j$ between $\delta_{j \in (J-1) ^{\mathrm{ss}}}$
and $m_j$, and we have 
\begin{equation}
  \label{eq:I-restr-iso}
\Hom_{\GR}(\Isigma, \pi) \xrightarrow{\sim}
 \Hom_{\GR}(\sigma_{(J-1) ^{\mathrm{ss}}}, \pi) \cong \V,
\end{equation}
where the map is given by restriction along
$\sigma_{(J-1)^{\mathrm{ss}}} \subseteq I(\sigma_{(J-1) ^{\mathrm{ss}}}, \sigma_{\underline{m}})$.
\end{enumerate}
\end{lemma}
\begin{proof}
(i)
Fix a nonzero $w \in \V$ and denote by $\psi_w$ the restriction
of $\pJ$ to $\F w \otimes_{\F} \VJ$.
We prove that $\ker(\psi_w) = \F w \otimes_{\F} \KJ$,
which is enough to conclude
by \cite[Lemma~3.1.5(i)]{Ber} applied to 
$S = \F [\GR]$, $D = \VJ$, $M_1 = \V \otimes_{\F} \KJ$, $M_2 = \ker \pJ$
(by the proof of 
\cite[Lemma~3.48]{BHHMS2},
the $\GR$-representation $\VJ$ is multiplicity free).

The Claim in the proof of \cite[Lemma~5.1(i)]{Wang24}
(recall that $\VJ$ is isomorphic to the 
$\GR$-representation $V$ of the Claim, 
cf.\ \Cref{rmk:QJ-angle-bracket})
gives us the inclusion
$\Wr \cap \JH(\VJ) \subseteq \JH \left( \IndI (\chi_J^s) \right)$.
In particular, for every Serre weight $\sigma$
in the socle of $\im (\psi_w) \subseteq \pi$ we have
$\sigma \in \Wr \cap \JH(\VJ) \subseteq \JH \left( \IndI (\chi_J^s) \right)$,
so
\begin{equation}
  \label{eq:socle-equality}
\socR \im(\psi_w) = \socR \psi_w 
\left(\F w \otimes_{\F} \IndI (\chi_J^s) \right),
\end{equation}
and we are left to show that \eqref{eq:socle-equality}
is isomorphic to $\sigma_{(J-1) ^{\mathrm{ss}}}$,
which by \cite[Lemma~2.1(ii)]{Wang24}
is isomorphic to $F(\mathfrak{t}_{\lambda_J}
(- \underline{e}^{J \triangle (J -1 ) ^{\mathrm{ss}}}))$,
and so is parametrised by 
$J \triangle (J -1 ) ^{\mathrm{ss}}$
under \eqref{eq:P-bijection-1}
and \eqref{eq:P-bijection-2}
(cf.\ also \Cref{rmk:Wang-notation}).

If we denote by $\psi_w'$ the restriction of
$\psi_w$ 
to $\F w \otimes_{\F} \IndI(\chi_J^s) \to \pi^{K_1 }$,
then using the fact that 
$\IndI (\chi_J^s)$ is $K_1 $-invariant
we see that $\psi_w'$ takes values in
$ \pi^{K_1 }$.
Notice that $\psi_w$ is nonzero, 
otherwise by Frobenius reciprocity
we would deduce that 
$\pJo (\F w \otimes_{\F} \PWJ) = 0$,
(cf.\ \eqref{eq:psi0-def}),
and this in turn would imply 
(upon applying $\Pi^{-1}$ to $\pJo$,
cf.\ the line before \eqref{eq:psi0-def}) 
that the inclusion
$\F w \otimes_{\F} \WJ \subseteq 
\F w \otimes_{\F} \DZero \xrightarrow{\ip^{-1}} \pi$
is zero, contradiction.
In particular, 
$\socR \im (\psi_w') 
\overset{\eqref{eq:socle-equality}}{= }
\socR \im (\psi_w) \neq 0$,
which implies $\im (\psi_w ') \neq 0$.
By Frobenius reciprocity,
we see that
$\psi_w'(\F w \otimes_{\F} \chi_J^s)
\subseteq \pi^{K_1}|_I$
is also nonzero, 
hence $\psi_w'(\F w \otimes_{\F} \chi_J^s)$ 
is necessarily isomorphic to $\chi_J^s$.
By \cite[Lemma~4.1(iii)]{Wang24},
we deduce that $\ip \circ \psi_w'$ takes values in 
$\F w \otimes_{\F} \Dx[\sigma_{(J-1) ^{\mathrm{ss}}}]$.
This forces
$\ip \left( \socR \im (\psi_w') \right)
= \socR \im (\ip \circ \psi_w') 
= \socR \left( \F w \otimes_{\F} \Dx[\sigma_{(J-1) ^{\mathrm{ss}}}] \right)
= \F w \otimes_{\F}\sigma_{(J-1) ^{\mathrm{ss}}}$,
completing the proof of (i).

(ii), (iii)
Recall that the proof of 
\cite[Lemma~3.48]{BHHMS2}
shows that the $\GR$-representation 
$\VJ$ is multiplicity free,
a fortiori
$\F w \otimes_{\F}\QJ$ 
(which by \Cref{lemma:Wang-first-lemma}(i)
is abstractly isomorphic to the $\GR$-representation
$\overline{V}$
defined in the proof 
\cite[Lemma~5.1(i)]{Wang24})
is also multiplicity free.
Then, the proof of \cite[Lemma~5.1(i),(ii)]{Wang24}
goes through without problems.

(iv)
Since $\QJ$
is abstractly isomorphic to the $\GR$-representation
$\overline{V}$
defined in the proof 
\cite[Lemma~5.1(i)]{Wang24},
we deduce from
\cite[Lemma~5.1(iii)]{Wang24} that 
$\QJ$ contains a 
(necessarily unique)
$\GR$-subrepresentation isomorphic to 
$I\left(\sigma_{(J-1) ^{\mathrm{ss}}}, \sigma_{\underline{m}}\right)$,
for all $\underline{m} \in \mathbb{Z}^f$
such that all $m_j$ are between 
$\delta_{j \in (J -1 ) ^{\mathrm{ss}}}$ and $c_j$;
we also deduce from
\cite[Lemma~5.1(iii)]{Wang24} that 
$\Isigma$ has constituents $\sigma_{\underline{b}}$
with each $b_j$ between 
$\delta_{j \in (J-1) ^{\mathrm{ss}}}$
and $m_j$.
We are only left to show \eqref{eq:I-restr-iso}.

It follows from the proof of
\cite[Lemma~5.1(iii)]{Wang24}
that any Jordan-H\"older constituent of 
$I\left(\sigma_{(J-1) ^{\mathrm{ss}}}, \sigma_{\underline{m}}\right)$
which is also an element of $\Wr$ must appear in 
$\Dx[\sigma_{(J-1) ^{\mathrm{ss}}}]$,
hence has to be $\sigma_{(J-1) ^{\mathrm{ss}}}$.
A dévissage on the Jordan-H\"older constituents of 
$\Isigma/\sigma_{ (J -1 )^{\mathrm{ss}}}$ 
shows that 
$\Hom_{\GR}(\Isigma /\sigma_{ (J -1 )^{\mathrm{ss}}}, \pi)
= 0$,
which implies that the restriction morphism
$\Hom_{\GR}(\Isigma, \pi) \hookrightarrow 
\Hom_{\GR}(\sigma_{(J-1) ^{\mathrm{ss}}}, \pi)
= \Hom_{\GR}(\sigma_{(J-1) ^{\mathrm{ss}}}, \pi^{K_1})
\overset{\ip}{\cong }\V$
is injective.
However, we also have an $\F$-linear embedding
$\V \hookrightarrow \Hom_{\GR}(\Isigma, \pi)$,
given by sending $w \in \V$ to the restriction of 
$\pbJ$ to $\F w \otimes_{\F}
I\left(\sigma_{(J-1) ^{\mathrm{ss}}}, \sigma_{\underline{m}}\right)
\subseteq \F w \otimes_{\F} \QJ$.
We conclude for dimension reasons.
\end{proof}

\begin{remark}
\label{rmk:interpret-inside-VQ}
Using \Cref{lemma:Wang-first-lemma}(i),
we can rewrite \eqref{eq:interpret-inside-VV}
as
\begin{equation}
  \label{eq:interpret-inside-VQ}
  \pbJ \left( 
  Y^{\underline{k}} {\mo}
  \cdot
  \big[\id, 
  \big( w \otimes \underline{Y}^{- \underline{i}'} v_J
  \big)\big]
  \right)
  = \underline{Y}^{\underline{k}} 
  {\mpi} \ip^{-1}
  \left( w \otimes \underline{Y}^{- \underline{i}'} v_J
  \right)
  \in \pi,
\end{equation}
where by an abuse of notation
we still denote by
$Y^{\underline{k}} {\mo}
\cdot \big[\id, 
\big( w \otimes \underline{Y}^{- \underline{i}'} v_J
\big)\big]$
the image of 
$Y^{\underline{k}} {\mo}
\cdot \big[\id, 
\big( w \otimes \underline{Y}^{- \underline{i}'} v_J
\big)\big] \in \V \otimes_{\F} \VJ$
inside $\V \otimes_{\F} \QJ$.
\end{remark}

\begin{remark}
  \label{rmk:uniqueness-of-Isigma}
Let $J, J' \subseteq \mathcal{J}$ be such that
$(J-1) ^{\mathrm{ss}} = (J' -1) ^{\mathrm{ss}}$,
and let $\underline{i}, \underline{i}' \in \mathbb{Z}^f$
be such that 
$\underline{0} \le \underline{i} \le  \underline{f} - \underline{e}^{J^{\mathrm{sh}}}$,
$\underline{0} \le \underline{i}' \le  \underline{f} - \underline{e}^{J^{\prime\mathrm{sh}}}$.
Fix a nonzero $w \in \V$.
A consequence of \Cref{lemma:Wang-first-lemma}(iv) is that,
whenever $\underline{m} \in  \mathbb{Z}^f$ 
satisfies the assumptions 
of \emph{loc.\ cit.}\ for the pairs
 $(J, \underline{i})$ and $(J', \underline{i}')$, 
then we have
\begin{equation}
  \label{eq:uniqueness-of-Isigma}
\pbJ[J, \underline{i}](\F w \otimes_{\F} I( \sigma_{(J-1) ^{\mathrm{ss}}}, \sigma_{\underline{m}}))
= 
\pbJ[J', \underline{i}'](\F w \otimes_{\F} I( \sigma_{(J'-1) ^{\mathrm{ss}}}, \sigma_{\underline{m}}))
\end{equation}
inside $\pi$.
\end{remark}

\paragraph{Extending \cite[§§1--4]{Wang24}:}
the hypothesis $r = 1$ is never used.
\paragraph{Extending \cite[§5]{Wang24}:}
as a general rule, 
the correct way to extend the results 
and proofs of \cite[§5]{Wang24}
is to fix a nonzero $w \in \V$ at the outset,
substitute $\underline{Y}^{- \underline{i}} v_J$
with $\ip^{-1} \left( w \otimes \underline{Y}^{- \underline{i}} v_J \right)$,
and $\underline{Y}^{\underline{k}}
\mpi (\underline{Y}^{- \underline{i}} v_J)$
with either side of
\eqref{eq:interpret-inside-VQ}
with $\underline{i} = \underline{i}'$
(keeping the same assumptions on $J$, $\underline{i}$).
Moreover,
$I (\sigma_{(J -1 ) ^{\mathrm{ss}}}, 
\sigma_{\underline{m}})$
should always be substituted by
$\pbJ (I (\sigma_{(J -1 ) ^{\mathrm{ss}}}, 
\sigma_{\underline{m}}))$
(keeping the same assumptions on $\underline{m}$).
From now on we will do this without further notice,
and only point out when more care is needed.
\begin{itemize}
\item 
\cite[Lemma~5.1]{Wang24}: addressed by \Cref{lemma:Wang-first-lemma}.
\item 
\cite[Proposition~5.4]{Wang24}: 
in the second line of the proof, 
we replace $\phi$ by $w \otimes \phi$
(cf.\ the paragraph after 
\eqref{eq:interpret-inside-VV}).
In the last sentence, 
we use that 
$\F w \otimes_{\F}\QJ[J, \underline{0}]$ 
is multiplicity free.
We conclude that there exists
a unique scalar 
$\mu_{J, (J-1) ^{\mathrm{ss}}} \in \F ^{\times}$
as in the statement of \emph{loc.\ cit.}
\end{itemize}
\begin{remark}
  \label{rmk:mu-uniform}
Notice that
$\mu_{J, (J-1) ^{\mathrm{ss}}} \in \F ^{\times}$
does not depend on the choice of 
$w \in \V \setminus \left\{ 0  \right\}$.
To do this, 
use \eqref{eq:interpret-inside-VQ}
to interpret (the analogue of)
\cite[Eq.~(16)]{Wang24}
as an identity inside 
$\F w \otimes_{\F} \QJ[J, \underline{0}]$,
and use that for all
$w_1 , w_2  \in \V \setminus \left\{ 0  \right\}$
there exists a $\GR$-equivariant isomorphism
$\F w_1  \otimes_{\F} \QJ[J, \underline{0}]
\xrightarrow{\sim} \F w_2  \otimes_{\F} 
\QJ[J, \underline{0}]$
(namely, the one obtained by tensoring
$\F w_1 \xrightarrow{\sim} \F w_2 $,
$w_1  \mapsto w_2 $, 
with the identity on $\QJ[J, \underline{0}]$)
that sends
$\mo\ip^{-1} \left( w_1 \otimes \Pi v_J \right)$
to
$\mo\ip^{-1} \left( w_2 \otimes \Pi v_J \right)$
for all $J \subseteq \mathcal{J}$.
\end{remark}

\begin{itemize}
\item 
\cite[Proposition~5.5]{Wang24}, 
\cite[Remark~5.6]{Wang24}, 
\cite[roposition~5.7]{Wang24}: 
their proofs go through without problems,
\emph{mutatis mutandis}.

\item 
\cite[Proposition~5.8]{Wang24}:
 following the paragraph before \cite[Lemma~5.9]{Wang24} 
we denote by 
$B_1 (w)$ (resp.\ $B_2 (w)$) 
the left-hand side (resp.\ right-hand side)
of (the analogue of) \cite[Eq.~(29)]{Wang24};
note that we have 
$B_1 (w) \in \pbJ(\F w \otimes_{\F} \QJ)$
and
$B_2 (w) \in \pbJ[J \setminus \{j_0 +2\}, \underline{i}']
(\F w \otimes_{\F} \QJ[J \setminus \{j_0 +2\},
\underline{i}'])$.
In the second paragraph of the proof
of \cite[Proposition~5.8]{Wang24},
we substitute
$B_1 , B_2 \in 
I(\sigma_{(J-1) ^{\mathrm{ss}}}, \sigma_{\underline{m}})
= 
I(\sigma_{(J-1) ^{\mathrm{ss}}}, \sigma_{\underline{m}'})
$
(cf.\ \cite[Lemma~5.9(i)]{Wang24}
for the definition of 
$\underline{m}, \underline{m}' 
\in \mathbb{Z}^f$)
with
\begin{equation*}
B_1 (w),
B_2 (w)\in \pbJ[J, \underline{i}] \left(  
\F w \otimes_{\F} I(\sigma_{(J-1) ^{\mathrm{ss}}}, \sigma_{\underline{m}})\right)
\overset{\eqref{eq:uniqueness-of-Isigma}}{= }
\pbJ[J \setminus \{j_0 +2\}, \underline{i}'] \left(  
\F w \otimes_{\F} I(\sigma_{(J-1) ^{\mathrm{ss}}}, \sigma_{\underline{m}'})\right),
\end{equation*}
in particular $B_2 (w) \in \pbJ(\F w \otimes_{\F} \QJ)$.
The rest of the proof goes through
without problems.

\item \cite[Proposition~5.10]{Wang24}:
following the proof of \emph{loc.\ cit.}\ we denote 
the left-hand side of 
(the analogue of) \cite[Eq.~(37)]{Wang24}
by $B(w) \in 
\pbJ[J, \underline{e}^{(J \cap J') ^{\text{nss}}}]
\left( \F w \otimes_{\F} \QJ[J, 
\underline{e}^{(J \cap J') ^{\text{nss}}}] \right)$.
In \cite[Eq.~(40)]{Wang24}
we substitute $\Dx[\sigma_{(J -1 ) ^{\mathrm{ss}}}]$
with 
$\ip^{-1}(\F w \otimes_{\F}
\Dx[\sigma_{(J -1 ) ^{\mathrm{ss}}}])$.
Observe that $I (\sigma_{(J -1 ) ^{\mathrm{ss}}},
\sigma_{J'})$
is $K_1 $-invariant by the inclusion
\cite[Eq.~(40)]{Wang24}.
By construction,
$\pbJ[J, \underline{e}^{(J \cap J') ^{\text{nss}}}]
|_{
\F w \otimes_{\F} I (\sigma_{(J -1 ) ^{\mathrm{ss}}},
\sigma_{J'})}
\linebreak
\in \Hom_{\GR}(\Isigma[\sigma_{J'}], \pi)$
corresponds to $w \in \V$
under \eqref{eq:I-restr-iso}.
It follows 
from \cite[Lemma~3.1.5(ii)]{Ber}
applied to $S = \F[\Gamma],
D = \DZero, W = \V, 
M = \pbJ[J, \underline{e}^{(J \cap J') ^{\text{nss}}}]
( \F w \otimes_{\F} I (\sigma_{(J -1 ) ^{\mathrm{ss}}},
\sigma_{J'}))$
that 
$ \pbJ[J, \underline{e}^{(J \cap J') ^{\text{nss}}}]
( \F w \otimes_{\F} I (\sigma_{(J -1 ) ^{\mathrm{ss}}},
\sigma_{J'}))
\subseteq \ip^{-1}(\F w \otimes_{\F}
\DZero
)$,
and even that 
$ \pbJ[J, \underline{e}^{(J \cap J') ^{\text{nss}}}]
( \F w \otimes_{\F} I (\sigma_{(J -1 ) ^{\mathrm{ss}}},
\sigma_{J'}))
\subseteq \ip^{-1}(\F w \otimes_{\F}
\Dx[\sigma_{(J -1 ) ^{\mathrm{ss}}}])$,
because $\sigma_{J'}$ is a 
Jordan-H\"older constituent of
$\Dx[\sigma_{(J -1 ) ^{\mathrm{ss}}}])$
(cf.\ the line after \cite[Eq.~(40)]{Wang24}).

The rest of the proof goes through
without problems.
We argue as in
\Cref{rmk:mu-uniform}
to show that $\mu_{J, J'}$
does not depend on $w$.
\end{itemize}
if $\Jrep = \emptyset $, define 
$\xJ[ \emptyset , \underline{r}](w) \defeq 
\mu_{ \emptyset , \emptyset }^{-1} \underline{Y}^{\underline{p} - \underline{1} - \underline{r}} \Pi
\ip^{-1}(w \otimes v_{ \emptyset })$,
so that $\underline{Y}^{\underline{r}} \xJ[ \emptyset , \underline{r}](w) 
= \ip^{-1}(w \otimes v_{ \emptyset })$
by (the $r \ge 1$ analogue of) \cite[Proposition~5.4]{Wang24}
applied to $J = \emptyset $.
This agrees with the definition of $\xJ(w)$
given in \Cref{thm:xJi-def} below.

\begin{itemize}
\item \cite[Proposition~5.12]{Wang24}:
after replacing $\phi$
by $w \otimes \phi$
(both times it appears in the proof)
like we did in the proof of 
\cite[Proposition~5.4]{Wang24},
the proof goes through without problems.
We argue as in
\Cref{rmk:mu-uniform}
to show that $\mu_{J, J'}$
does not depend on $w$.
\end{itemize}

\paragraph{Extending \cite[§6]{Wang24}:}
in addition to the previous indications,
as a general rule
the correct way to extend the results 
and proofs of 
\cite[§6]{Wang24}
is to fix a nonzero $w \in \V$ at the outset,
and substitute $\xJ$ with $\xJ(w)$
(keeping the same assumptions on $J$, $\underline{i}$).
From now on we will do this without further notice,
and only point out when more care is needed.

\begin{itemize}
\item \cite[Theorem~6.3]{Wang24}:
we report the extension for $r \ge 1$ below,
for ease of reference.
\end{itemize}
\begin{theorem}
  \label{thm:xJi-def}
Fix $w \in \V$.
Then, there exists a unique family 
$ \left\{ \xJ(w) \middle| J \subseteq \mathcal{J}, \ \underline{i} \in \mathbb{Z}^f \right\}$
of elements of $\pi$ satisfying the following properties:

\begin{enumerate}[(i)]
\item 
For each $J \subseteq \mathcal{J}$,
we have $\xJ[J, \underline{f}](w) = \ip^{-1} \big( w \otimes  \underline{Y}^{- (\underline{ f} - \underline{e}^{J ^{\mathrm{sh}}}) }v_J \big)$.

\item 
For each $J \subseteq \mathcal{J},\ \underline{i} \in \mathbb{Z}^f$ 
and $\underline{k} \in \mathbb{Z}_{\ge 0}^f$,
we have $\underline{Y}^{\underline{k}} \xJ(w) = \xJ[J, \underline{i} - \underline{k}](w)$.

\item 
For each $J \subseteq \mathcal{J}$ and $\underline{i} \in \mathbb{Z}^f$, 
we have
\[
\mpi
\xJ[J+1, \underline{i}](w) = 
\sum_{J ^{\mathrm{ss}} \subseteq J' \subseteq J} \varepsilon_{J'} \mu_{J+1, J'} \xJ[J', p \delta(\underline{i})+ \underline{c}^J + \underline{r}^{J \setminus J'}](w).
\]
\end{enumerate}
\end{theorem}
\begin{remark}
  \label{rmk:xJ-linear}
We observe that $\xJ(w)$ 
depends linearly on $w \in \V$,
in the sense that 
$\label{eq:xJ-linear}
\xJ(\lambda_1 w_1 + \lambda_2 w_2) = \lambda_1 \xJ(w_1) + \lambda_2 \xJ(w_2)$
for all $w_1 , w_2 \in \V$,
$\lambda_1 , \lambda_2 \in \F$.
Indeed, the right-hand side
of the previous equation
satisfies (i), (ii), (iii) of 
\Cref{thm:xJi-def} 
(applied to $w = \lambda_1 w_1 + \lambda_2 w_2$),
and we conclude from the uniqueness in
\Cref{thm:xJi-def}.
In particular, $\xJ (0) = 0$ for all
$J \subseteq \mathcal{J}$, 
$\underline{i} \in \mathbb{Z}^f$.
\end{remark}

As in \cite{Wang24}, 
for $J \subseteq \mathcal{J}$, $\underline{i} \in  \mathbb{Z}^{f}$,
one defines $\xJ(w)$
with a double induction on $|J|$ and on $\max _j i_j$
by the formulas 
\begin{equation}
  \label{eq:xJ-def-1}
\xJ(w) \defeq 
\begin{cases}
\ip^{-1} 
\left( w \otimes \underline{Y}^{- 
\big( \underline{i} - 
\underline{e}^{J ^{\mathrm{sh}}} \big)} 
v_J 
\right)
& \text{if }\underline{i} \ge \underline{e} ^{J^{\mathrm{sh}}}
, \\
0 & \text{otherwise}
\end{cases}
\end{equation}
for $\max_j i_j \le f$,
and by
\begin{equation}
  \label{eq:xJ-def-2}
\varepsilon_J \mu_{J+1, J} \xJ(w) \defeq 
\underline{Y}^{  \underline{\ell} }
\mpi
\xJ[J+1, \underline{i}'](w)
- 
\sum_{J ^{\mathrm{ss}} \subseteq J' \subsetneq J}
\varepsilon_{J'} \mu_{J+1, J'} \xJ[J', \underline{i} + \underline{r}^{J \setminus J'}](w)
\end{equation}
for $\max_j i_j > f$,
where $\underline{i}', \underline{\ell} \in \mathbb{Z}^f$
are uniquely determined by 
$\underline{i} = p \delta(\underline{i}') + \underline{c}^J - \underline{\ell}$,
$\underline{0} \le \underline{\ell} \le \underline{p} - \underline{1}$.
\begin{itemize}
\item \cite[Lemma~6.7]{Wang24}:
if we define
$z_J(w)$, $w_J(w)$ as in
(the analogue of)
\cite[Eq.~(61)]{Wang24},
then we replace $z_J, w_J \in \DZero$
in the statement of
\cite[Lemma~6.7]{Wang24} with
$z_J(w), w_J(w) \in \ip^{-1} 
\left( \F w \otimes \DZero \right)$.
We argue as in the $r \ge 1$ analogue of
\cite[Proposition~5.10]{Wang24}
to substitute 
$I (\sigma_{J ^{\mathrm{ss}}}, 
\sigma_{\underline{e}^{J ^{\mathrm{ss}}}
+ \underline{e}^{J ^{\mathrm{nss}}}})
\subseteq \Dx[\sigma_{J ^{\mathrm{ss}}}]$
with
$\pbJ[J, \underline{e}^{(J \cap J') ^{\mathrm{nss}}}]
(\F w \otimes_{\F} I(\sigma_{J ^{\mathrm{ss}}}, 
\sigma_{\underline{e}^{J ^{\mathrm{ss}}}
+ \underline{e}^{J ^{\mathrm{nss}}}}))
\subseteq \ip^{-1} (\F w \otimes_{\F}
\Dx[\sigma_{J ^{\mathrm{ss}}}])$
in the first paragraph of (i),
and to substitute
$I (\sigma_{(J') ^{\mathrm{ss}}}, 
\sigma_{\underline{m}})
\subseteq \Dx[\sigma_{(J') ^{\mathrm{ss}}}]$
with
$\pbJ[J, \underline{e}^{(J \cap J') ^{\mathrm{nss}}}]
(I (\sigma_{(J') ^{\mathrm{ss}}}, 
\sigma_{\underline{m}}))
\subseteq \ip^{-1} (\Dx[\sigma_{(J') ^{\mathrm{ss}}}])$
in the first paragraph of (ii).

The rest of the proof goes through without problems.

\end{itemize}

\paragraph{Extending \cite[§7]{Wang24}:}
after substituting $\ip^{-1}(w \otimes v_J)$ for $v_J$
and $\xJ(w)$ for $\xJ$ as previous indicated,
everything extends without problems.
In particular, we have the following generalisation
of \cite[Theorem~7.5]{Wang24}.
\begin{theorem}[Finiteness condition]
  \label{thm:finiteness-condition}
For $J \subseteq \mathcal{J}$, $w \in \V$, and $M \in \mathbb{Z}$,
the set $ \{\underline{i}  \in \mathbb{Z}^{f} \mid \xJ(w) \neq 0,\ \| \underline{i}\| = M\} $
is finite.
\end{theorem}

\paragraph{Extending \cite[§8]{Wang24}:}
for $w \in \V$, we define the sequence
$\xJ[J](w) = (\xJ[J, k] (w))_{k \ge 0}$
as $\xJ[J,k](w) \defeq \xJ[J, \underline{k}](w)$.
By (the $r \ge 1$ analogue of) \cite[Proposition~B.4]{Wang24},
we have
$\xJ[J,k](w) \in \pi \left[ \mI^{kf - |J ^{\mathrm{sh}}|+1} \right]$
for all $k \ge 0$,
hence $\xJ[J,k](w) \in \Hom_{\F} ^{\text{cont}}(\DA(\pi), \F)$,
by the description of \cite[Proposition~3.2.3]{BHHMS3}.

Let $(w_s)_{1 \le s \le r}$ be any basis of $\V$,
and recall the $A$-linear injection
\begin{equation}
  \label{eq:mu-ast}
\mu_{*} \colon \Hom_{A}(\DA (\pi), A) \hookrightarrow 
\Hom_{\F} ^{\mathrm{cont}}(\DA (\pi), \F)
\end{equation}
of \cite[Theorem~3.7.1]{BHHMS3}
given by composition with the 
map $\mu \in \Hom ^{\mathrm{cont}}_{\F}
(A, \F)$
of \cite[Proposition~3.3.1]{BHHMS3},
and where the $A$-module structure 
on the right-hand side of 
\eqref{eq:mu-ast}
is described by
\cite[Lemma~3.8.1(i)]{BHHMS3}
(whose proof does not use $r = 1$).

Recall from \cite[\S 3.1.1]{BHHMS2}
that $\F [\![N_0 ]\!]$
comes with the $\mathfrak{m}_{N_0 }$-adic
valuation $v_{N_0}$,
which defines an increasing filtration on
$\F [\![N_0 ]\!] $.
This filtration is inherited by
$\F [\![N_0 ]\!] [1/(Y_0 \dots Y_{f-1})]$
and by $A$,
which is thus a complete filtered ring,
and we denote its associated graded ring by $\gr (A)$.
Moreover, 
$A \otimes_{\F [\![N_0 ]\!] } \pi^\vee$
can be endowed with the tensor product filtration,
where $\pi^\vee $ is given the 
$\mI$-adic filtration,
or more generally any good filtration on 
$\pi^\vee $
in the sense of \cite[Definition~I.5.1]{LvO}
(cf.\ also the paragraph before
\cite[Proposition~3.12]{BHHMS2}).
The filtration on
$A \otimes_{\F [\![N_0 ]\!] } \pi^\vee$
is inherited by
$\DA (\pi)$,
and we denote its associated graded $\gr (A)$-module
by $\gr (\DA (\pi))$.
The following theorem,
which computes the rank of $\DA (\pi)$,
extends \cite[Theorem~8.2]{Wang24}
for $r \ge 1$.
\begin{theorem}
  \label{thm:xJ-basis}
The sequences $\{x_J (w_s) \mid J \subseteq \mathcal{J},\ 1 \le s \le r\}$
are contained in the image of the
$A$-linear injection
\[
\mu_{*} \colon \Hom_{A}(\DA (\pi), A) \to 
\Hom_{\F} ^{\mathrm{cont}}(\DA (\pi), \F)
\]
and form an $A$-basis of $ \Hom_{A}(\DA (\pi), A)$.
In particular, $\DA (\pi)$ is a free $A$-module of rank $r \cdot 2^f$.
\end{theorem}
\begin{proof}
We follow the proof of
\cite[Theorem~8.2]{Wang24}.
By \Cref{thm:finiteness-condition}
and by \cite[Lemma~3.3.7]{BHHMS3},
each $x_J(w_s)$ comes from an element of 
$ \Hom_{A}(\DA(\pi), A)$, which we still denote by $x_J(w_s)$.

To show that $(x_J(w_s) \mid J \subseteq \mathcal{J},\ 1 \le s \le r)$
is an $A$-basis of $ \Hom_{A}(\DA(\pi), A)$,
we instead consider
$x_J'(w_s) \defeq (x_{J,k}'(w_s))_{k \ge 0}$,
where $x_{J,k}(w)$
is defined as
$x'_{J, k}(w_s) \defeq x_{J, \underline{k} + \underline{e}^{J^{ \mathrm{sh}}}}(w_s)$.
Like before, $x_J'(w_s)$ comes from an element of 
$\Hom_{A}(\DA(\pi), A)$, which we still
denote by $x'_J(w_s)$.
We have $x'_J(w_s) =  \underline{Y}^{- \underline{e}^{J ^{\mathrm{sh}}}} x_J(w_s)$
by \cite[Lemma~3.8.1(i)]{BHHMS3}
(whose proof does not use $r = 1$),
using \Cref{thm:xJi-def}(ii)
and the fact that 
$\underline{Y}^{- \underline{e}^{J ^{\mathrm{sh}}}}$
is invertible in $A$.
And since 
$\underline{Y}^{- \underline{e}^{J ^{\mathrm{sh}}}}$
is invertible in $A$,
it suffices to show that 
$(x_J'(w_s) \mid J \subseteq \mathcal{J},\ 1 \le s \le r)$ 
is an $A$-basis of $ \Hom_{A}(\DA(\pi), A)$.

Consider the $\grL$-linear isomorphism
$\theta \colon 
\bigoplus_{\lambda \in \mathscr{P}} 
\left(\V^\vee \otimes_{\F}\! R/\mathfrak{a} (\lambda) 
\right)
\!
\xrightarrow{\sim} \grm (\pi^\vee )$
of \cite[Eq.~(14)]{BHHMS4}
(we have suppressed the twists
of \emph{loc.\ cit.}\ 
as we are not interested in the $H$-action here),
and localise it at $\prod_{j = 0}^{f-1} y_j$.
Using the bijection
\eqref{eq:lambda-bij-Jsh}
between $\Dss$ and 
$ \{\lambda \in \mathscr{P} \mid 
y_j \notin \mathfrak{a} (\lambda)
\ \forall j \in \mathcal{J}\}$,
and the bijection 
$\mu \mapsto J_\mu$
between $\Dss$ and $\{J \subseteq \mathcal{J}\}$
we obtain a $\gr (A)$-linear isomorphism
$\theta_A \colon 
\bigoplus_{J \subseteq \mathcal{J}}
\left(\V^\vee \otimes_{\F}
\gr (A)
\right)
\xrightarrow{\sim} 
\grm (\pi^\vee )[ \frac{1}{\prod_{j} y_j}]
\cong\gr \left( \DA(\pi) \right)$,
where the last isomorphism comes from
\cite[Lemma~3.1]{BHHMS2}.

Since $\DOne$ is multiplicity-free by
\cite[Lemma~4.1(ii)]{Wang24},
there exist unique $I$-eigenvectors 
$e_J \in \DOne^\vee$
such that $ \left\langle v_J, 
e_{J'} \right\rangle = \delta_{J = J'}$
for all $J, J' \subseteq \mathcal{J}$.
Moreover, let $(u_s)_{1 \le s \le r}$ 
be the dual basis of $(w_s)_{1 \le s \le r}$
in $\V^\vee$,
and set
$e_{J, s} \defeq \ip^\vee 
(u_s \otimes e_J) 
\in ( \pi^{I_1 } )^\vee$
for $J \subseteq \mathcal{J}$,
$1 \le s \le r$.
We can see from the proof of 
\cite[Theorem~2.1.2]{BHHMS4}
that the isomorphism $\theta_A$ 
of the previous paragraph 
can be chosen so that
it sends the element $u_s \otimes 1$
inside the copy of $V^\vee \otimes_{\F} \gr (A)$
indexed by 
$J \subseteq \mathcal{J}$
to the element 
$\underline{Y}^{-\underline{e}^{J ^{\mathrm{sh}}}}
e_{J
,s} \in \gr (\DA (\pi))$;
cf.\ also the line after
\eqref{eq:N-unwinding} below,
after observing that
$e_{\lambda,s}$ of 
\eqref{eq:explicit-dual-basis}
coincides with $e_{J,s}$ here,
by the choice of $v_J$
in the paragraph after
\Cref{rmk:Wang-notation},
and where $\lambda$ is defined as in 
\eqref{eq:lambda-bij-Jsh}
in terms of the unique $\mu \in \Dss$
such that $J = J_\mu$.
In particular, we see that
$(e_{J,s} \mid 
J \subseteq \mathcal{J},\ 1 \le s \le r)$
is a $\gr(A)$-basis of $\gr(\DA(\pi))$.

By \cite[Theorem~I.5.7]{LvO}
the filtration we endowed $\DA (\pi)$ with
is good
in the sense of \cite[Definition~I.5.1]{LvO},
in particular $\DA (\pi)$ 
is a finitely generated $A$-module.
Then, by \cite[Lemma~I.6.9]{LvO},
\cite[Lemma~I.6.6]{LvO}
and by
\cite[Lemma~I.4.2.5(2)]{LvO}
it suffices to show that 
\begin{equation}
  \label{eq:gr-basis}
(\gr(x_J'(w_s)) \mid J \subseteq \mathcal{J},\ 1 \le s \le r)
\end{equation}
is a $\gr (A)$-basis of $ \Hom_{\gr(A)}(\gr(\DA(\pi)), \gr(A)) \cong \gr \left(  \Hom_{A}(\DA(\pi), A) \right)$.
We claim that 
\[
\left\langle \gr(x_J'(w_s)), e_{J', s'} \right\rangle
= \delta_{J=J'} \delta_{s= s'} \underline{y}^{- \underline{1}}
\]
in $\gr (A)$
for all $J, J' \subseteq \mathcal{J}$
and all $1 \le s, s' \le r$,
which is enough to conclude
using that $\underline{y}^{-\underline{1}} 
\in \gr(A)^{\times }$ is a unit
(more precisely, this shows that
\eqref{eq:gr-basis} is the dual basis to
$(\underline{y}^{\underline{1}} e_{J,s}
\mid J \subseteq \mathcal{J},\ 1 \le s \le r)$).
The argument is completely analogous to 
that of
\cite[Theorem~3.7.1]{BHHMS3},
using (the $r \ge 1$ analogue of) 
\cite[Corollary~7.3]{Wang24}
and
\cite[Proposition~8.1]{Wang24}.
\end{proof}

\begin{itemize}
\item \cite[Proposition~8.3]{Wang24}:
when $r \ge 1$ we denote by $\Mat(\varphi)$
and $\Mat(a)$, for $a \in \mathcal{O}_K ^{\times}$,
the matrices for the action of $\varphi$ and 
$a \in \mathcal{O}_K ^{\times}$ on 
$\Hom_A(\DA(\pi), A)(1)$ with respect to the basis
$ \left( x_J(w_s) \mid J \subseteq \mathcal{J}, \ 1 \le s \le r \right)$.
The following proposition describes these matrices
when $a \in [\mathbb{F}_q ^{\times}]$.
\end{itemize}
\begin{proposition}
  \label{prop:Mat-rep}
\begin{enumerate}[(i)]
\item We have 
\[
\Mat(\varphi)_{(J', s'),(J+1, s)} = 
\begin{cases}
\gamma_{J+1, J'} \underline{Y}^{- \left( \underline{c} ^J + \underline{r}^{J \setminus J'} \right)} & \text{if }J ^{\mathrm{ss}} \subseteq J' \subseteq J, \ s = s' \\
0 & \text{otherwise.}
\end{cases}
\]

\item 
For $a \in \left[ \mathbb{F}_q^{\times } \right]$,
$\Mat(a)$ is a diagonal matrix with 
$\Mat(a)_{(J,s),(J,s)} = \overline{a}^{\underline{r}^{J^c}}$.
\end{enumerate}
\end{proposition}
\begin{proof}
Identical to the proof of
\cite[Propostion~8.3]{Wang24},
\emph{mutatis mutandis}.
\end{proof}

\paragraph{Extending \cite[§9]{Wang24}:}
\begin{itemize}
\item \cite[Lemma~9.1]{Wang24}:
the following lemma,
which is used in the proof of
\Cref{thm:rk-subrep} below,
generalises \cite[Lemma~9.1]{Wang24}
for $r \ge 1$.
\end{itemize}
Let $\pi_1 \subseteq \pi$ be 
a subrepresentation of $\pi$, 
and recall that 
for $\ell \in \{0, \dots, f\}$
the $\F$-vector space
$\VFil[\pi_1 ,\ell] \subseteq \V$
is defined in
\eqref{eq:VFil-rFil-def}
when $\rep$ is 
reducible nonsplit,
and in the paragraph after the proof of
\cite[Lemma~3.1.3]{Ber}
when $\rep$ is 
reducible split.

\begin{lemma}
  \label{lemma:ss-Serre-weight-mult}
Assume that $\rep$ is $\max \{6, 2f + 1\}$-generic,
and let $\pi_1 \subseteq \pi$ be 
a subrepresentation of $\pi$, 
and let $L \subseteq \V$ be a line.
Moreover, set $I_L \defeq \{ \ell \in  \{0, \dots, f \}
\mid 
L \subseteq \VFil[\pi_1 ,\ell] \}$,
and set
$S_L \defeq  \{J \subseteq \mathcal{J} \mid
|J| \in I_L\}$.
Then, $S_L$
is stable under $J \mapsto J -1 $,
and is moreover stable under taking subsets 
if $\Jrep \neq \mathcal{J}$.
Moreover, 
keeping the notation of 
\cite[Lemma~9.1]{Wang24},
we have
\begin{align}
  \label{eq:JH-modular-of-line}
  \JH \left(\pi_1^{K_1 } \cap
  \ip^{-1}(L \otimes_{\F} \DZero)\right)
  \cap \Wss
  &= 
  \left\{\sigma_{\underline{e}^J} \ \middle|\ 
  J \in S_L\right\} = 
  \left\{\sigma_{\underline{e}^J} \ \middle|\ 
  |J| \in I_L\right\}; \\
  \label{eq:JH-of-line}
  \JH \left(\pi_1^{K_1 } \cap
  \ip^{-1}(L \otimes_{\F} \DZero)\right)
  & = 
  \left\{\sigma_{\underline{b}} \in \JH (\DZero)
  \ \middle|\ 
  \{j \mid b_j \ge 1\} \in S_L\right\}.
\end{align}
\end{lemma}

\begin{proof}
Notice that \eqref{eq:JH-of-line}
implies \eqref{eq:JH-modular-of-line},
using \cite[Eq.~(4)]{Wang24}
applied to $\rep = \rss$ (or equivalently
to $\Jrep = \mathcal{J}$).

Assume first $\rep$ is reducible nonsplit,
or equivalently that $\Jrep \neq \mathcal{J}$.
We know from \Cref{thm:srep-structure}
that $\pi_1^{K_1 } \cap \ip^{-1}(L \otimes_{\F} \DZero)$
is isomorphic to $\DZeroFil[\ell (L)]$
as a $\GR$-representation,
where $\ell (L) \defeq 
\max \{\ell \in \{-1, \dots, f\} \mid 
L \subseteq \VFil[\pi_1 ,\ell] \}$.
We also see that 
$S_L = \{J \subseteq \mathcal{J} \mid
|J| \le \ell (L)\}$,
in particular
$S_L$ is stable under $J \mapsto J -1 $
and under taking subsets.
Notice that, by 
\cite[Lemma~4.1(i)]{Wang24}
applied to $\rep = \rss$ (or equivalently
to $\Jrep = \mathcal{J}$)
the set 
$J \defeq  \{j \mid b_j \ge 1\}$
of  \eqref{eq:JH-of-line}
gives the unique $\sigma_J \in \Wss$
such that $\sigma_{\underline{b}} \in 
\JH ( \Dssx[\sigma_J])$.
Then, we conclude 
from \eqref{eq:D-Dss-Fil-comparison}
applied with $\ell = \ell (L)$
that \eqref{eq:JH-of-line} 
(hence also \eqref{eq:JH-modular-of-line}) holds.

Assume now that $\rep$ is reducible split.
We see from 
\cite[Corollary~3.1.6]{Ber}
applied to $\pi' = \pi_1 $
that $\pi_1^{K_1 } \cap \ip^{-1}(L \otimes_{\F} \DZero)$
is isomorphic to 
$\bigoplus_{\ell \in I_L} 
(L \otimes_{\F} \DZeroEll[\ell])$
as a $\GR$-representation,
in particular
$S_L$ is stable under $J \mapsto J -1 $.
Then, the right-hand side of \eqref{eq:JH-of-line} 
becomes
$\left\{\sigma_{\underline{b}} \in \JH (\DZero)
\mid
|\{j \mid b_j \ge 1\}| \in I_L\right\}$,
and we conclude by definition of
$\DZeroEll[\ell]$,
cf.\ the equation after 
\cite[Eq.~(13)]{Ber}.
\end{proof}

Recall the numerical invariants of 
\eqref{eq:VFil-rFil-def}.
The following theorem extends 
\Cref{thm:xJ-basis}
to subrepresentations of $\pi$.
\begin{theorem}
  \label{thm:rk-subrep}
Assume that $\rep$ is $\max \{6, 2f + 1\}$-generic,
and let $\pi_1 $ be a subrepresentation of $\pi$.
Then, we have 
\begin{equation}
  \label{eq:rk-subrep}
\rk_A \DA(\pi_1 ) = 
\sum_{0 \le \ell \le f} \rFil[\pi_1, \ell] \binom{f}{\ell}.
\end{equation}
\end{theorem}
\begin{remark}
  \label{rmk:rk-subrep-r=1}
Note that \Cref{thm:rk-subrep}
generalises
\cite[Theorem~9.2]{Wang24}
to the case $r \ge 1$.
Indeed, if $r = 1$
then
\eqref{eq:rk-subrep} becomes
$ \rk_A \DA(\pi_1 ) = 
\sum_{\ell \in I_{\V}}
\binom{f}{\ell}$,
where $I_{\V}$
is the set of \Cref{lemma:ss-Serre-weight-mult}
for $L = \V$.
On the other hand,
\eqref{eq:JH-modular-of-line}
applied to $L = \V$ gives 
$| \JH (\pi_1^{K_1 })
\cap \Wss| = 
|\left\{\sigma_{\underline{e}^J} \ \middle|\ 
|J| \in I_\V\right\}|
= \sum_{\ell \in I_\V}  \binom{f}{\ell}$.
\end{remark}
\begin{proof}
We essentially adapt the proof of \cite[Theorem~9.2]{Wang24}.
Assume first that $\rep$ is reducible nonsplit.

For $\ell \in \{0, \dots, f\}$,
set $r_1 (\ell) \defeq \rFil[\pi_1, \ell]$,
and let $(w_s \mid 1 \le s \le r)$ be an $\F$-basis of $\V$
adapted to $\pi_1 \subseteq \pi$,
in the sense that $(w_s \mid 1  \le s \le  r_1(\ell))$
is an $\F$-basis of $\VFil[\pi_1 , \ell]$,
or equivalently that $(w_s \mid r_1(\ell + 1) + 1  \le s \le  r_1(\ell))$
is sent to an $\F$-basis of
$\frac{\VFil[\pi_1 , \ell]}{\VFil [\pi_1 , \ell +1]}$.
Then, we prove the statement by showing that 
$\mathcal{B}  \defeq   \left( x_J(w_s) \mid J \subseteq \mathcal{J}, \ 1 \le s \le r_1(|J| ) \right)$
is an $A$-basis of $ \Hom_{A}(\DA(\pi_1 ), A)$.

\paragraph{Step 1.}
We prove that $x_J(w_s)$ is an element of 
$\Hom_A(\DA(\pi_1 ), A)$
for all $J \subseteq \mathcal{J}$, $1 \le s \le r_1 (\ell)$.

Observe first that
$\Hom_A(\DA(\pi_1 ), A) = \Hom_A(\DA(\pi ), A) 
\cap \Hom_{\F} ^{\mathrm{cont}}(\DA(\pi_1 ), \F)$,
where the intersection is taken inside
$\Hom_{\F} ^{\mathrm{cont}}(\DA(\pi ), \F)$,
and that $x_J(w_s) \in \Hom_A(\DA(\pi ), A)$
by \Cref{thm:finiteness-condition} and by \cite[Lemma~3.3.7]{BHHMS3}.
It remains to show that $x_J(w_s) \in 
\Hom_{\F} ^{\mathrm{cont}}(\DA(\pi_1 ), \F)$.
By \cite[Proposition~3.2.3]{BHHMS3}
(and since $\pi_1 [\mI^n] = \pi_1  \cap \pi[\mI^n]$
for all $n \in \mathbb{N}$),
it is enough to show that $\xJ(w) \in \pi_1 $
for all $J \subseteq \mathcal{J}$ 
and for all $w \in \VFil[\pi_1 , |J|]$.
This can be proven with a double induction on $|J|$
and on $\max_j i_j$,
using \eqref{eq:xJ-def-1}, \eqref{eq:xJ-def-2}.

\paragraph{Step 2.}
Clearly the elements
$ \left\{ x_J(w_s) \mid J \subseteq \mathcal{J}, \ 1 \le s \le r_1(|J| ) \right\}$
are linearly independent over $A$, since they are a subset of the $A$-basis of \Cref{thm:xJ-basis}.
We now show that these elements generate
$ \Hom_{A}(\DA(\pi_1 ), A)$.

If this is not the case, then there exist an $A$-linear combination
\begin{equation}
  \label{eq:absurd-lin-comb}
\sum_{ \substack{
 J \subseteq  \mathcal{J} \\ 
r_1(|J|)+1 \le s \le r}} a_{J,s} x_J(w_s) =
\sum_{s = 1}^{r} 
\sum_{ \substack{J \subseteq  \mathcal{J} \\
r_1(|J|) < s }}
a_{J,s} x_J(w_s) \in 
\Hom_{A}(\DA (\pi_1 ), A),
\end{equation}
with $a_{J,s} \in A$ not all zero.

Let $\mathfrak{S}$ be the set of pairs $(J,s)$, with
$J \subseteq \mathcal{J},\ r_1(|J|) +1 \le  s \le r$
and such that $a_{J,s} \neq 0$;
let $\mathfrak{S}_0 \subseteq  \mathfrak{S}$
be the set of pairs  $(J,s) \in \mathfrak{S}$ such that the quantity
$\deg(a_{J,s}) - | \partial J|$ is minimal in $\mathfrak{S}$.
Pick $(J_0 ,s_0 ) \in \mathfrak{S}_0$ such that
$J_0$ is moreover $ \subseteq $-maximal 
inside $\mathfrak{S}_0$.

Up to rescaling \eqref{eq:absurd-lin-comb}
with a suitable $\lambda \underline{Y}^{\underline{t}} \in A ^{\times}$,
with $\lambda \in \F ^{\times}$ and $\underline{t} \in \mathbb{Z}^f$,
we may assume that
\begin{equation}
  \label{eq:coeff-with-hot}
a_{J_0 , s_0 } = \underline{Y}^{- \underline{e}^{J_0 \setminus \partial J_0 }}
+ \left( \text{terms of degree }\ge - \left | J_0 \setminus \partial J_0  \right | \text{ and not in } \F \underline{Y}^{- \underline{e}^{J_0 \setminus \partial J_0 }} \right),
\end{equation}
which has degree $- \left | J_0 \setminus \partial J_0  \right | $.
As $(J_0 , s_0 ) \in  \mathfrak{S}_0 $,
we have $\deg(a_{J,s}) \ge - |J_0 | + | \partial J| \ge -f$
for all $(J,s) \in \mathfrak{S}$,
and moreover by the maximality of $J_0$
we have 
$\deg(a_{J,s}) \ge - |J_0 | + | \partial J| + 1
 = - (|J \setminus \partial J| - |J \setminus J_0 | -1 )$
for all $(J,s) \in \mathfrak{S}$
such that $J \supsetneq J_0$.

By \cite[Proposition~3.2.3]{BHHMS3},
we can identify \eqref{eq:absurd-lin-comb} with a sequence
$(z_k)_{k \ge 0}$ of elements of $\pi_1 $ with
$z_k \in \pi_1  \left[ \mI^{kf - |J ^{\mathrm{sh}}|+1} \right]$,
in particular $z_0 \in \pi_1 [ \mI ] = \pi_1^{I_1 } \subseteq \pi_1^{K_1 }$,
which by \Cref{thm:srep-structure} and by
our choice of the vectors $w_s \in \V$ can uniquely be written as
$z_0 =  \sum_{s = 1}^{r} z_0 (s)$,
where $z_0 (s) \in \pi_1^{K_1 }$ is such that
$\ip(z_0 (s)) \in \F w_s \otimes_{\F} \DZero$.

By \cite[Lemma~3.8.1(i)]{BHHMS3}
(whose proof does not use $r \ge 1$), 
for every $\lambda \in \F ^{\times}$ and $\underline{t} \in \mathbb{Z}^f$ we have
\begin{equation}
  \label{eq:monomial-action-1}
 \lambda \underline{Y}^{\underline{t}} x_J(w_s)
= \left( \lambda \underline{Y}^{\underline{\ell} + \underline{t} - \underline{k}} x_{J, \ell}(w_s) \right)_{k \ge 0},
\end{equation}
for $\ell \gg _k 0$.
By \Cref{thm:xJi-def}(ii), \eqref{eq:monomial-action-1}
becomes
\begin{equation}
  \label{eq:monomial-action-2}
 \lambda \underline{Y}^{\underline{t}} x_J(w_s)
= \left( \lambda  x_{J, \underline{k} - \underline{t}}(w_s) \right)_{k \ge 0}.
\end{equation}
We now claim that 
\begin{equation}
  \label{eq:z0-linewise}
z_0(s) = 
\sum_{ J \subseteq  \mathcal{J} \mid
r_1(|J|) < s } \left( a_{J,s} x_J(w_s) \right)_0.
\end{equation}

We know that $\deg(a_{J,s}) \ge -f$,
so for each monomial $\lambda \underline{Y}^{\underline{t}}$
(where $\lambda \in \F ^{\times}$, 
$\underline{t} \in  \mathbb{Z}^f$)
in $a_{J,s}$ we have $\| -\underline{t} \| \le f$.
If there exists $j \in \mathcal{J}$ such that 
$-t_j \le -1$, then 
$(\lambda \underline{Y}^{\underline{t}} x_J(w_s))_0 = 0$
by \eqref{eq:monomial-action-2}
and (the $r \ge 1$ analogue of) 
\cite[Proposition~7.2]{Wang24}.
So, we can assume that $-\underline{t} \ge 0$,
which together with $\| - \underline{t} \| \le f$
implies $-\underline{t} \le \underline{f}$.
Therefore, we can apply \eqref{eq:xJ-def-1}
(with $\underline{i} = - \underline{t}$)
to deduce that 
$( \lambda \underline{Y}^{\underline{t}} x_J(w_s))_0
\overset{\eqref{eq:monomial-action-2}}{= }
 \lambda  x_{J, -\underline{t}}(w_s) \in 
\F w_s \otimes_{\F} \DZero$
for all  monomials
$\lambda \underline{Y}^{\underline{t}}$ in $a_{J,s}$.
In particular,
$\ip \left( a_{J,s} x_J(w_s) \right)_0 \in 
\F w_s \otimes_{\F} \DZero$,
which proves \eqref{eq:z0-linewise}.

If we take $s = s_0 $, 
then \eqref{eq:z0-linewise} is a nonempty sum.
Consider the $\GR$-subrepresentation of $\DZero$
\[
V ^{\mathrm{Wang}} \defeq 
\begin{cases}
\sum_{J \not\supseteq J_0 } I \left( \sigma_{\underline{e}^{J ^{\mathrm{ss}}}}, \sigma_{\underline{e}^{J ^{\mathrm{ss}}} + \underline{\varepsilon}^J} \right) 
&  \text{if } \Jrep \neq \mathcal{J}, \\
\sum_{J \neq  J_0 } I \left( \sigma_{\underline{e}^{J ^{\mathrm{ss}}}}, \sigma_{\underline{e}^{J ^{\mathrm{ss}}} + \underline{\varepsilon}^J} \right) 
= \bigoplus_{J \neq J_0 } \Dx[\sigma_J]
&  \text{if } \Jrep = \mathcal{J} ,
\end{cases}
\]
appearing in the paragraph before
\cite[Eq.~(85)]{Wang24}
(we include the case $\Jrep = \mathcal{J}$,
even though we are assuming
$\rep$ reducible nonsplit,
so we can refer to it below).
We show as in the proof of
\cite[Theorem~9.2]{Wang24}
that $V ^{\mathrm{Wang}}$
has Jordan-H\"older constituents 
$\sigma_{\underline{b}}$
with $\{j \mid b_j \ge 1\} \not \supseteq J_0 $
if $\Jrep \neq \mathcal{J}$,
and
with $\{j \mid b_j \ge 1\} \neq  J_0 $
if $\Jrep = \mathcal{J}$.
So, it follows from 
\Cref{lemma:ss-Serre-weight-mult}
that
$\ip (\pi_1^{K_1 }) \cap 
(\F w_{s_0} \otimes_{\F} \DZero)
\subseteq \F w_{s_0} \otimes_{\F} V ^{\mathrm{Wang}}$.

The following claim
parallels the Claim 
in the proof of 
\cite[Theorem~9.2]{Wang24},
and the proof of \emph{loc.\ cit.}\ goes through
in our case,
\emph{mutatis mutandis}.
\paragraph{Claim.} 
We have the following properties:
\begin{enumerate}[(i)]
\item 
  \label{eq:claim-lincomb-i}
If $J \not\supseteq J_0 $ and $ \| \underline{i} \| \le f$,
then $\ip \left( \xJ(w_{s_0 }) \right) \in \F w_{s_0 }\otimes_{\F}V ^{\mathrm{Wang}}$.

\item 
  \label{eq:claim-lincomb-ii}
If $J \supsetneq J_0 $ and 
$ \| \underline{i} \| \le |J \setminus \partial J| - |J \setminus J_0 |  -1$,
then $\ip \left( \xJ(w_{s_0 }) \right) \in \F w_{s_0 } \otimes _{\F} V ^{\mathrm{Wang}}$.

\item 
  \label{eq:claim-lincomb-iii}
If $J =  J_0 $ and 
$ \| \underline{i} \| \le |J_0  \setminus \partial J_0 |$,
then we have $\ip \left( \xJ(w_{s_0 }) \right) \in \F w_{s_0 } \otimes_{\F} V ^{\mathrm{Wang}}$
if $\underline{i} \neq \underline{e}^{J_0 \setminus \partial J_0 }$,
and 
$\ip ( 
\xJ[J, \underline{e}^{J_0 \setminus \partial J_0 }]
(w_{s_0 }) ) \in
(\F w_{s_0 } \otimes_{\F} \DZero)
\setminus
\left( \F w_{s_0 } \otimes_{\F} V ^{\mathrm{Wang}} \right)$.
\end{enumerate}
It follows from the line after
\eqref{eq:coeff-with-hot} 
that $z_0(s_0 ) =
\sum_{ J \subseteq  \mathcal{J} :
r_1(|J|) < s_0} \left( a_{J,s_0 } x_J(w_{s_0 }) \right)_0$
can be expressed as a linear combination of $\xJ(w_{s_0 })$
satisfying the assumptions of 
\ref{eq:claim-lincomb-i},
\ref{eq:claim-lincomb-ii},
\ref{eq:claim-lincomb-iii}
of the above claim, with exactly one of the terms equal to $\xJ[J_0 , \underline{e}^{J_0 \setminus \partial J_0 }]$,
hence $\ip ( z_0 (s_0 )  ) \in
(\F w_{s_0 } \otimes_{\F} \DZero) \setminus 
\left( \F w_{s_0 } \otimes_{\F} V ^{\mathrm{Wang}} \right)$.
On the other hand, 
we have $ \ip \left( z_0 (s_0 )  \right) \in
\ip (\pi_1^{K_1 }) \cap (\F w_{s_0} \otimes_{\F} \DZero)
\subseteq \F w_{s_0} \otimes_{\F} V ^{\mathrm{Wang}}$
by construction, absurd.

Assume now that $\rep$ is reducible split.
In this case, we let
$(w_{\ell,s} \mid 1 \le s \le r)$ be 
an $\F$-basis of $\V$
such that $(w_{\ell, s} \mid 1  \le s \le  r_1(\ell))$
is an $\F$-basis of $\Vell[\pi_1 , \ell]$,
and we show that 
$\mathcal{B} \defeq 
\left( x_J(w_{|J|, s}) \mid J \subseteq \mathcal{J}, \ 1 \le s \le r_1(|J| ) \right)$
is an $A$-basis of $\Hom_{A}(\DA(\pi_1 ), A)$.
Notice first that
$\left( x_J(w_{|J|, s}) \mid J \subseteq \mathcal{J}, 
\ 1 \le s \le r \right)$
is an $A$-basis of $\Hom_{A}(\DA(\pi ), A)$,
because it differs from the basis of
\Cref{thm:xJ-basis}
by (an invertible) linear change of coordinates,
cf.\ \Cref{rmk:xJ-linear}.

After substituting $x_J(w_s)$
with $x_J(w_{|J|,s})$,
the same proof goes through without problems,
except that \eqref{eq:absurd-lin-comb}
becomes 
$\sum_{ J \subseteq  \mathcal{J}} 
\sum_{ r_1(|J|)+1 \le s \le r} a_{J,s} x_J(w_{|J|,s})
\in \Hom_{A}(\DA (\pi_1 ), A)$
and that,
in the paragraph after
\eqref{eq:coeff-with-hot},
we uniquely write $z_0 $ as
$z_0 = \sum_{\ell = 0}^{f} \sum_{s = 1}^{r_1 (\ell)}  
z_0 (\ell, s)$,
where $z_0 (\ell, s) \in \pi_1^{K_1 }$
is such that
$\ip(z_0 (\ell, s)) \in 
\F w_{\ell, s} \otimes_{\F} \DZeroEll[\ell]$
by \cite[Corollary~3.1.6]{Ber}
applied to $\pi' = \pi_1$.
Then, \eqref{eq:z0-linewise}
becomes
$z_0 (\ell, s) = \sum_{|J| = \ell}
\left( a_{J,s} x_J (w_{|J|, s}) \right)_0 $,
and in the paragraph
following \eqref{eq:z0-linewise}
$\F w_s \otimes_{\F} \DZero$
is replaced by
$\F w_{|J|,s} \otimes_{\F} \DZeroEll[|J|]$
both times it occurs.

We then take $(J,s) = (J_0 , s_0)$,
the above Claim still holds,
and we reach a contradiction by deducing that
$\ip (z_0 (\ell, s)) \in 
(\F w_{|J_0 |, s_0 } \otimes_{\F} \DZeroEll)
\setminus (\F w_{|J_0 |, s_0 } 
\otimes_{\F} V ^{\mathrm{Wang}})$
and 
$\ip (z_0 (\ell, s)) \in 
\ip (\pi_1^{K_1 }) \cap 
(\F w_{|J_0 |, s_0 } \otimes_{\F} \DZeroEll)
\subseteq 
\F w_{|J_0 |, s_0 }
\otimes_{\F} V ^{\mathrm{Wang}}$.
\end{proof}

The following corollary
will be crucially used
in Step 4 of the proof of 
\Cref{prop:CM-via-N},
and in the proof of
\Cref{thm:gen-by-inv}(iii).
\begin{corollary}
  \label{cor:Wang}
Assume that $\rep$ is $ \max \{6,2f + 1\}$-generic,
let $\pi_1$
be a subrepresentation of $\pi$.
Then,
\begin{equation}
  \label{eq:Wang-rk-Dvee}
\dimFX \Dvee (\pi_1 ) = 
\sum_{\ell = 0}^{f} 
\rFil[\pi_1 , \ell]  \cdot 
\binom{f}{\ell}.
\end{equation}
\end{corollary}
\begin{proof}
We know from \Cref{thm:et-and-rk}
that $\DA (\pi_i) = \DA (\pi_i) ^{\mathrm{\acute{e}t}}$,
and conclude from
\cite[Theorem~1.13]{BHHMS2}.
\end{proof}

\begin{remark}
  \label{rmk:D-eq-m-non-ss}
The statement of
\cite[Proposition~3.87(i)]{BHHMS2}
also holds when $\rep$ is semisimple, the same proof
goes through \emph{verbatim}.
On the other hand, 
\cite[Proposition~3.87(ii)]{BHHMS2}
does not hold in the non-semisimple case.
\Cref{cor:Wang} gives
the correct replacement.
\end{remark}

\begin{itemize}
\item \cite[Corollary~9.3]{Wang24}:
Implied by
\Cref{thm:gen-by-inv}(i) below.
\end{itemize}

Following \cite{Wang24}
(cf.\ the line after \cite[Eq.~(1)]{Wang24}),
we let $\mathcal{C}$ be 
the abelian category of admissible smooth
representations $\overline{\pi}$ of $\GF$ over $\F$
with a central character and such that 
$\gr(\DA(\overline{\pi}))$
is a finitely generated $\gr(A)$-module
(where we endow $\overline{\pi}^\vee$
with the $\mI$-adic filtration,
which induces a filtration on $\DA (\overline{\pi})$
as in the paragraph before
\Cref{thm:xJ-basis}).
One has the following generalisation of
\cite[Theorem~1.1]{Wang24}.
\begin{theorem}
  \label{thm:et-and-rk}
Assume that $\rep$ is $\max \{6, 2f +1 \}$-generic.
Let $\pi_1 \subseteq \pi$ be a subrepresentation.
Then $\pi_1  \in \mathcal{C}$,
the $A$-module map $\beta$ 
of \eqref{eq:beta-DA-definition}
(for $\overline{\pi} = \pi_1$)
is an isomorphism,
and
$\rk_A \DA(\pi_1 ) = \sum_{0 \le \ell \le f} 
\rFil[\pi_1, \ell] \binom{f}{\ell}$.
\end{theorem}
\begin{proof}
It is clear that $\pi_1$
is admissible smooth with a central character,
and it follows from the proof of \Cref{thm:xJ-basis}
and \Cref{thm:rk-subrep} that $\gr(\DA(\pi_1))$ 
is a finitely generated $\gr(A)$-module
(indeed, we provide a finite basis
of its $\gr(A)$-linear dual).
Hence, $\pi_1 \in \mathcal{C}$.

It only remains to show that $\beta$ is an isomorphism, 
or equivalently that $\Hom_A(\DA(\pi_1 ), A)(1)$
is an étale $(\varphi, \mathcal{O}_K ^{\times})$-module.
Consider the $A$-basis
$\mathcal{B}  \defeq  \left( x_J(w_s) \mid J \subseteq \mathcal{J}, \ 1 \le s \le \rFil[\pi_1, |J|]  \right)$
of $\Hom_A(\DA(\pi_1 ), A)(1)$
that appears in the proof of \Cref{thm:rk-subrep}.
It is enough to show that $\Mat(\varphi) \in \GL_{|\mathcal{B}|}(A)$ 
is invertible.
It follows from \Cref{prop:Mat-rep} that
$\Mat(\varphi)$ is upper triangular 
(for any total order $ \preceq $ on $\mathcal{B}$ 
such that $J' \subseteq J$ implies $(J', s') \preceq (J, s)$ for all $(J,s), (J',s') \in \mathcal{B}$) 
with invertible entries on the diagonal.
Hence, $\Mat(\varphi)$ is invertible.
\end{proof}

\paragraph{Extending \cite[Appendix~A]{Wang24}:}
After substituting $\ip^{-1}(w \otimes v_J)$ for $v_J$
and $\xJ(w)$ for $\xJ$ as previous indicated,
all the proofs in this section
go through when $r \ge 1$.
\paragraph{Extending \cite[Appendix~B]{Wang24}:}
\begin{itemize}
\item \cite[Lemma~B.1]{Wang24}:
its proof works \emph{verbatim}.
Indeed, the same proof holds for any $\pi \in \mathcal{C}$
such that $\gr(\pi)$ is annihilated by 
$(y_jz_j, z_jy_j\mid 0 \le j \le f-1)$.

\item \cite[Lemmas~B.2,B.3,B.4]{Wang24}:
the same proofs, \emph{mutatis mutandis}, 
go through when $r \ge 1$.
Note that the argument 
in the second paragraph of the proof of
\cite[Proposition~3.5.1]{BHHMS3}
(which is invoked in the proof of 
\cite[Lemma~B.2]{Wang24})
does not use $r = 1$.
\end{itemize}

\paragraph{Extending \cite[Appendix~C]{Wang24}:}
This section contains the only result that we were not able generalise to $r \ge 1$.
Fortunately, it is not used in the rest of the paper.
\begin{itemize}
\item \cite[Lemma~C.1, Lemma~C.2]{Wang24}:
they only concern the structure of the ring $A$,
hence they do not need the hypotheses $r = 1$.
\end{itemize}

\paragraph{Extending \cite[Appendix~D]{Wang24}:}
everything extends \emph{verbatim}.

\begin{remark}
  \label{eq:Wang-not-extended}
Let $\Mat(\varphi)'$ and $\Mat(a)'$ 
be the matrices for the action of $\varphi$ and 
$a \in \mathcal{O}_K ^{\times}$ on 
$\Hom_A(\DA(\pi), A)(1)$ with respect to the basis
\begin{equation}
  \label{eq:basis-for-computations}
\mathcal{B}' 
\defeq 
\left( \underline{Y}^{- \underline{r}^{J^c}} x_J(w_s) 
\ \middle|\  J \subseteq \mathcal{J}, 
\ 1 \le s \le r \right).
\end{equation}
We were not able
to extend
\emph{loc.\ cit.}\ for $r \ge 1$,
the main obstacle being
that we do not know how to prove that
$(\Mat(a)')_{(J', s'), (J, s)}$
vanishes for $s \neq s'$.
This is the only result in 
\cite{Wang24}
that we could not extend.
\end{remark}

\section{Cohen-Macauliness and finite length}
  \label{sec:CM-and-fl}
The goal of this section is to show,
in \Cref{prop:CM-via-N}(ii) below,
that $\pi$ and all its nontrivial subquotients 
are Cohen-Macaulay
of grade $2f$ and,
in \Cref{cor:fl} below,
that $\pi$ has finite length,
with an explicit upper bound
depending on the numerical invariants
\eqref{eq:VFil-rFil-def}.

Throughout this section we assume that 
$\rep$ is reducible nonsplit and $0$-generic,
and we assume that $\pi$ satisfies
hypotheses
\ref{hypothesis:i} to \ref{hypothesis:iv}
of \Cref{sec:hypotheses}.

\subsection{Cohen-Macauliness}
  \label{sec:CM}
The goal of this section is to prove
that $\pi$ and all its nontrivial subquotients
are Cohen-Macaulay of grade $2f$;
we start with a general lemma
on graded modules.

For  a graded ring $S$ 
(not necessarily commutative),
a graded $S$-module $N$,
and an integer $d$,
we let $N(d)$ be the $S$-module $N$
with grading $N(d)_j \defeq N _{j + d}$.
The following lemma shows that 
the isomorphisms of \cite[Lemma~3.65]{BHHMS2}
are graded, up to shift.
We will use it in \Cref{lemma:fraka-dual} below.
\begin{lemma}
\label{lemma:graded-decalage}
\begin{enumerate}[(i)]
\item 
Let $S$ be a noetherian graded ring
(not necessarily commutative),
suppose there exists
an $S$-regular sequence of homogeneous
central elements
$(u _1, \dots, u_s)$
for some integer $s \ge 0$,
let $U$ be the ideal 
generated by $u_1 , \dots , u_s$,
and set 
$d_U \defeq \sum _{s' = 1}^s \deg u _{s'}$.
Let $N$ be a finitely generated graded $S$-module 
annihilated by $U$.
Then,
we have an isomorphism
of graded $S$-modules for $i \ge 0$
\[
  \Ext_S ^{s + i}(N , S) \cong
  \Ext_{S/U}^{i}( N, S/U )(d_U).
\]
\item 
If $M$ is a finitely generated graded $\grL$-module
with compatible $H$-action
that is annihilated by $(h_0 , \dots , h_{f-1})$,
then we have an isomorphism
of graded $\grL$-modules with compatible
 $H$-action for $i \ge 0$
\[
  \Ext _{\grL} ^{f + i}(M , \grL) \cong
  \Ext_{R}^{i}( M, R )(-2f).
\]

If moreover $M$ is annihilated by $J$,
then we have an isomorphism
of graded $\grL$-modules with compatible
 $H$-action for $i \ge 0$
\[
  \Ext _{\grL} ^{2f + i}(M , \grL) \cong
  \Ext_{R}^{f + i}( M, R )(-2f) \cong 
  \Ext_{\Rbar}^{i}( M, \Rbar )(-4f).
\]
\end{enumerate}
\end{lemma}
The finiteness assumptions in the lemma
are only there to ensure that
the functor
$\Hom_{S}(-, -)$
coincides with
$\bigoplus_{i \in \mathbb{Z}} \Hom_{S}(-, -)_i$
(i.e.\ the \emph{graded} hom functor).
\begin{proof}
(i)
We follow the proof of
\cite[Lemma~5.1.3]{BHHMS1},
but need to be careful about the grading.
First, we show by induction on $s$ that
$\Ext _{S} ^{i}( S/U, S) \cong (S/U)(d_U)$
if $i = s$ and $0$ otherwise.
The case $s = 0$ is immediate.
For $s > 0$, let $U' \defeq (u_1 , \dots , u _{s-1})$.
By assumption we have a short exact sequence
of graded $S$-modules
$ 0 \to S/U' \xrightarrow{u_s \cdot }
(S/U')(- \deg u_s) \to (S/U)(- \deg u_s) \to 0$,
taking $\Hom_{S}( -, S )$ and using
the inductive hypotheses
we find that
$\Ext _{S} ^{i}(S/U, S)$ vanishes
for all $i \ge 0$ except possibly 
 $i = s-1, s$,
and moreover we have an exact sequence 
of graded $S$-modules
\[
  0 \to \Ext_{S}^{s-1}( S/U , S )
  \to \Ext_{S}^{s-1}( S/U', S ) 
  \xrightarrow{u_s \cdot }
  \Ext_{S}^{s-1}( S/U', S )(\deg u_s) 
  \to \Ext_{S}^{s}( S/U, S) \to 0.
\]
Again by the inductive hypothesis,
we have the following graded isomorphisms
of $S$-modules
\begin{align*}
 \Ext_{S}^{s-1}( S/U , S ) & \cong 
 \ker 
 \left( (S/U')(d _{U'}) \xrightarrow{u_s \cdot }
 (S/U')(d_U)\right) = 0,\\ 
 \Ext_{S}^{s}( S/U , S ) & \cong 
 \coker \left( (S/U')(d _{U'}) \xrightarrow{u_s \cdot }
 (S/U')(d_U)\right) = (S/U)(d_U).
\end{align*}

Since $N$ is annihilated by $U$,
we have an equality of functors
$\Hom_{S/U}( N, \Hom_{S}( S/U, - )  ) 
  = \Hom_{S}( N, -)$
(we always implicitly consider the graded hom functor,
hence both sides take values in
the category of graded $S$-modules),
and so there is a Grothendieck spectral sequence 
$E ^{i,j}_2 =
\Ext_{S/U}^{i}( N, \Ext_{S}^{j}( S/U, S )) 
\Rightarrow \Ext_{S}^{i + j}( N, S )$,
which by 
the previous paragraph degenerates to
the graded $S$-module isomorphism
$ \Ext_S ^{s + i}(N , S) \cong
\Ext_{S/U}^{i}( N, S/U )(d_U)$.

(ii)
A direct consequence of (i),
using that $\deg h_j = -2 = \deg y_j z_j$
for $j  \in \{ 0, \dots, f-1 \} $,
that $(h_0, \dots , h_{f-1})$
is a regular sequence of central elements in $\grL$,
and that $(y_0 z_0 , \dots , y_{f-1} z _{f-1})$
is a regular sequence in the commutative ring $R$
(cf.\ \eqref{eq:R-definition}
and \eqref{eq:R-bar-definition}).
Compatibility with the $H$-action can be seen
by retracing the proof of (i)
and using that $H$ acts trivially on 
 $h_j$ and $y_j z_j $,
for $j  \in \{ 0, \dots, f-1 \} $
(or by using
 \cite[Lemma~3.65]{BHHMS2} directly).
\end{proof}

The following definitions will be needed
to state \Cref{lemma:fraka-dual} below.
\begin{definition}[{\cite[Definition~3.57]{BHHMS2}}]
  \label{def:tj}
To $\lambda \in  \mathscr{P}$
we associate an ideal 
$\mathfrak{a}(\lambda)$ of $R$
(cf.\ \eqref{eq:R-definition})
as follows:
$\mathfrak{a}(\lambda)=
(t_0, \dots, t_{f-1})$, where
\[
  t_j = 
\begin{cases}
  z_j & \text{if }\lambda_j(x_j) \in \{x_j,p-3-x_j\}, j \in J_{\rep}, \\
  y_j & \text{if }\lambda_j(x_j) \in \{x_j+2,p-1-x_j\}, j \in J_{\rep},\\
  y_jz_j & \text{if }\lambda_j(x_j) \in \{x_j,p-1-x_j\}, j \notin J_{\rep} , \\
  y_jz_j & \text{if }\lambda_j(x_j) \in \{x_j+1,p-2-x_j\}.
\end{cases}
\]
For $t_j $ as above,
we also set
$d_\lambda \defeq | \{j \mid t_j = y_j z_j \}|$.
\end{definition}
Observe that $(y_jz_j \mid 0 \le j \le f-1) \subseteq \mathfrak{a}(\lambda)$
for all $\lambda \in \mathscr{P}$.
Whenever convenient, we will identify $\mathfrak{a}(\lambda)$ 
with the unique ideal of $\Rbar$ 
(cf.\ \eqref{eq:R-bar-definition})
that pulls back to $\mathfrak{a}(\lambda)$ in $R$.

\begin{definition}[{\cite[Definition~4.2.4]{BHHMS4}}]
  \label{def:IJJ-def}
Suppose that $J_1, J_2$ are disjoint subsets 
of $\{0, \dots,f-1\}$ and that $d \in \mathbb{Z}$.
We define the ideal $I(J_1,J_2, d)$ of $\Rbar$
as follows: if $d \ge 1$ we let 
\begin{equation}
  \label{eq:IJJ-def}
I(J_1, J_2, d) \defeq \left( 
\prod_{j \in J_1'} y_j 
\prod_{j \in J_2'} z_j \,\middle|\,
J_1' \subseteq J_1, J_2' \subseteq J_2, 
|J_1'|+ |J_2'| =d\right);
\end{equation}
if $d \le 0$ we let $I(J_1, J_2, d) \defeq \Rbar$.
\end{definition}

For $\lambda \in \mathscr{P}$ and
$\ell \in \{-1, \dots,f\}$, define the ideal
of $\Rbar$
\begin{equation}
  \label{eq:a-fil-def}
\atwo[\ell](\lambda) \defeq 
\mathfrak{a}(\lambda)+
I(J_1(\lambda), J_2(\lambda), \ell
+ 1 - |J_\lambda|),
\end{equation}
where
\begin{align}
  \label{eq:J1-def}
 J_1 (\lambda) & \defeq  \{j \in  J_{\rep}^c  \mid 
\lambda_j(x_j)= p-1-x_j\} ,\\
  \label{eq:J2-def}
 J_2 (\lambda) & \defeq  \{j \in J_{\rep}^c  \mid 
\lambda_j(x_j)= x_j\},
\end{align}
and where $J_\lambda$
is defined in \eqref{eq:J-lambda-def}.
Observe that $J_\lambda$, $J_1 (\lambda)$
and $J_2 (\lambda)$
are pairwise disjoint subsets of
$\{ 0, \dots, f-1\}$.
In particular, 
$\atwo[-1] (\lambda) = \Rbar$
and
$\atwo[f] (\lambda) = \mathfrak{a} (\lambda)$.
More generally,
$\atwo[\ell] (\lambda) = \Rbar$
if and only if 
$\ell \le |J_\lambda| -1$,
and $\atwo[\ell] (\lambda) = \mathfrak{a} (\lambda)$ 
if and only if
$\ell \ge |J_1 (\lambda)| + |J_2  (\lambda)|
+ |J_\lambda|$.

\begin{definition}[{\cite[Definition~3.62]{BHHMS2}}]
  \label{def:lambda-star}
Given $\lambda \in \mathscr{P}$,
we define another $f$-tuple $\lambda^\ast$
as follows:
\[
\lambda^\ast_j(x_j) \defeq 
\begin{cases}
p-3-\lambda_j(x_j) & \text{if }t_j = z_j, \\
p+1-\lambda_j(x_j) & \text{if }t_j = y_j, \\
p-1- \lambda_j(x_j) & \text{if }t_j = y_jz_j.
\end{cases}
\]
\end{definition}
It is shown in 
\cite[Lemma~3.63(i)]{BHHMS2} 
that $\lambda^\ast \in \mathscr{P}$
and that
$\mathfrak{a}(\lambda^\ast)=
\mathfrak{a}(\lambda)$.
From this it follows, using
\eqref{eq:J1-def}, \eqref{eq:J2-def}, \Cref{def:lambda-star}
and the fact that $j \notin J_{\rep}$ implies $t_j = y_jz_j$
(cf.\ \Cref{def:tj}), that
\begin{equation}
  \label{eq:atwo-lambda-star}
\atwo[\ell](\lambda^\ast) = 
\mathfrak{a}(\lambda) +
I(J_2 (\lambda), J_1 (\lambda), \ell+1- |J_{\lambda^\ast} |).
\end{equation}
Moreover, by \cite[Lemma~4.1.4]{BHHMS4} we have
for all $\lambda \in  \mathscr{P}$
\begin{equation}
  \label{eq:lambda-lambda-ast-identity}
|J_{\lambda}| + |J_{\lambda^*}| + |J_1(\lambda) | + |J_2 (\lambda)| 
= f.
\end{equation}

Following 
\cite[Lemma~3.64]{BHHMS2},
we define the $\mathbb{F}_{q} ^{\times}$-character 
$\eta(a) \defeq \chi_{\lambda}
\left( \left( \begin{smallmatrix}
a & 0 \\
0 & a
\end{smallmatrix} \right) \right)$,
which doesn't depend on the choice of 
$\lambda \in \mathscr{P}$
(for example because $\pi$ has 
a central $K^\times $-character,
hence $\pi^{I_1 } \cong \V \otimes_{\F}\DOne$
has a central $\mathbb{F}_q ^{\times}$-character).
If we set
$\Eg[j](-) \defeq  \Ext^j_{\grL} (-, \grL)$,
then the following lemma is a variation on 
\cite[Proposition~3.66]{BHHMS2}
that will be needed in the proof of 
\Cref{prop:CM-via-N}.
\begin{lemma}
  \label{lemma:fraka-dual}
For $\lambda \in \mathscr{P}$ and for
$-1 \le \ell \le f$ we have
an isomorphism
\begin{equation}
  \label{eq:fraka-dual}
\Eg\left(\chi_\lambda^{-1} \otimes
\frac{\atwo[\ell](\lambda)}
{\mathfrak{a}(\lambda)}\right)
\cong 
\left( \chi_{\lambda^\ast}^{-1} 
\otimes \eta\right) \otimes 
(R/\atwo[f-1-\ell](\lambda^\ast)) (-4 f + 
(f - d_\lambda ))
\end{equation}
of graded $\grL$-modules with compatible $H$-action,
where $d_\lambda $
is defined in \Cref{def:tj}.
\end{lemma}
\begin{proof}
By \Cref{lemma:graded-decalage}(ii)
applied to $i = 0, M = \chi_{\lambda^\ast}^{-1}
\otimes R/\mathfrak{a}(\lambda)$,
the left hand side of 
\eqref{eq:fraka-dual} is isomorphic to
$\Hom_{\Rbar}
\left(\chi_\lambda^{-1} \otimes
\frac{\atwo[\ell](\lambda)}
{\mathfrak{a}(\lambda)}, \Rbar\right)(-4 f)$
as a graded $\grL$-module with compatible $H$-action.

Twisting by $\chi_\lambda^{-1} $ and using 
\cite[Lemma~3.64]{BHHMS2},
it is enough to construct an isomorphism
of graded $\grL$-modules with compatible $H$-action
\[
\Hom_{\Rbar}
\left( \frac{\atwo[\ell](\lambda)}
{\mathfrak{a}(\lambda)}, \Rbar\right)
\cong 
\eta_{\lambda}^{-1} \otimes 
R/\atwo[f-1-\ell](\lambda^\ast)(f - d_\lambda),
\]
where $\eta_\lambda$ is the character of $H$
acting on $ \prod _{j = 0}^{f-1} t_j$
(cf.\ \Cref{def:tj} for the definition of $t_j$).

\paragraph{Step 1.}
Consider the natural graded $\Rbar$-module maps
\begin{align}
  \label{eq:mult-map-def}
\Rbar &\to \Hom_{\Rbar}
\left( \mathfrak{a}(\lambda), \Rbar\right),
& \Rbar &\to \Hom_{\Rbar}
\left( \atwo[\ell](\lambda), \Rbar\right),\\
a &\mapsto (m_a: x \mapsto ax), &
\nonumber
a &\mapsto (m_a: x \mapsto ax),
\end{align}
with compatible $H$-action,
and with respective kernels
$\Rbar[ \mathfrak{a}(\lambda)]$ and $\Rbar[ \atwo[\ell](\lambda)]$.
We claim that the maps \eqref{eq:mult-map-def}
are surjections.
More generally, let
$\mathfrak{b}$ be an ideal of 
$\Rbar \cong 
\F [y_j, z_j\mid 0 \le j \le f-1] / (y_j z_j\mid 0 \le j \le f-1)$
generated by monomials,
or more precisely by the image in $\Rbar$ 
of a set of monomials
$\{u_1 , \dots, u_n\}$ of
$R \cong 
\F [y_j, z_j\mid 0 \le j \le f-1]$
(this is true for
$\mathfrak{b} \in \left\{ \mathfrak{a}(\lambda), \atwo[\ell](\lambda) \right\}$,
cf.\ \Cref{def:tj} and \Cref{def:IJJ-def}).
We claim that the natural graded $\Rbar$-module map
with compatible $H$-action
\begin{align}
  \label{eq:mult-map-gen-def}
\frac{\Rbar }{\Rbar [\mathfrak{b}]}&\to \Hom_{\Rbar}
\left( \mathfrak{b}, \Rbar\right), \\
\nonumber
a &\mapsto (m_a: x \mapsto ax)
\end{align}
is an isomorphism.

We prove this by induction on the number $n$ of (monomial) generators of $\mathfrak{b}$.
The case when $\mathfrak{b} = (u_1)$ is generated by a single monomial is immediate.
Now assume the claim true for $n$,
and let $\mathfrak{b} = (u_1 , \dots, u_{n+1})$
be generated by $n+1$ monomials.
If we set $\mathfrak{b}_0 \defeq (u_1 , \dots, u_n)$,
then we have the following commutative diagram with exact rows
\begin{equation}
  \label{eq:diagram-b-induction}
\begin{tikzcd}[column sep = small]
0 \ar[r] & \frac{\Rbar}{\Rbar[\mathfrak{b}]} \ar[r]\ar[d] 
&\frac{\Rbar}{\Rbar[\mathfrak{b}_0 ]}\oplus  \frac{\Rbar}{\Rbar[u_{n+1}]} \ar[r]\ar[d, "\wr"] 
& \frac{\Rbar}{\Rbar[\mathfrak{b}_0] + \Rbar[ u_{n+1}]} \ar[r]\ar[d] & 0 \\
0 \ar[r] & \Hom_{\Rbar}(\mathfrak{b}, \Rbar) \ar[r] 
& \Hom_{\Rbar}(\mathfrak{b}_0 \oplus (u_{n+1}), \Rbar) \ar[r] 
&  \Hom_{\Rbar}(\mathfrak{b}_0 \cap (u_{n+1}), \Rbar),
\end{tikzcd}
\end{equation}
where the middle vertical map is an isomorphism
by the inductive hypothesis,
and where the top row is exact
by a version of the Chinese remainder theorem
(we are using that
$\Rbar [\mathfrak{b}] =
\Rbar [\mathfrak{b}_0 ] \cap \Rbar [u_{n+1}]$).
By the snake lemma, it is enough to show
that the vertical map on the right is injective to conclude.
First, we remark that
$\mathfrak{b}_0 \cap (u_{n+1}) = 
(u_1, \dots, u_n) \cap (u_{n+1}) = (\lcm(u_1 , u_{n+1}), \dots, \lcm(u_n, u_{n+1}))$,
where the least common multiple is computed inside
the unique factorisation domain $R$.
Indeed, 
we can compute this intersection 
by pulling back the two ideals on the left
to $R \cong \F[y_j, z_j \mid 0 \le j \le f-1]$,
in which case it becomes a standard fact of monomial ideals.

Then, we show that 
$\Rbar[\mathfrak{b}_0 ] + \Rbar[u_{n+1}] = \Rbar[\mathfrak{b} \cap (u_{n+1})]$,
and in fact they are both equal (as sets) to 
$\Rbar[\mathfrak{b}_0 ] \cup   \Rbar[u_{n+1}]$:
we always have
$\Rbar[\mathfrak{b}_0 ] + \Rbar[u_{n+1}] \subseteq  \Rbar[\mathfrak{b} \cap (u_{n+1})]$
and $\Rbar[\mathfrak{b}_0 ] \cup \Rbar[u_{n+1}]
\subseteq \Rbar[\mathfrak{b}_0 ] +\Rbar[u_{n+1}]$,
we are only left to prove
$\Rbar[\mathfrak{b} \cap (u_{n+1})] \subseteq 
\Rbar[\mathfrak{b}_0 ] \cup \Rbar[u_{n+1}]$.

When $n =1$, this amounts to the statement that,
for all $a \in \Rbar$,
 $(u_1 a = 0) \text{ or } (u_2 a = 0)$ is equivalent to $ \lcm(u_1, u_2 )a = 0$.

In the general case, we have
\begin{align*}
\Rbar[\mathfrak{b}_0  \cap (u_{n+1})]  &= \Rbar[(u_1, \dots, u_n) \cap (u_{n+1})] = 
\Rbar[(\lcm(u_1 , u_{n+1}), \dots, \lcm(u_n, u_{n+1}))] \\
&= \bigcap_{i = 1} ^n \Rbar[\lcm(u_i, u_{n+1})] \overset{(*)}{=}
\bigcap_{i = 1} ^n \left(\Rbar[u_{i}] \cup \Rbar[u_{n+1}]   \right) 
=  \left(\bigcap_{i = 1} ^n\Rbar[u_{i}]\right) \cup \Rbar[u_{n+1}]    \\
&=  \Rbar[\mathfrak{b}_0] \cup \Rbar[u_{n+1}],
\end{align*}
where in $(*)$ we have used the case $n = 1$.

Hence, the vertical map on the right in \eqref{eq:diagram-b-induction} becomes
$\frac{\Rbar}{\Rbar[\mathfrak{b}_0 \cap (u_{n+1})]}
\to  \Hom_{\Rbar}(\mathfrak{b}_0 \cap (u_{n+1}), \Rbar)$,
which is an isomorphism by the inductive hypothesis.

\paragraph{Step 2.}
We show that there is 
a graded $\Rbar$-module isomorphism
$\frac{\Rbar[\mathfrak{a}(\lambda)]}{\Rbar[\mathfrak{a}^{\ell}(\lambda)]}
\xrightarrow{\sim} 
\Hom_{\Rbar}(\frac{\atwo[\ell](\lambda)}{\mathfrak{a}(\lambda)}, \Rbar)$
with compatible $H$-action.

Consider the following commutative diagram 
of graded $\Rbar$-modules
with exact rows
\[
\begin{tikzcd}
0 \ar[r] & \frac{\Rbar[\mathfrak{a}(\lambda)]}{\Rbar[\mathfrak{a}^{\ell}(\lambda)]} \ar[r]\ar[d] 
& \frac{\Rbar}{\Rbar[\atwo[\ell](\lambda)]} \ar[r]\ar[d, "\wr"] 
& \frac{\Rbar}{\Rbar[\mathfrak{a}(\lambda)]} \ar[r]\ar[d, "\wr"] & 0 \\
0 \ar[r] & \Hom_{\Rbar}(\frac{\atwo[\ell](\lambda)}{\mathfrak{a}(\lambda)}, \Rbar) \ar[r] 
& \Hom_{\Rbar}(\atwo[\ell](\lambda), \Rbar) \ar[r] 
&  \Hom_{\Rbar}(\mathfrak{a}(\lambda), \Rbar)
\end{tikzcd}
\]
where the vertical isomorphisms in the middle and on the right
come from Step 1.
We conclude that the vertical map on the left is
an isomorphism by applying the snake lemma.
Compatibility with the $H$-action 
comes from the fact that 
$ \frac{\Rbar}{\Rbar[\atwo[\ell](\lambda)]}
\xrightarrow{\sim}
\Hom_{\Rbar}(\atwo[\ell](\lambda), \Rbar)$
is compatible with $H$-action,
again by Step 1.

\paragraph{Step 3.}
We produce
a graded $\Rbar$-module isomorphism
$(\eta_{\lambda}^{-1} \otimes \Rbar/\atwo[f-1-\ell](\lambda^*))(f - d_\lambda)
\xrightarrow{\sim} 
\frac{\Rbar[\mathfrak{a}(\lambda)]}{\Rbar[\atwo[\ell](\lambda)]}$
with compatible $H$-action,
thus concluding the proof.

Set $t_j' \defeq y_jz_j/t_j$,
and set $t' \defeq \prod_{j = 0}^{f-1}t_j'$.
First, we show the equalities
$t' \Rbar =  \Rbar[\mathfrak{a}(\lambda)]$
and
$ t'\atwo[f-1- \ell](\lambda^*) =  
\Rbar[\atwo[\ell](\lambda)]$
of ideals of $\Rbar$.
For the first equality, we have
\begin{align}
  \label{eq:t'-computation}
\Rbar [\mathfrak{a}(\lambda)] & = \bigcap_{j = 0}^{f-1} \Rbar[t_j] = \bigcap_{j = 0}^{f-1} \Rbar t_j' = t' \Rbar,
\end{align}
where on the last equality 
we have used that the $t_j'$ are pairwise coprime.

We prove the inclusion 
$ t'\atwo[f-1- \ell](\lambda^*) \subseteq 
\Rbar[\atwo[\ell](\lambda)]$.
By \eqref{eq:atwo-lambda-star}
and using the fact that $t' t_j = 0$ for all $j \in \{0, \dots,f-1\}$,
we only need to check that,
for fixed
$J_1' \subseteq J_1 (\lambda),
J_2' \subseteq J_2 (\lambda)$
such that $|J_1'|+ |J_2'| = f- \ell - |J_{\lambda^*}|$,
we have
$t' \prod_{j \in J_1'} z_j \prod_{j \in J_2'} y_j 
 \in \Rbar[\atwo[\ell] (\lambda)]$ .
Since again $t' t_j = 0$ 
for all $j \in \{0, \dots,f-1\}$,
by \eqref{eq:a-fil-def} we only need to check that 
\[
t' \prod_{j \in J_1'} z_j \prod_{j \in J_2'} y_j
\prod_{j \in J_1''} y_j \prod_{j \in J_2''} z_j = 0
\]
for all $J_1'' \subseteq J_1(\lambda), J_2'' \subseteq J_2(\lambda)$
such that $|J_1''|+ |J_2''| = \ell + 1 - |J_{\lambda}|$.

This certainly holds if $J_1 ' \cap J_1 '' \neq \emptyset $
or if $J_2 ' \cap J_2 '' \neq \emptyset $,
and if 
$J_1 ' \cap J_1 '' = J_2 ' \cap J_2 '' =  \emptyset $
then
\begin{align*}
|J_1 (\lambda)| + |J_2(\lambda) | &\ge 
|J_1 ' \sqcup J_1 ''| + |J_2 ' \sqcup J_2 '' | 
= |J_1 '| + |J_2 ' |+ | J_1 ''| +  | J_2 '' | =\\
&=(\ell + 1 - |J_{\lambda^*}|) + (f - \ell - |J_{\lambda}|)
= f + 1 - |J_{\lambda^*}| - |J_{\lambda}|,
\end{align*}
which contradicts \eqref{eq:lambda-lambda-ast-identity}.

We prove the inclusion 
$ \Rbar[\atwo[\ell](\lambda)] \subseteq 
t'\atwo[f-1- \ell](\lambda^*)$. 
Write $\Rbar [\atwo[\ell](\lambda)]$ as
\begin{align}
\Rbar[\atwo[\ell](\lambda)]
&\overset{\eqref{eq:atwo-lambda-star}}{=} 
\Rbar[\mathfrak{a}(\lambda)] \cap 
\bigcap_{
\substack{
J_1'' \subseteq J_1(\lambda),\ J_2'' \subseteq J_2(\lambda)
\\
|J_1''|+ |J_2''| = \ell + 1 - |J_{\lambda}|}
}\Rbar 
\left[\prod_{j \in J_1 ''} y_j  \prod_{j \in J_2 ''} z_j\right]
\\
  \label{eq:annihilator-ugly-intersection}
& \overset{\eqref{eq:t'-computation}}{=}
t' \Rbar \cap 
\bigcap_{
\substack{
J_1'' \subseteq J_1(\lambda),\ J_2'' \subseteq J_2(\lambda)
\\
|J_1''|+ |J_2''| = \ell + 1 - |J_{\lambda}|}
}
\big( 
\left(z_j\mid j \in J_1 ''\right) + 
\left(y_j \mid j \in J_2 ''\right)
 \big).
\end{align}
It is clear from this description that
$\Rbar [\atwo[\ell](\lambda)]$
is a monomial ideal in $\Rbar$
generated by squarefree monomials,
so it suffices to show that every squarefree monomial
$u \in \Rbar [\atwo[\ell](\lambda)]$
belongs to $t' \atwo[f-1- \ell](\lambda^*)$,
or even that it belongs to 
$\atwo[f-1- \ell](\lambda^*)$.
We can uniquely write 
$u \in \Rbar [\atwo[\ell](\lambda)]$
as a product
$u = \prod_{j \in \widetilde{J}_1 } z_j 
\prod_{j \in \widetilde{J}_2 } y_j$,
for some $\widetilde{J}_1 ,\ \widetilde{J}_2 \subseteq 
\{0, \dots, f-1\}$
such that 
$\widetilde{J}_1 \cap \widetilde{J}_2 = \emptyset$,
by \eqref{eq:atwo-lambda-star}
it suffices to show that, if we set
$J_1 ' \defeq \widetilde{J}_1 \cap J_1 (\lambda)$,
$J_2 ' \defeq \widetilde{J}_2 \cap J_2 (\lambda)$,
then we have $|J_1 '| + |J_2 '| \ge 
f- \ell -|J_{\lambda^*}|$.
Suppose for the sake of contradiction that
$|J_1 '| + |J_2 '| <
f- \ell -|J_{\lambda^*}|$.
It follows that 
\[
|J_1 (\lambda) \setminus J_1 '| + 
|J_2 (\lambda) \setminus J_2 '| > 
|J_1 (\lambda)| + |J_2 (\lambda)| - 
(f - \ell -|J_{\lambda^*}| )
\overset{\eqref{eq:lambda-lambda-ast-identity}}{=}
\ell - |J_\lambda|
,
\]
hence we can find 
$J_1'' \subseteq J_1(\lambda),\ J_2'' \subseteq J_2(\lambda)$
such that $J_1 '' \cap J_1 ' = \emptyset $,
$J_2 '' \cap J_2 ' = \emptyset $,
and such that 
$|J_1''|+ |J_2''| = \ell + 1 - |J_{\lambda}|$.
By \eqref{eq:annihilator-ugly-intersection}
(and since $u$ is a monomial),
we either have $z_j | u$ for some
$j \in J_1 ''$,
or $y_j | u$ for some
$j \in J_2 ''$.
In the first case this contradicts 
$J_1 '' \cap J_1 ' = \emptyset $ 
and the definition of $J_1'$,
while in the second case 
this contradicts 
$J_1 '' \cap J_1 ' = \emptyset $ 
and the definition of $J_2'$.

Recall that $\eta_\lambda$ is the character of $H$
acting on $\prod_{j = 0}^{f-1} t_j$,
hence that
$\eta_\lambda^{-1}$ is the character of $H$
acting on $t' = \prod_{j = 0}^{f-1} t_j' = \prod_{j = 0}^{f -1 } (y_j z_j/t_j)$.
We can now consider the following commutative diagram with exact rows
of graded $\Rbar$-modules with compatible $H$-action
\[
\begin{tikzcd}[cramped]
0  \ar[r] & \eta_\lambda^{-1} \otimes 
{\atwo[f-1- \ell](\lambda^*)} 
\ar[r]\ar[d, "{t' \cdot}"', two heads] 
& \eta_\lambda^{-1} \otimes 
\Rbar
\ar[r]\ar[d, "{t'\cdot}"', two heads] 
& \eta_\lambda^{-1} \otimes 
{\frac{\Rbar}{\atwo[f-1- \ell](\lambda^*)}}
\ar[r]\ar[d, "{t' \cdot }"', two heads] & 0 \\
0 \ar[r] &  {\Rbar {[\atwo[\ell](\lambda) ]} } (\deg t')
 \ar[r] 
&  {\Rbar[\mathfrak{a}(\lambda)]} (\deg t')\ar[r] 
& \frac{\Rbar[\mathfrak{a}(\lambda)]}
{\Rbar {[\atwo[\ell](\lambda)]}} ( \deg t')\ar[r] & 0,
\end{tikzcd}
\]
where the vertical maps are well defined 
and surjective
by the previous paragraphs.
After observing that $\deg t' = -f + d_\lambda$,
we conclude by showing 
that the vertical arrow on the right is injective.
By the snake lemma,
it is enough to show that 
$\atwo[f-1- \ell](\lambda^*)[t'] = \Rbar [t']$,
which follows from 
$\Rbar [t'] = \mathfrak{a}(\lambda) =
\mathfrak{a}(\lambda^*) \subseteq \atwo[f-1- \ell](\lambda^*)$.
\end{proof}

For $\ell \in \{-1, \dots, f\}$ and
$\lambda \in \mathscr{P}_\ell$,
define
\begin{equation}
  \label{eq:Nfil-def}
\Rfil[\ell](\lambda) \defeq 
\chi_\lambda^{-1} \otimes 
\frac{\Rbar}{\atwo[\ell](\lambda)},
\quad 
\Nfil[\ell] \defeq 
\bigoplus _{\lambda \in \mathscr{P}}
\Rfil[\ell](\lambda)
\end{equation}
which come with natural surjections
of $\Rbar$-modules with compatible $H$-action
$\Rfil[\ell'](\lambda) \twoheadrightarrow 
\Rfil[\ell] (\lambda)$,
$\Nfil[\ell'] \twoheadrightarrow 
\Nfil[\ell]$
for $\ell \le \ell'$.
When $\ell = f$, we set
\begin{equation}
  \label{eq:Nnull-def}
\Nnull \defeq \Nfil[f] = \bigoplus_{\lambda \in \mathscr{P}} 
\left( \chi_\lambda^{-1} \otimes_{\F} \frac{\Rbar}{\mathfrak{a}(\lambda)} \right)
\end{equation}

Observe that a minimal prime $\mathfrak{q} \subseteq \Rbar$
is uniquely of the form $\mathfrak{q} = (v_0 , \dots, v_{f-1})$,
for $v_j \in \left\{ y_j, z_j \right\}$.
Then, we define 
\begin{equation}
  \label{eq:Yq-Zq-def}
Y(\mathfrak{q}) = \{j \in \left\{ 0, \dots, f-1 \right\} \mid v_j = y_j\}, \quad 
Z(\mathfrak{q}) = \{j \in \left\{ 0, \dots, f-1 \right\} \mid v_j = z_j\}.
\end{equation}

We recall here 
\cite[Definition~3.2.4]{Ber}
(in \emph{loc.\ cit.}\ we assume
$\rep$ semisimple,
even though the definition does not need
this assumption).
By hypothesis \ref{hypothesis:iii}
and by \cite[Corollary~1.8]{Koh17},
we can fix once and for all
a $\GF$-equivariant isomorphism 
\begin{equation}
	\label{eq:kappa-pi-def}
\kappa_\pi \colon \pi \xrightarrow{\sim}
\E(\pi^\vee)^\vee \otimes (\mytwist)
\end{equation}
of $\Lambda$-modules.
Let $\pi_1$ be a subrepresentation of $\pi$,
and set $\pi_2  \defeq \pi/ \pi_1$.
Since $\Lambda$ is Auslander regular 
(cf.\ the discussion before
\Cref{def:multiplicity})
and using hypothesis \ref{hypothesis:iii},
it follows from
\cite[Cor.~III.2.1.6]{LvO}
that $\pi_2 ^{\vee}$ 
is of grade $\ge 2f$,
hence $\E[2f -1] (\pi_2 ^{\vee})=0$.
In particular, the natural inclusion
$\pi_1  \subseteq \pi$ induces a surjection
$ \E(\pi^\vee)^\vee 
\twoheadrightarrow 
\E(\pi_1 ^{\vee})^\vee$. 

\begin{definition}
	\label{def:conjugate-srep}
We define the subrepresentation
$ \widetilde{\pi}_1 \subseteq \pi$
as 
\begin{equation}
	\label{eq:conjugate-def}
\widetilde{\pi}_1  \defeq 
\ker \left( \pi 
\xrightarrow[\sim]{\kappa_\pi}
\E(\pi^\vee)^\vee \otimes (\mytwist)
\twoheadrightarrow 
\E(\pi_1 ^{\vee})^\vee \otimes (\mytwist) \right) .
\end{equation}
\end{definition}

If $F'$, $F''$
are two good filtrations
(in the sense of \cite[§I.5]{LvO})
on a finitely generated $\Lambda$-module
$M$, such that $\gr_{F'}(M)$
(or equivalently $\gr_{F''} (M)$,
by the discussion before
\cite[Proposition~3.1.2.11]{BHHMS2})
is annihilated by some power of $J$,
then 
$\gr_{F'}(M)$ and $\gr_{F''} (M)$
are finitely generated as $\grL$-modules
by \cite[Lemma~I.5.4]{LvO},
and moreover
$\mathcal{Z}(\gr_{F'}(M)) =\mathcal{Z}(\gr_{F''} (M))$
 by \cite[Lemma~3.3.4.3]{BHHMS2}.
Therefore, we can define 
$m_{\mathfrak{q}} (M)$
to be  $m_{\mathfrak{q}} (\gr_F(M))$
and $\mathcal{Z}(M)$
to be  $\mathcal{Z}(\gr_F(M))$,
where $F$ is any good filtration on $M$.
The proof of \cite[Lemma~3.2.7]{Ber}
goes through for $\rep$ non-semisimple,
and we obtain the following lemma.
\begin{lemma}
	\label{lemma:Z-equal-Z}
Keeping the above notation, we have
\[
\mathcal{Z}(\pi_2 ^{\vee})
=
\mathcal{Z}(\widetilde{\pi}_1 ^{\vee}).
\]
\end{lemma}
\begin{remark}
	\label{rem:conjugate-when-CM}
Keep the notation of \Cref{def:conjugate-srep},
and assume
$\E[2f+1](\pi_1 ^{\vee} )=0$.
Then, arguing as in 
\cite[Remark~3.2.5]{Ber}
we obtain
the following commutative diagram
of $\Lambda$-modules
with exact rows:
\begin{equation}
  \label{eq:conjugate-when-CM}
\begin{tikzcd}
0 \ar[r] & \E(\pi_1 ^{\vee})\ar[r]
&  \E(\pi^\vee)  
\ar[r]\ar[d,"{ \kappa_\pi^\vee \otimes (\mytwist) }"', "{\wr}"] 
& 
\E(\pi_2 ^{\vee}) 
\ar[d,"{ \kappa_\pi^\vee \otimes  (\mytwist)}", "{\wr}"'] \ar[r]
& 0 \\
& 
& \pi^\vee \otimes (\mytwist)  \ar[r]
&\widetilde{\pi}_1 ^{\vee}\otimes (\mytwist)
\end{tikzcd}\end{equation}
\end{remark}

The following proposition
establishes the Cohen-Macaulayness
of $\pi$ and of all 
its nontrivial subrepresentations and quotients,
generalising
\cite[Proposition~4.4.3]{BHHMS4} for $r \ge 1$.
Its corollary \Cref{cor:CM-via-N-subquotients}
will be used in the proof of
\Cref{thm:gen-by-inv} below.
\begin{proposition}
  \label{prop:CM-via-N}
Assume that $\rep$ is 
$\max \{9,2f + 1\}$-generic.
Let $\pi_1 \subseteq \pi$
be a subrepresentation of $\pi$ and let
$\pi_2 \defeq \pi/\pi_1$.
Endow $\pi_2^{\vee }$
with the submodule filtration $F$
induced from the 
$\mI$-adic filtration on $\pi ^\vee$.

For all $\ell \in \{-1, \dots, f \}$,
pick a complementary subspace
$W_1(\ell)$ of $\VFil[\pi_1,\ell+1]$ 
inside $\VFil[\pi_1,\ell]$ 
(cf.\ \eqref{eq:VFil-rFil-def})
and set
\begin{align}
  \label{eq:N-and-N1-def}
N \defeq \V^\vee \otimes_{\F} \Nnull ,
\quad
N_1 \defeq \bigoplus_{\ell=0}^f
\left( W_1(\ell)^\vee \otimes_{\F} 
\Nfil[\ell]\right), \quad
N_2 \defeq \ker (N \twoheadrightarrow N_1).
\end{align}
\begin{enumerate}[(i)]
\item 
There is a commutative diagram 
\begin{equation}
  \label{diag:thetas}
\begin{tikzcd}
0 \ar[r] & \gr_F( \pi _2 ^\vee   ) \ar[r] &  
\gr_{\mI} ( \pi ^\vee ) \ar[r] &
\gr_{\mI} ( \pi_1 ^\vee ) \ar[r] & 0 \\
0 \ar[r] & N_2 \ar[u, "\theta_2"]  \ar[r] 
&N \ar[u, " \theta "] \ar[r]  & 
N_1  \ar[u, " \theta _1"] \ar[r] & 0
\end{tikzcd}
\end{equation}
of 
graded $\Rbar$-modules with compatible $H$-action
with exact rows such that
the vertical maps are isomorphisms.

\item 
Assume now that $0 \subsetneq \pi_1 \subsetneq \pi$
is a proper subrepresentation of $\pi$.
Then, both $\gr_{\mI} (\pi_1 ^\vee ) $and
$\gr_{F}  (\pi_2 ^\vee ) $
are Cohen-Macaulay
$\grL   $-modules
of grade $2f $.
In particular, 
$\pi_1 ^\vee  $
and
$\pi_2 ^\vee  $
are Cohen-Macaulay $\Lambda $-modules of grade $2f $.
\end{enumerate}
\end{proposition}
\begin{remark}
  \label{rmk:F-not-mI}
As in \cite[Remark~4.4.4]{BHHMS4},
if $0 \subsetneq \pi_1 \subsetneq \pi$
is a proper subrepresentation of $\pi$,
then we remark that $F$ does not coincide with the
$\mI$-adic filtration in general,
for example because $\pi_2^{I_1 }$ 
is in general larger than
$\pi^{I_1 }/ \pi_1^{I_1 }$.
Indeed, when $\Jrep =  \emptyset$,
we see from \eqref{eq:srep-I1-decomposition} and
\Cref{rmk:I1-does-not-reconstruct-pi}
that $\pi^{I_1 }/ \pi_1^{I_1 }$ vanishes
if $\pi_1 $ is such that $\VFil[\pi_1 ,\ell] = \V$ 
for all $\ell \le f/2$,
whereas $\pi_2^{I_1}$ is always nonzero
(as $I_1 $ is a pro-$p$ group).
\end{remark}

\begin{proof}
We broadly follow the proof of \cite[Proposition~4.4.3]{BHHMS4},
but in Step 3 our argument adapts 
the proof of \cite[Proposition~3.2.2]{BHHMS4}
to allow us to obtain the inequality
\eqref{eq:hype-mult-inequality} below.

We take the graded $\Rbar$-module map 
$\theta \colon N \to \grm (\pi^\vee )$
with compatible $H$-action 
of \eqref{diag:thetas}
to be given by the isomorphism of 
\cite[Eq.~(14)]{BHHMS4}.
For an explicit construction of $\theta$,
cf.\ Step 1 and 2 of the proof of
\cite[Lemma~3.2.1]{Ber}, applied to $\pi_2 = \pi$
(although \emph{loc.\ cit.}\ assumes 
$\rep$ semisimple, 
the argument goes through \emph{verbatim});
notice that the construction of 
\emph{loc.\ cit.}\ only depends 
on the choice of a basis
$\mathcal{B}$ of the
$\F$-vector space
$\left( \pi^{I_1 } \right)^\vee $
consisting of $H$-eigenvectors.

Recall the integer $\rFil[\pi_1 , \ell]$ 
defined in \eqref{eq:VFil-rFil-def},
and pick a basis $(w_s \mid 1 \le s \le r)$
of $\V$ adapted to the direct sum decomposition
$\V = \bigoplus_{\ell =  -1} ^f W_1(\ell)$,
in the sense that 
$(w_s \mid \rFil[\pi_1 , \ell +1]+1 \le s \le 
\rFil[\pi_1 , \ell])$
is an $\F$-basis of $W_1 (\ell)$
for all $\ell \in \{-1, \dots, f\}$.
Let $(u_s \mid 1 \le s \le r)$
be the dual basis of 
$(w_s \mid 1 \le s \le r)$
in $\V^\vee$.
Our choice of $\F$-basis $\mathcal{B}$ 
of $\left( \pi^{I_1 } \right)^\vee $
is as follows:
recall the $\F$-basis 
$(v_\lambda)_{\lambda \in \mathscr{P}}$ 
of $\DOne$ defined in the paragraph of
\eqref{eq:v-iso-def},
and let 
$(e_\lambda)_{\lambda \in \mathscr{P}}$ 
be its dual basis in $\DOne^\vee$.
Finally, define
\begin{equation}
    \label{eq:explicit-dual-basis}
e_{\lambda, s} \defeq \ip^\vee 
(u_s \otimes e_\lambda) \in \left( \pi^{I_1 } \right)^\vee, \quad \mathcal{B} \defeq \left( e_{\lambda, s}
\mid \lambda \in \mathscr{P},\ 
1 \le s \le r\right),
\end{equation}
where $\ip^\vee \colon \V^\vee \otimes_{\F}\DOne^\vee 
\xrightarrow{\sim}
(\pi^{I_1})^\vee $
is the $\F$-linear dual of 
$\ip \colon \pi^{I_1 } \xrightarrow{\sim} \DOne$.

Unwinding definitions, we have
\begin{equation}
    \label{eq:N-unwinding}
N \cong \bigoplus_{\ell = -1} ^f 
\left( W_1(\ell)^\vee \otimes_{\F} \Nnull \right)
\overset{\eqref{eq:Nnull-def}}{\cong}
\bigoplus_{\ell = -1} ^f 
\bigoplus_{\lambda \in \mathscr{P}} 
\left( W_1(\ell)^\vee \otimes_{\F} 
\Rbar (\lambda) \right),
\end{equation}
and $\theta$ is characterised as 
the unique $\Rbar$-linear
homomorphism sending
$u_s \otimes 1 \in  W_1(\ell)^\vee \otimes_{\F} 
\Rbar (\lambda)
\overset{\eqref{eq:N-unwinding}}{\subseteq} N$
to $e_{\lambda,s} \in 
\left( \pi^{I_1 } \right)^\vee
= \grm \left( \pi^\vee  \right)_0 
\subseteq \grm \left( \pi^\vee  \right)$,
for all $\lambda \in \mathscr{P},\ 1 \le s \le r$.

\paragraph{Step 1.} 
We show that the composition
\[
N \xrightarrow{\theta} \grm \left( 
\pi^\vee \right) \twoheadrightarrow 
\grm \left( \pi_1^\vee  \right)
\]
factors via $N \twoheadrightarrow N_1$.

Fix 
$\lambda \in \mathscr{P}$,
$\ell \in \{-1, \dots, f\}$,
and $s \in 
\{\rFil[\pi_1 , \ell +1 ] + 1 , \dots, \rFil[\pi_1 , \ell]\}$.
By construction of $\theta$ above,
it suffices to show that
the ideal $\atwo[\ell](\lambda)$
of $\Rbar$ annihilates
the image of $e_{\lambda,s}$ in
$\big( \pi_1 ^{I_1 } \big)^\vee
= \grm \left( \pi_1 ^\vee  \right)_0 
\subseteq \grm \left( \pi_1 ^\vee  \right)$,
which we continue to denote by $e_{\lambda,s}$
by an abuse of notation.
We already know by \cite[Theorem~3.67]{BHHMS2}
(whose hypotheses are satisfied by (i) and (ii) of
\Cref{sec:hypotheses})
that the ideal $\mathfrak{a}(\lambda)$ annihilates
$e_{\lambda,s} \in \grm \left( \pi^\vee  \right)$,
hence a fortiori it annihilates its image in
$\grm \left( \pi_1 ^\vee  \right)$.
By \eqref{eq:a-fil-def}, we are left to show that 
$I(J_1(\lambda),J_2(\lambda), \ell+1- |J_{\lambda} |)$
(defined in \eqref{eq:IJJ-def})
also annihilates
$e_{\lambda,s} \in 
\grm \left( \pi_1 ^\vee  \right)$.
So, let us consider a monomial 
$t \defeq \prod_{j \in J_1 '} y_j\prod_{j \in J_2 '} z_j$,
with $J_1 ' \subseteq J_1(\lambda) , J_2 ' \subseteq J_2 (\lambda)$,
$|J_1 '| + |J_2 '| = \ell +1 - |J_\lambda|$.

If we define $\lambda' \in \mathscr{P}^{\mathrm{ss}}$ by
\begin{equation}
  \label{eq:lambda-prime-def}
\lambda_j'(x_j) \defeq 
\begin{cases}
\lambda_j(x_j)-2 & \text{if }j \in J_1 ', \\
\lambda_j(x_j)+2 &  \text{if }j \in J_2 ',\\
\lambda_j(x_j) & \text{otherwise},
\end{cases}
\end{equation}
then one checks as in the proof of 
\cite[Proposition~4.4.3]{BHHMS4}
that multiplication by $t$
sends the $\chi_\lambda^{-1}$-eigenspace of
$\grm(\pi_1^\vee )$ to the 
$\chi_{\lambda'}^{-1}$-eigenspace of $\grm(\pi_1^\vee )$,
that $\lambda' \in \mathscr{P} ^{\mathrm{ss}} \setminus \mathscr{P}$,
and that
$|J_{\lambda'}| = |J_\lambda| + (\ell + 1 - |J_\lambda|)
= \ell +1$.

We see that
$t e_{\lambda,s} \in \grm (\pi^\vee _1) $ 
is homogeneous of degree
$\deg t + \deg e_{\lambda,s} 
= (|J_1 '| + |J_2 '|) + 0  \le f$
(because $\grm (\pi^\vee _1)$
is a graded $\Rbar$-module),
hence 
$t e_{\lambda,s} \in 
\bigoplus_{m = 0} ^f \grm (\pi_1^\vee )_{-m}
\subseteq \grm (\pi_1^\vee )$.
Observe that
$\bigoplus_{m = 0} ^f \grm (\pi_1^\vee )_{-m}
\cong \grm (\pi_1 [\mI^{f+1}]^\vee )$
and that
$\bigoplus_{m = 0} ^f \grm (\pi^\vee )_{-m}
\cong \grm (\pi [\mI^{f+1}]^\vee )$,
and consider the following commutative diagram
of $\Rbar$-modules
\begin{equation}
    \label{eq:i-funny-diagram}
\begin{tikzcd}
\grm (\pi [\mI^{f +1 }]^\vee ) \ar[r, "i^\vee ", "\sim"']\ar[d, two heads] 
& \grm (\tau^{(f +1 )} (\pi)[\mI^{f +1 }]^\vee ) 
\ar[d, two heads] \\
\grm (\pi_1  [\mI^{f +1 }]^\vee ) \ar[r, "i^\vee ", "\sim"'] &   
\grm ( i^{-1} (\pi_1  [\mI^{f +1 }])^\vee ),
\end{tikzcd}
\end{equation}
where the $I$-representation 
$\tau^{(f +1 )} (\pi)[\mI^{f +1 }]^\vee $
is defined in \eqref{eq:tau1d-def}
applied to $\pi_1 = \pi$
(note that the genericity condition of
\emph{loc.\ cit.}\ is satisfied here),
and $i$ is the $I$-equivariant
injection \eqref{eq:i-tau-iso-def},
inducing an isomorphism
$i \colon  \tau^{(f+1)}(\pi)[\mI^{f+1}] 
\xrightarrow{\sim} \pi[\mI^{f+1}]$
by the line below \eqref{eq:i-tau-iso-def},
hence an isomorphism
$i \colon  i^{-1} (\pi_1 [\mI^{f +1 }])
\xrightarrow{\sim} \pi_1 [\mI^{f+1}]$.
Note that $i^\vee $ sends
$e_{\lambda,s } \in \grm (\pi_1  [\mI^{f +1 }]^\vee )$
to
\begin{align}
    \label{eq:e-lambda-s-computation}
i^\vee (e_{\lambda,s})
&\overset{\eqref{eq:explicit-dual-basis}}{= }
i^\vee (\ip^\vee (u_s \otimes e_\lambda))
= (\ip \circ i)^\vee (u_s \otimes e_\lambda)
\overset{\eqref{eq:i-ip-compatibility}}{= }
(\id_{\V} \otimes \alphaI)^\vee (u_s \otimes e_\lambda)
\\ 
\nonumber
&= 
\id_{\V^\vee } (u_s) \otimes \alphaI^\vee (e_\lambda)
= u_s \otimes \alphaI^\vee (e_\lambda)
\in 
\grm ( i^{-1} (\pi_1  [\mI^{f +1 }])^\vee )_0.
\end{align}
By \eqref{eq:e-lambda-s-computation},
and using that
$i^\vee \colon \grm (\pi_1  [\mI^{f +1 }]^\vee
\xrightarrow{\sim}   
\grm ( i^{-1} (\pi_1  [\mI^{f +1 }])^\vee )$
is an $\Rbar$-module isomorphism,
it is enough to show that every 
$t \in I(J_1 (\lambda), J_2 (\lambda), \ell +1)$
as in the second paragraph of Step 1 annihilates
$u_s \otimes \alpha^\vee (e_\lambda) \in 
\grm ( i^{-1} (\pi_1  [\mI^{f +1 }])^\vee )_0$.

Let us fix some notation:
for $\mu \in \mathscr{P}$, we let
$T_\mu \defeq \tau_\mu^{(f +1 )} [\mI^{f +1 }]$,
and we let $M_\mu$ be the unique 
$I$-subrepresentation of 
$\V \otimes_{\F} T_\mu$
such that 
$i^{-1} (\pi_1 [\mI^{f +1 }])
= \bigoplus_{\mu \in \mathscr{P}} M_\mu$,
by \cite[Lemma~3.1.5(ii)]{Ber}
applied to $S = \F[I],
D = \bigoplus_{\mu \in \mathscr{P}} T_\mu, 
W = \V, M = i^{-1} (\pi_1 [\mI^{f +1 }])$.
Then, the rightmost vertical surjection of
\eqref{eq:i-funny-diagram}
can be rewritten as
\[
\bigoplus_{\mu \in \mathscr{P}} \left( 
\V^\vee \otimes_{\F} \grm(T_\mu^\vee) \right)
\cong 
\grm (\tau^{(f +1 )} (\pi)[\mI^{f +1 }]^\vee ) 
\twoheadrightarrow 
\grm ( i^{-1} (\pi_1  [\mI^{f +1 }])^\vee )
\cong 
\bigoplus_{\mu \in \mathscr{P}} 
\grm(M_\mu^\vee).
\]
If we consider
$u_s \otimes \alpha^\vee (e_\lambda) \in 
\grm (\tau^{(f+1)}(\pi)[\mI^{f+1}]^\vee )$,
then it lies in
$\F u_s \otimes_{\F} \grm (T_\lambda^{\vee})$
(for example, because 
$u_s \otimes \alpha^\vee (e_\lambda)$
is homogeneous of degree $0$, 
and an $H$-eigenvector of eigencharacter 
$\chi_\lambda^{-1}$);
the image of
$u_s \otimes \alpha^\vee (e_\lambda)$ in 
$\grm ( i^{-1} (\pi_1  [\mI^{f +1 }])^\vee )$
then lies in
$\grm (M_\lambda^\vee )$.
In particular,
if we denote by $\psi$ the composition
$\psi \colon \F u_s \otimes_{\F} T_\lambda^\vee 
\subseteq \V^\vee \otimes_{\F} T_\lambda^\vee 
\twoheadrightarrow M_\lambda^\vee $,
then $u_s \otimes \alpha^\vee (e_\lambda)$ in
$\grm ( i^{-1} (\pi_1  [\mI^{f +1 }])^\vee )$
lies in the image of 
$\grm (\psi) \colon 
\F u_s \otimes_{\F} \grm (T_\lambda^\vee)
\to \grm (M_\lambda^\vee)$,
hence the same is true of 
$t (u_s \otimes \alpha^\vee (e_\lambda))$.
Recall that 
$t (u_s \otimes \alpha^\vee (e_\lambda))$
is a (possibly zero) 
$H$-eigenvector of $H$-eigencharacter
$\chi_{\lambda'}^{-1}$
(cf.\ \eqref{eq:lambda-prime-def}
for the definition of 
$\lambda' \in \mathscr{P} ^{\mathrm{ss}}$).
To show that 
$t (u_s \otimes \alpha^\vee (e_\lambda))$ vanishes,
it suffices to show that the $H$-character
$\chi_{\lambda'}^{-1}$ does not occur in
$\im (\grm (\psi))$,
or equivalently in $\im (\psi)$.

Set $Q \defeq \im \psi$. It is easier to work with duals,
i.e.\ to show that 
$\chi_{\lambda'}$ does not occur in $Q^\vee $.
Taking duals of the following commutative diagram
of $I$-representations
\[
\begin{tikzcd}
\F u_s \otimes_{\F} T_\lambda^\vee  
\ar[d, two heads]\ar[r, phantom, " \subseteq "]
\ar[rd, "\psi"] & 
V^\vee \otimes_{\F} T_\lambda^\vee \ar[d, two heads] \\
Q \ar[r, hook] & M_\lambda^\vee,
\end{tikzcd}
\]
we see that we can describe
$Q^\vee \hookrightarrow \F w_s \otimes_{\F} T_\lambda$
as the image of the composition 
\[
M_\lambda \hookrightarrow \V \otimes_{\F} T_\lambda
\twoheadrightarrow \F w_s \otimes_{\F} T_\lambda,
\]
where the last surjection is given by the projection
$\V \otimes_{\F} T_\lambda
= \bigoplus_{s' = 1} ^r \F w_{s'} \otimes_{\F} T_\lambda
\twoheadrightarrow \F w_s \otimes_{\F} T_\lambda$
onto the $s$-th coordinate.
By \Cref{prop:char-only-occurs} 
applied to $\lambda = \lambda'$ (and $d = \ell + 1$),
we deduce that $\chi_{\lambda'}$ can only occur in
the $I$-subrepresentation 
$\left( \VFil[\pi_1 ,\ell +1] \otimes_{\F} T_\lambda \right)
\cap M_\lambda$ of $M_\lambda$.
Therefore, to show that $\chi_{\lambda'}$ does not occur in
$Q^\vee $ it is enough to show that the composition
\[
\left( \VFil[\pi_1 ,\ell +1] \otimes_{\F} T_\lambda \right)
\cap M_\lambda
\subseteq 
M_\lambda \hookrightarrow \V \otimes_{\F} T_\lambda
\twoheadrightarrow \F w_s \otimes_{\F} T_\lambda
\]
vanishes, and it is a fortiori enough 
to show that the composition
$\VFil[\pi_1 ,\ell +1] \otimes_{\F} T_\lambda 
\hookrightarrow \V \otimes_{\F} T_\lambda
\twoheadrightarrow \F w_s \otimes_{\F} T_\lambda$
vanishes,
which in turn follows from the fact that
$\VFil[\pi_1 ,\ell +1] \cap \F w_s = 0$.
Indeed, recall that $s \in 
\{\rFil[\pi_1 , \ell +1 ] + 1 , \dots, \rFil[\pi_1 , \ell]\}$,
hence
$w_s \in W_1 (\ell)$ by our choice of $\mathcal{B}$ 
(cf.\ the paragraph of \eqref{eq:explicit-dual-basis}),
and recall that
$\VFil[\pi_1 ,\ell +1] \cap W_1(\ell) = 0$
by definition of $W_1 (\ell)$
as a complementary subspace
of $\VFil[\pi_1,\ell+1]$ 
inside $\VFil[\pi_1,\ell]$.

As a consequence, we can construct
the graded $\Rbar$-module morphism 
$\theta_1 $ 
of \eqref{diag:thetas},
and the graded $\Rbar$-module morphism 
$\theta_2$ is constructed
by diagram chasing.

\paragraph{Step 2.}
In the notation of \Cref{lemma:fraka-dual},
we set
\begin{equation}
  \label{eq:def-N-prime-2}
N_2' \defeq \bigoplus_{\ell = 0} ^f W_1(f-1-\ell) 
\otimes_{\F} \left( 
 \bigoplus_{\lambda \in \mathscr{P}}
 \chi _{\lambda }^{-1}  \otimes 
 (R/\atwo[\ell] (\lambda))(-4f + (f-d_\lambda))\right),
\end{equation}
and claim that there exists 
an isomorphism 
of $\Rbar$-modules with compatible $H$-action
$\Eg(N_2) \cong N_2' \otimes 
\eta$.

Notice first that by \eqref{eq:a-fil-def}
and the line after \eqref{eq:IJJ-def}
we have $\atwo[-1](\lambda) = \Rbar$ 
for all $\lambda \in \mathscr{P}$,
in particular $\Nfil[-1] = 0$.
Therefore, we can and will extend the direct sum
in \eqref{eq:def-N-prime-2} to $\ell = -1$,
and likewise we can write
$N_1 = \bigoplus_{\ell=-1}^f
\left( W_1(\ell)^\vee \otimes_{\F} 
\Nfil[\ell]\right)$.
Using \eqref{eq:N-unwinding},
we compute
\begin{align}
  \label{eq:N2-computation-1}
N_2 &\cong  \ker \left( \bigoplus_{\ell = -1} ^f 
\left( W_1 (\ell)^{\vee} \otimes_{\F} \Nnull \right)
 \twoheadrightarrow 
\bigoplus_{\ell=-1}^f
\left( W_1(\ell)^\vee \otimes_{\F} 
\Nfil[\ell]\right)
 \right)\\
  \label{eq:N2-computation-2}
&
\overset{\eqref{eq:Nfil-def}}{\cong}  \bigoplus_{\ell = -1} ^f
\left(  W_1 (\ell)^\vee \otimes_{\F}  
\bigoplus_{\lambda \in \mathscr{P}} \left(
\chi_{\lambda}^{-1} \otimes  \frac{\atwo[\ell](\lambda)}{\mathfrak{a}(\lambda)} \right) \right).
\end{align}
We have the following chain of graded $\Rbar$-module
isomorphisms
with compatible $H$-action:
\begin{align}
  \label{eq:N-duality-computation}
\Eg(N_2) &\overset{\eqref{eq:N2-computation-2}}{\cong} 
\bigoplus_{\ell = -1} ^f \bigoplus_{\lambda \in \mathscr{P}}
W_1 (\ell) \otimes_{\F}
\Eg\left(
\chi_{\lambda}^{-1} \otimes \frac{\atwo[\ell](\lambda)}{\mathfrak{a}(\lambda)} 
\right) 
\\
  \nonumber
& \overset{\eqref{eq:fraka-dual}}{\cong} 
\bigoplus_{\ell = -1} ^f \bigoplus_{\lambda \in \mathscr{P}}
W_1 (\ell) \otimes_{\F}
  \left(   \chi_{\lambda^\ast}^{-1} \otimes
 R/\atwo[f-1-\ell](\lambda^\ast) \right)
 (-4f + (f - d_{\lambda ^\ast})) \otimes \eta \\
  \label{eq:extended-range-substitution}
& =
\bigoplus_{\ell' = -1} ^f \bigoplus_{\lambda' \in \mathscr{P}}
W_1 (f-1- \ell') \otimes_{\F}
 \left(  
\chi_{\lambda'}^{-1} \otimes 
R/\atwo[\ell'](\lambda')  \right)
 (-4f + (f - d_{\lambda'}))\otimes \eta 
\\
  \nonumber
& \overset{\eqref{eq:def-N-prime-2}}{\cong} 
N_2 ' \otimes \eta ,
\end{align}
where in \eqref{eq:extended-range-substitution} we have made the substitutions
$\ell' \defeq f-1 - \ell$,
$\lambda' \defeq \lambda^*$,
and used that $d _{\lambda } = d _{\lambda ^\ast}$
by \cite[Lemma~3.63(i)]{BHHMS2}.

If we set
\begin{equation}
  \label{eq:def-N-prime}
N' \defeq \V
\otimes_{\F} \left( 
 \bigoplus_{\lambda \in \mathscr{P}}
 \chi _{\lambda }^{-1}  \otimes 
 (R/\mathfrak{a} (\lambda))(-4f + (f-d_\lambda))\right),
\end{equation}
then by \eqref{eq:N-duality-computation}
applied to $\pi_1 = 0$ and $\pi_2 = \pi$ 
(which implies $N_2 = N$)
we have a graded $\Rbar$-module isomorphism
with compatible $H$-action
 $ \Eg (N) \cong N' \otimes \eta$,
and a graded $\Rbar$-linear surjection
with compatible $H$-action
\begin{equation}
\label{eq:N-prime-surjection}
N' \cong \Eg(N) \otimes \eta ^{-1}
\twoheadrightarrow 
\Eg(N_2 ) \otimes \eta ^{-1}
 \overset{\eqref{eq:N-duality-computation}}{\cong }N_2 '
\end{equation}
coming from the inclusion
$N_2 \overset{\eqref{eq:N-and-N1-def}}{ = }
\ker (N \twoheadrightarrow N_1 ) \subseteq N$.
Notice that $N'$ is abstractly isomorphic to $N$,
as $\Rbar$-modules (without grading)
with compatible $H$-action.

\paragraph{Step 3.}
Let $\widetilde{\pi}_1 \subseteq \pi$
be the subrepresentation of $\pi$
defined in \eqref{eq:conjugate-def}.
For all $\ell \in \{-1, \dots, f \}$,
pick a complementary subspace
$\widetilde{W}_1(\ell)$ of $\VFil[\widetilde{\pi}_1,\ell+1]$ 
inside $\VFil[\widetilde{\pi}_1,\ell]$,
and set
\begin{equation}
  \label{eq:def-N-tilde}
\widetilde{N}_1 \defeq \bigoplus_{\ell = -1} ^f
\widetilde{W}_1(\ell)^\vee  \otimes_{\F} \Nfil[\ell].
\end{equation}
We prove that 
$\widetilde{N}_1$ and $ N_2'$
are abstractly isomorphic
as $\Rbar$-modules 
(without grading)
with compatible $H$-action,
which is equivalent to showing that
 $\dim _{\F} W_1(f-1- \ell ) = 
 \dim _{\F} \widetilde{W}_1(\ell)$
for all $\ell  \in \{ 0, \dots, f \} $.

Recall the $\mathbb{F}_q ^{\times}$-character
$\eta$ defined in the paragraph before 
\Cref{lemma:fraka-dual} 
and notice that, for all 
$a \in \mathbb{F}_q ^{\times}$
we have 
$\label{eq:eta-mytwist-iso}
\eta 
\left( 
a
\right)
= 
(\mytwist) 
\left( 
[a]
\right)$,
cf.\ the paragraph before
\eqref{eq:weight-characterisation}.
Therefore, 
whenever we consider twists for the $H$-action,
we will freely interchange $\eta$ with $\mytwist$
in the rest of the proof.
Recall that if $M$ is a finitely generated 
left $\Lambda$-module with a good filtration,
then for all $i \ge 0$
the right $\Lambda$-module 
$\E[i](M)$ carries a canonical and functorial 
good filtration, cf.\ \cite[\S A]{BHHMS4}.
In particular, we endow
$\E \!( \pi _{1} ^\vee  )$
and $\E \!( \pi ^\vee )$
with this filtration,
we let $F'$ be the unique filtration on 
$\pi ^\vee \otimes \eta$
that makes 
$\E (\pi ^{\vee})
\xrightarrow[\sim]{\kappa _\pi ^{\vee}} 
\pi ^{\vee} \otimes \eta$
a filtered isomorphism,
and we endow the quotient
$\pi ^{\vee} \otimes \eta 
\twoheadrightarrow 
 \widetilde{\pi }_1 ^{\vee} \otimes \eta$
with the induced filtration,
which we still denote by $F'$.
As in Step 2 of the proof of
\cite[Proposition~3.2.2]{BHHMS4}
we can construct the following commutative diagram 
with exact rows
of graded $\grL$-modules with compatible $H$-action:
\begin{equation}
\label{diagr:grE-to-EgN}
\begin{tikzcd}
  0 \ar[r] & \gr(\E \!( \pi _{1} ^\vee ) ) \ar[r] \ar[d, hook]  & 
  \gr(\E \!( \pi ^\vee )) \ar[d, "\wr"] \ar[r]   &
 \gr _{F'} (\widetilde{ \pi }_1 ^\vee \otimes  \eta) 
  \ar[dd, two heads, "\gamma"]  \ar[r] & 0 \\
  0 \ar[r] & \Eg \!( \grm (\pi _{1} ^\vee )) \ar[r] \ar[d, hook]  & 
   \Eg\!( \grm (\pi ^\vee )) \ar[d, equals]  & \\
  0  \ar[r] & \Eg (N_1) \ar[r] & 
  \ar[r]\Eg (N) \ar[r] & 
  \ar[r] \Eg (N_2) \ar[r] & 0 
\end{tikzcd}
\end{equation}
(this construction extends 
\emph{verbatim} to $\rep$ non-semisimple
and to $r \ge 1$).
There is a also a surjection
\begin{equation}
\label{eq:theta-tilde-def}
\widetilde{\theta}_1 \colon 
\widetilde{N}_1 \twoheadrightarrow 
\grm \left( 
\widetilde{\pi}_1^\vee  \right)
\end{equation}
of graded $\Rbar$-modules 
with compatible $H$-action,
by Step 1 applied to 
$\pi_1 = \widetilde{\pi}_1$
(note that here $\widetilde{\pi}_1^\vee$
carries the $\mI$-adic filtration).

We claim that there exists a surjective homomorphism 
of $H $-representations
$\F \otimes _{\grL} \widetilde{N}_1 =
(\widetilde{N}_1)_0 
 \twoheadrightarrow  
\F \otimes _{\grL} {N}'_2$.
Consider the composition
 $ \gr _{F'} (\widetilde{\pi }_1 ^\vee \otimes \eta )
 \xtwoheadrightarrow{\gamma}
 \Eg (N_2 ) \overset{\eqref{eq:N-duality-computation}}{\cong }
 N_2' \otimes \eta \twoheadrightarrow 
 \F \otimes_{\grL} N_2 ' \otimes \eta$,
by \eqref{eq:def-N-prime-2}
the right-hand side is supported
in degrees $ [ 3f ; 4f ] $.
After twisting by $\eta ^{-1}$,
by the semisimplicity of $\F [H]$
we deduce that 
$\F \otimes_{\grL} N_2 '$
is a quotient 
(as an $H$-representation over $\F$)
 of $F' _{4f}(\widetilde{\pi }_1 ^\vee)/
 F' _{3f - 1}(\widetilde{\pi }_1 ^\vee)$.
Moreover,
by
\eqref{eq:N-duality-computation}
and by the graded $\grL$-module isomorphism 
 $\gr(\E (\pi ^\vee)) \cong \Eg (N)$
of \eqref{diagr:grE-to-EgN}
applied to $\pi_1 = \pi$,
we see that 
 $\gr(\E (\pi ^\vee))$
is supported in degrees 
$ \le  4f $,
in particular so are
$\gr _{F'}(\widetilde{\pi}_1  ^\vee))$
and 
 $F' _{4f}(\widetilde{\pi }_1 ^\vee) = 
 \widetilde{\pi }_1 ^\vee$.
Since $\Lambda$ carries the
 $\mI$-adic filtration,
by definition of filtered $\Lambda$-module
we have
 $\mI ^{f +1} \widetilde{\pi }_1 ^\vee = 
\mI ^{f +1} F' _{4f}\widetilde{\pi }_1 ^\vee
\subseteq F' _{3f - 1}(\widetilde{\pi }_1 ^\vee)$,
hence a natural surjection of $\Lambda $-modules
with compatible $H$-action
 $\widetilde{\pi }_1 ^\vee/
 \mI ^{f +1} \widetilde{\pi }_1 ^\vee
 \twoheadrightarrow
 \widetilde{\pi }_1 ^\vee/
 F' _{3f - 1}(\widetilde{\pi }_1 ^\vee)$,
and a fortiori 
$\F \otimes_{\grL} N_2 '$
is a quotient of
 $\widetilde{\pi }_1 ^\vee/
 \mI ^{f +1} \widetilde{\pi }_1 ^\vee$
(or equivalently of
$ \bigoplus_{i = 0}^f 
\grm(\widetilde{\pi }_1 ^\vee) _{- i}$)
as an $H$-representation.
By \eqref{eq:theta-tilde-def},
this implies that
$\F \otimes_{\grL} N_2 '$
is a quotient of
$ \bigoplus_{i = 0}^f 
 (\widetilde{N}_1 ) _{- i}$
as an $H$-representation.
Pick an $H$-equivariant surjection
$ \bigoplus_{i = 0}^f 
(\widetilde{N}_1 ) _{- i}
\twoheadrightarrow \F \otimes_{\grL} N_2 '$
of $\F$-vector spaces.

We now show that the composition
$( \widetilde{N}_1 )_0 
 \subseteq \bigoplus_{i = 0}^f 
(\widetilde{N}_1 ) _{- i}
\twoheadrightarrow \F \otimes_{\grL} N_2 '$
is still surjective.
It is enough to show that 
$ \JH(\F \otimes_{\grL} N_2') \cap 
\JH( (\widetilde{N}_1 ) _{- i})
  = \emptyset$
for $1 \le i \le f$.
By construction, 
there are graded $\grL$-module surjections
with compatible $H$-action
$N \twoheadrightarrow \widetilde{N}_1$
and $\F \otimes_{\grL} N' 
\overset{\eqref{eq:N-prime-surjection}}
{\twoheadrightarrow} \F \otimes_{\grL} N_2'$,
hence it suffices to show that 
$ \JH(\F \otimes_{\grL} N') \cap 
\JH( N  _{- i})
  = \emptyset$
for $1 \le i \le f$.
Since $N'$ is abstractly isomorphic to $N$
as $\Rbar$-modules (without grading)
with compatible $H$-action,
we have an $H$-equivariant isomorphism
 $\F \otimes_{\grL} N' \cong 
\F \otimes_{\grL} N \cong N_0$.
By definition of $N$ (cf.\ \eqref{eq:N-and-N1-def})
it suffices to show that
$ \JH(\Nnull _0) \cap \JH( \Nnull _{- i})
  = \emptyset$,
which holds because
$\bigoplus_{i=0} ^f (\Nnull)_{-i}$
is a multiplicity free $H$-representation
(it follows from \cite[Lemma~2.3.7]{BHHMS4}
with $n = f+1$,
using that $\rep$ is $(2f+1)$-generic).

By the line after \eqref{eq:J2-def} we have
 $\F \otimes_{\grL} \widetilde{N}_1 
 \overset{\eqref{eq:def-N-tilde}}{\cong}
 \bigoplus_{\ell = 0} ^f \widetilde{W}_1(\ell )^\vee
 \otimes_{\F}
 \left( \bigoplus_{\lambda \in \mathscr{P}\mid
 \ell (\lambda ) \le \ell }
 \chi _\lambda ^{-1} \right) $
and 
 $\F \otimes_{\grL} N_2' 
 \overset{ \eqref{eq:def-N-prime-2}}{\cong}
 \bigoplus_{\ell = 0} ^f W_1(f-1- \ell )
 \otimes_{\F}
 \left( \bigoplus_{\lambda \in \mathscr{P}\mid
 \ell (\lambda ) \le \ell }
 \chi _\lambda ^{-1}  \right) $
as $H$-representations over $\F$.
Fix $\ell_0  \in \{ 0, \dots, f \} $,
and fix $\lambda _0 \in \mathscr{P}$
such that $\ell (\lambda_0) = \ell_0 $.
It follows
from the $H$-equivariant surjection
$(\widetilde{N}_1)_0 \twoheadrightarrow 
\F \otimes_{\grL} {N}'_2$
that we have 
$[(\widetilde{N}_1)_0 : \chi _{\lambda_0} ^{-1} ]
\ge [\F \otimes_{\grL}{N}'_2 : \chi _{\lambda_0}^{-1}]$.
We compute both sides to be
\begin{align}
  \label{eq:hype-mult-inequality}
\dim_{\F} 
\VFil[\widetilde{\pi}_1, \ell_0 ]
=
\sum_{\ell = \ell_0}^{f} 
\dim_{\F} 
\widetilde{W}_1  (\ell)
&\ge 
\sum_{\ell = \ell_0}^{f} 
\dim_{\F} 
W_1 (f-1-\ell)
=
\sum_{\ell' = -1}^{f-1-\ell_0} 
\dim_{\F} 
W_1 (\ell')
\\
\nonumber
&= r - \sum_{\ell' = f-\ell_0}^{f} 
\dim_{\F} 
W_1 (\ell')
= r - \dim_{\F} 
\VFil[\pi_1, f - \ell_0 ],
\end{align}
where we have used the substitution
$\ell' = f-1-\ell$ in the second equality,
and where the first and last equalities
follow from
$ \VFil[\widetilde{\pi}_1, \ell_0 ]
=
\bigoplus_{\ell = \ell_0}^{f} 
\widetilde{W}_1  (\ell)$
and
$\VFil[\pi_1 , f- \ell_0 ] 
=\bigoplus_{\ell' = f-\ell_0}^{f} 
W_1 (\ell')$,
using the description of
$\widetilde{W}_1 (\ell)$
and $W_1 (\ell')$
as complementary subspaces.

We show that
\eqref{eq:hype-mult-inequality}
is an equality for all $\ell_0
\in {0, \dots, f}$.
Recall the characteristic cycle $\mathcal{Z}(N)$
of an $R$-module $N$,
defined in \cite[Definition~3.79]{BHHMS2},
and consider
\begin{equation}
  \label{eq:Z-eq-easy}
\mathcal{Z}(\widetilde{\pi}_1^\vee )
= \mathcal{Z}(\pi_2^\vee )
= \mathcal{Z}(\pi^\vee )-
\mathcal{Z}(\pi_1^\vee ),
\end{equation}
where the first equality comes from
\Cref{lemma:Z-equal-Z}
and where the second equality comes from
the additivity of $\mathcal{Z}(-)$
(cf.\ \cite[Lemma~3.82]{BHHMS2}).

For $\ell \in \{0, \dots, f\}$,
recall that 
$\rFil[\widetilde{\pi}_1, \ell] = 
\dim_{\F} 
\VFil[\widetilde{\pi}_1, \ell ]$
and
$\rFil[\pi_1, \ell]= 
\dim_{\F} 
\VFil[\pi_1, \ell ]$
by \eqref{eq:VFil-rFil-def}.
We have
\begin{align*}
\sum_{\ell = 0}^{f} \rFil[\widetilde{\pi}_1, \ell] 
\binom{f}{\ell} 
\overset{ \eqref{eq:Wang-rk-Dvee} }{=}
&
\dimFX \Dvee(\widetilde{\pi}_1)
=
\mpp(\grm(\widetilde{\pi}_1^\vee ))
\overset{\eqref{eq:Z-eq-easy}}{=}
\mpp(\grm(\pi^\vee ))- \mpp(\grm(\pi_1^\vee ))=\\
&= 
\dimFX \Dvee(\pi) -
\dimFX \Dvee(\pi_1)
\overset{ \eqref{eq:Wang-rk-Dvee} }{=}
\sum_{\ell = 0}^{f} (r -\rFil[\pi_1, f- \ell]) 
\binom{f}{\ell},
\end{align*}
where the unlabeled equalities
come from 
\cite[Proposition~3.87(i)]{BHHMS2}
(cf.\ \Cref{rmk:D-eq-m-non-ss}).
This forces equality to always hold in 
\eqref{eq:hype-mult-inequality}.

\paragraph{Step 4.}
We now show (i) and (ii) of the statement
by proving that 
$N_1$ and $N_2 = \ker(N \twoheadrightarrow N_1)$
are either zero
or Cohen-Macaulay of grade $2f$,
and proving that
$\theta_1 \colon N_1 \twoheadrightarrow 
\grm(\pi_1^\vee)$ is an isomorphism
(for (ii) we only need 
that $N_1$ is nonzero
whenever $\pi_1 $ is nonzero,
and that $N_2$ is nonzero
whenever $\pi_2 $ is nonzero).
This implies that
$\pi_1 ^\vee$ 
and $\pi_2 ^\vee$
are either zero 
or Cohen-Macaulay $\Lambda $-modules of grade $2f$
by \cite[Proposition~III.2.2.4]{LvO}.

Set
\begin{align}
  \label{eq:N1-prime-def}
N_1' & \defeq  
\ker \left( 
\V \otimes_{\F} \Nnull 
\twoheadrightarrow 
\bigoplus_{\ell=-1}^f
\left( W_1(f-1 - \ell) \otimes_{\F} 
\Nfil[\ell]\right)
 \right)\\
 \nonumber
&= \ker \left( \bigoplus_{\ell = -1} ^f 
\left( W_1 (f -1 - \ell) \otimes_{\F} \Nnull \right)
 \twoheadrightarrow 
\bigoplus_{\ell=-1}^f
\left( W_1(f-1- \ell) \otimes_{\F} 
\Nfil[\ell]\right)
 \right)\\
 \nonumber
&
\overset{\eqref{eq:Nfil-def}}{\cong}  
\bigoplus_{\ell = -1} ^f
\left(  W_1 (f-1 - \ell) \otimes_{\F}  
\bigoplus_{\lambda \in \mathscr{P}} \left(
\chi_{\lambda}^{-1} \otimes  \frac{\atwo[\ell](\lambda)}{\mathfrak{a}(\lambda)} \right) \right).
\end{align}
A computation entirely analogous to
the one in \eqref{eq:N-duality-computation}
shows that $\Eg (N_1') \cong N_1 \otimes 
\eta$.
We can then apply
\cite[Proposition~III.4.2.8(1)]{LvO}
with $R = \grL$, $M = N_1'$
to conclude that
$N_1$ is \emph{pure} of grade $2f$ over $\grL$ 
in the sense of
\cite[Definition~III.4.2.7]{LvO}.
In particular (taking $\pi_1 = \pi$),
we find that $N$ is pure of grade $2f$ over $\grL$,
and by \cite[Proposition~III.4.2.9]{LvO}
we conclude that
$N_2 \subseteq N$
is also pure of grade $2f$ over $\grL$.

It is proven in Step 3 of the proof of
\cite[Proposition~4.4.3]{BHHMS4}
that the $\grL$-module
$\Nfil[\ell]$ of \eqref{eq:Nfil-def}
(which only depends on $\rep$, 
not on $\pi$ nor $r$)
is either zero (for $ \ell = -1$) 
or Cohen-Macaulay of grade $2f$
(for $ \ell \in \{0, \dots, f\} $).
It follows from
\eqref{eq:N-and-N1-def}
that the same holds for $N$ and $N_1$.
Using the bottom row of \eqref{diag:thetas}
together with $N$, $N_1 $ 
Cohen-Macaulay over grade $2f$ over $\grL$,
one shows that $\Eg[i] (N_2 ) = 0$
for $i > 2f$,
then one concludes
from the fact that $N_2 $ is of grade $2f$ over $\grL$.
Now, we claim that
$\mathcal{Z}(\grm(\pi_1^\vee ))
=\mathcal{Z}(N_1)$.
By the additivity of $\mathcal{Z}(-)$
(cf.\ \cite[Lemma~3.82]{BHHMS2})
and the fact that $\theta$ is an isomorphism,
the claim is equivalent to
$\mathcal{Z}(\grm(\pi_2^\vee ))
=\mathcal{Z}(N_2)$.

We have the chain of inequalities
\begin{align}
\mathcal{Z}(N_2') = 
\mathcal{Z}(N_2) \overset{\theta_2}{\le }
\mathcal{Z}(\gr_F (\pi_2^\vee )) \overset{
\eqref{eq:Z-eq-easy}}{=}
\mathcal{Z}(\grm (\widetilde{\pi}_1^\vee )) 
\overset{\widetilde{\theta}_1}{\le}
\mathcal{Z}( \widetilde{N}_1)
= \mathcal{Z}(N_2'),
\end{align}
where the first equality follows from 
Step 2,
and where the last equality follows from
 \eqref{eq:hype-mult-inequality},
which we showed to be an equality,
together with the definitions
\eqref{eq:def-N-prime-2}
and \eqref{eq:def-N-tilde}.
This forces equality everywhere.
If we set $S \defeq \ker (\theta_1) \subseteq N_1$,
then by the additivity of $\mathcal{Z} (-)$
we have $\mathcal{Z}(S) = 0$.
Suppose for the sake of contradiction 
that $S \neq 0$.
Then,
\cite[Proposition~III.4.2.9]{LvO}
(applied to $M = N_1 $, $N = S$)
implies that $S$ is also pure of grade $2f$ over
$\grL$,
hence of grade $0$ over $\Rbar$
by the second statement in 
\cite[Lemma~3.65]{BHHMS2}.
This implies
$\Ext^0_{\Rbar}(S, \Rbar)=
\Hom_{\Rbar}(S, \Rbar)\neq 0$,
so we can fix 
a nonzero $\Rbar$-linear homomorphism 
$h \colon S \to \Rbar.$

It follows from $\mathcal{Z}(S) = 0$
that for all minimal primes
$\mathfrak{q}$ of $\Rbar$,
we have
$\Hom_{\Rbar}(S, \Rbar)_{\mathfrak{q}} 
\cong 
\Hom_{\Rbar}(S_{\mathfrak{q}} , \Rbar_{\mathfrak{q}})
=0$,
hence $h_{\mathfrak{q}} = 0$
and in particular
$\im(h)_{\mathfrak{q}} = \im(h_{\mathfrak{q}} ) = 0$
by the exactness of localisation.
So, $I_h \defeq \im(h) \subseteq \Rbar$
is an ideal of $\Rbar$
whose localisation at all minimal primes is zero.
But $\Rbar$ is reduced, 
hence the localisation $\Rbar_{\mathfrak{q}} $
is a field (being artinian local and reduced),
which implies that the localisation map
$ \Rbar \to \Rbar_{\mathfrak{q}} $
has kernel $\mathfrak{q}$.
This implies that $I_h \subseteq \mathfrak{q}$
for all $\mathfrak{q} \subseteq \Rbar$
minimal, hence $I_h \subseteq 
\bigcap_{\mathfrak{q}\ \text{minimal}} \mathfrak{q} = \sqrt{0} =0$,
using that $\Rbar$ is reduced.

In conclusion,
$I_h = \im(h) = 0$, so $h = 0$, contradition.
We then must have $S=0$,
i.e.\ $\theta_1 \colon N_1 \twoheadrightarrow \grm ( \pi _1 ^\vee ) $
is an isomorphism.
\end{proof}
\begin{remark}
\label{rmk:gr-F-prime}
It follows from
\Cref{prop:CM-via-N}(ii)
and \cite[Proposition~3.84]{BHHMS2}
that the vertical arrow
$ \gr(\E( \pi _{1} ^\vee ) ) 
 \hookrightarrow 
\Eg ( \grm (\pi _{1} ^\vee )) $
of \eqref{diagr:grE-to-EgN}
is an isomorphism,
and it follows from
\Cref{prop:CM-via-N}(i)
that the vertical arrow
$\Eg ( \grm (\pi _{1} ^\vee ))
 \hookrightarrow \Eg (N_1)$
of \eqref{diagr:grE-to-EgN}
is an isomorphism.
We deduce that the map $\gamma $
in \eqref{diagr:grE-to-EgN}
is an isomorphism
of graded $\grL$-modules with compatible $H$-action.
Keep the notations from the proof of
 \Cref{prop:CM-via-N}.
Composing with
 \eqref{eq:N-duality-computation}
and using \eqref{eq:def-N-prime-2},
we find
\[
\gr _{F'}( \widetilde{\pi }_1 ^\vee) \cong   
\bigoplus_{\ell = 0} ^f W_1(f-1-\ell) 
\otimes_{\F} \left( 
 \bigoplus_{\lambda \in \mathscr{P}}
 \chi _{\lambda }^{-1}  \otimes 
 (R/\atwo[\ell] (\lambda))(-4f + (f-d_\lambda))\right)
\]
as graded $\grL$-modules with compatible $H$-action.
Comparing with
$\grm ( \widetilde{\pi }_1 ^\vee) \cong   
\bigoplus_{\ell = 0} ^f \widetilde{W}_1(\ell)^\vee
\otimes_{\F} \left( 
 \bigoplus_{\lambda \in \mathscr{P}}
 \chi _{\lambda }^{-1}  \otimes 
(R/\atwo[\ell] (\lambda))\right)$
(coming from 
\Cref{prop:CM-via-N}(i)
applied to $\pi_1 = \widetilde{\pi }_1$),
and using that
 $\widetilde{W}_1 (\ell )^\vee$
and $W_1 (f - 1 - \ell )$
are $\F$-vector spaces of the same dimension
(cf.\ Step 3 of the proof of
 \Cref{prop:CM-via-N}),
we see that
 $\grm ( \widetilde{\pi }_1 ^\vee)$
and
 $\gr _{F'} ( \widetilde{\pi }_1 ^\vee)$
only differ from certain shifts in the grading,
which sheds some light on
 \cite[Remark~2.1.4]{BHHMS4}.
\end{remark}

We prepare to give a variation of \eqref{diag:thetas}
for subquotients, i.e.\ when the term
$\grm (\pi^\vee )$ in the middle
is replaced with $\grm (\pi_0 ^\vee)$,
for some subrepresentation $\pi_0$
of $\pi$ containing $\pi_1$.
We begin with an elementary linear algebra lemma.
\begin{lemma}
  \label{lemma:double-filtration}
Let $U$ be a finite-dimensional $\F$-vector space.
Let $m \ge 0$ be a nonnegative integer,
let
$U_1(m) \subseteq \dots \subseteq 
U_1 (0) \subseteq U_1 (-1) = U$
and
$U_0(m) \subseteq \dots \subseteq 
U_0 (0) \subseteq U_0 (-1) = U$
be two $(m+2)$-step filtrations of $U$,
and assume that $U_1 (\ell) \subseteq U_0 (\ell)$
for all $\ell \in \{-1, \dots, m\}$.
Then, there exists a family
\begin{equation}
  \label{eq:family-of-bases}
\{
 \mathcal{B}(\ell_0 , \ell_1 )
\mid 
\ell_0 , \ell_1  \in \{-1, \dots, m\},
 \ell_0  \ge \ell_1 \}
\end{equation}
of (possibly empty) 
pairwise disjoint subsets of $U$ such that, 
for all $\ell \in \{-1, \dots, m\}$
and $i \in \{ 0, 1\}$,
the subset 
$\bigsqcup_{
-1 \le \ell_1 \le \ell_0 \le f \ \mid \ 
\ell_i \ge \ell}
\mathcal{B}(\ell_0, \ell_1) $
of $U$
is an $\F$-basis of $U_i (\ell)$.
\end{lemma}
\begin{proof}
Set $U(\ell_0 , \ell_1) \defeq 
U_0(\ell_0) \cap U_1(\ell_1)$
for every 
$ \ell_0, \ell_1  \in \{-1, \dots, m\}$
such that $\ell_1 \le  \ell_0$.
We define $ \mathcal{B}(m, m),
 \mathcal{B}(m, m-1), \dots, \mathcal{B}(m, 0),
 \mathcal{B}(m, -1)$
to be an $\F$-basis of $U_0 (m)$
adapted to the filtration 
 $U_1(m) = U(m,m)
 \subseteq U (m, m-1)
 \subseteq \dots U(m, 0) \subseteq 
 U (m, -1) = U_0(m)$
 of $U_0(m)$,
in the sense that
$\bigsqcup_{  \ell' \ge \ell}
\mathcal{B}(m, \ell') $
is an $\F$-basis of $U_0(m, \ell)$
for all $\ell  \in \{ -1, \dots, m\}$.
For $m' < m$,
we define $ \{\overline{\mathcal{B}}(m', \ell)
\mid -1 \le \ell \le m'\}$
to be an $\F$-basis of $U_0 (m')/U_0(m'+1)$
adapted (in the same sense as above)
to the decreasing filtration 
\[
  \left( \frac{ U(m' ,\ell ) + U_0(m'+1) }{U_0(m' +1)}
  \cong U(m' ,\ell )/U(m' +1, \ell) \mid
  -1 \le \ell \le m'
  \right)
\]
of $U_0 (m')/U_0(m'+1)$.
Finally, we set $\mathcal{B}(m', \ell)$
to be any preimage in $U$ of 
$\overline{\mathcal{B}}(m', \ell)$.
One readily checks that this defines a family
as in \eqref{eq:family-of-bases}
satisfying all the desiderata.
\end{proof}

Let $\pi_1 \subseteq \pi_0$
be two subrepresentations of $\pi$, and let
$\pi_2 \defeq \pi_0 /\pi_1$.
Assume that $\rep$ is 
$6$-generic.
We can apply \Cref{lemma:double-filtration}
to $m = f$, $U_i (\ell) = \VFil[\pi_i,\ell]$
(where $\ell \in \{-1, \dots, f\}$ and $i \in \{0, 1\}$),
to obtain a family 
$\{
\mathcal{B}(\ell_0 , \ell_1 )
\mid 
\ell_0, \ell_1  \in \{-1, \dots, f\},\ 
 \ell_0 \ge \ell_1 \}$
as in \eqref{eq:family-of-bases}
(note that the hypothesis
$U_1 (\ell) \subseteq U_0 (\ell)$ is satisfied
by \Cref{cor:VFil-monotone}).
In particular, for $i \in \{0, 1\}$ 
and $\ell \in \{-1, \dots, f\}$,
we see that
\begin{equation}
  \label{eq:Wi-def}
W_i (\ell) \defeq \bigoplus_{
\substack{
-1 \le \ell_1 \le \ell_0 \le f \\ \mid \ 
\ell_i \ge \ell}} 
\bigoplus_{w \in \mathcal{B}(\ell_0, \ell_1)} 
\F w
\end{equation}
can be chosen to be the complementary subspace 
of $\VFil[\pi_i, \ell + 1]$ inside
$\VFil[\pi_i, \ell]$
that appears in \Cref{prop:CM-via-N}
applied to $\pi_1 = \pi_i$.
The following corollary to \Cref{prop:CM-via-N}
is in the spirit of \cite[Lemma~3.2.1(i)]{Ber},
and will be used in the proof of 
\Cref{thm:gen-by-inv} below.
\begin{corollary}
  \label{cor:CM-via-N-subquotients}
Keep the above notation,
and assume moreover that $\rep$ is 
$\max \{9,2f + 1\}$-generic.
For $i \in \{0, \dots, 1\}$, set 
\begin{equation}
  \label{eq:Ni-def}
N_i \defeq \bigoplus_{\ell=-1}^f
\left( W_i(\ell)^\vee \otimes_{\F} 
\Nfil[\ell]\right).
\end{equation}
Then, there is a surjection
$h \colon N_0 \twoheadrightarrow N_1 $ 
of $R$-modules with compatible $H$-action,
fitting in a commutative square 
\begin{equation}
  \label{diag:thetas-subquot}
\begin{tikzcd}
\grm ( \pi_0 ^\vee ) \ar[r, two heads] &
\grm ( \pi_1 ^\vee ) \\
N_0  \ar[u, " \theta_0", "\wr"'] \ar[r, "h"', two heads]  & 
N_1  \ar[u, " \theta _1"', "\wr"]
\end{tikzcd}
\end{equation}
of graded
$R$-modules with compatible $H$-action
such that the vertical maps are isomorphisms.
\end{corollary}
\begin{proof}
For $\ell_0, \ell_1  \in \{-1, \dots, f\}$
such that
 $\ell_0 \ge \ell_1$,
let $\mathcal{B}^\vee (\ell_0, \ell_1)$
be the dual basis to 
$\mathcal{B}(\ell_0, \ell_1)$.
We begin by constructing the surjection 
$h \colon  N_0 \twoheadrightarrow N_1$
as
\begin{align*}
h \colon N_0 &\overset{\eqref{eq:Ni-def}}{= }
\bigoplus_{\ell=-1}^f
\big( W_0(\ell)^\vee \otimes_{\F} 
\Nfil[\ell]\big)
\overset{\eqref{eq:Wi-def}}{= }
\bigoplus_{\ell=-1}^f
\big(
\big( \bigoplus_{
\substack{
-1 \le \ell_1 \le \ell_0 \le f \\ \mid \ 
\ell_0 \ge \ell}} 
\bigoplus_{u \in \mathcal{B}^\vee(\ell_0, \ell_1)} 
 \F u \big)
\otimes_{\F} 
\Nfil[\ell]\big)
\\
& \twoheadrightarrow 
\bigoplus_{\ell=-1}^f
\big( 
\big( \bigoplus_{
\substack{
-1 \le \ell_1 \le \ell_0 \le f \\ \mid \ 
\ell_1 \ge \ell}} 
\bigoplus_{u \in \mathcal{B}^\vee(\ell_0, \ell_1)} 
 \F u \big)
\otimes_{\F} 
\Nfil[\ell]\big) 
\overset{\eqref{eq:Wi-def}}{= }
\bigoplus_{\ell=-1}^f
\big( W_1(\ell)^\vee \otimes_{\F} 
\Nfil[\ell]\big)
\overset{\eqref{eq:Ni-def}}{= }
N_1,
\end{align*}
where the surjection comes from 
the inclusion of sets
$\{(\ell_0 , \ell_1 )\mid 
\ell_0, \ell_1  \in \{-1, \dots, m\},
 \ell_0 \ge  \ell_1, \ell_1 \ge \ell \}
  \subseteq 
\{(\ell_0 , \ell_1 )\mid 
\ell_0, \ell_1  \in \{-1, \dots, m\},
 \ell_0 \ge  \ell_1, \ell_0 \ge \ell \}$.
For $i \in \{0, 1\}$,
by \Cref{prop:CM-via-N}(i)
applied to $\pi_1 = \pi_i$
we have an $\Rbar$-module isomorphism
$N_i \xrightarrow[\sim]{\theta_i} \grm (\pi_i^\vee )$ 
with compatible $H$-action.
By the paragraph before Step 1 of the proof of
\Cref{prop:CM-via-N}(i)
we see that,
for all 
 $i  \in \{ 0, \dots, 1 \} $,
 all
 $\ell   \in \{ -1, \dots, f \} $,
and all
$u \in \bigsqcup_{
-1 \le \ell_1 \le \ell_0 \le f \ \mid \ 
\ell_i \ge \ell}
\mathcal{B}^\vee(\ell_0, \ell_1)$,
$\theta_i$ sends 
$u \otimes 1 \in  \F u \otimes_{\F} 
\Rfil[
\ell](\lambda)
\subseteq N_i$
to $\ip^\vee (u \otimes e_\lambda) \in 
\big( \pi_i^{I_1 } \big)^\vee
= \grm \big( \pi_i^\vee  \big)_0 
\subseteq \grm \big( \pi_i^\vee  \big)$,
for all $\lambda \in \mathscr{P}$.
Then, the commutativity of \eqref{diag:thetas-subquot}
follows from a simple diagram chase.
\end{proof}

We conclude the section
by studying the construction
$\pi_1 \mapsto \widetilde{\pi}_1$
of \Cref{def:conjugate-srep}.
The following corollary
will be used in the proof of
\Cref{thm:PS}.
\begin{corollary}
  \label{cor:tilde-involution}
Assume that $\rep$ is 
$\max \{9,2f + 1\}$-generic.
Let $\pi_0, \pi_1$ be two subrepresentations of $\pi$,
and assume that $\pi_1 \subseteq \pi_0$.
Set $\pi_2 \defeq \pi_0 /\pi_1 $.
\begin{enumerate}[(i)]
\item 
If $\pi_1 = 0$, 
then $\widetilde{\pi}_1 = \pi$.
If $\pi_1 = \pi$, then $\widetilde{\pi}_1 = 0$.
\item 
We have 
$\widetilde{\pi}_0 \subseteq \widetilde{\pi}_1$.
Moreover, 
if we set
$\widetilde{\pi}_2  \defeq
\widetilde{\pi}_1 / \widetilde{\pi}_0 $,
then there is an isomorphism
$\E (\pi_2 ^{\vee })\cong 
\widetilde{\pi}_2 ^{\vee }\otimes (\mytwist)$
of $\Lambda$-modules with compatible
actions of $\GF$.
\item 
There exists a $\GF$-equivariant
isomorphism 
$\gamma \colon \pi \xrightarrow{\sim} \pi$,
only depending on $\pi$
and $\kappa_\pi$,
such that we have 
$\gamma(\widetilde{\widetilde{\pi}}_1) = \pi_1$.
\end{enumerate}
In particular, 
$\pi \mapsto \widetilde{\pi}_1$
defines an order-reversing
automorphism of the lattice
$\{\pi_1 \mid \pi_1 \subseteq \pi\}$,
ordered by inclusion.
\end{corollary}
\begin{proof}
(i)
Follows immediately from
\eqref{eq:conjugate-def}.

(ii)
The proof of
\cite[Corollary~4.3.12(vi)]{Ber}
goes through \emph{verbatim}
after replacing $\pi_i$ with $\pi_{i-1}$,
for $i \in \{1, 2\}$,
and $\pi'$ with $\pi_2$.

(iii)
If $\pi_1 = 0$ or $\pi$,
this follows from (i).
Assume now that 
$0 \subsetneq \pi_1 \subsetneq \pi$
is a proper subrepresentation of $\pi$,
and set $\pi_0 = \pi$.
The assumption of 
\Cref{rem:conjugate-when-CM}
is satisfied because $\pi_1^\vee $
is Cohen-Macaulay of grade $2f$ over $\Lambda$
by \Cref{prop:CM-via-N}(ii),
and the diagram \eqref{eq:conjugate-when-CM}
extends to the following
commutative diagram of $\Lambda$-modules
with exact rows
and vertical isomorphisms
(we suppress the twists
as we will not consider the $H$-action):
\begin{equation}
  \label{eq:to-dualise-for-inv}
\begin{tikzcd}
0 \ar[r] & \E(\pi_1 ^{\vee})\ar[r]
\ar[d, "\wr"]
&  \E(\pi^\vee)  
\ar[r]\ar[d,"{ \kappa_\pi^\vee }"', "{\wr}"] 
& 
\E(\pi_2 ^{\vee}) 
\ar[d,"{ \kappa_\pi^\vee}"', "{\wr}"] \ar[r]
& 0 \\
0 \ar[r]
& 
\widetilde{\pi}_2^\vee 
\ar[r]
& \pi^\vee   \ar[r]
&\widetilde{\pi}_1 ^{\vee}
\ar[r]
&
0.
\end{tikzcd}\end{equation}
Apply the functor $\Hom_{\Lambda}(-, \Lambda)$
to \eqref{eq:to-dualise-for-inv}
and use that $\pi^\vee, \pi_1^\vee, \pi_2^\vee$
are Cohen-Macaulay of grade $2f$ over $\Lambda$
by \Cref{prop:CM-via-N}(ii)
to obtain the following
commutative diagram of $\Lambda$-modules
with exact rows
and vertical isomorphisms:
\begin{equation}
  \label{eq:dualised-for-inv}
\begin{tikzcd}
0 \ar[r] & \E \!(\E(\pi_2 ^{\vee}))
\ar[r]
&  \E \!(\E(\pi^\vee))
\ar[r] 
& 
\E \!(\E(\pi_1 ^{\vee}))
\ar[r]
& 0 \\
0 \ar[r]
& 
\E (\widetilde{\pi}_1^\vee) 
\ar[u,"{\E( \kappa_\pi^\vee})", "{\wr}"']
\ar[r]
& \E(\pi^\vee) 
\ar[r]
\ar[u,"{\E( \kappa_\pi^\vee})", "{\wr}"']
&
\E(\widetilde{\pi}_2 ^{\vee})
\ar[u, "{\wr}"']
\ar[r]
&
0.
\end{tikzcd}\end{equation}
For any finitely generated pseudocompact
(in the sense of \cite{Koh17})
$\Lambda$-module $M$,
recall that the natural map
of complexes of $\Lambda$-modules
\begin{equation}
  \label{eq:natural-reflexive}
  M \to R \Hom_{\Lambda}
  (R \Hom_{\Lambda}(M, \Lambda), \Lambda)
\end{equation}
is a quasi-isomorphism,
by the second paragraph of the proof of
\cite[Theorem~3.5]{Koh17}.
Since $\pi^\vee $, $\pi_1^\vee $, $\pi_2^\vee $,
$\widetilde{\pi}_1^\vee $ and $\widetilde{\pi}_2^\vee $
are Cohen-Macaulay $\Lambda$-modules of grade $2f$
by \Cref{prop:CM-via-N}(ii),
taking the $0$-th cohomology group of
\eqref{eq:natural-reflexive}
(and using \eqref{eq:to-dualise-for-inv})
we deduce an isomorphism 
$M \xrightarrow{\sim} \E( \E (M))$
of $\Lambda$-modules
for $M \in \{\pi^\vee , \pi_1^\vee , \pi_2^\vee \}$.
Consider the following commutative diagram
of $\Lambda$-modules with exact rows
and vertical isomorphisms:
\begin{equation}
  \label{eq:glueing-squares-dual}
\begin{tikzcd}
0 \ar[r]
& 
\pi_2^\vee 
\ar[d, "\wr", " {\eqref{eq:natural-reflexive}}"']
\ar[r]
& \pi^\vee   \ar[r]
\ar[d, "\wr", " {\eqref{eq:natural-reflexive}}"']
&\pi_1 ^{\vee}
\ar[r]
\ar[d, "\wr", " {\eqref{eq:natural-reflexive}}"']
&
0
\\
0 \ar[r] & \E \!(\E(\pi_2 ^{\vee}))
\ar[r]
&  \E \!(\E(\pi^\vee))
\ar[r] 
& 
\E \!(\E(\pi_1 ^{\vee}))
\ar[r]
& 0 \\
0 \ar[r]
& 
\E (\widetilde{\pi}_1^\vee) 
\ar[u,"{\E( \kappa_\pi^\vee})", "{\wr}"']
\ar[d, "{\wr}"] 
\ar[r]
& \E(\pi^\vee) 
\ar[r]
\ar[u,"{\E( \kappa_\pi^\vee})", "{\wr}"']
\ar[d,"{ \kappa_\pi^\vee}"', "{\wr}"] 
&
\E(\widetilde{\pi}_2 ^{\vee})
\ar[u, "{\wr}"']
\ar[r]
\ar[d,"{ \kappa_\pi^\vee}"', "{\wr}"] 
&
0 \\
0 \ar[r] & \widetilde{\widetilde{\pi}}_2 ^{\vee}
\ar[r]
&  \pi^\vee
\ar[r]
& 
\widetilde{\widetilde{\pi}}_1 ^{\vee}
\ar[r]
& 0,
\end{tikzcd}\end{equation}
where the upper squares commute
by the naturality of 
\eqref{eq:natural-reflexive},
and where the lower 
squares commute by
\eqref{eq:to-dualise-for-inv}
applied to $\pi_1 = \widetilde{\pi}_1$ and
$\pi_2 = \widetilde{\pi}_2$.
By \eqref{eq:glueing-squares-dual}
we induce a commutative diagram of 
admissible $\GF$-representations
over $\F$
with exact rows and vertical isomorphisms
\begin{equation}
  \label{eq:dual-of-glueing-squares-dual}
\begin{tikzcd}
0 \ar[r]
& 
\pi_1
\ar[d, "\wr"]
\ar[r]
& \pi  \ar[r]
\ar[d, "\wr", "\gamma"']
&\pi_2
\ar[r]
\ar[d, "\wr"]
&
0
\\
0 \ar[r] & \widetilde{\widetilde{\pi}}_1
\ar[r]
&  \pi
\ar[r]
& 
\widetilde{\widetilde{\pi}}_2
\ar[r]
& 0,
\end{tikzcd}\end{equation}
where $\gamma \colon \pi \xrightarrow{\sim} \pi$
is the unique $\GF$-linear isomorphism
such that $\gamma^\vee = 
\kappa_\pi^\vee \circ 
\E (\kappa_\pi^\vee )^{-1} 
\circ \eqref{eq:natural-reflexive}$
is the composition
of the middle vertical arrows in 
\eqref{eq:dualised-for-inv}.
This concludes the proof.
\end{proof}
\begin{remark}
  \label{rmk:tilde-involution}
Recall the $\F$-linear subspaces of $\V$
defined in
\eqref{eq:VFil-rFil-def}.
We suspect that,
using $\kappa_\pi$,
one should be able to construct
a nondegenerate pairing
$b_\pi \colon \V \otimes_{\F} \V \to \F$,
and moreover that
$\VFil[\widetilde{\pi}_1 ,\ell]
= \VFil[\pi_1 , f - \ell]^{\perp}$
for $\ell \in \{0, \dots, f\}$,
where for an $\F$-linear subspace $W$ of $\V$
we set $W^{\perp} \defeq 
\{v \in \V \mid b_\pi (W, v) = 0 \}$.

Notice that this would agree with
Step 4 of the proof of
\Cref{prop:CM-via-N}, which shows that
\eqref{eq:hype-mult-inequality} is an equality.
This would also agree with
\Cref{cor:tilde-involution},
using the well-known identities
$0^{\perp} = \V$, $\V^{\perp} = 0$,
$W_1 \subseteq W_0 \implies W_0^{\perp} \subseteq W_1^{\perp},W_1^{\perp} / W_0^{\perp} 
\cong \Hom_{\F}(W_0 / W_1, \F)$,
and
$(W_1 ^\perp)^ \perp = W_1 $,
valid for $\F$-linear subspaces 
$W_0 , W_1 \subseteq \V$.

In any case, we will not need such a result
in the sequel.
\end{remark}

\subsection{Finite length}
  \label{sec:fl}
Assume that $\rep$ is $\max \{9, 2f +1\}$-generic,
and keep the hypotheses
\ref{hypothesis:i} to
\ref{hypothesis:iv}
of \Cref{sec:hypotheses}.
The goal of this section is to prove,
in \Cref{cor:fl} below,
that $\pi$ has finite length.

For a $\GR$-subrepresentation $\pi_1 $ of $\pi$,
recall the vector subspaces
$\VFil[\pi_1 ,\ell]$ of $\V$
defined in \eqref{eq:VFil-rFil-def}.
The following theorem 
establishes that 
the inclusion $\pi_1  \subseteq \pi_0 $
between two subrepresentations $\pi_0 , \pi_1 $
of $\pi$
can be checked at the level of $K_1$-invariants.
It generalises
\cite[Theorem~4.4.8]{BHHMS4}
to the case $r \ge 1$.
\begin{theorem}
  \label{thm:gen-by-inv}
Assume that $\rep$ is $\max \{9, 2f +1\}$-generic.
Let $\pi_0$ and $\pi_1 $
be two $\GF$-subrepresentations of $\pi$.
\begin{enumerate}[(i)]
\item 
Assume that $\pi_1 \subseteq \pi_0 $,
and set $\pi_2  \defeq \pi_0 /\pi_1 $.
Then, $\pi_2 $ is generated by its $K_1$-invariants.

\item 
The following are equivalent:
\begin{enumerate}[(a)]
\item $\pi_1  \subseteq \pi_0 $;
\item $\pi_1 ^{K_1} \subseteq 
\pi_0 ^{K_1} $ inside $\pi$;
\item $ \VFil[\pi_1 ,\ell] \subseteq 
\VFil[\pi_0 ,\ell]$ for all 
$\ell \in \{0, \dots, f\}$.
\end{enumerate}
Moreover, the statement remains true if we substitute
$ \subseteq $ with $ \subsetneq $ in (a), (b), (c).

\item 
Assume that $\pi_1 \subseteq \pi_0 $,
and set $\pi_2  \defeq \pi_0 /\pi_1 $.
If the inclusion $\pi_1  \subsetneq \pi_0 $
is strict, then the \'etale $(\varphi, \Gamma)$-module
$\Dvee(\pi_2 )$ 
of \Cref{def:Dvee} is nonzero.
\end{enumerate}
\end{theorem}
\begin{proof}
(i)
The quotient of any $\GF$-representation
generated by its $K_1$-invariants is also 
generated by its $K_1$-invariants,
hence it suffices to consider the case $\pi_1 = 0$,
i.e.\ $\pi_2  =  \pi_0$.
Let $\pi_0' \defeq \left\langle
\GF \cdot \pi_0^{K_1}  \right\rangle$
be the subrepresentation of $\pi_0$
generated by $\pi_0^{K_1} $,
in particular $\pi_0^{K_1} = \pi_0^{\prime K_1}$,
and $\VFil[\pi_0,\ell] = \VFil[\pi_0',\ell]$
for all $\ell \in \{0, \dots, f\}$
by \Cref{thm:srep-structure}
(applied to $\pi_1 = \pi_0$).
Note that, in the commutative square
\eqref{diag:thetas-subquot}
applied to $\pi_0 = \pi_0 $, $\pi_1 = \pi_0 '$,
the surjection $h$ is an isomorphism,
because 
$\VFil[\pi_0,\ell] = \VFil[\pi_0',\ell]$
for all $\ell \in \{0, \dots, f\}$.
It follows that the natural map
$ \grm(\pi_0^\vee ) \twoheadrightarrow 
\grm(\pi_0^{\prime \vee } )$ is also an isomorphism,
hence $\pi_0^{\prime \vee } /\mK^n=
\pi_0^\vee/\mK^n$ for all $n \in \mathbb{N}$,
for dimension reasons, hence
$\pi_0^{\prime \vee }=\pi_0^\vee$,
and we deduce $\pi_0 = \pi_0'$.

(ii)
We have $(a) \Leftrightarrow (b)$ by (i),
$(c) \Rightarrow (b)$ by \Cref{thm:srep-structure}
and $(b)\Rightarrow (c)$
by \Cref{cor:VFil-monotone}.

For the last statement in (ii) it is enough to observe that,
for every partial order $(X, \le )$ on a set $X$,
and for every two elements $x_1, x_0$ of $X$,
we have $x_1 \lneq x_0$ 
(i.e.\ $x_1 \le x_0$ and $x_1 \neq x_0$)
if and only if 
we have $x_1 \le x_0 $ and 
$x_0 \not \le x_1 $.

(iii)
By \Cref{cor:Wang} and by part (ii),
we must have
$\dimFX \Dvee(\pi_1)< \dimFX \Dvee(\pi_0)$,
hence $\dimFX \Dvee(\pi_2 )\neq 0 $
by the exactness of $\Dvee(-)$,
which implies $\Dvee(\pi_2 ) \neq 0$.
\end{proof}

Finally, we can prove that $\pi$ 
has finite $\GF$-length.
\begin{corollary}
  \label{cor:fl}
Assume that $\rep$ is $\max \{9, 2f +1\}$-generic.
Let $\pi_1 \subseteq \pi_0$
be two $\GF$-subrepresentations of $\pi$.
Then, we have
\begin{equation}
  \label{eq:lg-inequality}
\lg_{\GF}(\pi_0  /\pi_1) 
\le \sum_{\ell = 0}^{f} 
\left( \rFil[\pi_0  , \ell]-
\rFil[\pi_1, \ell] \right),
\end{equation}
where $\rFil[\pi_0, \ell], \rFil[\pi_1, \ell]$
are defined in \eqref{eq:VFil-rFil-def}.
In particular, 
$\pi$ has finite $\GF$-length,
bounded above by $r \cdot (f + 1)$.
\end{corollary}
\begin{proof}
Let $n \defeq \sum_{\ell = 0}^{f} 
\left( \rFil[\pi_0 , \ell]-
\rFil[\pi_1, \ell] \right)$,
and suppose for the sake of a contradiction
that there exists a strictly increasing chain
\[
\pi_1  \subsetneq \pi'_n \subsetneq \pi_{n-1}' \subsetneq \dots \subsetneq 
\pi'_2 \subsetneq \pi'_{1} \subsetneq  \pi_0 
\]
of nonzero $\GF$-subrepresentations.
For $i \in \{1, \dots, n\}$ and
$\ell \in \{0, \dots, f\}$,
recall the vector subspaces $\VFil[\pi_i', \ell],
\VFil[\pi_0, \ell], \VFil[\pi_1, \ell]$
of \eqref{eq:VFil-rFil-def}.
By \Cref{thm:gen-by-inv}(ii),
the chain 
\[
\bigoplus_{ \ell=0} ^f \VFil[\pi_1 ,\ell]
 \subsetneq
\bigoplus_{ \ell=0} ^f \VFil[\pi'_n ,\ell]
\subsetneq
\bigoplus_{ \ell=0} ^f \VFil[\pi'_{n-1} ,\ell]
\subsetneq
\cdots
\subsetneq
\bigoplus_{ \ell=0} ^f \VFil[\pi'_1 ,\ell]
\subsetneq
\bigoplus_{ \ell=0} ^f \VFil[\pi_0 ,\ell]
\]
is strictly increasing, contradicting
\[
\dim_{\F}\left( \frac{\bigoplus_{\ell=0} ^f \VFil[\pi_0 ,\ell]}
{\bigoplus_{\ell=0} ^f \VFil[\pi_1 ,\ell]} \right)
=\sum_{\ell=0} ^f (\rFil[\pi_0, \ell]-\rFil[\pi_1,\ell])
=n.
\]
As for the last statement, taking
$\pi_0 = \pi, \pi_1  = 0$ we find
$\lg_{\GF}(\pi) \le \sum_{\ell = 0}^{f} (r - 0) = 
r \cdot (f+1)$.
\end{proof}

\begin{remark}
  \label{rmk:fl-want-equality}
Even though we expect equality to always hold in
\eqref{eq:lg-inequality},
proving such a result seems currently
out of reach.
\end{remark}

Write 
$\rep \cong 
\left( 
\begin{smallmatrix}
\chi_1  & \ast \\
0 & \chi_2 
\end{smallmatrix}
\right)$
for some $\Gal( \overline{K}/K)$-characters 
$\chi_1 ,\chi_2 $
(and $\ast \neq 0$).
Let $B \subseteq \GL_{2, \mathcal{O}_K}$
denote the algebraic subgroup of
upper-triangular matrices,
and let $\pi_0 , \pi_f$
be the principal series
$\pi_0 \defeq 
\IndBK (\chi_2 \otimes \chi_1 \omega^{-1})$
and
$\pi_f \defeq 
\IndBK (\chi_1 \otimes \chi_2 \omega^{-1})$,
where we have used
local class field theory,
normalised so that uniformisers correspond to geometric Frobeniuses,
to regard
$\chi_2 \otimes \chi_1 \omega^{-1}$,
$\chi_1 \otimes \chi_2 \omega^{-1}$
as $K^\times$-characters.
The following result
identifies all the principal series 
which occur in $\pi$.
\begin{theorem}
  \label{thm:PS}
Keep the above notation,
and assume that $\rep$ is $\max \{9, 2f +1\}$-generic.
Then,
$\pi$ is of the form
\[
\pi \cong 
\left( 
\begin{xy} (0,0)*+{
\V \otimes_{\F} \pi_0 
}="a";
(15,0)*+{\pi'}="b";
(30,0)*+{
\V \otimes_{\F} \pi_f
}="c"; 
{\ar@{-}"a";"b"};
{\ar@{-}"b";"c"}\end{xy}
\right),
\]
where the notation means that
$\socF \pi \cong  \V \otimes_{\F} \pi_0$,
$\cosocF \pi \cong  \V \otimes_{\F} \pi_f$,
and that $\pi' \cong 
\ker(\pi/\V \otimes_{\F} \pi_0 \twoheadrightarrow 
\V \otimes_{\F} \pi_f)$.
Moreover, we have 
$\lgF (\pi') \le  r \cdot (f -1 )$,
every Jordan-H\"older constituent
of $\pi'$ is supersingular,
and $\pi^{\prime \vee}$
is essentially self-dual
(as in hypothesis \ref{hypothesis:iii}
of \Cref{sec:hypotheses}).
\end{theorem}
In the proof we will use
that $\pi$ is of the form 
\eqref{eq:piShi-dev},
i.e.\ our arguments are global.
Without a global input, 
we would not be able to show that the 
principal series occurring inside $\pi$
are precisely $\pi_0 $ and $\pi_f$.
\begin{proof}
Let $\sigma_0\in \Wr$
be the unique Serre weight of length $0$
in $\Wr$,
which by \Cref{def:J-and-ell}
is $\sigma_0 =(r_0, \dots, r_{f-1})
\otimes \det^{e(x_0 , \dots, x_{f-1})
(r_0,\dots,r_{f-1})}$,
and let $\chi_{\sigma_0} $ be the character
of $I$ acting on $\sigma_0^{I_1}$.
If $ \cInd$ denotes the compact induction from 
$\GRZ $ to $\GF $,
we start by studying the action on 
the Hecke algebra
$\End_{\GF}(\cInd \sigma_0 )$
on 
\begin{align}
  \label{eq:Hecke-on-pi}
  \Hom_{\GF}&(\cInd \sigma_0 , \pi)
  \xrightarrow{\sim} \Hom_{\GR}(\sigma_0 , \pi|_{\GR})
  = \Hom_{\GR}(\sigma_0 , \pi^{K_1})
  \\
  \nonumber
  &\xrightarrow{\sim} 
  \Hom_{H}(\sigma_0 ^{I_1}, \pi^{I_1})
  \xrightarrow[\sim]{\ip}
  \V \otimes_{\F}
  \Hom_{H}(\sigma_0 ^{I_1}, \DOne)
  \xrightarrow[\sim]{\id_{\V} \otimes \alphaI^{-1}} 
  \V
\end{align}
where the first isomorphism is given by
Frobenius reciprocity,
the second one follows from the fact that
$\sigma_0$ is generated by its $I_1$-invariants
and that $\sigma_0 ^{I_1 }$ only occurs 
in the $H$-subrepresentation 
$\socR (\pi)^{I_1} \subseteq \pi^{I_1} \cong 
\V \otimes_{\F} \DOne$
of $\pi^{I_1}$
(use that $\DOne$ is multiplicity free),
and where the last one is given by composition
with the $H$-equivariant isomorphism
\eqref{eq:v-iso-def},  using that
$\chi_{\sigma_0 } = \chi_{(x_0 , \dots, x_{f-1})}$.
Following the notation of the paragraph before
\cite[Lemma~3.5]{NewYork},
for $g \in \GF$ and $x \in \sigma_0 $
we let $[g, x] \in \cInd \sigma_0 $
be the unique element with support on $\GRZ g^{-1}$
which sends $g^{-1}$ to $x$.
Recall that, by \cite[Proposition~8]{BL94},
$\End_{\GF}(\cInd \sigma_0 ) \cong \F[T]$
is a polynomial algebra with generator $T$
and that, by \cite[Eqq.~(5),(8)]{NewYork}
and \cite[Lemma~3.10(i),(ii)]{NewYork},
the endomorphism $T$ sends 
$[g, v] \in \cInd \sigma_0$ to
\begin{equation}
  \label{eq:explicit-Hecke-action}
  T([g, v]) = 
  \sum_{a \in \Fres } 
  \left[ g 
  \left( \begin{smallmatrix}
  p & [a] \\ 0 & 1
  \end{smallmatrix} \right),
  \varphi \left( 
  \left( \begin{smallmatrix}
  1 & 0 \\ 0 & p^{-1}
  \end{smallmatrix} \right)
  \right)
  \left( \begin{smallmatrix}
  1 & -[a] \\ 0 & 1
  \end{smallmatrix} \right)
  v \right],
\end{equation}
with the convention $[0] = 0$.
We claim that $T$ 
acts on $\Hom_{\GF}(\cInd \sigma_0 , \pi)$,
(or equivalently on $\V$ via \eqref{eq:Hecke-on-pi})
by multiplication by a scalar.
Equivalently,
we claim that every line $L \subseteq \V$
is stable under the action of $T$.

If we pick $g = \id$,
$v_0 \defeq v_{(x_0 , \dots, x_{f-1})} 
\in \sigma_0^{I_1}$
(cf.\ the line before \eqref{eq:v-iso-def}),
then \eqref{eq:explicit-Hecke-action}
simplifies to
$T([1, v]) = 
\sum_{a \in \Fres } 
\left[ \left( \begin{smallmatrix}
p & [a] \\ 0 & 1
\end{smallmatrix} \right),
v_0  \right]$,
where we have used 
$ \left( \begin{smallmatrix}
1 & -[a] \\ 0 & 1
\end{smallmatrix} \right) \in I_1 $
and \cite[Lemma~3.10(iii)]{NewYork}.
Fix a nonzero $w \in \V$,
let $L \defeq \F w$,
and let
$h_w \colon \cInd \sigma_0 \to \pi$
be the unique $\GF$-equivariant map 
such that
$h_w \colon [\id, v_0 ]\mapsto \ip^{-1}(w \otimes v_0)$,
by \eqref{eq:Hecke-on-pi}.
Writing
$ \left( \begin{smallmatrix}
p & [a] \\ 0 & 1
\end{smallmatrix} \right)
= \left( \begin{smallmatrix}
[a] & 1 \\ 1 & 0
\end{smallmatrix} \right)
\Pi$,
we see that $T\cdot h_w$ sends
$[\id, v_0 ]$ to 
$\sum_{a \in \Fres} \left( \begin{smallmatrix}
[a] & 1 \\ 1 & 0
\end{smallmatrix} \right)
\Pi \ip^{-1}(w \otimes v_0)
\overset{\eqref{eq:iota-pi-diagram-global}}{=} 
\sum_{a \in \Fres} 
\ip^{-1}\left(
w \otimes
\left( \begin{smallmatrix}
[a] & 1 \\ 1 & 0
\end{smallmatrix} \right)
\Pi  v_0 \right)
\in \ip^{-1}(L \otimes \DZero)$,
so the unique $w' \in \V$
such that $T \cdot h_w = h_{w'}$
must belong to $L$.
It follows that the $\F[T]$-module
$\Hom_{\GF}(\cInd \sigma_0 , \pi)$
factors as a $\F[T]/(T- \mu_0)$-module
for a unique $\mu_0 \in \F$.

For each $\F$-vector subspace $W \subseteq \V$,
we let
\begin{equation}
	\label{eq:pi0-line}
	\pi_0(W) \defeq \left\langle \GF \cdot \iota_\pi^{-1}(W \otimes \sigma_0) \right\rangle
\hookrightarrow \pi,
\end{equation}
and we show that $\pi_0 (L)$
is an irreducible principal series
when $L = W \subseteq \V$ is a line.
By the previous paragraph,
the surjective $\GF$-linear morphism 
$\cInd \sigma_0 \twoheadrightarrow \pi_0(L)$
induced by Frobenius reciprocity from
$\sigma_0 \cong L \otimes \sigma_0
\xhookrightarrow{\iota_\pi^{-1}}
  \pi_0(L)$
factors through 
$\cInd \sigma_0/(T- \mu_0)$.
By \cite[Eq.~(57)]{Ber}
(applied to $\rep = \rss$)
we have
$\JH (\IndI \chi_{\sigma_0 }) \cap \Wss 
= \{\sigma_0 \}$,
and since
$\sigma_0 \in 
\JH (\IndI \chi_{\sigma_0 }) \cap \Wr$
we conclude that 
$
  \JH (\IndI \chi_{\sigma_0 }) \cap \Wr
  = \{\sigma_0 \}.
$
So, the hypotheses of \cite[Lemma~5.14]{HW22}
are satisfied, and
by the third paragraph of 
the proof of \emph{loc.\ cit.}\ 
we must have $\mu_0 \in \F^\times$.
By \cite[Theorem~30]{BL94},
$\cInd \sigma_0/(T- \mu_0)$
is irreducible and isomorphic to some
principal series 
$\cInd \sigma_0/(T- \mu_0) \cong 
\IndBK \chi_0,$
where $\chi_0: B(K) \to \F^{\times}$
is a smooth character.
Then,
\begin{equation}
	\label{eq:pi-L}
\Ind_{B(K)} ^{\GF}\chi_0 \cong 
	\cInd \sigma_0/(T- \mu_0) \xrightarrow{\sim}
\pi_0(L)
\end{equation}
is an isomorphism,
since $\pi_0(L)$ is generated by $ \sigma_0$.
Let $[H] \defeq 
\left\{ \left(\begin{smallmatrix}
[a] & 0 \\ 0 & [b]
\end{smallmatrix}\right)
\mid a, b \in \Fres \right\} 
\subseteq \GR$,
and recall the isomorphism
$(\IndBK \chi_0)^{K_1}
 \cong (\Ind _{B(\mathcal{O}_K)} ^{\GR} 
 \chi_0 | _{[H]}) ^{K_1 }
\cong \IndB (\chi_0 |_{[H]})$
of $\Gamma$-representations.
By construction of $\pi_0 (L)$
we see that 
$L \otimes_{\F} \sigma_0 
\subseteq \ip(\socR (\pi_0 (L)))$,
in particular 
$L \otimes_{\F} \DZero \cap \ip(\pi_0 (L)^{K_1 })
\neq 0$,
and it follows from
\Cref{thm:srep-structure}
that
$L \subseteq \VFil[\pi_0(L), 0]$
(in the notation of \eqref{eq:VFil-rFil-def}),
or equivalently that
$L \otimes_{\F}\DZeroFil[0]
\subseteq \ip(\pi_0 (L)^{K_1 })$.
We claim that this inclusion is an equality.

First, we show that
$\DZeroFil[0] = \DssZeroEll[0]$.
We have
$\DZeroFil[0] = \DZeroEll[0]
\subseteq \DssZeroEll[0]$
by the line after 
\eqref{eq:DZero-filtration-definition},
on the other hand,
the proof of \cite[Proposition~5.2]{Hu16}
shows that
$\DZerogr[0]$ is 
the largest subrepresentation of $\Dx[\sigma_0]$
such that $\JH(\DZerogr[0]) \cap \Wss = \{\sigma_0\}$
Unwinding the definition \eqref{eq:diagr-from-I},
this is equivalent to showing that
$\ell(\sigma_0 ,\tau) = 
\ell(\rss, \tau) < \infty$
implies
$\JH(I(\sigma_0 , \tau)) \cap \Wss 
= \{\sigma_0 \}$,
which follows from
\cite[Lemma~12.8]{BP12}
(applied to $\rho = \rss$).

By the $\Gamma$-equivariant isomorphism
$(\IndBK \chi_0)^{K_1}
\cong \IndB (\chi_0 |_{[H]})$
after \eqref{eq:pi-L},
we see that
$\ip(\pi_0 (L)^{K_1 })$
is a $\Gamma$-representation
with $2^f$ Jordan-H\"older constituents
(cf.\ the paragraph
of \eqref{eq:pi-L} 
and the paragraph
of \eqref{eq:P-bijection-1}).
On the other hand, it contains
$L \otimes_{\F}\DZeroFil[0]
  = L \otimes_{\F} \DssZeroEll[0] =
L \otimes_{\F} \Dssx[\sigma_0]$,
and by \cite[Remark~14.9]{BP12}
we know that 
$\Dssx[\sigma_0]
\cong \IndB \chi_{\sigma_0 }^s$,
which is a $\Gamma$-representation
with $2^f$ Jordan-H\"older constituents.
This shows that
$L \otimes_{\F}\DZeroFil[0]
\subseteq \ip(\pi_0 (L)^{K_1 })$
is an equality
(and moreover that
$\chi_{\sigma_0 }^s \cong \chi_0 |_{[H]}$
as $H$-characters).
By
\Cref{thm:srep-structure},
this implies that
\[
  \VFil[\pi_0 (L),\ell] = 
\begin{cases}
L & \text{if }\ell = 0, \\
0 &  \text{otherwise.}
\end{cases}
\]

From the previous paragraph it follows that 
the $\GF$-equivariant map
$i_0 \colon \V \otimes_{\F} \pi_0 
\cong 
\V \otimes_{\F} \cInd \sigma_0 /(T - \mu_0)
\to \pi$
induced from Frobenius reciprocity
restricts to an isomorphism
$i_0 \colon L \otimes_{\F} \pi_0 
\xrightarrow{\sim} \pi_0 (L)$
for every line $L \subseteq \V$,
in particular $i_0 $ is injective
and $i_0 (W \otimes_{\F} \pi_0 ) = \pi_0 (W)$
for every $\F$-vector subspace $W \subseteq \V$.
We prove that
that $ \pi_0 (\V) = \socF \pi$
by showing that 
every irreducible subrepresentation
$\pi'$ of $\pi$
is contained inside $\im (i_0 ) = \pi_0 (\V)$.
Indeed, it follows from
\Cref{thm:gen-by-inv}(ii)
that $\pi'$
contains the nonzero subrepresentation
$\pi_0 (\VFil[\pi',0])$,
hence $\pi' =  \pi_0 (\VFil[\pi',0]) \subseteq 
\pi_0 (\V)$
(and moreover $\VFil[\pi',0]$ is a line).

We now determine the form of $\chi_0$.
Recall the \emph{ordinary part} functor
$\Ord_B$
defined by Emerton in \cite{Eme},
from the category of smooth representations
of $\GF$ over $\F$
to the category of smooth representations
of $T_{2}(K)$ over $\F$.
We proceed as in the proof of 
\cite[Proposition~10.8]{HW22}:
by \cite[Proposition~5.1(i)]{HW22}
and using $\IndBK \chi_0 =
\Ind_{\overline{B}(K)}^{\GF} \chi_0 ^s$,
where $\overline{B} \subseteq \GL_2 $
denotes the algebraic subgroup of
lower-triangular matrices,
it is enough to show that 
$\Ord_B (\pi)$ is isotypical
of $T_2 (K)$-eigencharacter
$\chi_1 \omega^{-1} \otimes \chi_2
= (\chi_2 \otimes \chi_1 \omega^{-1})^s$,
which follows from
\cite[Proposition~7.4]{HW22}
applied to $\mathbb{W} = W(\F)$.

Recall that
$\pi_f \defeq 
\IndBK (\chi_1 \otimes \chi_2 \omega^{-1})$.
By \cite[Lemma~10.7]{HW22},
which is based on
\cite[Proposition~5.4]{Koh17},
for $i \in \{0, f\}$ we have
\begin{equation}
	\label{eq:PS-self-duality}	
	\pi_i^{\vee} \otimes(\mytwist) 
    \cong \E({\pi_{f-i} ^\vee }).
\end{equation}
Arguing as in the proof of 
\cite[Corollary~3.92]{BHHMS2}, 
one computes the $\GR$-socle of $\pi_f$
to be $\sigma_f$.
Let $\widetilde{\pi}_0 (\V) \subseteq \pi$
be the subrepresentation 
of \Cref{eq:conjugate-def}
applied to $\pi_1  = \pi_0 (\V)$.
We show that 
$\pi/ \widetilde{\pi}_0 (\V) = \cosocF \pi$,
or equivalently 
that for every subrepresentation 
$\pi' \subseteq \pi $
such that $\pi'' \defeq \pi/\pi'$ is irreducible
we have $\widetilde{\pi}_0 (\V) \subseteq \pi'$.
By \Cref{cor:tilde-involution}(ii),(iii)
this is equivalent to showing that
$\widetilde{\pi}' \subseteq 
\widetilde{\widetilde{\pi}_0 (\V)}
= \pi_0 (\V)$,
which follows from
$\pi_0 (\V) = \socF \pi$.

Let $\pi' \defeq 
\ker(\pi/\pi_0 (\V) \twoheadrightarrow 
\V \otimes_{\F} \pi_f)
\cong  \ker(\pi \twoheadrightarrow 
\V \otimes_{\F} \pi_f)/\pi_0 (\V)
= \widetilde{\pi}_0 (\V)/\pi_0 (\V)$,
where the last equality follows 
from the previous paragraph.
By \Cref{cor:fl}, $\pi'$
has $\GF$-length bounded above by
$r \cdot (f -1 )$.
Then, \Cref{cor:tilde-involution}(ii)
applied to 
$(\pi_0, \pi_1, \pi_2,
\widetilde{\pi}_0, \widetilde{\pi}_1, \widetilde{\pi}_2) =
(\widetilde{\pi}_0 (\V), \pi_0 (\V), \pi',
\pi_0 (\V), \widetilde{\pi}_0 (\V), \pi')$
provides an isomorphism
$\E (\pi^{\prime \vee }) \cong 
\pi^{\prime \vee} \otimes (\mytwist)$
of $\Lambda$-modules with compatible actions
of $\GF$,
i.e.\ it shows that $\pi^{\prime \vee }$
is essentially self-dual.

To show that every Jordan-H\"older factor 
$\overline{\pi}$ of $\pi'$
is supersingular,
write $\overline{\pi} =
\overline{\pi}_0 / \overline{\pi}_1$
for two subrepresentations
$\overline{\pi}_1 \subsetneq \overline{\pi}_0 $
of $\pi$ such that
$\pi_0 (\V) \subseteq \overline{\pi}_1 $
and 
$\overline{\pi}_0  \subseteq \widetilde{\pi}_0 (\V)$.
By \Cref{thm:gen-by-inv}(ii),
there exists $\ell \in \{0, \dots, f\}$
such that $\VFil[\overline{\pi}_1 ,\ell] \subsetneq 
\VFil[\overline{\pi}_0 ,\ell]$.
and since $\pi_0 (\V) \subseteq \overline{\pi}_1 $
we have 
$\V = \VFil[\overline{\pi}_1 ,0] = 
\VFil[\overline{\pi}_0 ,0]$,
so necessarily $\ell >0$.
Since 
$\overline{\pi}_0  \subseteq \widetilde{\pi}_0 (\V)$
we have 
$0 = \VFil[\overline{\pi}_1 , f] = 
\VFil[\overline{\pi}_0 , f]$,
so necessarily $\ell < f$.
Pick a nonzero $w \in 
\VFil[\overline{\pi}_0 ,\ell] \setminus
\VFil[\overline{\pi}_1 ,\ell]$,
and let $0 < \ell_0  \le \ell < f$ 
be the smallest integer 
such that
$w \notin \VFil[\overline{\pi}_1 ,\ell_0 ]$.
By \Cref{thm:srep-structure},
there exists a $\GR$-equivariant injection
$\F w \otimes_{\F} \DZerogr[\ell_0 ]
\hookrightarrow 
 \overline{\pi}_0^{K_1 } 
/ \overline{\pi}_1^{K_1 }
\subseteq 
\overline{\pi}^{K_1}$.
By \eqref{eq:socle-DZerogr},
we obtain the inequality
$\lgR (\socR \overline{\pi}^{K_1}) \ge 
|\Wss_{\ell_0 }| =  \binom{f}{\ell_0}$,
which is at least two
since $0 < \ell_0 < f$.

Suppose for the sake of contradiction
that $\overline{\pi}$ is not supersingular.
Then, by the characterisation of 
smooth irreducible representations
given in \cite[Theorem~33]{BL94},
$\overline{\pi}$ is isomorphic to
a character, a principal series,
or a special series.
Consider
$\lg_{\GR} ( \socR \overline{\pi})$:
we know that it is $1$ in all cases except
for a principal series 
$\Ind_{B(K)} ^{\GF} (\theta_1 \otimes \theta_2)$
for two characters $\theta_1 , \theta_2 \colon 
K ^{\times} \to \F ^{\times}$
such that  $\theta_1 | _{\mathcal{O}_K ^{\times}} = 
\theta_2| _{\mathcal{O}_K ^{\times}}$,
in which case it is $2$
(cf.\ for example \cite[Corollary~2.38]{AWS25}).
In the latter case we would have
$\socR \overline{\pi}^{K_1} = \socR \IndB 1
\subseteq \Wss$,
which contradicts the 
$0$-genericity of $\rep$.
This concludes the proof.
\end{proof}

\section{Global results}
  \label{sec:global}
In this section we recall our global setup
(which coincides with the ``indefinite
case'' 
of \cite[§2.6]{BHHMS4}) in more detail,
and then prove
\Cref{thm:global-part-fl},
which states that the representation
$\pi$ of \eqref{eq:piShi-def}
has finite length.

We fix a totally real number field $F$, with
ring of integers $\mathcal{O}_F$,
and we let $S_p$ denote the set of places above $p$.
We assume that $F$ is unramified
at $S_p$.
For a place $w$, we denote by $F_w$
the completion of $F$ at $w$, and
by $\mathcal{O}_{F_w} $ its ring of integers.
We let 
$\Frob_w$
be a geometric Frobenius element at $w$.
We fix a quaternion algebra
$D$, with centre $F$, which is split at all places
above $p$ and at most one infinite place,
and we let $S_D$  denote the set of places 
where $D$ ramifies.
Consequently, if $\mathcal{O}_D$
is a fixed maximal order in $D$,
we can also fix isomorphisms
$(\mathcal{O}_D)_w \defeq
\mathcal{O}_D \otimes_{\mathcal{O}_F}\mathcal{O}_{F_w}
\xrightarrow{\sim} \MR[F_w]$
for $w \notin S_D$.

We fix a continuous representation
$\repgl \colon
\aGal[F] \to \GL_2(\F)$,
and we set 
$\repgl_w \defeq \repgl|_{\aGal[F_w]} $
for a finite place $w$ of $F$.
Let $S_{\repgl} $ be the set of places
where $\repgl$ ramifies. 
Then, we make the following assumptions 
on $\repgl$:
\begin{itemize}
\item the restriction
$\repgl|_{\Gal(\overline{F}/F(\mu_p))} $
is absolutely irreducible;
\item for all places $w \in  S_p$,
$\repgl_w$ is $0$-generic, 
in particular $S_p \subseteq S_{\repgl} $;
\item for all places
$w \in (S_D \cup S_{\repgl} ) \setminus 
 S_p$,
the universal framed deformation ring
of $\repgl_w$ is formally smooth over 
$W(\F)$.
\end{itemize}
If $D$ splits at exactly one infinite place
(the ``indefinite case''), 
we make the following choices.

We denote by $\mathbb{A}_F^{ \infty }$
the ring of finite adèles of $F$.
Given a compact open subgroup $V$ of 
$(D \otimes_{F}\mathbb{A}_F^{\infty})^{\times}$
(which should not be confused
with the $\F$ vector space $\V = \F^r$), 
then we can consider the associated smooth
projective Shimura curve $X_{V}$, 
which is defined over $F$, with the conventions of
\cite[§3.1]{BD14}.
We choose:
\begin{enumerate}[(i)]
\item
	\label{condition:place-ii}
 a finite set $S$ of finite places of $F$
such that: 
\begin{enumerate}
\item we have
$ S_D \cup S_{\repgl} \subseteq S$
\item for all $w \in S \setminus S_p$, 
the framed deformation ring 
$R_{\repgl_w ^\vee } $
of $\repgl_w ^\vee $ is formally smooth
over $W(\F)$;
\end{enumerate}
\item 
	\label{condition:place-iii}
compact open subgroups
$V_w$ of 
$(\mathcal{O}_D)_w ^{\times}$,
as $w$ runs over the set of places of $F$,
such that if we set
$V \defeq \prod_{w} V_w$,
then we have:
\begin{enumerate}
\item an equality
$V_w = (\mathcal{O}_D)_w^{\times}$
for $w \notin S $;
\item an inclusion
$V_w \subseteq 1+ p\MR[F_w]$
and
$V_w \unlhd  (\mathcal{O}_D)_w^{\times}$
for $w \in S_p$;
\item 
	\label{condition:place-iii-b}
the nonvanishing of
\begin{equation}
	\label{eq:pre-pi-nonvanishing}
 \Hom_{\aGal[F]}(\repgl, 
 H^{1}_{ \text{ét}}(X_{V} \times_{F} \overline{F}, \F) ) \neq 0.
\end{equation}
\end{enumerate}
\end{enumerate}

If $D$ splits at no infinite places 
(the ``definite'' case)
we make the same choices as 
\ref{condition:place-ii} and
\ref{condition:place-iii} above,
replacing \eqref{eq:pre-pi-nonvanishing}
by the condition $S(V,\F)[\mgl] \neq 0$,
where
$S(V, \F) \defeq 
\{f \colon D^{\times} \backslash\!
(D \otimes_{F}\mathbb{A}_F^{\infty })^{\times}
/V \to \F\}$,
and where $\mgl$ is generated by 
$T_w-S_w \tr(\repgl(\Frob_w))$,
$\Norm(w)-S_w \det(\repgl(\Frob_w))$,
for every $w \notin S$ such that
$V_w = (\mathcal{O}_D)_w^{\times}$,
with $T_w$, $S_w$ acting on $S(V,\F)$
via right translation on functions by
$V \left( \begin{smallmatrix}
\varpi_w & 0 \\
0 & 1 \\
\end{smallmatrix} \right)V$,
$V \left( \begin{smallmatrix}
\varpi_w & 0 \\
0 & \varpi_w \\
\end{smallmatrix} \right)V$
respectively (where $\varpi_w$
is any choosen uniformiser of $F_w$).

For convenience, set
\begin{equation}
	\label{eq:small-global-pi-def}
\pi(V) \defeq \begin{cases}
\Hom_{\Gal(\overline{F}/F)}\left(\overline{r}, H^1_{ \text{ét}} (X_{V} \times_{F}\overline{F}, \F)\right)
 & \text{in the indefinite case,} \\
S(V, \F)[\mgl]
 & \text{in the definite case,} \\
\end{cases}
\end{equation}
so that \ref{condition:place-iii}(c)
becomes $\pi(V) \neq 0$.

\begin{remark}
	\label{rmk:gl-conditions-for-free}
Up to enlarging $S$, we can always assume that
\ref{condition:place-iii}(a) is satisfied
(this, of course, makes 
\ref{condition:place-ii}(b) more restrictive).
Recall that,
by \cite[Remark~2.6.1]{BHHMS4},
we need not assume the existence 
of a finite place $w_1$
satisfying conditions
(ii)(a), (ii)(b), (ii)(c)
of \cite[\S 8.1]{BHHMS1}.
\end{remark}

Now, fix a finite place $v \in  S_p$,
and let $K \defeq F_v$,
$\rep \defeq  \repgl^\vee$,
which we assume to be non-semisimple.
If $V^{v} \defeq \prod _{w \neq v} V_w$,
then we want to study
the admissible smooth representation
of $\GF$ over $\F$
\begin{equation}
	\label{eq:global-pi-def}
\pi(V^{v}) \defeq \begin{cases}
\varinjlim_{V_v} \pi(V^{v}V_v)
 & \text{in the indefinite case,} \\
\varinjlim_{V_v} \pi(V^{v}V_v)
 & \text{in the definite case,} \\
\end{cases}
\end{equation}
where the limit is taken over all
compact open subgroups 
$V_v \unlhd (\mathcal{O}_D)_v^{\times}
\cong \GR$
contained inside
$1+ p \MR$.
It follows from \eqref{eq:pre-pi-nonvanishing}
that \eqref{eq:global-pi-def} is nonzero.

If we let
$\psi \colon \aGal[F] \to W(\F)^{\times}$
be the Teichm\"uller lift
of $\det (\repgl) \omega$,
and denote by $\psi_w$ its restriction to
$\aGal(F_w)$,
and by $\overline{\psi}_w$ 
the reduction modulo $p$ of $\psi_w$,
then notice that \eqref{eq:global-pi-def}
has central character 
$\overline{\psi}_v^{-1} = \mytwist$.

\subsection{The diagram associated to \texorpdfstring{$\pi$}{pi}}
  \label{sec:global-diagram}
The main result of this section 
is \Cref{thm:diag} below,
where we compute the diagram
$ (\pi^{I_1} \hookrightarrow \pi^{K_1 })$
associated to $\pi$,
using global methods.
We start by recalling patching functors 
in the setting of \cite[\S 8]{BHHMS1},
except here $\rep$ is assumed to be non-semisimple.
Throughout this section
we assume that $\rep$ is $12$-generic.

Set $\mathcal{O} \defeq W (\F)$,
$E \defeq \mathcal{O} [1/p]$,
and fix a tame inertial type $\tau_w$ of $F_w$
and a $\GR[F_w]$-invariant lattice
$\sigma^{0} (\tau_w^\vee )$
in $\sigma (\tau_w^\vee ) = \sigma (\tau_w)^\vee$
for each $w \in S_p \setminus \{v\}$.
As in \cite[\S 6]{EGS} we can define a 
\emph{patching functor} $\Moo$ 
from the category of continuous representations
$\sigma_v$ of $\GR[F_v]$
on finite type $\mathcal{O}$-modules 
with central character 
$\psi|_{I_{F_v}}^{-1} \circ
\Art_{F_v}|_{\mathcal{O}_{F_v} ^{\times}}$
(where $\Art_{F_v}$ is the Artin map,
normalised so that uniformisers
are sent to geometric Frobenius elements)
to the category of finite type $\Roo$-modules,
where $\Roo \cong 
R^{\psi_v}_{\repgl_v} [\![X_1 , \dots, X_g]\!] $
for some positive integer $g$,
and where 
$R^{\psi_v}_{\repgl_v}$
denotes the framed deformation ring of $\repgl_v$
with fixed determinant $\varepsilon^{-1} \psi_v$.
We denote by $\moo$ the maximal ideal of $\Roo$.
Moreover, for a tame inertial type $\tau$
we define $\Roo^{(1, 0), \tau} \defeq 
\Roo \otimes_{R_{\repgl_v^\vee }}
R_{\repgl_v^\vee }^{(1, 0), \tau}$,
where
$R_{\repgl_v^\vee }^{(1, 0), \tau}$
parametrises potentially Barsotti-Tate
deformations of $\repgl_v^\vee $.
Let $K \defeq F_v$ and 
recall from \Cref{sec:preliminaries}
that $\Gamma =  \Gres$.
Let $\mK$ be the  maximal ideal of 
$\F [\![K_1/Z_1]\!]$,
and set 
$\Gt \defeq \F [\![\GR/Z_1]\!]/\mK^{2}$.
For a Serre weight $\sigma$,
we let $P_\sigma \defeq \ProjG \sigma$
be the projective envelope of $\sigma$
in the category of $\F [\Gamma]$-modules
and $\widetilde{P}_\sigma$
be the projective $\mathcal{O} [\Gamma]$-module
lifting $P_\sigma$.
The following results on freeness
for patching functors,
which are due to Yitong Wang,
will be crucial in the proof of 
\Cref{thm:diag}.
\begin{proposition}[{\cite[Propositions~6.1]{Wang23}}]
  \label{prop:Wang-freeness}
There exists an integer $r \ge 1$ such that 
\begin{enumerate}[(i)]
\item 
for any $\sigma \in \Wr$
the module $M_\infty (\sigma)$
is free of rank $r$ 
over its scheme-theoretic support,
which is formally smooth over $\F$ 
(hence a domain);

\item 
for any tame inertial type $\tau$
such that $\JH (\overline{\sigma (\tau)}) \cap \Wr 
\neq \emptyset $
and any $\GR$-invariant
$W(\F)$-lattice $\sigma^0 (\tau)$
in $\sigma (\tau)$ 
with irreducible cosocle,
the module $M_\infty (\sigma^0 (\tau))$
is free of rank $r$ 
over its scheme-theoretic support
$R_\infty^{ (1,0), \tau}$,
which is a domain;

\item 
for any $\sigma \in \Wr$ the modules 
$M_\infty (\widetilde{P}_\sigma)$ and
$M_\infty (P_\sigma)$ 
are free of rank $r$ 
over their respective scheme-theoretic support.
\end{enumerate}
\end{proposition}
To be precise, 
\cite[Proposition~6.1(iii)]{Wang23}
only treats the case $M_\infty (\widetilde{P}_\sigma)$
in \Cref{prop:Wang-freeness}(iii).
For the case $M_\infty (P_\sigma)$ 
we remark that
$\widetilde{P}_\sigma / p \widetilde{P}_\sigma 
= P_\sigma$
by construction,
so by exactness of $\Moo(-)$ 
(cf.\ \cite[Definition~6.1.3]{EGS})
we have
$\Moo(P_\sigma) = 
\Moo(\widetilde{P}_\sigma )/
p \Moo (\widetilde{P}_\sigma)$.

Let $\pi$ be the $\GF$-representation 
of \eqref{def:global-small-pi-def} below,
and assume that it satisfies 
one of the equivalent conditions of
\eqref{eq:small-pi-and-modular-weights} below.
Another crucial result of Yitong Wang
is the following theorem on the $K_1 $-invariants
of $\pi$.
\begin{theorem}[{Consequence of 
\cite[Theorem~6.3(ii)]{Wang23}}]
There is an isomorphism 
$\pi^{K_1} \cong \V \otimes_{\F}\DZero$
of $\GR$-representations.
\end{theorem}
For the rest of the section,
we fix once and for all
such an isomorphism
$\ip \colon  \pi^{K_1} \xrightarrow{\sim} 
\V \otimes_{\F}\DZero$.
Recall that the diagram $\Diagr$
is only well defined
up to the choice of certain parameters,
namely the scalars $\nu_i$
defined in the second paragraph of
\cite[§6]{Bre11}.
The following theorem 
states that there is a unique choice of $\nu_i$
that makes
\eqref{eq:iota-pi-diagram-global} 
below possible,
generalising
\cite[Theorem~1.3]{DL21} 
to the case $r \ge 1$.
\begin{theorem}
  \label{thm:diag}
Keep the above notation
and recall that $\V \defeq \F^r$,
where $r$ is the integer 
of \Cref{prop:Wang-freeness}.
Then, 
there exists a diagram
$\Diagr = (\DOne \hookrightarrow \DZero)$
as in \Cref{eq:diag-def},
only depending on $\rep$,
such that the $\GR$-equivariant isomorphism $\ip$ 
promotes to an isomorphism of diagram
\begin{equation}
  \label{eq:iota-pi-diagram-global}
\ip \colon 
(\pi^{I_1 } \hookrightarrow \pi^{K_1}) 
\xrightarrow{\sim}
\V \otimes_{\F} \Diagr
\end{equation}
\end{theorem}
\begin{remark}
  \label{rmk:ss-case-known}
The diagram $\Diagr$ of \Cref{thm:diag}
coincides with the diagram 
$\mathcal{D}(\pi _{\mathrm{glob}}(\rep))$
of \cite[Theorem~1.3]{DL21}.
\end{remark}

For $\rep$ semisimple
\Cref{thm:diag} is proven in
\cite[Theorem~3.93]{BHHMS2},
we broadly follow their line of argument.

Following \cite[Definition~4.1]{DL21},
we define 
$R \colon \pi^{I_1 } \to (\socR \pi)^{I_1 }$ 
as the unique $\F$-linear map
such that 
its restriction $R [\chi]$ to 
the $\chi$-eigenspace
$\pi^{I_1 } [\chi]$
is nonzero whenever 
$\pi^{I_1 } [\chi]$ is nonzero,
and moreover given by some $S_{i(\chi)}$
for the (uniquely defined) integer
$0 \le i(\chi) < p^f -1$
such that $S_{i(\chi)} \pi^{I_1 } [\chi] \neq 0$ 
when $\chi$ does not occur in 
$(\socR \pi)^{I_1}$
(cf.\ \cite[\S 2.1]{DL21} for the definition
of $S_{i(\chi)}$),
and by $R[\chi] = \id$
when $\chi$ occurs in 
$(\socR \pi)^{I_1}$.
We can similarly define
$R \colon \DOne \to \socR (\socR\DZero)^{I_1 }$,
and this definition
makes the following diagram commute
\[
\begin{tikzcd}
\pi^{I_1 }  
\ar[r, "\ip", "\sim"']\ar[d, "R"'] & 
\V \otimes_{\F} \DOne
\ar[d, "\id_{\V} \otimes R"] \\
(\socR \pi)^{I_1 }
\ar[r, "\ip"', "\sim"] & 
\V \otimes_{\F} (\socR\DZero)^{I_1 }.
\end{tikzcd}
\]
Since $\DOne$ is multiplicity free, 
for $\chi \colon I \to \F ^{\times}$
a character occurring in $\DOne$, 
we denote by $R \chi$
the character of $I$ on
$R (\DOne [\chi])$.
Moreover, if $\chi$ occurs in 
$(\socR\DZero)^{I_1}$,
then we let $\sigma (\chi) \subseteq \socR \DZero$
be the unique Serre weight such that 
$I$ acts by $\chi$ on $\sigma(\chi)^{I_1 }$.
We define as in \cite[Proposition~4.14]{DL21}
an isomorphism 
\[
\overline{h}_\chi \colon
\Moo (\sigma (R \chi^s))/\moo
\xrightarrow{\sim}
\Moo (\sigma (R \chi))/\moo,
\]
which is moreover dual to $g_\chi$
in \emph{loc.\ cit.}
The proof of \emph{loc.\ cit.}\ carries through
in our case, we only need to replace
``one-dimensional by Theorem~4.6''
by 
``of the same dimension by \Cref{prop:Wang-freeness}''.
Before proving
\Cref{thm:diag}
we will need the following proposition,
which generalises \cite[Proposition~3.99]{BHHMS2}
for $\rep$ non-semisimple.
\begin{proposition}
  \label{prop:full-circle}
Let $k \ge 1$ and let 
$\chi_0, \dots, \chi_{k-1}$
be arbitrary characters of $I$ 
which occur in $\pi^{I_1 }$
(or equivalently on $\DOne$)
such that $R \chi_i^s = R \chi_{i+1}$
for $i \in \{0, \dots, k-2\}$,
and $R \chi_{k-1}^s = R \chi_{0}$.
Then, the isomorphism
\[
\overline{h}_{\chi_1 } \circ 
\overline{h}_{\chi_2 } \circ \dots
\overline{h}_{\chi_{k-1} } \circ 
\overline{h}_{\chi_0 }
\colon 
M_\infty (\sigma (R \chi_0^s))/\moo
\xrightarrow{\sim}
M_\infty (\sigma (R \chi_0^s))/\moo
\]
is the multiplication by a scalar in $\F ^{\times}$
which depends neither on $r$ nor on $M_\infty $.
In particular, 
this scalar is the same as in 
\cite[Eq.~(34)]{DL21}.
\end{proposition}
\begin{proof}
We adapt the proof of
\cite[Proposition~3.99]{BHHMS2}
to the case where $\rep$ is not assumed to be semisimple.
The argument there carries through \emph{verbatim}, 
except \cite[Theorem~3.94(iii)]{BHHMS2}
should always be replaced by
\Cref{prop:Wang-freeness}(ii).
We should also point out that 
\cite[Lemma~3.97]{BHHMS1} and
\cite[Proposition~3.98]{BHHMS1}
(which are used in the proof of 
\cite[Proposition~3.99]{BHHMS2})
also hold for
$\rep$ not necessarily semisimple.

Indeed, the proof of
\cite[Lemma~3.97]{BHHMS1}
carries through \emph{verbatim},
except when the proof of 
\cite[Proposition~8.2.3]{BHHMS1}
(which assumed $\rep$ semisimple)
is invoked.
However, the assumption $\rep$ semisimple 
is not needed in the proof of
\cite[Proposition~8.2.3]{BHHMS1},
cf.\ for example the proof of
\cite[Proposition~6.1(ii)]{Wang23}.
As for \cite[Lemma~3.98]{BHHMS1},
the assumption $\rep$ semisimple 
is never used in the proof.
\end{proof}

We can now prove \Cref{thm:diag}.
\begin{proof}[Proof of {\Cref{thm:diag}}]
The proof of 
\cite[Theorem~3.93]{BHHMS2}
goes through \emph{verbatim},
except in the second-to-last paragraph,
where 
``Proposition~3.99''
is replaced by 
``\Cref{prop:full-circle}'',
and where
``by (the analogue of) [\textbf{DL21},\ Prop.4.14]''
is replaced by ``by the paragraph before
\Cref{prop:full-circle}''.
\end{proof}

\subsection{Finite length}
  \label{sec:global-fl}
We conclude with the main result of this section,
and of the article.
\begin{theorem}
	\label{thm:global-part-fl}
Suppose that $\rep$ is
$\max \{12,2f+1\}$-generic.
The representation $\pi(V^{v})$ of
\eqref{eq:global-pi-def}
has finite length as a 
$\GF$-representation. 
\end{theorem}
Notice that for a smaller
 $V^{\prime v} \subseteq V^{v}$
we have 
$\pi(V^v) \subseteq \pi(V^{\prime v})$.
In particular, if $\pi(V^{\prime v})$ has finite length,
then so does $\pi(V^v)$,
and so we can always assume that
\ref{condition:place-iii}(b) 
is satisfied.

To prove this theorem, 
we consider a larger auxiliary 
prime-to-$v$ level
$U^{v} \defeq  \prod _{w \neq v} U_w$, where 
\begin{equation}
	\label{eq:U-lvl-def}
\begin{cases}
U_w = V_w & \text{for } w \notin S_p, \\
U_w = (\mathcal{O}_D)_w^{\times}
\cong \GR[F_w]
 &  \text{for } w \in S_p \setminus \{v\}.
\end{cases}
\end{equation}
By \ref{condition:place-iii}(b),
$U^{v}$ normalises $V^{v}$.
We also consider
an auxiliary
$\GF$-representation:
for a choice
of a Serre weight $\sigma_w$ of $\GR[F_w])$
for each $w \in S_p \setminus \{v\} $,
set 
$\sigma_p^{v} \defeq 
\bigotimes_{w \in S_p \setminus \{v\} } \sigma_w $,
and define $\pi(V^{v},\sigma_p^{v})$ to be
\begin{equation}
	\label{def:global-small-pi-def}
\pi(V^{v},\sigma_p^{v}) \defeq 
\varinjlim _{V_v} \Hom_{U^{v}/V^{v}}
\left(\sigma_p^{v}, 
\pi(V^{v}V_v)
\right).
\end{equation}
Then, we have the following result,
 coming from \cite[§5.5]{GK14}:
\begin{equation}
	\label{eq:small-pi-and-modular-weights}
\pi(V^{v},\sigma_p^{v}) \neq 0
\iff \sigma_w \in W(\overline{r}_w ^\vee )
\ 
\forall w \in S_p \setminus \{v\} ,
\end{equation}
where $W(\overline{r}_w ^\vee )$
is defined as in 
\cite[§3]{BDJ10}.

Assume that \eqref{def:global-small-pi-def}
satisfies \eqref{eq:small-pi-and-modular-weights}.
Here is a compilation of results 
concerning the representation 
\eqref{def:global-small-pi-def}.
\begin{theorem}
	\label{thm:compilation}
For $w \in  S_p \setminus \{v\}$,
we let $\sigma_w$ be a Serre weight in
$W(\repgl_w^\vee)$,
and set
$\sigma_p^{v} \defeq 
\bigotimes_{w \in S_p \setminus \{v\} } \sigma_w $.
Assume that $\rep$ is $12$-generic.
\begin{enumerate}[(i)]
\item 
(\Cref{thm:diag})
the hypothesis \ref{hypothesis:i}
of \Cref{sec:hypotheses}
holds for
$\pi(V^{v}, \sigma_p^{v})$;
	\item 
(\cite[Proposition~6.4.6]{BHHMS1} and \cite[Theorem~6.3(ii)]{Wang23})
the hypothesis \ref{hypothesis:ii}
of \Cref{sec:hypotheses}
 holds for
$\pi(V^{v}, \sigma_p^{v})$;
	\item
	(\cite[Theorem~8.2]{HW22} 
and \cite[Theorem~6.3(i)]{Wang23})
the hypothesis \ref{hypothesis:iii}
of \Cref{sec:hypotheses}
holds for
$\pi(V^{v}, \sigma_p^{v})$;
\item 
(\cite[Proposition~2.6.3]{BHHMS4})
the hypothesis \ref{hypothesis:iv}
of \Cref{sec:hypotheses}
 holds for
$\pi(V^{v}, \sigma_p^{v})$.
\end{enumerate}
\end{theorem}

We return to \Cref{thm:global-part-fl}.
\begin{proof}[Proof of 
\Cref{thm:global-part-fl}]
Notice that 
$ H^{1}_{\text{ét}}(X_{V^{v}V_v} \times_{F} \overline{F},\F )$
is a $U^{v}/V^{v}$-representation,
in particular it is $V^{v}$-invariant,
and we can rewrite $\pi(V^{v})$ as
\begin{align}
	\label{eq:pi-rewrite-dumb}
\pi(V^{v}) \cong &
\varinjlim _{V_v} \Hom_{V^{v}}\left(1, 
\pi(V^{v}V_v)
\right)\\
	\label{eq:pi-rewrite-Frobenius}
\cong& 
\varinjlim _{V_v} \Hom_{U^{v}/V^{v}}\left(
 \bigotimes _{w \in S_p \setminus \{v\} } 
\Ind_{V_w} ^{U_w}1, 
\pi(V^{v}V_v)
\right).
\end{align}
In \eqref{eq:pi-rewrite-dumb},
$1$ denotes the trivial representation
of $V^{v}$,
and in \eqref{eq:pi-rewrite-Frobenius}
we have used Frobenius reciprocity
from $V^{v}$ to $U^{v}$.
Since the two levels only differ at 
places in
$S_p \setminus \{v\}$,
$\Ind_{V^v} ^{U^v}1$
coincides with 
$
\Ind_{\prod_{w \in S_p \setminus \{v\}} V_w} 
^{\prod_{w \in S_p \setminus \{v\}} U_w}1 
\cong 
\bigotimes _{w \in S_p \setminus \{v\} } 
\Ind_{V_w} ^{U_w}1
\overset{\eqref{eq:U-lvl-def}}{\cong} 
\bigotimes _{w \in S_p \setminus \{v\} } 
\Ind_{V_w} ^{\GR[F_w]}1 $.

Recall that 
the centre of $\GR[F_w]$
acts by $\overline{\psi}^{-1} |_{I(\overline{F}_w/F_w)} $
on \eqref{eq:global-pi-def}.
If
$(\Ind_{V_w} ^{\GR[F_w]}1)_Z $
denotes the maximal quotient of 
$\Ind_{V_w} ^{\GR[F_w]}1$
on which 
the centre of $\GR[F_w]$
acts by the character $\overline{\psi}^{-1} |_{I(\overline{F}_w/F_w)} $,
then
\begin{equation}
	\label{eq:pi-rewrite}
\pi(V^{v}) \cong
\varinjlim _{V_v} 
\Hom_{U^{v}/V^{v}}\!\!\left(
\bigotimes _{w \in S_p \setminus \{v\} } 
(\Ind_{V_w} ^{\GR[F_w]}1)_Z, 
\pi(V^{v}V_v)
\right)
.
\end{equation}
If we set
$\JH_{w}
\defeq \JH((\Ind_{V_w} ^{\GR[F_w]}1)_Z )$,
then a dévissage 
on the Jordan-H\"older constituents of
$ \bigotimes _{w \in S_p \setminus \{v\} } 
(\Ind_{V_w} ^{\GR[F_w]}1)_Z$
(which are all of the form
$\bigotimes_{S_p \setminus \{v\} }  \sigma_w$,
for $\sigma_w \in \JH_{w}$)
gives
\begin{equation}
	\label{eq:length-devissage}
\lg_{\GF}  \left(\pi(V^{v})\right)
\le 
\sum _{w \in S_p \setminus \{v\} } 
\sum _{\sigma_w \in \JH_w} 
\lg_{\GF}  \left(
\pi\left(V^{v}, \textstyle\bigotimes_{w \in S_p \setminus \{v\} }  \sigma_w\right)
\right).
\end{equation}
Finally, notice that
each term
$ \pi\left(V^{v}, \bigotimes_{w \in S_p \setminus \{v\} }  \sigma_w\right) $
is either zero
(by \eqref{eq:small-pi-and-modular-weights})
or satisfies hypotheses
\ref{hypothesis:i}
to 
\ref{hypothesis:iv}
 of \Cref{sec:hypotheses}
(by \Cref{thm:compilation}),
 and hence has finite length 
by \Cref{cor:fl},
using that $\rep$ is
$\max \{12,2f+1\}$-generic.
Since the sum in the right-hand side
of \eqref{eq:length-devissage}
is finite, we conclude
that $\pi(V^{v})$ has finite length.
\end{proof}

\printbibliography
\end{document}